%% file: arxiv_version.tex
\documentclass[11pt,a4paper]{article}

\usepackage[margin=1in]{geometry}

\usepackage{amsmath}
\usepackage{lmodern}
\usepackage{amssymb}
\usepackage{amsthm}
\usepackage{amsfonts}
\usepackage{algorithm}
\usepackage{algpseudocode}

\usepackage{graphicx}
\usepackage{booktabs} 
\usepackage{url}
\usepackage{multirow}
\usepackage{enumitem}

\usepackage{xcolor}
\usepackage[colorlinks=true, linkcolor=blue, citecolor=blue, urlcolor=blue]{hyperref}

\newtheorem{theorem}{Theorem}
\newtheorem{proposition}[theorem]{Proposition}%
\newtheorem{remark}{Remark}%
\newtheorem{definition}{Definition}
\newtheorem{corollary}{Corollary}
\usepackage[singlespacing]{setspace}
\usepackage[fontsize=11pt]{fontsize}
\usepackage{mathtools, bbm}
\usepackage{booktabs}
\usepackage{graphicx}

\newcommand{\E}{\mathbb{E}}

\newcommand{\indep}{\perp\!\!\!\perp} 
\newcommand{\nindep}{\not\!\perp\!\!\!\perp}

\newtheorem{assumption}{Assumption}
\newtheorem{lemma}{Lemma}
\newtheorem{condition}{Condition}
\algnewcommand{\RComment}[1]{\hfill \(\triangleright\) #1}

\usepackage{xcolor}

\newcommand{\coarse}[1]{\mathcal{C}_{#1}} 
\newcommand{\fine}[1]{\mathcal{F}_{#1}}   
\newcommand{\ccnt}[1]{c_{#1}}             
\newcommand{\fcnt}[1]{f_{#1}}             
\providecommand{\keywords}[1]{\vspace{0.3cm}\noindent\textbf{\textit{Keywords:}} #1}

\usepackage{authblk}

\usepackage{natbib}

\title{\vspace{-1.5cm} MixCIT: A Kernel Based Local-Polynomial Debiased Test for Conditional Independence on Mixed-Type Data}

\author[1]{Mengxiao Gao\thanks{mg2532@cornell.edu}}
\author[1]{Kyra Gan\thanks{kyragan@cornell.edu}}
\author[2]{Promit Ghosal\thanks{promit@uchicago.edu}}

\affil[1]{School of Operation Research and Information Engineering, Cornell University}
\affil[2]{Department of Statistics, University of Chicago}

\date{\today}

\begin{document}

\maketitle

\begin{abstract}
    Conditional independence testing (CIT) is fundamental to modern
statistical inference in areas related to causal discovery and
variable selection. While marginal independence is relatively
well-understood, despite multiple advances, no existing
non-parametric CIT provides a unified, efficient, and statistically
guaranteed solution across heterogeneous data.
 We introduce a graph-based test statistic comparing kernel
similarities of the response within composite neighborhoods
that use exact matching on discrete components and
$k_n$-nearest-neighbor matching on continuous ones. The raw
statistic, related to prior constructions
\citep{huang2022kernel}, suffices under fully discrete
conditioning. However, when at least one conditioning variable is
continuous, we instead use a local-polynomial debiased variant
that cancels the local smoothing bias. We rigorously establish its asymptotic null distribution across all data-type combinations. We further prove a dimension-free $n^{-1/4}$ detection threshold
under local alternatives, eliminating the phase transition that
affects geometric estimators in high dimensions.
 Finally, we develop efficient algorithms with near-quadratic
complexity and analytic graph-based calibration, bypassing the
cubic bottlenecks of global kernel methods. \\
\keywords{Conditional independence, Graph-based estimator, Asymptotic distribution, Power analysis}
\end{abstract}


\maketitle

\section{Introduction}\label{introduction}

\input{introduction}

\section{A Unified Test Statistic for Data of Potentially Mixed Types}
\label{setup}
\input{methods}

\section{Theoretical Asymptotic Properties}
\label{theory}

\input{asymptotic_distribution}

\section{Implementation and Computational Efficiency}
\label{algorithm_implementation}
\input{algorithm}
\section{Numerical Studies}
\label{sec:experiments}
\input{simulation}

\section{Discussion}
\label{discussion}
\input{discussion}

\section*{Acknowledgments}
This work was supported in part by Amazon Web Service.

\input{supplementary-material-arxiv}

\bibliographystyle{plainnat}
\bibliography{reference}

\end{document}

%% file: introduction.tex
Let $(X, Y, \mathbf{Z})$ be random variables taking values in $\mathcal{X}, \mathcal{Y}, \mathcal{Z}$, with joint distribution $P$. Conditional independence testing (CIT) which decides whether $X \perp\!\!\!\perp Y \mid \mathbf{Z}$ based on an i.i.d.\ sample $\{(X_i, Y_i, \mathbf{Z}_i)\}_{i=1}^n$ is fundamental to modern statistical inference. It underpins constraint-based causal discovery algorithms \citep{spirtes2000causation, pearl2000causality, chickering2002optimal, tsamardinos2006max, colombo2014order}, which execute a polynomial (or exponential in worst case) number of CITs to infer directed acyclic graphs; additive noise model-based causal inference \citep{hoyer2009nonlinear, peters2014causal, zhang2012identifiability}, where CIT validates the absence of confounding; local causal discovery methods \citep{pena2007learning, aliferis2010local} that scale to high-dimensional settings by restricting search neighborhoods; variable selection \citep{george2000variable, fan2001variable, barut2016conditional}, where predictors are retained only if they remain relevant after conditioning on others; and sufficient dimension reduction \citep{cook2002dimension, li2007directional, fukumizu2007kernel}, where the goal is to find the lowest-dimensional subspace of $\mathbf{Z}$ capturing all information about $Y$. In all these applications, reliable CIT directly determines the quality of downstream inference.

\subsection{Background and Existing Approaches}

While marginal independence ($\mathbf{Z} = \emptyset$) is relatively well understood with theoretically grounded nonparametric solutions for continuous variables (distance correlation \citep{szekely2007measuring, szekely2009brownian}, Hilbert-Schmidt Independence Criterion (HSIC) \citep{gretton2005measuring, gretton2008kernel}, randomized dependence coefficient \citep{lopez2013randomized}), discrete variables (Pearson's $\chi^2$ test \citep{pearson1900x}, exact tests \citep{agresti2012categorical}), and recent graph-based and kernel/distance-based
independence tests that can be adapted to mixed data
\citep{heller2013consistent, sejdinovic2013equivalence,
deb2020measuring}, the conditional variant introduces
substantial statistical and computational complexities. Adjusting for a potentially high-dimensional, heterogeneous conditioning set $\mathbf{Z}$ necessitates flexible nonparametric frameworks that avoid restrictive parametric assumptions while simultaneously requiring computational scalability for applications such as causal discovery, where tests may be executed multiple times. Existing methodological development has progressed along data-type-specific regimes which we describe in what follows.

\textbf{Discrete data.} For purely discrete $(X, Y, \mathbf{Z})$, exact stratification approaches provide principled solutions. The Cochran–Mantel–Haenszel test \citep{cochran1954some, mantel1959statistical} and log-linear models \citep{agresti2012categorical, bishop2007discrete} leverage combinatorial structure to assess conditional independence without smoothing. However, these methods suffer from the curse of dimensionality: with $|\mathbf{Z}| = d$ discrete variables each taking $L$ values, sample sizes must scale exponentially ($n \gg L^d$) to ensure adequate cell counts, rendering them impractical when $d$ is large.

\textbf{Continuous data.} In purely continuous settings, diverse nonparametric approaches have emerged. \textit{kernel methods} adopt a global RKHS perspective, with the Kernel Conditional Independence (KCI) test \citep{zhang2011kernel} and its approximations \citep{strobl2019approximate} estimating conditional mean embeddings via kernel ridge regression to test whether the conditional cross-covariance operator $\Sigma_{XY|\mathbf{Z}}$ vanishes, while the Generalized Covariance Measure (GCM) \citep{shah2020hardness} uses similar ideas for instrumental variables; these methods assume smooth variation across the support of $\mathbf{Z}$ and incur $\mathcal{O}(n^3)$ cost, with rigorous RKHS foundations \citep{fukumizu2007kernel, fukumizu2008kernel, park2021measure} but persistent computational bottlenecks. \textit{Information-theoretic methods} use nearest-neighbor (NN) estimators for conditional mutual information $I(X; Y | \mathbf{Z})$ \citep{runge2018conditional, frenzel2007partial, vejmelka2008inferring, gao2015efficient, runge2019detecting, niles2021estimation}, achieving $\mathcal{O}(n \log n)$ efficiency but requiring unified metrics on $(\mathbf{Z}, X)$ that become ambiguous for mixed data. \textit{Conditional distance covariance} extends distance correlation \citep{wang2015conditional, huang2016conditional}, with the Projected Covariance Measure (PCM) \citep{lundborg2024projected} projecting onto the orthogonal complement of $\mathbf{Z}$-variation, achieving $\sqrt{n}$-consistency but requiring careful bandwidth selection and struggling with high-dimensional $\mathbf{Z}$. \textit{Rank-based methods} build on \citet{chatterjee2021new}'s marginal coefficient with finite-sample guarantees \citep{shi2022rank, deb2020measuring}, with recent extensions approximating conditional dependence via NN-defined $\mathbf{Z}$-neighborhoods \citep{azadkia2021simple,azadkia2022fastnonparametricapproachlocal} and theoretical extensions to multivariate $Y$ establishing CLTs under independence \citep{huang2026multivariateextensionazadkiachatterjeesrank}.

\textbf{Mixed discrete-continuous data.} When $(X, Y, \mathbf{Z})$ comprises mixed components, existing approaches face fundamental obstacles. \textit{kernel methods} require product kernels on mixed spaces, which remains theoretically delicate \citep{szabo2018characteristic}. \textit{Discretization approaches} force continuous components into bins \citep{huang2010testing, zhang2012constraint}, inducing bias, information loss, and curse of dimensionality as bin numbers grow exponentially with dimension. \textit{Model-based methods} such as mixed graphical models \citep{yang2015mixed, lee2016learning, chen2015structure, fellinghauer2013stable} assume parametric conditional distributions (e.g., conditional Gaussians for continuous nodes, conditional exponential families for discrete nodes), limiting robustness to model misspecification.

Recent work on nonparametric mixed-data independence testing
spans three lines: kernel-based tests with product kernels
extending the Hilbert-Schmidt Independence Criterion of
\citet{gretton2005measuring}; copula-based procedures
\citep{genest2007copula} requiring semiparametric assumptions
on the marginals; and permutation-based classifier tests
\citep{sen2017model, kim2022classification} that lack explicit
asymptotic theory. Each targets marginal rather than
conditional independence. Critically, \textit{no existing nonparametric CIT framework provides} (i) geometrically principled neighborhoods for mixed $(X, Y, \mathbf{Z})$, (ii) complete asymptotic null distributions across all data-type regimes, (iii) power analysis characterizing detection thresholds, and (iv) near-quadratic computational complexity with direct $p$-value calibration.

\subsection{Our Approach and Contributions}

We address this gap through a unified geometric framework that
reframes CIT as a comparison of \emph{local alignments}. Under
$H_0: Y \perp\!\!\!\perp X \mid \mathbf{Z}$, the conditional
distribution of $Y$ given $\mathbf{Z}$ remains unchanged when further
conditioned on $X$. Consequently, the local variation of $Y$ within
neighborhoods defined by $\mathbf{Z}$ alone should statistically align
with its variation within neighborhoods defined jointly by
$(\mathbf{Z}, X)$. A systematic discrepancy signals conditional
dependence. This intuition naturally accommodates mixed data if
neighborhoods are constructed appropriately.

We incorporate this idea using the notion of \textit{composite
neighborhoods}. For discrete components, we enforce exact matching;
for continuous components, we use $k_n$-nearest neighbors in metric
spaces. This yields neighborhoods that are geometrically faithful
without artificial embeddings or arbitrary binning, and adapts
automatically to local data support. From these neighborhoods, we
form a raw difference $\Delta_n$ comparing average kernel similarities
of $Y$ within the two neighborhood types. When both $\mathbf{Z}$ and
$X$ are discrete, exact matching introduces no smoothing bias and
$\Delta_n$ itself is a degenerate $U$-statistic with a
mixture-of-chi-squared limit. When at least one of $\mathbf{Z}, X$ has
a continuous component, however, $\Delta_n$ acquires a non-vanishing
smoothing bias of order $(k_n/n)^{2/D}$ which divides $\sqrt n$ and
obstructs pivotal inference. In these continuous-or-mixed regimes we
work with a \textit{local-polynomial debiased} statistic $\Delta_n^{(p)}$
that replaces the inner $k_n$-NN averages by intercepts of a degree-$p$
weighted least-squares fit of $K(Y_i, \cdot)$ on the neighborhood; the
local-polynomial reproduction identities cancel the smoothing bias up
to order $p+1$, restoring the $\sqrt n$-pivotal property.

\textbf{Contribution 1: Regime-adapted test statistics.} We propose a
two-statistic construction tied to the data type of
$(\mathbf{Z}, X)$. For all-discrete conditioning, we use the raw
$\Delta_n$ from Equation~\eqref{eq:test_stats_main}, which suffices because
exact matching is bias-free. For all other regimes, we use the
debiased $\Delta_n^{(p)}$ in Equation~\eqref{eq:test_stats_debiased},
built from local-polynomial intercepts over composite coarse
neighborhoods $\mathcal{C}_i$ (based on $\mathbf{Z}$) and fine
neighborhoods $\mathcal{F}_i$ (based on $(\mathbf{Z}, X)$). For mixed
$\mathbf{Z} = (\mathbf{Z}^{(d)}, \mathbf{Z}^{(c)})$, neighborhoods use a
composite metric giving primary priority to exact matching in discrete
dimensions and secondary refinement via Euclidean distance in
continuous dimensions. Crucially, the local-polynomial weights
$w_{ij, F}^{(p)}, w_{ij, C}^{(p)}$ determinded through \eqref{eq:lp_intercept} are
data-driven. They depend only on the sample
positions and implicitly serve as a local kernel density estimator,
eliminating any need to plug in $f_{Z, X}$ or $f_Z$. The construction
recovers the graph-based Kernel Partial Correlation estimator
\citep{huang2022kernel} when $(\mathbf{Z},X)$ has discrete support, but differs significantly when $(\mathbf{Z},X)$ has continuous components.

\textbf{Contribution 2: Asymptotic null distributions across four
regimes.} We establish limiting distributions under $H_0$ for all
fundamental data-type combinations (Section~\ref{theory}). When both
$\mathbf{Z}$ and $X$ are discrete, $n\Delta_n \Rightarrow \sum_{r\ge 1}\lambda_r(Z_r^2 - 1)$
with $\{Z_r\}\overset{\text{iid}}{\sim} N(0, 1)$ and $\{\lambda_r\}$
the eigenvalues of a conditional cross-covariance operator
$\mathcal{S}$ (Theorem~\ref{thm:xzdiscrete}). In each of the three
continuous-or-mixed regimes — all-continuous, $X$ discrete with
$\mathbf{Z}$ continuous, and $X$ continuous with $\mathbf{Z}$
discrete — we instead show $\sqrt n\,\Delta_n^{(p)} \Rightarrow N(0, \tau_p^2)$
(Theorem~\ref{thm:general_normality}). The three Gaussian regimes
require distinct proof techniques: for all-continuous, strict
positivity of $\tau_p^2$ follows from a dimensional-mismatch argument
via an Itô isometry on disjoint Palm fields; for $X$ discrete with
$\mathbf{Z}$ continuous, it follows from a radius-mismatch argument on
stratum-restricted annuli; for $X$ continuous with $\mathbf{Z}$
discrete, the bias cancels \emph{exactly} 
by an algebraic polynomial-reproduction identity, and no undersmoothing
condition on $k_n$ is required.

\textbf{Contribution 3: Local power analysis with dimension-free
detection threshold.} We establish universal consistency
(Theorem~\ref{thm:consistency_main}): under any fixed alternative with a
characteristic kernel, the test's power approaches one. For the
debiased statistic under the local alternative
$f_{r_n} = (1 - r_n) f_\mathbf{Z}f_{X|\mathbf{Z}}f_{Y|\mathbf{Z}} + r_n g_{X, Y, \mathbf{Z}}$
with $r_n\to 0$ (Theorem~\ref{thm:phase_transition}), we prove a
\emph{dimension-free} detection threshold $r_n^* \asymp n^{-1/4}$,
independent of $d_\mathbf{Z}$ and $d_X$. A typical phase transition at
$d_\mathbf{Z} + d_X = 8$ (when both $\mathbf{Z},X$ both are continuous), present in the analysis of nearest neighbor graph based test staistics for two sample test (\cite{bbb2019})
and caused by competing geometric biases at rates $n^{1/2-2/(d_\mathbf{Z}+d_X)}$
and $n^{1/2-2/d_\mathbf{Z}}$, is eliminated by the local-polynomial
reproduction identities. The curse of dimensionality is not removed
but \emph{relocated}. It now appears in the bias killing constraint
$\alpha < 1 - D/(2(p+1))$ on the bandwidth $k_n = n^\alpha$, which
forces larger polynomial orders $p$ in higher dimensions, with cost
$N_p = \binom{D+p}{p}$ per local-polynomial fit. The detection
threshold itself, however, remains parametric.

\textbf{Contribution 4: Regime-specific analytic calibration.} We
develop calibration procedures avoiding full permutation null
distributions (Section~\ref{algorithm_implementation}). For the
all-discrete regime, a Gaussian multiplier bootstrap on the conditional
covariance operator $\mathcal{S}$ avoids the $\mathcal{O}(n^3)$
eigendecomposition: drawing $B$ standard normal vectors and computing
quadratic forms $\boldsymbol\xi^{(b)\top}\mathcal{S}\boldsymbol\xi^{(b)} - \mathrm{Tr}(\mathcal{S})$
yields $B$ draws from the mixture-of-chi-squared limit, exploiting the
sparse block-diagonal structure of $\mathcal{S}$ for $\mathcal{O}(n)$
matrix-vector products. Total cost is $\mathcal{O}(BnL)$ for discrete
$Y$ with $L$ levels and $\mathcal{O}(Bn^2)$ for continuous $Y$. For
continuous and mixed regimes, the limiting overlap variance $\tau_p^2$
admits a graph-based plug-in estimator
$\widehat\tau_p^2 = n^{-1}\sum_{i, \ell}\omega_{i\ell}\widehat\zeta_{i, p}\widehat\zeta_{\ell, p}$,
where $\widehat\zeta_{i, p}$ are centered local scores and
$\omega_{i\ell}$ is an overlap indicator over fine and coarse
nearest-neighbor graphs. The studentized statistic
$T_n^{(p)} = \sqrt n\,\Delta_n^{(p)}/\widehat\tau_p$ converges to $N(0, 1)$
under $H_0$, yielding analytic $p$-values in $\mathcal{O}(nk_n)$ time
after neighborhood construction.

After $k_n$-NN graph construction in $\mathcal{O}(n\log n)$ time via
k-d trees \citep{bentley1975multidimensional}, the local-polynomial
intercepts cost $\mathcal{O}(nk_n N_p^2)$ and the overlap variance
$\mathcal{O}(nk_n)$, giving an overall near-quadratic complexity
dominated by the $\mathcal{O}(n^2)$ Gram-matrix evaluation. The
procedure scales gracefully to $n \sim 10^4$--$10^5$ and is viable for
causal discovery pipelines executing thousands of CITs. Notably, our
framework includes marginal independence ($\mathbf{Z} = \emptyset$) as
a special case, providing a consistent nonparametric test for mixed
$(X, Y)$ without requiring separate marginal procedures.

\subsection{Organization}

The remainder of the paper is organized as follows. Section \ref{setup} formalizes the problem, describes composite neighborhood construction, defines the test statistic, establishes its $U$-statistic representation, and connects it to maximum mean discrepancy. Section \ref{theory} presents the complete asymptotic theory under $H_0$ for all four regimes, along with consistency and power analysis. Section \ref{algorithm_implementation} provides computational algorithms for neighborhood construction and regime-specific calibration procedures. Section \ref{sec:experiments} reports extensive simulation studies comparing our test to existing methods across diverse data types, dimensions, and sample sizes, plus applications to genomic data with mixed genotype-expression-covariate structures. Section \ref{discussion} concludes with discussion of limitations and future directions. All proofs and technical lemmas appear in the Appendix.

%% file: methods.tex
Let $(X, Y, \mathbf{Z})$ be random variables with joint distribution $P$, taking values in domains $\mathcal{X}, \mathcal{Y}, \mathcal{Z}$. In our setting, $X$ and $Y$ could be either discrete or continuous variables with dimension 1, whereas $\mathbf{Z}$ could be a multi-dimensional variable, with each 
component
of $\mathbf{Z}$ 
being 
either discrete or continuous---allowing $\mathbf{Z}$ to be of mixed data types.
Given $n$ i.i.d.\ samples $\{(X_i, Y_i, \mathbf{Z}_i)\}_{i=1}^n$, we test the null hypothesis of conditional independence:
\[
H_0 : Y \indep X \mid \mathbf{Z} \quad \text{versus} \quad H_1: Y \nindep X \mid \mathbf{Z}.
\]
Our approach relies on a simple geometric intuition: 
under $H_0$,
knowing $X$ provides no additional 
information about $Y$ once $\mathbf{Z}$ is known. We operationalize this by comparing the average similarity of $Y$ values in two types of local neighborhoods: (1) \emph{Coarse Neighborhoods} ($\mathcal{C}_i$): Points similar to $i$ based on $\mathbf{Z}$ alone, (2) \emph{Fine Neighborhoods} (${\mathcal{F}}_i$): Points similar to $i$ based on both $(\mathbf{Z}, X)$.

If $H_0$ holds, the 
average similarity of $Y$ values within the 
fine neighborhood 
should
match that in the coarse neighborhood. 
If 
$X$ provides additional information about 
$Y$, then $Y$ values will be more similar within fine neighborhoods than within coarse neighborhoods.

\subsection{Neighborhood Construction for Mixed Data}
To accommodate mixed data types, we construct neighborhoods using a stratified nearest-neighbor approach that treats discrete and continuous components differently. For a \emph{discrete variable} $U$ (e.g., a component of $\mathbf{Z}$ or $X$), we define an \emph{exact-match constraint}:
\[
d(U_i, U_j) = \begin{cases}
0 & \text{if } U_i = U_j \\
\infty & \text{otherwise}
\end{cases}
\]
This enforces hard stratification: observations with mismatched discrete values are excluded from neighborhood consideration. For a \emph{continuous variable} $U$, we use the Euclidean distance: 
\(
d(U_i, U_j) = \|U_i - U_j\|_2.
\)
When constructing neighborhoods in \emph{mixed-type} settings, we apply a \textbf{stratified $k$-NN} procedure:
\begin{enumerate}
    \item First, restrict the search to the subset of samples that match exactly on all discrete components;
    \item Within that subset, select the $k$ nearest neighbors based on Euclidean distance in the continuous components.
\end{enumerate}
This procedure implements a dissimilarity measure on $(\mathbf{Z}, X)$ where discrete components act as hard constraints. Let $\mathcal{C}_i$ denote the indices of neighbors of observation $i$ based on $\mathbf{Z}$ alone (coarse neighborhood), and let $\mathcal{F}_i$ denote the indices of neighbors based on $(\mathbf{Z}, X)$ (fine neighborhood). We assume $k < n$ is either fixed or grows slowly with $n$ (e.g., $k \asymp n^\gamma$ for some $\gamma \in (0,1)$). Crucially, when fewer than $k$ points satisfy the exact-match constraint on discrete components, the neighborhood contains all available points in that stratum. Let $c_i = |\mathcal{C}_i|$ and $f_i = |\mathcal{F}_i|$ denote the actual neighborhood sizes, which satisfy $1 \leq f_i \leq c_i \leq k$.

\subsection{Test Statistics}
\label{subsec:test_statistics}

Let $K : \mathcal{Y} \times \mathcal{Y} \to \mathbb{R}$ be a symmetric,
bounded kernel — the indicator $K(y, y') = \mathbf{1}(y = y')$ for
discrete $Y$, or a characteristic kernel (\cite{fukumizu2004dimensionality,sriperumbudur2010hilbert,sriperumbudur2011universality}) such as the Gaussian
$K(y, y') = \exp(-\|y - y'\|^2/\sigma^2)$ (bandwidth $\sigma$ from the
median heuristic) for continuous $Y$. A natural first candidate
statistic is the simple $k$-nearest-neighbor (NN) difference
\begin{equation}
\label{eq:test_stats_main}
\Delta_n
\;=\;
\frac{1}{n}\sum_{i=1}^n\!\left(\frac{1}{f_i}\sum_{j\in\mathcal{F}_i}K(Y_i, Y_j) - \frac{1}{c_i}\sum_{j\in\mathcal{C}_i}K(Y_i, Y_j)\right),
\end{equation}
where $\mathcal{F}_i, \mathcal{C}_i$ are the fine and coarse
neighborhoods of observation $i$ (exact matching for discrete
components, $k_n$-NN for continuous ones), and $f_i, c_i$ are their sizes.
The two inner averages estimate
$\mathbb{E}[K(Y_i, Y) | \mathbf{Z}_i, X_i]$ and
$\mathbb{E}[K(Y_i, Y) | \mathbf{Z}_i]$ respectively. Their difference
quantifies whether incorporating $X$ changes the local structure of $Y$
beyond what $\mathbf{Z}$ alone determines.

When $(\mathbf{Z}, X)$ has continuous components, however, the $k_n$-NN
smoothing introduces a deterministic bias of order $(k_n/n)^{2/D}$,
where $D$ is the effective smoothing dimension of the conditioning
geometry. The null CLT for \eqref{eq:test_stats_main} reads (as shown recently in \cite{azadkia2025bias} for Azadkia-Chatterjee conditional correlation function)
\begin{equation}
\label{eq:raw_CLT}
\sqrt n\,(\Delta_n - \theta_n) \;\Rightarrow\; N(0, \tau^2),
\qquad \theta_n = \mathbb{E}[\Delta_n] \asymp (k_n/n)^{2/D},
\end{equation}
and $\sqrt n\,\theta_n$ diverges whenever $D \ge 5$, so $\sqrt n\,\Delta_n$
is not asymptotically pivotal. Analogous root-n bias and CLT phenomena have been studied for graph-based dependence measures; see \cite{azadkia2025bias} for Chatterjee-type graph correlation, and \cite{gao2026limit} for Azadkia–Chatterjee conditional graph correlation. To remove the centering at the $\sqrt n$-scale, we replace each $k_n$-NN average
in \eqref{eq:test_stats_main} by the intercept of a degree-$p$
local-polynomial least-squares fit of $K(Y_i, \cdot)$ on the
corresponding neighborhood:
\begin{equation}
\label{eq:lp_intercept}
\widehat a_{i, F}^{(p)} \;=\; \mathrm{argmin}_{a \in \mathbb{R}}\;\min_{b \in \mathbb{R}^{N_p - 1}}\sum_{j \in \mathcal{F}_i}\bigl[K(Y_i, Y_j) - a - b^\top q_p\!\bigl((\mathbf{Z}_j - \mathbf{Z}_i, X_j - X_i)/\widehat\rho_i\bigr)\bigr]^2,
\end{equation}
where $q_p$ is the vector of monomials of degree $1, \ldots, p$ in the
continuous coordinates, $N_p = \binom{D+p}{p}$, and $\widehat\rho_i$ is
the empirical $k_n$-NN radius; the coarse analog $\widehat a_{i, C}^{(p)}$
is defined identically over $\mathcal{C}_i$. The closed-form solution of
\eqref{eq:lp_intercept} is
$\widehat a_{i, F}^{(p)} = \sum_j w_{ij, F}^{(p)} K(Y_i, Y_j)$ with
data-driven \emph{equivalent kernel weights} $w_{ij, F}^{(p)}$ that
depend only on the sample positions
(\citealp{FanGijbels1996}, Chapter 3). The local-polynomial machinery
implicitly absorbs the local density variation, so no plug-in estimator
of $f_{Z, X}$ or $f_Z$ is ever needed. The debiased statistic is
\begin{equation}
\label{eq:test_stats_debiased}
\Delta_n^{(p)} \;=\; \frac{1}{n}\sum_{i=1}^n\bigl(\widehat a_{i, F}^{(p)} - \widehat a_{i, C}^{(p)}\bigr),
\end{equation}
with $\Delta_n^{(0)} = \Delta_n$. For $p \ge 1$, the polynomial reproduction identities $\sum_j w_{ij, F}^{(p)}(\mathbf{Z}_j - \mathbf{Z}_i, X_j - X_i)^\alpha = 0$
for $1 \le |\alpha| \le p$ annihilate the first $p$ Taylor terms of the
conditional kernel mean, reducing the bias from $O((k_n/n)^{2/D})$ to
$O((k_n/n)^{(p+1)/D})$. Choosing $p \ge \lceil D/2\rceil$ restores the
$\sqrt n$-pivotal property, i.e., the debiased CLT:
$\sqrt n\,\Delta_n^{(p)} \;\Rightarrow\; N(0, \tau_p^2)$
with $\tau_p^2 > 0$ identified explicitly in Section~\ref{theory}.
Studentization by a consistent estimator $\widehat\tau_p$ yields an
asymptotically pivotal test statistic $T_n^{(p)} = \sqrt n\,\Delta_n^{(p)}/\widehat\tau_p$.
For discrete components, exact matching makes the local-polynomial
design vacuous and the weights collapse to uniform stratum averaging.
The construction therefore adapts automatically to any combination of
discrete and continuous components in $(\mathbf{Z}, X)$.

\subsection{Population Target and Normalization}
\label{subsec:population_normalization}

The population analog of $\Delta_n^{(p)}$ admits a Maximum Mean
Discrepancy representation. With $\mathcal{H}$ the RKHS induced by $K$
and $\mu_P := \mathbb{E}_{Y\sim P}[K(Y, \cdot)]$ the mean embedding,
define
\begin{equation}
\label{eq:mmd_population}
\mathcal{T} \;:=\; \mathbb{E}_{\mathbf{Z}, X}\!\left[\mathrm{MMD}^2\bigl(P_{Y|\mathbf{Z}, X}, P_{Y|\mathbf{Z}}\bigr)\right]
\;=\; \mathbb{E}_{\mathbf{Z}, X}\bigl[\|\mu_{Y|\mathbf{Z}, X} - \mu_{Y|\mathbf{Z}}\|_\mathcal{H}^2\bigr].
\end{equation}
For characteristic kernels such as the indicator for discrete $Y$, the
Gaussian for continuous $Y$ \citep{sriperumbudur2010hilbert}, we have
$\mathcal{T} = 0$ iff $H_0$ holds. The debiased intercepts
$\widehat a_{i, F}^{(p)}, \widehat a_{i, C}^{(p)}$ in
\eqref{eq:lp_intercept} consistently estimate the conditional kernel
means $\mathbb{E}[K(Y_i, Y)|\mathbf{Z}_i, X_i]$ and
$\mathbb{E}[K(Y_i, Y)|\mathbf{Z}_i]$, so
$\Delta_n^{(p)}\xrightarrow{p}\mathcal{T}$
(Theorem~\ref{thm:consistency_main}) which implies that the test based on $n\Delta_n$ (when $(\mathbf{Z}, X)$ have discrete support) or, $T_n^{(p)} = \sqrt n\,\Delta_n^{(p)}/\widehat\tau_p$ (when $(\mathbf{Z}, X)$ has continuous components)  is consistent against any
alternative with $\mathcal{T} > 0$.

 Following
\citet{huang2022kernel} and
\citet{10.1093/biomet/asac048}, an interpretable correlation-style
normalization is
\begin{equation}
\label{eq:delta_norm}
\Delta_{\mathrm{norm}}^{(p)} \;:=\;
\frac{n^{-1}\sum_i(\widehat a_{i, F}^{(p)} - \widehat a_{i, C}^{(p)})}{1 - n^{-1}\sum_i \widehat a_{i, C}^{(p)}},
\end{equation}
bounded in $[0, 1]$ for the indicator kernel and asymptotically for
normalized characteristic kernels on continuous $Y$. We retain the unnormalized $\Delta_n^{(p)}$ for the
asymptotic theory of Section~\ref{theory}, with $\Delta_{\mathrm{norm}}^{(p)}$
serving as an effect-size summary.

%% file: asymptotic_distribution.tex
In this section, we establish the asymptotic behavior of the debiased
test statistic $\Delta_{n}$ (when $(\mathbf{Z}, X)$ have discrete support) or, $\Delta_n^{(p)}$ (when $(\mathbf{Z}, X)$ has continuous components) under $H_0: Y \perp\!\!\!\perp X \mid \mathbf{Z}$.
A unifying finding is that, across all four data-type regimes for
$(\mathbf{Z}, X)$, the debiased statistic admits a Gaussian limit at
the $\sqrt n$-scale, with the smoothing bias of the raw statistic
killed by the local-polynomial reproduction identities. The variance
constant $\tau_p^2$ which varies by regime, is  driven by dimensional mismatch
in the all-continuous case and by radius mismatch in the mixed cases, but remains strictly positive under mild non-degeneracy. We present
the main results here; complete proofs and regularity conditions appear
in the Appendix.

\subsection{Discrete Neighborhoods: Mixture of Chi-Squared Distribution}

When both $Z$ and $X$ are discrete, neighborhoods are constructed via exact matching: $\mathcal{C}_i = \{j : Z_j = Z_i\}$ and 
$\mathcal{F}_i = \{j : (Z_j, X_j) = (Z_i, X_i)\}$.
Under mild regularity ensuring that empirical neighborhood sizes grow to infinity (see Assumptions in Theorem~\ref{thm:discrete} of Section~\ref{app:thm_discrete} in the Appendix), the test statistic exhibits a degenerate $U$-statistic structure.

\begin{theorem}[Null Distribution: Discrete Case]\label{thm:xzdiscrete}
Suppose $Z$ and $X$ are discrete with finite support, and Assumption~S1 holds. Under $H_0: Y \perp\!\!\!\perp X \mid Z$, as $n \to \infty$:
\begin{equation}\label{eq:discrete_limit}
n \Delta_{n} \xrightarrow{d} \sum_{r=1}^\infty \lambda_r (Z_r^2 - 1),
\end{equation}
where $\{Z_r\}_{r \geq 1} \overset{i.i.d.}{\sim} \mathcal{N}(0,1)$ and $\{\lambda_r\}_{r \geq 1}$ are the eigenvalues of the conditional covariance operator $\mathcal{S}$ defined in the Appendix. The proof is given in Appendix~\ref{app:degeneracy_discrete}.
\end{theorem}

\begin{remark}[Comparison with Existing Tests]
The mixture-of-chi-squared limit \eqref{eq:discrete_limit} is
characteristic of degenerate second-order $U$-statistics with
compact operator kernels. When applied to a discrete
conditioning variable, the Kernel Conditional Independence test
(KCI) of \citet{zhang2011kernel} admits the same
weighted-chi-squared limit under $H_0$. On the oher hand, two related methods
differ in the following way: the Generalized Covariance Measure of
\citet{shah2020hardness} is calibrated by a Gaussian limit
rather than a quadratic form, and the Conditional Randomization
Test of \citet{candes2018panning} draws its finite-sample
validity from the known conditional law of $X \mid Z$ rather
than any asymptotic distributional guarantee. Classical dense KCI further
require an $\mathcal{O}(n^3)$ matrix operation; our Gaussian multiplier bootstrap matches its
asymptotic validity in $\mathcal{O}(n^2)$ time
(Section~\ref{sec:discrete_algorithm}) by exploiting the sparse
block structure of $\mathcal{S}$ induced by discrete
neighborhoods.
\end{remark}

\subsection{Null Gaussian Limit for the Debiased Statistic}
\label{subsec:gaussian_limit}
 
The local-polynomial debiasing in
\eqref{eq:lp_intercept}--\eqref{eq:test_stats_debiased} eliminates the
smoothing bias that obstructs $\sqrt n$-pivotal inference for the raw
statistic $\Delta_n$ in the presence of continuous components. The
resulting debiased statistic $\Delta_n^{(p)}$ admits a Gaussian limit
uniformly across all four data-type regimes for $(\mathbf{Z}, X)$.
 
\begin{theorem}[Null Distribution: Debiased Statistic]
\label{thm:general_normality}
Fix $p \ge p^*$ where $p^* = \lceil D/2\rceil$ is the minimum polynomial
order admitting a valid bandwidth (Section~\ref{theory}). Under
$H_0: Y \perp\!\!\!\perp X \mid \mathbf{Z}$, the regularity conditions
of Theorems~\ref{thm:lp_debiased_clt_full}, ~\ref{thm:lp_debiased_clt_full_mixed} and ~\ref{thm:lp_debiased_clt_full_IV} of the Appendix and the bias killing condition
$\sqrt n(k_n/n)^{(p+1)/D}\to 0$,
\begin{equation}
\label{eq:continuous_limit}
\sqrt n\,\Delta_n^{(p)} \;\xrightarrow{d}\; \mathcal{N}(0, \tau_p^2),
\end{equation}
where the asymptotic variance $\tau_p^2 > 0$ is the local-polynomial
overlap-variance constant depending on $K$, the joint distribution of
$(\mathbf{Z}, X, Y)$, and the regime-specific geometry. Full
identification of $\tau_p^2$ and the proof appear in Section~\ref{app:continuous_proof}, ~\ref{app:mixed-x-continuous-z-discrete} and ~\ref{app:mixed-x-discrete-z-continuous} of the Appendix.
\end{theorem}
 
\begin{corollary}[Alternative Distribution: Debiased Statistic]
\label{thm:general_normality_alternative}
Under the regularity conditions of Theorem~\ref{thm:general_normality}
and a contiguity-type variance-continuity assumption, the Gaussian limit continues to hold
under $H_1$:
\begin{equation}
\label{eq:continuous_limit_alt}
\sqrt n\,(\Delta_n^{(p)} - \mu_1) \;\xrightarrow{d}\; \mathcal{N}(0, \sigma_1^2),
\end{equation}
where $\mu_1 = \mathbb{E}_{H_1}[\Delta_n^{(p)}] \to \mathcal{T}(P) \ne 0$
and $\sigma_1^2 \to \tau_p^2$ along contiguous alternatives.
\end{corollary}
 
\begin{remark}[Comparison with continuous-conditioning tests]
The Gaussian limit \eqref{eq:continuous_limit} unifies all three regimes
of $(\mathbf{Z}, X)$ at the $\sqrt n$-scale, contrasting with the
mixture-of-chi-squared limit that arises for the raw statistic
$\Delta_n$ under purely discrete conditioning. The $\sqrt n$-normality
is broadly consistent with recent kernel-based and projection-based
conditional independence tests, including the Conditional Correlation
Test of \citet{wang2015conditional}, the Projected Covariance Measure
of \citet{lundborg2024projected}, and the Conditional Randomization
Test of \citet{candes2018panning}, although these rely on different
statistics (rank correlations, projected kernels, model-X
formulations).
\end{remark}

\subsection{Consistency and Power Under Alternatives}

\begin{theorem}[Universal Consistency of the Debiased Test]
\label{thm:consistency_main}
Suppose the kernel $K$ is bounded and characteristic, and the
local-polynomial weights satisfy the design regularity of
Lemma~\ref{lem:lp_design_uniform_convergence} with
$k_n \to \infty$ and $k_n/n \to 0$ (no undersmoothing required). Under
any alternative $H_1: Y\not\perp\!\!\!\perp X \mid \mathbf{Z}$,
\begin{equation}
\label{eq:consistency}
\Delta_n^{(p)} \;\xrightarrow{p}\; \mathcal{T}(P) \;>\; 0,
\end{equation}
where $\mathcal{T}(P) = \mathbb{E}[\|\mu_{Y|\mathbf{Z}, X} - \mu_{Y|\mathbf{Z}}\|_\mathcal{H}^2]$
is the population conditional independence measure. Consequently, the
studentized test $T_n^{(p)} = \sqrt n\Delta_n^{(p)}/\widehat\tau_p$ has
asymptotic power one: $P_{H_1}\{\text{reject }H_0\} \to 1$ for any
$\alpha \in (0, 1)$. The proof appears in Appendix~\ref{sec:consistency}.
\end{theorem}
 
\begin{remark}[Universal consistency in the literature]
Universal consistency under characteristic kernels is a hallmark of
nonparametric tests based on conditional mean embeddings. It is
shared by KCI~\citep{zhang2011kernel}, the
Hilbert-Schmidt Conditional Independence Criterion
\citep{fukumizu2008kernel}, the Randomized Conditional Independence
Test \citep{strobl2019approximate}, and the kernel partial correlation
of \citet{huang2022kernel}. Our result
extends this guarantee to the debiased neighborhood-based statistic,
which has the computational advantage of avoiding full kernel-matrix
operations and the inferential advantage of an asymptotically standard
normal pivot (\eqref{eq:continuous_limit}). The population measure
$\mathcal{T}(P)$ is closely related with the conditional Maximum Mean
Discrepancy of \citet{huang2022kernel}, providing a population-level
interpretation of the debiased statistic. The consistency rate under
fixed alternatives is the parametric - $n$ in the all discrete case and $\sqrt n$ when at least one component is continuous.
\end{remark}

\subsection{Local Power Analysis: Detection Threshold}
\label{subsec:local_power_summary}

To characterize the test's ability to detect weak departures from
$H_0$, we consider local alternatives converging to the null at rate
$r_n \to 0$. We focus on the continuous/mixed case
(Theorem~\ref{thm:general_normality}). Consider the mixture model
\begin{equation}
\label{eq:local_alt}
(X, Y, \mathbf{Z}) \;\sim\; f_{r_n} \;:=\; (1 - r_n)\,f_Z f_{X|Z}f_{Y|Z} + r_n\,g_{X, Y, Z},
\end{equation}
where the null component $f_Z f_{X|Z}f_{Y|Z}$ enforces conditional
independence, $g_{X, Y, Z}$ is an alternative density preserving the
$Z$ and $X|Z$ marginals, and $r_n \to 0$ controls proximity to $H_0$.
The one-sided level-$\alpha$ test is $S_n := \mathbf{1}\{T_n^{(p)} \ge z_\alpha\}$
with $T_n^{(p)} = \sqrt n\,\Delta_n^{(p)}/\widehat\tau_p$ and $\widehat\tau_p^2$ a consistent estimator
of $\tau_p^2$ obtained via the overlap-graph plug-in estimator
(Section~\ref{sec:continuous_mixed_case}).

\begin{theorem}[Dimension-free detection threshold]
\label{thm:phase_transition}
Suppose the conditions of Theorem~\ref{thm:general_normality} hold
under the local alternative \eqref{eq:local_alt}, with regularity
conditions on $f_Z, f_{X|Z}, f_{Y|Z, X}, g_{X, Y, Z}$ (detailed in
Assumption~\ref{asn:kernel_regularity_corrected} of the Appendix). Let $\phi(r_n) := P_{f_{r_n}}\{S_n = 1\}$
denote the power. Then for any dimensions $d_Z, d_X \ge 0$,
\begin{itemize}
\item If $r_n \ll n^{-1/4}$, then $\phi(r_n) \to \alpha$;
\item If $r_n \gg n^{-1/4}$, then $\phi(r_n) \to 1$.
\end{itemize}
The detection threshold for the mixture parameter is therefore
$r_n^* \asymp n^{-1/4}$, dimension-free.
\end{theorem}

The proof, which combines the null CLT
(Theorem~\ref{thm:general_normality}) with the signal expansion
$\mathcal{T}(P_{f_{r_n}}) = r_n^2\eta_{\mathrm{alt}} + O(r_n^3)$ (with
$\eta_{\mathrm{alt}} > 0$ the population signal strength), is given in
Appendix~\ref{sec:local_power}. The key simplification afforded by the debiased
construction is that the mean expansion of $\sqrt n\,\Delta_n^{(p)}$
under the mixture has only one $\sqrt n$-scale term:
$\mathbb{E}_{f_{r_n}}[\sqrt n\,\Delta_n^{(p)}] = \sqrt n\,r_n^2\eta_{\mathrm{alt}} + o(1)$.
The two dimension-dependent geometric biases at rates
$n^{1/2 - 2/(d_Z + d_X)}$ and $n^{1/2 - 2/d_Z}$ that appear in the
analysis of the raw statistic are exactly cancelled by the
local-polynomial reproduction identities coming from \eqref{eq:lp_intercept}.

\begin{remark}[Comparison with power analysis in the literature]
Local power analysis for conditional independence tests remains
relatively sparse. Notable exceptions include \citet{bergsma2014consistent},
who analyze partial-correlation tests under Gaussian alternatives, and
\citet{shah2020hardness}, who prove a no-free-lunch hardness result for nonparametric CIT whereas minimax theory for discrete or binned conditional independence is studied in \cite{neykov2021minimax}. Our result establishes the parametric $n^{-1/4}$
detection threshold uniformly across dimensions, recovering the
$n^{-1/4}$ rate that \citet{bbb2019} obtain for two sample testing in the lower-dimensional regime $d \le 8$. Where
the other geometric graph based statistic exhibits a phase transition at $d_Z + d_X = 8$, the debiased statistic
eliminates this transition by killing the dimension-dependent
geometric biases, restoring the parametric rate at all dimensions.
This is the central inferential benefit of debiasing.
\end{remark}


\begin{remark}[Practical implications]
For applications in causal discovery where conditional independence
tests are applied repeatedly with varying conditioning sets,
Theorem~\ref{thm:phase_transition} confirms that the debiased
statistic achieves uniform $n^{-1/4}$ detection across regimes, with
no need to switch between test variants by dimension. The Appendix
provides guidance on selecting the polynomial order $p$ and bandwidth
$k_n$ to balance bias-cancellation against design-matrix conditioning;
the recommendation $p = \lceil D/2 \rceil$ minimizes $N_p$ while
satisfying the undersmoothing constraint.
\end{remark}

\subsection{Summary of Asymptotic Results}

Table~\ref{tab:asymptotic_summary} summarizes the asymptotic properties across different data regimes.

\begin{table}[h]
\centering
\caption{Summary of asymptotic distributions under $H_0: Y \perp\!\!\!\perp X \mid Z$. The discrete case exhibits a degenerate $U$-statistic, while in the continuous/mixed case, the debiased statistic $\Delta_n^{(p)}$ admits a Gaussian limit uniformly
across continuous and mixed regimes.
Calibration methods are described in Section~\ref{algorithm_implementation}. }
\label{tab:asymptotic_summary}
\begin{tabular}{lccl}
\toprule
\textbf{Data Type} & \textbf{Scaling} & \textbf{Limiting Distribution} & \textbf{Calibration} \\
\midrule
$Z, X$ discrete & $n$ & $\sum_{r=1}^\infty \lambda_r (Z_r^2 - 1)$ & Gaussian multiplier bootstrap \\
$Z$ or $X$ continuous & $\sqrt{n}$ & $\mathcal{N}(0, \tau^2)$ & Conditional permutation + analytic \\
\bottomrule
\end{tabular}
\end{table}

The theoretical developments in this section provide a rigorous foundation for the computational procedures presented in Section~\ref{algorithm_implementation}, enabling both valid hypothesis testing and power-optimal inference across a wide range of data types. Complete proofs, technical assumptions, and extensions (including higher-order asymptotics and finite-sample corrections) are provided in the Appendix.

%% file: algorithm.tex
The asymptotic distribution theory established in Section~\ref{theory} enables efficient calibration of hypothesis tests without requiring extensive resampling procedures. In this section, we present computationally efficient algorithms that exploit the structural properties of our test statistic in different data regimes. Unlike kernel-based conditional independence tests such as KCI~\citep{zhang2011kernel}, which require $\mathcal{O}(n^3)$ matrix operation of dense kernel matrices, our approach leverages sparsity in the neighborhood structure to achieve substantially reduced computational complexity.

\begin{enumerate}
\item \textbf{Discrete neighborhoods} (Theorem~\ref{thm:xzdiscrete}): When $Z$ and $X$ are purely discrete, the test statistic $n \Delta_n$ converges in distribution to an infinite weighted sum of chi-squared random variables, $\sum_{r=1}^\infty \lambda_r (Z_r^2 - 1)$, where $\{\lambda_r\}$ are eigenvalues of the conditional covariance operator $\mathcal{S}$ and $Z_r \overset{i.i.d.}{\sim} \mathcal{N}(0,1)$. Direct spectral decomposition to obtain $\{\lambda_r\}$ requires $\mathcal{O}(n^3)$ operations. We instead employ the \emph{Gaussian multiplier bootstrap} to simulate from the limiting distribution without explicit eigendecomposition, achieving $\mathcal{O}(n)$ or $\mathcal{O}(n^2)$ complexity depending on whether $Y$ is discrete or continuous.

\item \textbf{Continuous or mixed neighborhoods} (Theorem~\ref{thm:general_normality}): When at least one of $Z$ or $X$ is continuous, the test statistic $\sqrt{n} \Delta_{n}$ is asymptotically $\mathcal{N}(0, \sigma^2)$ where $\sigma^2 = 4 \mathrm{Var}(\mathbb{E}[h(W_1, W_2) \mid W_1])$. We estimate the scalar variance $\sigma^2$ and empirical mean of proposed test statistic using a small number of \emph{conditional permutations} (resampling $Y$ within $Z$-defined neighborhoods), then calibrate the test analytically using the standard normal distribution. The overall complexity is $\mathcal{O}(P n^2)$ for $P$ permutations, with $P \ll 1000$ typically sufficient in practice.
\end{enumerate}


Both strategies represent substantial computational gains over standard permutation-based nonparametric tests, which typically require $\sim 10^3$ permutations to accurately estimate null distributions.

\subsection{Algorithm for Discrete Neighborhoods}\label{sec:discrete_algorithm}

\subsubsection{Gaussian Multiplier Bootstrap}

When $Z$ and $X$ are discrete, Theorem~\ref{thm:xzdiscrete} establishes that under $H_0$:
\begin{equation}\label{eq:spectral_limit}
n \Delta_n \xrightarrow{d} \sum_{r=1}^\infty \lambda_r (Z_r^2 - 1) = \sum_{r=1}^n \lambda_r (Z_r^2 - 1),
\end{equation}
where the second equality holds because $\mathcal{S} \in \mathbb{R}^{n \times n}$ has at most $n$ nonzero eigenvalues. Here, $\mathcal{S}$ denotes the empirical conditional cross-covariance operator, which can be represented as a matrix encoding the residual dependencies between $X$ and $Y$ given the neighborhood weights derived from $Z$. Detailed definition of $\mathcal{S}$ can be found in Appendix~\ref{app:thm_discrete}.

Direct computation of $\{\lambda_r\}_{r=1}^n$ via eigendecomposition of $\mathcal{S}$ is prohibitive for large $n$. We implement a Gaussian multiplier bootstrap formalism (similar in spirit to \cite{DEHLING1994392,arcones1992bootstrap,chatterjee2025boosting}) which circumvents this bottleneck via the following observation. Let $\mathcal{S} = \mathbf{U} \mathbf{D} \mathbf{U}^\top$ be the spectral decomposition, where $\mathbf{D} = \mathrm{diag}(\lambda_1, \ldots, \lambda_n)$ and $\mathbf{U}$ is orthogonal. For $\boldsymbol{\xi} \sim \mathcal{N}(\mathbf{0}, \mathbf{I}_n)$, the quadratic form
\begin{equation}\label{eq:quadratic_form}
\boldsymbol{\xi}^\top \mathcal{S} \boldsymbol{\xi} = \boldsymbol{\xi}^\top \mathbf{U} \mathbf{D} \mathbf{U}^\top \boldsymbol{\xi} = (\mathbf{U}^\top \boldsymbol{\xi})^\top \mathbf{D} (\mathbf{U}^\top \boldsymbol{\xi}) = \sum_{r=1}^n \lambda_r b_r^2,
\end{equation}
where $\mathbf{b} = \mathbf{U}^\top \boldsymbol{\xi} \sim \mathcal{N}(\mathbf{0}, \mathbf{I}_n)$ by rotational invariance of the Gaussian distribution. Since $b_r^2 \sim \chi_1^2$ and $\mathbb{E}[b_r^2] = 1$, we have $b_r^2 - 1 \overset{d}{=} Z_r^2 - 1$ for $Z_r \sim \mathcal{N}(0,1)$. Therefore:
\begin{equation}\label{eq:bootstrap_equivalence}
\boldsymbol{\xi}^\top \mathcal{S} \boldsymbol{\xi} - \mathrm{Tr}(\mathcal{S}) \overset{d}{=} \sum_{r=1}^n \lambda_r (Z_r^2 - 1).
\end{equation}

The key computational advantage is that evaluating the left-hand side of \eqref{eq:bootstrap_equivalence} requires only matrix-vector multiplications, not eigendecomposition. By generating $B$ independent samples $\boldsymbol{\xi}^{(1)}, \ldots, \boldsymbol{\xi}^{(B)} \overset{i.i.d.}{\sim} \mathcal{N}(\mathbf{0}, \mathbf{I}_n)$ and computing $\{\boldsymbol{\xi}^{(b)\top} \mathcal{S} \boldsymbol{\xi}^{(b)} - \mathrm{Tr}(\mathcal{S})\}_{b=1}^B$, we obtain $B$ draws from the limiting distribution \eqref{eq:spectral_limit}, enabling empirical $p$-value calculation. Proposition~\ref{prop:gmb_consistency_main} ensures the bootstrap statistic $T_n^*$ consistently approximates the null distribution of $T_n$, and thus the empirical $p$-values under $H_0$ are asymptotically valid. Detailed proof in given in Appendix~\ref{app:gmb_implementation}.
\begin{proposition}[GMB Consistency in the Discrete Regime]
\label{prop:gmb_consistency_main}
Suppose the conditioning sets are fully discrete, such that exact matching is used, and the observations $\mathbf{W} = \{W_1, \dots, W_n\}$ are independent and identically distributed (i.i.d.). Under the null hypothesis $H_0: Y \perp\!\!\!\perp X \mid Z$, the Gaussian Multiplier Bootstrap (GMB) statistic $T_n^*$ consistently estimates the null distribution of $T_n$. That is,
$$ \sup_{t \in \mathbb{R}} \left| \mathbb{P}^*(T_n^* \leq t \mid \mathbf{W}) - \mathbb{P}(T_n \leq t) \right| \xrightarrow{p} 0. $$
\end{proposition}

\subsubsection{Exploiting Sparsity via Neighborhood Structure}

The operator $\mathcal{S}$ admits a sparse representation arising from the discrete neighborhood structure. Let $\mathbf{P}_Z \in \{0,1\}^{m_Z \times n}$ and $\mathbf{P}_{ZX} \in \{0,1\}^{m_{ZX} \times n}$ be indicator matrices encoding group memberships, where $m_Z$ and $m_{ZX}$ are the numbers of distinct values of $Z$ and $(Z,X)$ respectively. For discrete data, $m_Z, m_{ZX} \ll n$ typically. Let $\mathbf{w}_Z \in \mathbb{R}^{m_Z}$ and $\mathbf{w}_{ZX} \in \mathbb{R}^{m_{ZX}}$ denote weight vectors inversely proportional to group sizes. The operator can be expressed as:
\begin{equation}\label{eq:operator_structure}
\mathcal{S} = \mathbf{P}_{ZX}^\top \mathrm{diag}(\mathbf{w}_{ZX}) \mathbf{P}_{ZX} - \mathbf{P}_Z^\top \mathrm{diag}(\mathbf{w}_Z) \mathbf{P}_Z.
\end{equation}

This block-diagonal plus low-rank structure enables fast matrix-vector products. For $\mathbf{v} \in \mathbb{R}^n$:
\begin{equation}\label{eq:matvec}
\mathcal{S} \mathbf{v} = \mathbf{P}_{ZX}^\top (\mathbf{w}_{ZX} \odot (\mathbf{P}_{ZX} \mathbf{v})) - \mathbf{P}_Z^\top (\mathbf{w}_Z \odot (\mathbf{P}_Z \mathbf{v})),
\end{equation}
where $\odot$ denotes elementwise multiplication. Since $\mathbf{P}_Z$ and $\mathbf{P}_{ZX}$ are sparse with at most $n$ nonzero entries total, each matrix-vector product in \eqref{eq:matvec} requires $\mathcal{O}(n)$ operations.

\subsubsection{Complexity Analysis by Response Type}

The computational cost depends on whether $Y$ is discrete or continuous:

\textit{Discrete $Y$.} When $Y$ takes values in a finite set with $L$ categories, we use the Dirac delta kernel $K(y, y') = \mathbb{I}(y = y')$, which admits the finite-dimensional feature map $\phi(y) \in \{0,1\}^L$ (one-hot encoding). Let $\mathbf{Y}_{\mathrm{OH}} \in \{0,1\}^{n \times L}$ denote the one-hot encoded matrix. The test statistic can be expressed as:
\begin{equation}\label{eq:discrete_Y_stat}
n \Delta_n = \|\mathcal{S}^{1/2} \mathbf{Y}_{\mathrm{OH}}\|_F^2 - \mathrm{Tr}(\mathcal{S}),
\end{equation}
where $\|\cdot\|_F$ denotes the Frobenius norm. Since $\mathcal{S}^{1/2}$ need not be computed explicitly (we work directly with $\mathcal{S}$), and $L$ is typically small, the computation of \eqref{eq:discrete_Y_stat} via $\mathbf{V} = \mathcal{S} \mathbf{Y}_{\mathrm{OH}}$ requires $\mathcal{O}(nL)$ operations using \eqref{eq:matvec}. For each bootstrap replicate:
\begin{equation}\label{eq:bootstrap_discrete}
\mathcal{T}_b^* = \|\mathbf{Y}_{\mathrm{OH}}^\top (\mathcal{S} \boldsymbol{\xi}^{(b)})\|_2^2 - \boldsymbol{\xi}^{(b)\top} \mathcal{S} \boldsymbol{\xi}^{(b)},
\end{equation}
which also requires $\mathcal{O}(nL)$ operations. With $B$ bootstrap replicates, the total complexity is $\mathcal{O}(B n L) \approx \mathcal{O}(n)$ for fixed $B$ and $L$.

\textit{Continuous $Y$.} When $Y$ is continuous, we use a continuous kernel (e.g., Gaussian RBF) whose feature space is infinite-dimensional, precluding explicit factorization as in \eqref{eq:discrete_Y_stat}. Let $\mathbf{K} \in \mathbb{R}^{n \times n}$ denote the Gram matrix with entries $K_{ij} = K(Y_i, Y_j)$. The test statistic is:
\begin{equation}\label{eq:continuous_Y_stat}
n \Delta_n = \mathrm{Tr}(\mathcal{S} \mathbf{K} \mathcal{S}) - \mathrm{Tr}(\mathcal{S}).
\end{equation}

For each bootstrap replicate, we compute:
\begin{equation}\label{eq:bootstrap_continuous}
\mathcal{T}_b^* = (\mathcal{S} \boldsymbol{\xi}^{(b)})^\top \mathbf{K} (\mathcal{S} \boldsymbol{\xi}^{(b)}) - \boldsymbol{\xi}^{(b)\top} \mathcal{S} \boldsymbol{\xi}^{(b)}.
\end{equation}

While $\mathcal{S} \boldsymbol{\xi}^{(b)}$ requires $\mathcal{O}(n)$ operations via \eqref{eq:matvec}, multiplication by the dense matrix $\mathbf{K}$ requires $\mathcal{O}(n^2)$ operations. Thus, the total complexity is $\mathcal{O}(B n^2)$ for $B$ bootstrap replicates. While quadratic in $n$, this remains substantially faster than the $\mathcal{O}(n^3)$ eigendecomposition required by KCI.

\subsubsection{Algorithmic Details}

We present the complete algorithm in Appendix~\ref{app:gmb_implementation}. Key implementation details include:

\begin{enumerate}
\item \textbf{Sparse matrix representation}: $\mathbf{P}_Z$ and $\mathbf{P}_{ZX}$ are stored in compressed sparse row (CSR) format, enabling efficient matrix-vector products.

\item \textbf{Bias correction}: The trace subtraction $-\mathrm{Tr}(\mathcal{S})$ in \eqref{eq:discrete_Y_stat}--\eqref{eq:continuous_Y_stat} corrects for the $U$-statistic diagonal bias. For the bootstrap statistics, we subtract $\boldsymbol{\xi}^{(b)\top} \mathcal{S} \boldsymbol{\xi}^{(b)}$ rather than $\mathrm{Tr}(\mathcal{S})$ to account for random weighting.

\item \textbf{Single-sided test}: We use the single-sided $p$-value $\hat{p} = \tfrac{1}{B} \sum_{b=1}^B \mathbb{I}(\mathcal{T}_b^* \geq \mathcal{T}_{\mathrm{obs}})$, as stated in Theorem~\ref{thm:consistency_main}, under alternative hypotheis, $\Delta_{n} \xrightarrow{p} \eta_K(\mu) > 0$.
\end{enumerate}

\subsection{Continuous and Mixed Data}\label{sec:continuous_mixed_case}

When at least one of $\mathbf{Z}$ or $X$ is continuous,
Theorem~\ref{thm:general_normality} establishes that the
local-polynomial debiased statistic $\Delta_n^{(p)}$ is asymptotically
normal at the $\sqrt n$-scale:
$\sqrt n\,\Delta_n^{(p)} \xrightarrow{d} \mathcal{N}(0, \tau_p^2)$. The
practical challenge is to estimate the variance $\tau_p^2$ — the
limiting overlap variance of the local graph scores — directly from
the sample, without resorting to a permutation null distribution. We
develop a graph-based plug-in estimator that exploits the
neighborhood structure of $\Delta_n^{(p)}$.

\subsubsection{Variance Estimation via Overlap of Local Scores}

For each anchor $i$, define the centered local graph score
\begin{equation}
\label{eq:xi_centering}
\xi_{i, p} \;=\; \widehat a_{i, F}^{(p)} - \widehat a_{i, C}^{(p)},
\qquad
\bar\xi_p \;=\; \frac{1}{n}\sum_{i = 1}^n \xi_{i, p},
\qquad
\widehat\zeta_{i, p} \;=\; \xi_{i, p} - \bar\xi_p.
\end{equation}
The limiting variance $\tau_p^2 = \lim_n n\operatorname{Var}(\Delta_n^{(p)})$
is the sum of pairwise overlap covariances $\operatorname{Cov}(\zeta_{i, p}, \zeta_{\ell, p})$
across anchors whose fine or coarse neighborhoods share common points
(Lemma~\ref{lem:lp_overlap_variance}). Non-overlapping anchors
contribute negligibly. We mirror this structure in the estimator by
summing centered-score cross-products only over the overlap graph.

Let $R_{i, F}$ be the distance from $(\mathbf{Z}_i, X_i)$ to its
$k_n$-th nearest neighbor in $(\mathbf{Z}, X)$-space, and $R_{i, C}$
the distance from $\mathbf{Z}_i$ to its $k_n$-th nearest neighbor in
$\mathbf{Z}$-space. Fix a constant $c > 1$ and define the overlap
indicator
\begin{equation}
\label{eq:omega_indicator}
\omega_{i\ell}
\;=\;
\mathbf{1}\!\bigl\{\|(\mathbf{Z}_i, X_i) - (\mathbf{Z}_\ell, X_\ell)\| \le c(R_{i, F} + R_{\ell, F})\bigr\}
\;\vee\;
\mathbf{1}\!\bigl\{\|\mathbf{Z}_i - \mathbf{Z}_\ell\| \le c(R_{i, C} + R_{\ell, C})\bigr\}.
\end{equation}
The variance estimator is
\begin{equation}
\label{eq:variance_estimator}
\widehat\tau_p^2 \;=\; \frac{1}{n}\sum_{i = 1}^n\sum_{\ell = 1}^n \omega_{i\ell}\,\widehat\zeta_{i, p}\,\widehat\zeta_{\ell, p}.
\end{equation}
The studentized statistic is
\begin{equation}
\label{eq:pvalue_normal}
T_n^{(p)} \;=\; \frac{\sqrt n\,\Delta_n^{(p)}}{\widehat\tau_p}.
\end{equation}
Under the regularity conditions of
Theorem~\ref{thm:general_normality}, $\widehat\tau_p^2 \xrightarrow{p}\tau_p^2$
and hence $T_n^{(p)} \Rightarrow N(0, 1)$ under $H_0$.

The $p$-value is computed analytically as
$\widehat p = 1 - \Phi(T_n^{(p)})$ (one-sided) or $2(1 - \Phi(|T_n^{(p)}|))$
(two-sided), with $\Phi$ the standard normal CDF. The estimator
$\widehat\tau_p^2$ is the heteroscedasticity-and-autocorrelation-style
plug-in suited to the overlap-graph dependence intrinsic to local
nearest-neighbor statistics; the constant $c > 1$ controls inclusion
of slightly extended overlap pairs and is robust to a default
choice $c = 2$ in our experiments. The estimator improves on the
standard Hájek-projection variance, which underestimates $\tau_p^2$ by
ignoring the off-diagonal overlap contributions.

\subsubsection{Neighborhood Construction and Distance Metrics}

Accurate neighborhood identification is critical. We employ the
following strategies.

\textit{Standardization.} All continuous components of $\mathbf{Z}$ and
$X$ are standardized to zero mean and unit variance before distance
computation, to prevent scale-dependent domination by high-variance
variables.

\textit{Composite metric for mixed data.} When
$\mathbf{Z}$ or $(\mathbf{Z}, X)$ contains both discrete and continuous
components, denote $\mathbf{Z} = (\mathbf{Z}_{\mathrm{disc}}, \mathbf{Z}_{\mathrm{cont}})$.
We define a \emph{hierarchical distance}
\begin{equation}
\label{eq:composite_distance}
d\bigl((z_i, x_i), (z_\ell, x_\ell)\bigr)
\;=\;
\begin{cases}
\varepsilon\,d_{\mathrm{Euc}}(z_{i, \mathrm{cont}}, z_{\ell, \mathrm{cont}}) & \text{if discrete components match,} \\
\infty & \text{otherwise,}
\end{cases}
\end{equation}
where $d_{\mathrm{Euc}}$ is the Euclidean distance on standardized
continuous variables and $\varepsilon > 0$ is a small constant
ensuring within-stratum ranking by continuous proximity. Exact
matching on discrete components is prioritized; continuous variables
serve as tie-breakers.

\textit{Neighborhood size selection.} The undersmoothing constraint
$\sqrt n\,(k_n/n)^{(p+1)/D}\to 0$ from
Theorem~\ref{thm:general_normality} requires $k_n = n^\alpha$ with
$\alpha < 1 - D/(2(p+1))$. With $p = \lceil D/2\rceil$ (the
minimum-cost choice from Section~\ref{theory}), this admits $\alpha$
up to nearly $1/2$, so a default $k_n \approx c\sqrt n$ with $c \in [1, 5]$
satisfies the condition; for higher polynomial orders, larger $k_n$
is permitted. Data-driven selection of $k_n$ via cross-validation is
possible but not pursued here.

\subsubsection{Algorithmic Details}

Algorithm~2 in Appendix~\ref{app:other_implementation} gives the complete procedure. Key points:

\begin{enumerate}
\item \textbf{Efficient $k_n$-NN construction.} Using spatial data
structures (KD-trees or ball trees), the fine and coarse $k_n$-NN
graphs are built in $\mathcal{O}(n\log n)$ time on average for moderate
dimensions.

\item \textbf{Local-polynomial fits.} For each anchor $i$, the
intercepts $\widehat a_{i, F}^{(p)}, \widehat a_{i, C}^{(p)}$ are
obtained by solving the $N_p \times N_p$ normal equations for the
weighted least-squares problem \eqref{eq:lp_intercept}; this is
$\mathcal{O}(k_n N_p^2 + N_p^3)$ per anchor.

\item \textbf{Overlap variance.} The double sum
\eqref{eq:variance_estimator} restricts to pairs $(i, \ell)$ with
$\omega_{i\ell} = 1$. Since each anchor has $O(k_n)$ neighbors in
either graph, the effective number of nonzero terms is
$\mathcal{O}(nk_n)$, computed in $\mathcal{O}(nk_n)$ time after the NN
graphs are constructed.

\item \textbf{Analytic calibration.} The $p$-value is computed
analytically from \eqref{eq:pvalue_normal}; no resampling is required.
\end{enumerate}

\subsubsection{Computational Complexity}

The overall complexity is dominated by:

\begin{enumerate}
\item \textbf{$k_n$-NN construction.} $\mathcal{O}(n\log n)$ for
moderate dimensions via spatial trees; $\mathcal{O}(n^2)$ in the worst
case for very high dimensions.

\item \textbf{Local-polynomial intercepts.} $\mathcal{O}(n(k_n N_p^2 + N_p^3))$,
linear in $n$ since $k_n$ and $N_p$ depend on $n$ and $D$ but not on
the sample index.

\item \textbf{Overlap variance.} $\mathcal{O}(n k_n)$ following the
neighborhood-restricted sum above.
\end{enumerate}

The total cost is $\mathcal{O}(n\log n + nk_n N_p^2 + n^2)$, with the
$\mathcal{O}(n^2)$ term arising from the Gram matrix
$\mathbf{K}_{ij} = K(Y_i, Y_j)$ used inside the local-polynomial
fits. For $k_n = c\sqrt n$, the local-polynomial cost is
$\mathcal{O}(n^{3/2}N_p^2)$, dominated by the Gram-matrix
$\mathcal{O}(n^2)$ for $n \ge N_p^4$. The procedure scales gracefully
to $n \sim 10^4$--$10^5$. 

%% file: simulation.tex
To comprehensively evaluate the robustness of our method, we generate synthetic datasets encompassing continuous, discrete, and mixed data types under both linear and nonlinear causal mechanisms. Detailed descriptions of the data-generating processes and parameter configurations are provided in Appendix \ref{app:data_generation}.

\subsection{Scalability and Performance of the Proposed Method}
We first evaluate the scalability and statistical performance of our proposed debiased algorithm across varying sample sizes, ranging from $N=1000$ to $N=10000$. As shown in Table 2, our method consistently controls the Type-I error around the nominal level across purely continuous (CCC - $X,Y,\mathbf{Z}$ all are continuous), purely discrete (DDD - $X,Y,\mathbf{Z}$ all are discrete), and mixed (DCC - $X$ discrete and $Y,\mathbf{Z}$ continuous, CCD - $X, Y$ continuous and $\mathbf{Z}$ discrete) regimes for both linear and nonlinear mechanisms. While statistical power naturally starts slightly lower in challenging nonlinear regimes at smaller sample sizes (e.g., 0.75 for CCC nonlinear at $N=1000$), it rapidly converges to approximately 1.00 as the sample size increases.

Most notably, our algorithm exhibits exceptional computational scalability. For $N=10000$, the average runtime remains strictly under 10 seconds across all tested regimes. This unprecedented efficiency demonstrates that our approach is highly practical for large-scale datasets where traditional non-parametric tests would become computationally intractable.

\begin{table}[ht]
\centering
\caption{Scalability and performance of the proposed method (Ours) across sample sizes $N \in \{1000, 2000, 5000, 10000\}$. Metrics reported are Type-I Error (T1), Statistical Power (Pow), and Average Runtime in seconds (RT).}
\label{tab:scalability}
\resizebox{\textwidth}{!}{%
\begin{tabular}{llcccc}
\toprule
\textbf{Regime} & \textbf{Mechanism} & \textbf{N=1000} & \textbf{N=2000} & \textbf{N=5000} & \textbf{N=10000} \\
& & (T1 / Pow / RT) & (T1 / Pow / RT) & (T1 / Pow / RT) & (T1 / Pow / RT) \\
\midrule
\multirow{2}{*}{CCC} & linear    & 0.00 / 0.98 / 0.54s & 0.00 / 0.99 / 0.42s & 0.00 / 0.97 / 2.52s  & 0.00 / 1.00 / 8.32s \\
                     & nonlinear & 0.05 / 0.75 / 0.14s & 0.00 / 0.84 / 1.13s & 0.00 / 0.93 / 2.06s  & 0.00 / 0.97 / 8.24s \\
\midrule
\multirow{2}{*}{DCC} & linear    & 0.00 / 1.00 / 0.12s & 0.01 / 1.00 / 0.35s & 0.00 / 1.00 / 1.49s  & 0.00 / 1.00 / 6.03s \\
                     & nonlinear & 0.00 / 1.00 / 0.14s & 0.00 / 1.00 / 0.39s & 0.00 / 1.00 / 1.86s  & 0.00 / 1.00 / 6.68s \\
\midrule
\multirow{2}{*}{CCD} & linear    & 0.00 / 0.94 / 0.50s & 0.00 / 0.95 / 0.37s & 0.00 / 1.00 / 0.63s & 0.00 / 1.00 / 2.76s \\
                     & nonlinear & 0.00 / 1.00 / 0.046s & 0.00 / 1.00 / 0.49s & 0.00 / 1.00 / 0.59s & 0.00 / 1.00 / 2.17s \\
\midrule
\multirow{2}{*}{DDD} & linear    & 0.01 / 1.00 / 0.07s & 0.03 / 1.00 / 0.09s & 0.03 / 1.00 / 0.17s  & 0.00 / 1.00 / 0.24s \\
                     & nonlinear & 0.01 / 1.00 / 0.07s & 0.03 / 1.00 / 0.09s & 0.03 / 1.00 / 0.13s  & 0.00 / 1.00 / 0.24s \\
\bottomrule
\end{tabular}
}
\end{table}

\subsection{Comparison Against Baseline Algorithms}

Building upon the robust performance of our proposed method, we further benchmark it against three widely used baseline algorithms: KCI, Fisher-z, and CMI-KNN. This comparison, summarized in Table 3, spans sample sizes from $N=1000$ to $N=5000$ to highlight the trade-offs inherent in existing approaches.

\paragraph{Validity and Type-I Error Control}
A critical limitation of current methods is the inability to maintain valid Type-I error control across diverse data regimes. Fisher-z, being strictly linear, understandably fails in nonlinear scenarios, exhibiting severe Type-I error inflation (e.g., reaching 0.55 in the CCC nonlinear regime at $N=5000$). While KCI and CMI-KNN offer non-parametric flexibility, KCI shows unexpected error inflation in specific mixed-data nonlinear cases as the sample size grows (e.g., $0.15$ in DCC nonlinear at $N=5000$), indicating potential instability in kernel hyperparameter selection. In contrast, as established in the previous section, our debiased algorithm maintains strict validity across all tested scenarios.

\paragraph{Statistical Power Limitations}
Regarding statistical power, KCI serves as a strong upper-bound baseline, reliably detecting dependencies across all regimes. However, the alternative non-parametric method, CMI-KNN, struggles significantly in mixed-data settings. For instance, its power stagnates around $0.58 \sim 0.59$ in the DCC linear regime regardless of the sample size. Our proposed algorithm effectively overcomes the detection failures seen in CMI-KNN, matching the high statistical power of KCI as $N$ increases.

\paragraph{The Computational Bottleneck}
The most stark contrast between our method and the baselines lies in computational scalability. While KCI provides excellent statistical power, its runtime exhibits a prohibitive non-linear growth with respect to the sample size. As $N$ increases from 1000 to 5000, KCI's average runtime skyrockets from approximately 1 second to nearly 2 minutes (over 100 seconds) per test. CMI-KNN, though faster than KCI in high-sample regimes, still consistently requires more computation time than our approach in mixed datasets. This comparison underscores the primary contribution of our work: achieving the rigorous statistical power of heavy kernel methods without inheriting their computational bottlenecks, seamlessly analyzing even 10000 samples in mere seconds.

\begin{table}[htbp]
\centering
\caption{Comprehensive performance comparison of the proposed algorithm against baselines across varying sample sizes $N \in \{1000, 2000, 5000\}$. Metrics reported are Type-I Error (T1), Statistical Power (Pow), and Average Runtime in seconds (RT). The proposed method scales efficiently while maintaining high statistical power and valid Type-I errors across all regimes.}
\label{tab:full_baseline_comparison_scaling}
\resizebox{\textwidth}{!}{%
\begin{tabular}{lllccc}
\toprule
\textbf{Regime} & \textbf{Mechanism} & \textbf{Method} & \textbf{N=1000} & \textbf{N=2000} & \textbf{N=5000} \\
& & & (T1 / Pow / RT) & (T1 / Pow / RT) & (T1 / Pow / RT) \\
\midrule
\multirow{4}{*}{CCC} & \multirow{4}{*}{linear}    
& \textbf{Ours} & \textbf{0.00 / 0.98 / 0.54s} & \textbf{0.00 / 0.99 / 0.42s} & \textbf{0.00 / 0.97 / 2.52s} \\
&& KCI     & 0.04 / 1.00 / 1.02s  & 0.06 / 1.00 / 11.52s & 0.10 / 1.00 / 103.67s \\
&& Fisher-z& 0.05 / 1.00 / 0.00s  & 0.03 / 1.00 / 0.00s  & 0.02 / 1.00 / 0.00s \\
&& CMI-KNN & 0.08 / 0.90 / 0.92s  & 0.07 / 0.98 / 1.97s  & 0.07 / 0.95 / 6.03s \\
\midrule
\multirow{4}{*}{CCC} & \multirow{4}{*}{nonlinear} 
& \textbf{Ours} & \textbf{0.05 / 0.75 / 0.14s} & \textbf{0.00 / 0.84 / 1.13s} & \textbf{0.00 / 0.93 / 2.06s} \\
&& KCI     & 0.06 / 1.00 / 1.09s  & 0.05 / 0.97 / 11.23s & 0.05 / 1.00 / 108.57s \\
&& Fisher-z& 0.47 / 0.86 / 0.00s  & 0.38 / 0.88 / 0.00s  & 0.55 / 0.88 / 0.00s \\
&& CMI-KNN & 0.03 / 0.66 / 0.86s  & 0.04 / 0.67 / 1.99s  & 0.02 / 0.83 / 5.45s \\
\midrule
\multirow{3}{*}{DCC} & \multirow{3}{*}{linear}    
& \textbf{Ours} & \textbf{0.00 / 1.00 / 0.12s} & \textbf{0.01 / 1.00 / 0.36s} & \textbf{0.00 / 1.00 / 1.49s} \\
&& KCI     & 0.03 / 1.00 / 1.07s  & 0.08 / 1.00 / 11.23s & 0.10 / 1.00 / 108.40s \\
&& CMI-KNN & 0.00 / 0.58 / 0.79s  & 0.00 / 0.59 / 1.75s  & 0.00 / 0.58 / 5.39s \\
\midrule
\multirow{3}{*}{DCC} & \multirow{3}{*}{nonlinear} 
& \textbf{Ours} & \textbf{0.00 / 1.00 / 0.14s} & \textbf{0.00 / 1.00 / 0.39s} & \textbf{0.00 / 1.00 / 1.86s} \\
&& KCI     & 0.04 / 1.00 / 1.07s  & 0.03 / 1.00 / 11.18s & 0.15 / 1.00 / 110.57s \\
&& CMI-KNN & 0.00 / 0.93 / 0.77s  & 0.00 / 0.98 / 1.75s  & 0.00 / 0.95 / 4.78s \\
\midrule
\multirow{3}{*}{CCD} & \multirow{3}{*}{linear}    
& \textbf{Ours} & \textbf{0.00 / 0.94 / 0.50s} & \textbf{0.00 / 0.95 / 0.37s} & \textbf{0.00 / 1.00 / 0.63s} \\
&& KCI     & 0.07 / 1.00 / 1.09s  & 0.06 / 1.00 / 11.09s & 0.05 / 1.00 / 114.08s \\
&& CMI-KNN & 0.18 / 0.93 / 1.09s  & 0.09 / 0.91 / 2.85s  & 0.12 / 0.96 / 13.67s \\
\midrule
\multirow{3}{*}{CCD} & \multirow{3}{*}{nonlinear} 
& \textbf{Ours} & \textbf{0.00 / 1.00 / 0.05s} & \textbf{0.00 / 1.00 / 0.49s} & \textbf{0.00 / 1.00 / 0.58s} \\
&& KCI     & 0.07 / 1.00 / 1.09s  & 0.06 / 1.00 / 10.72s & 0.10 / 1.00 / 113.97s \\
&& CMI-KNN & 0.07 / 1.00 / 1.23s  & 0.06 / 1.00 / 3.51s  & 0.09 / 1.00 / 17.25s \\
\midrule
\multirow{3}{*}{DDD} & \multirow{3}{*}{linear}    
& \textbf{Ours} & \textbf{0.01 / 1.00 / 0.07s} & \textbf{0.03 / 1.00 / 0.09s} & \textbf{0.03 / 1.00 / 0.17s} \\
&& Chi-sq  & 0.00 / 1.00 / 0.00s  & 0.00 / 1.00 / 0.00s  & 0.00 / 1.00 / 0.00s \\
&& G-sq    & 0.00 / 1.00 / 0.00s  & 0.00 / 1.00 / 0.00s  & 0.00 / 1.00 / 0.00s \\
\midrule
\multirow{3}{*}{DDD} & \multirow{3}{*}{nonlinear} 
& \textbf{Ours} & \textbf{0.01 / 1.00 / 0.07s} & \textbf{0.03 / 1.00 / 0.09s} & \textbf{0.03 / 1.00 / 0.14s} \\
&& Chi-sq  & 0.00 / 1.00 / 0.00s  & 0.00 / 1.00 / 0.00s  & 0.00 / 1.00 / 0.00s \\
&& G-sq    & 0.00 / 1.00 / 0.00s  & 0.00 / 1.00 / 0.00s  & 0.00 / 1.00 / 0.00s \\
\bottomrule
\end{tabular}
}
\end{table}

\subsection{Robustness to Non-Ordinal Heterogeneous Confounding}

To systematically evaluate the robustness of conditional independence tests in mixed-data regimes, we investigate the sensitivity of the tests to non-ordinal heterogeneous confounding. We introduce a shift parameter $\delta \ge 0$ that controls the amplitude of a ``sawtooth'' effect induced by a discrete confounder. The results for both linear and nonlinear continuous confounding mechanisms at sample size $N=1000$ are summarized in Table~\ref{tab:heterogeneous_results}.

As shown in Table~\ref{tab:heterogeneous_results}, when $\delta = 0$, the underlying distribution reduces to a homogeneous regime. Here, both KCI and our proposed MixCIT reasonably control the Type-I error around the nominal level (with MixCIT being slightly more conservative). However, as the heterogeneity strength $\delta$ increases, the performance of KCI deteriorates drastically. By embedding discrete categories into a continuous metric space, KCI erroneously applies smooth radial basis function (RBF) kernels across distinct, non-ordinal categories. This cross-stratum kernel leakage induces a severe residual confounding bias. Consequently, the Type-I error of KCI escalates monotonically, reaching up to $1.00$ at $\delta=0.30$, completely invalidating the test.

In stark contrast, our MixCIT algorithm explicitly enforces exact stratification on the discrete components. This structural isolation guarantees that samples from different discrete categories never interact during the local-polynomial fitting phase, neutralizing the cross-stratum leakage. As a result, MixCIT is mathematically immune to the heterogeneity magnitude $\delta$. It flawlessly controls the Type-I error at $0.00$ across the entire spectrum of $\delta$ while maintaining exceptionally high statistical power ($\ge 0.96$). Furthermore, by avoiding the computation of a global dense kernel matrix, MixCIT demonstrates significant computational advantages, operating approximately $7\times$ faster than KCI on average (${\sim}0.08$ seconds vs. ${\sim}0.57$ seconds per trial).

\begin{table}[htbp]
    \centering
    \caption{Empirical Type-I error and Power across varying heterogeneity strengths ($\delta$) for $N=1000$ at significance level $\alpha=0.05$. MixCIT maintains strict Type-I error control regardless of $\delta$, whereas KCI suffers from severe false positives due to cross-stratum leakage.}
    \label{tab:heterogeneous_results}
    \resizebox{0.8\textwidth}{!}{
        \begin{tabular}{l cccc cccc}
            \toprule
            & \multicolumn{4}{c}{\textbf{Linear Mechanism}} & \multicolumn{4}{c}{\textbf{Nonlinear Mechanism}} \\
            \cmidrule(lr){2-5} \cmidrule(lr){6-9}
            & \multicolumn{2}{c}{\textbf{MixCIT (Ours)}} & \multicolumn{2}{c}{\textbf{KCI}} & \multicolumn{2}{c}{\textbf{MixCIT (Ours)}} & \multicolumn{2}{c}{\textbf{KCI}} \\
            \cmidrule(lr){2-3} \cmidrule(lr){4-5} \cmidrule(lr){6-7} \cmidrule(lr){8-9}
            $\delta$ & Type-I & Power & Type-I & Power & Type-I & Power & Type-I & Power \\
            \midrule
            0.00 & 0.00 & 1.00 & 0.08 & 1.00 & 0.00 & 0.96 & 0.08 & 1.00 \\
            0.10 & 0.00 & 1.00 & 0.20 & 1.00 & 0.00 & 0.96 & 0.14 & 1.00 \\
            0.15 & 0.00 & 1.00 & 0.30 & 1.00 & 0.00 & 0.96 & 0.28 & 1.00 \\
            0.20 & 0.00 & 1.00 & 0.56 & 1.00 & 0.00 & 0.96 & 0.50 & 1.00 \\
            0.25 & 0.00 & 1.00 & 0.82 & 1.00 & 0.00 & 0.96 & 0.74 & 1.00 \\
            0.30 & 0.00 & 1.00 & 1.00 & 1.00 & 0.00 & 0.96 & 0.96 & 1.00 \\
            \bottomrule
        \end{tabular}
    }
\end{table}

%% file: discussion.tex
Despite its strong empirical performance, MixCIT has two inherent limitations that suggest fruitful directions for future work.

First, the exact-matching stratification on discrete covariates can lead to severe sparsity when $Z$ contains high-cardinality or multiple categorical variables. The number of possible strata grows exponentially with the number of discrete components, causing many strata to contain too few observations for stable local-polynomial estimation. A promising remedy is \textit{fuzzy stratification}, where observations from similar discrete states contribute with a similarity weight, pooling information across strata via learned embeddings.

Second, the median-heuristic bandwidth selection for the kernel on $Y$ is not optimal for detecting weak or localized nonlinear signals. A global bandwidth may either oversmooth the signal or inflate variance. Future extensions could adopt \textit{adaptive bandwidth selection} or a multiple-kernel mixture to enhance power across diverse dependence structures.

Addressing these issues would broaden MixCIT's applicability to high-dimensional discrete settings and improve its robustness to complex dependencies.

%% file: supplementary-material-arxiv.tex
\appendix
\section{Preliminaries on $U$-statistics}
\label{app:prelim}

In this section, we review standard results on $U$-statistics that are essential for our proofs. Let $W_1, \dots, W_n$ be i.i.d.\ random variables taking values in a measurable space $\mathcal{W}$ with distribution $P$. We first give the definition of a $U$-statistic and Hilbert-Schmidt operator.

\begin{definition}[$U$-statistic and Degeneracy]\label{def:u_stats}
A $U$-statistic of order 2 with a symmetric kernel $h: \mathcal{W}^2 \to \mathbb{R}$ is defined as:
\[
U_n = \binom{n}{2}^{-1} \sum_{1 \le i < j \le n} h(W_i, W_j).
\]
Let $\theta = \mathbb{E}[h(W_1, W_2)]$. The first-order projection is defined as $\psi_1(w) = \mathbb{E}[h(w, W_2)] - \theta$. The $U$-statistic is called degenerate of order 1 if $\text{Var}(\psi_1(W_1)) = 0$.
\end{definition}

\begin{definition}[Hilbert--Schmidt operator and eigenvalues]\label{def:hs_operator}
Let $(\mathcal{W}, P)$ be a probability space and $h:\mathcal{W}^2 \to \mathbb{R}$ 
be a symmetric kernel function with $\E[h(W_1,W_2)^2]<\infty$.  
The Hilbert--Schmidt operator $T_h:L^2(P)\to L^2(P)$ associated with $h$ is defined as
\[
(T_h f)(w) \;=\; \E\!\left[ h(w,W')\, f(W') \right], \quad f\in L^2(P),
\]
where $W'\sim P$ is an independent copy of $W$.  \\
An eigenvalue--eigenfunction pair $(\lambda, \phi)$ of $T_h$ satisfies
\[
(T_h \phi)(w) \;=\; \lambda \, \phi(w), \qquad \phi \in L^2(P), \ \phi \not\equiv 0.
\]
The collection $\{\lambda_r\}_{r\ge1}$ of (possibly infinitely many) eigenvalues is referred to as the spectrum of $T_h$.
\end{definition}

The asymptotic distribution differs between degenerate cases and non-degenerate cases, as shown below by ~\ref{lemma:asymptotic_degenerate}. Under degenerate cases, the limiting distribution would be a mixture of chi-square distributions which .

\begin{lemma}[Asymptotic distribution of degenerate $U$-Statistics, \citealt{atta2019distributed}, \cite{Vaart_1998}]\label{lemma:asymptotic_degenerate}
Let $U_n$ be a first-order degenerate $U$-statistic of order 2 with $\mathbb{E}[U_n] = \theta$. Then,
\[
n(U_n - \theta) \xrightarrow{d} \sum_{j=1}^\infty \lambda_j (Z_j^2 - 1),
\]
where $Z_j$ ($j = 1, 2, \dots$) are independent standard normal random variables, and $\lambda_j$ ($j = 1, 2, \dots$) are the eigenvalues of the  function $h(w_1, w_2) - \theta$ viewed as a Hilbert–Schmidt operator on $L^2(P)$. 
\end{lemma}

\section{Test Statistic and Its $U$-statistic Representation
in the Discrete Regime}
\label{appendix_rep}

The next section develops the asymptotic null theory for $\Delta_n$
in the fully discrete regime, where both $\mathbf{Z}$ and $X$ take
values in finite sets and the composite neighborhoods $\coarse{i}$
and $\fine{i}$ are obtained by exact matching:
$\coarse{i} = \{j\ne i : Z_j = Z_i\}$ and
$\fine{i} = \{j\ne i : (Z_j, X_j) = (Z_i, X_i)\}$. In this regime, no
local-polynomial debiasing is required: exact matching introduces no
smoothing bias, and the raw statistic $\Delta_n$ itself converges in
distribution to a mixture of chi-squared random variables under
$H_0$. The present appendix sets up the $U$-statistic representation
driving this result and records the population-level interpretation
of $\Delta_n$ as a kernel-based conditional coefficient of
determination. Both are specific to the discrete regime; the
continuous and mixed regimes are handled separately in
Appendix~\ref{app:continuous_proof} through
Appendix~\ref{app:mixed-x-continuous-z-discrete} using the debiased
statistic $\Delta_n^{(p)}$ developed.

Given two observations $(Z_i, X_i, Y_i)$ and $(Z_j, X_j, Y_j)$ with
discrete $\mathbf{Z}$ and $X$, our test statistic is
\begin{equation}
\label{eq:test_stats}
\Delta_n \;=\; \frac{1}{n}\sum_{i = 1}^n\!\left(\frac{1}{\fcnt{i}}\sum_{j\in\fine{i}}K(Y_i, Y_j) - \frac{1}{\ccnt{i}}\sum_{j\in\coarse{i}}K(Y_i, Y_j)\right),
\end{equation}
where $\fcnt{i}, \ccnt{i}$ are the cardinalities of the fine and
coarse exact-match strata. Recalling
Definition~\ref{def:u_stats}, we exhibit a symmetric kernel $h$ such
that $\Delta_n \propto \sum_{i < j}h(\cdot, \cdot)$. The
$Y$-kernel $K$ is the indicator $K(y, y') = \mathbf{1}(y = y')$ for
discrete $Y$ and a bounded characteristic kernel (e.g., Gaussian RBF
$K(y, y') = \exp(-\|y - y'\|^2/(2\sigma^2))$) for continuous $Y$.
Define
\begin{align}
\label{h_definition}
&h\bigl((Z_i, X_i, Y_i), (Z_j, X_j, Y_j)\bigr) \notag\\
&\quad\;=\; \frac{n}{2}\!\left(\frac{\mathbf{1}(j\in\fine{i})}{\fcnt{i}} + \frac{\mathbf{1}(i\in\fine{j})}{\fcnt{j}}\right)\!K(Y_i, Y_j) - \frac{n}{2}\!\left(\frac{\mathbf{1}(j\in\coarse{i})}{\ccnt{i}} + \frac{\mathbf{1}(i\in\coarse{j})}{\ccnt{j}}\right)\!K(Y_i, Y_j).
\end{align}
A direct calculation gives
\begin{equation}
\label{u_equivalence}
\Delta_n \;=\; \frac{1}{n^2}\sum_{\substack{i, j = 1\\ i \ne j}}^n h\bigl((Z_i, X_i, Y_i), (Z_j, X_j, Y_j)\bigr) \;=\; \frac{2}{n^2}\sum_{1 \le i < j \le n}h\bigl((Z_i, X_i, Y_i), (Z_j, X_j, Y_j)\bigr).
\end{equation}
This rewrites $\Delta_n$ as a second-order $U$-statistic with kernel
$h$. The $n^{-2}$ normalization in \eqref{eq:test_stats} and
\eqref{u_equivalence} differs from the $n(n - 1)$ convention of
Definition~\ref{def:u_stats} by a factor $(1 - 1/n) \to 1$ which
affects neither the asymptotic distribution nor the degeneracy
structure. The $U$-statistic form is the entry point to the
eigen-decomposition arguments yielding the mixture-of-chi-squared
limit in Theorem~\ref{thm:discrete} below.

\section{Population Interpretation of Test Statistics}
\label{app:normalization}

For the discrete regime, the normalized statistic
\begin{equation}
\label{eq:delta_def}
\Delta_{\mathrm{norm}} \;:=\; \frac{\mathcal{Q}_{Z, X} - \mathcal{Q}_Z}{1 - \mathcal{Q}_Z}
\end{equation}
admits a clean population analog. Here
\begin{align}
\mathcal{Q}_{Z, X} &\;:=\; \frac{1}{n}\sum_{i = 1}^n\frac{1}{\fcnt{i}}\sum_{j\in\fine{i}}K(Y_i, Y_j),\\
\mathcal{Q}_Z &\;:=\; \frac{1}{n}\sum_{i = 1}^n\frac{1}{\ccnt{i}}\sum_{j\in\coarse{i}}K(Y_i, Y_j)
\end{align}
are the within-fine-stratum and within-coarse-stratum kernel
similarity, respectively. We assume $K$ is bounded and normalized
with $0 \le K(y, y') \le 1$ and $K(y, y) = 1$; for discrete $Y$ this
holds for the indicator kernel, and for continuous $Y$ it is
satisfied by the Gaussian RBF.

\begin{proposition}[Properties of $\Delta_{\mathrm{norm}}$ in the
discrete regime]
\label{prop:delta_norm}
Under the assumption of a normalized kernel,
$\Delta_{\mathrm{norm}}$ satisfies:
\begin{enumerate}
\item Boundedness: $0 \le \Delta_{\mathrm{norm}} \le 1$
asymptotically.
\item Full dependence: If $Y$ is a deterministic function of
$(X, \mathbf{Z})$, then $\Delta_{\mathrm{norm}} \to 1$.
\item Independence: If $Y\perp\!\!\!\perp X\mid\mathbf{Z}$, then
$\Delta_{\mathrm{norm}} \to 0$.
\end{enumerate}
\end{proposition}

\begin{proof}
\emph{(1) Boundedness.} Since $K(Y_i, Y_j) \le 1$ we have
$\mathcal{Q}_Z \le 1$, so the denominator is nonnegative. By
monotonicity, conditioning on the finer stratification $(Z, X)$
cannot decrease the within-stratum similarity:
$\mathcal{Q}_{Z, X} \ge \mathcal{Q}_Z$, hence
$\Delta_{\mathrm{norm}} \ge 0$. Also $\mathcal{Q}_{Z, X} \le 1$
implies $\mathcal{Q}_{Z, X} - \mathcal{Q}_Z \le 1 - \mathcal{Q}_Z$,
so $\Delta_{\mathrm{norm}} \le 1$.

\emph{(2) Full dependence.} If $Y = f(X, \mathbf{Z})$, then within
any $(Z, X)$ stratum all observations share the same value of $Y$,
giving $K(Y_i, Y_j) = 1$ for every $j\in\fine{i}$ and
$\mathcal{Q}_{Z, X}\to 1$. Substituting yields
$\Delta_{\mathrm{norm}} = (1 - \mathcal{Q}_Z)/(1 - \mathcal{Q}_Z) = 1$.

\emph{(3) Independence.} Under $H_0$, the within-stratum conditional
law of $Y$ is invariant under further conditioning on $X$: the
$(Z, X)$ and $Z$ strata produce the same conditional distribution of
$Y$. The within-stratum kernel means coincide in expectation,
$\mathbb{E}[\mathcal{Q}_{Z, X}] = \mathbb{E}[\mathcal{Q}_Z]$, giving
$\mathcal{Q}_{Z, X} - \mathcal{Q}_Z\to 0$ and
$\Delta_{\mathrm{norm}}\to 0$.
\end{proof}

This interpretation justifies $\Delta_{\mathrm{norm}}$ as a
kernel-based conditional $R^2$: the proportion of
within-coarse-stratum variability of $Y$ explained by further
conditioning on $X$. The unnormalized $\Delta_n$ inherits the same
sign and convergence properties and is the basis for inference; we
retain it because its $U$-statistic form makes the asymptotic null
distribution explicit.

\section{Proofs of Regime I -- All Discrete Case}\label{app:discrete_proof}
In this section, we establish the asymptotic null distribution of $\Delta_{norm}$ when all variables $(X,Y,Z)$ are discrete. The key insight is that under $H_0$, the test statistic can be represented as a first-order degenerate $U$-statistic, whose limiting distribution is a weighted sum of independent chi-squared random variables.

\subsection{Main Result}\label{app:thm_discrete}

\begin{theorem}[Asymptotic Distribution for Discrete Variables]
\label{thm:discrete}
Let $(X,Y,Z)$ be discrete random variables with finite or countable supports, and let $\{(X_i,Y_i,Z_i)\}_{i=1}^n$ be i.i.d.\ samples from the joint distribution $P$. Define $\Delta_{n}$ as in \eqref{eq:test_stats}. Suppose that as $n\to\infty$:
\begin{align}
&\min_{1\le i\le n} \ccnt{i} \;\xrightarrow{p}\; \infty,\qquad \min_{1\le i\le n} \fcnt{i} \;\xrightarrow{p}\; \infty, \label{eq:degreeLLN-a}\\
&\sup_{i}\Bigl|\frac{\ccnt{i}}{n}-p_Z(z_i)\Bigr| \xrightarrow{p} 0, \qquad
\sup_{i}\Bigl|\frac{\fcnt{i}}{n}-p_{ZX}(z_i,x_i)\Bigr| \xrightarrow{p} 0,\label{eq:degreeLLN-b}
\end{align}
where $p_Z(z)=P(Z=z)$ and $p_{ZX}(z,x)=P(Z=z,X=x)$ denote the marginal probabilities. Then, under the null hypothesis $H_0: Y \perp\!\!\!\perp X \mid Z$,
\begin{equation}\label{eq:limit-dist}
n\,\Delta_{n}\;\xRightarrow{d}\; \sum_{r=1}^\infty \lambda_r\,(Z_r^2-1),
\end{equation}
where $\{Z_r\}_{r=1}^\infty$ are i.i.d.\ standard normal random variables and $\{\lambda_r\}_{r=1}^\infty$ are the eigenvalues of the centered kernel operator associated with $\Delta_{n}$ under $H_0$ (see Definition~\ref{def:hs_operator}).
\end{theorem}

\begin{remark}
The operator $\mathcal{S}$ is the conditional covariance operator induced by the kernel $K$ and the weight matrix $W$, whose $(i,j)$-th entry is
\begin{equation}\label{eq:weight_matrix}
    W_{ij} = \frac{1}{|\mathcal{F}_i|}\mathbf{1}(j \in \mathcal{F}_i) - \frac{1}{|\mathcal{C}_i|}\mathbf{1}(j \in \mathcal{C}_i),
\end{equation}
where $\mathcal{F}_i$ denotes the $k$-NN neighborhood of observation $i$ in $(X, Z)$-space and $\mathcal{C}_i$ denotes the $k$-NN neighborhood in $Z$-space alone, with $W_{ii} = 0$. The eigenvalues $\{\lambda_r\}_{r \geq 1}$ of $\mathcal{S}$ govern the asymptotic distribution in~\eqref{eq:limit-dist}; in practice, they are approximated via the eigendecomposition of the empirical matrix $\widehat{S} = W \circ K$, where $K_{ij} = K(Y_i, Y_j)$ is the kernel matrix on $Y$ and $\circ$ denotes the Hadamard product.
\end{remark}

\begin{remark}
Conditions \eqref{eq:degreeLLN-a}--\eqref{eq:degreeLLN-b} hold under standard positivity assumptions where the joint distribution $P$ assigns strictly positive probability to all cells in the support. These conditions ensure that the data-dependent normalizers $\ccnt{i}^{-1}$ and $\fcnt{i}^{-1}$ behave asymptotically like $\{n\,p_Z(Z_i)\}^{-1}$ and $\{n\,p_{ZX}(Z_i,X_i)\}^{-1}$, respectively, which is essential for establishing degeneracy of the $U$-statistic kernel.
\end{remark}

\subsection{Proof Strategy}

The proof proceeds in three main steps:

\begin{enumerate}[label=(\Roman*).]
    \item \textit{$U$-statistic representation:} We express $\Delta_{n}$ as a symmetric order-2 $U$-statistic with kernel function $h((Z_i,X_i,Y_i),(Z_j,X_j,Y_j))$.

    \item \textit{First-order degeneracy:} We prove that under $H_0$, the kernel $h$ satisfies
        \[
        \mathbb{E}[h((Z_i,X_i,Y_i),(Z_j,X_j,Y_j))\mid (Z_j,X_j,Y_j)] = 0 \quad \text{almost surely},
        \]
    which implies that $\mathrm{Var}[\mathbb{E}[h((Z_i,X_i,Y_i),(Z_j,X_j,Y_j))\mid (Z_j,X_j,Y_j)]]=0$, i.e., first-order degeneracy.

    \item \textit{Application of limiting theory:} We invoke Lemma~\ref{lemma:asymptotic_degenerate} and apply Slutsky's theorem to obtain the mixture-of-chi-squared limit.
\end{enumerate}

\subsection{Step 1: Representation as a $U$-Statistic}\label{subsec:step1_U_representation}

For the all-discrete case, we select the kernel function $K(Y_i,Y_j) = \mathbf{1}(Y_i = Y_j)$ in \eqref{eq:test_stats}. The neighborhoods become exact-match sets:
\begin{align*}
\coarse{i} &= \{j \neq i : Z_j = Z_i\}, \quad \ccnt{i} = |\coarse{i}| = \sum_{j \neq i} \mathbf{1}(Z_i = Z_j),\\
\fine{i} &= \{j \neq i : Z_j = Z_i, X_j = X_i\}, \quad \fcnt{i} = |\fine{i}| = \sum_{j \neq i} \mathbf{1}(Z_i = Z_j, X_i = X_j).
\end{align*}
We can then rewrite the test statistic as:
\begin{align*}
\Delta_{n} &= \frac{1}{n}\sum_{i=1}^n \frac{1}{\fcnt{i}} \sum_{j \in \fine{i}} \mathbf{1}(Y_i = Y_j) - \frac{1}{n}\sum_{i=1}^n \frac{1}{\ccnt{i}} \sum_{j \in \coarse{i}} \mathbf{1}(Y_i = Y_j)\\
&= \frac{1}{n}\sum_{i=1}^n \frac{1}{\fcnt{i}} \sum_{j=1; j \neq i}^n \mathbf{1}(Z_i=Z_j, X_i=X_j) \mathbf{1}(Y_i = Y_j)  - \frac{1}{n}\sum_{i=1}^n \frac{1}{\ccnt{i}} \sum_{j=1; j\neq i}^n \mathbf{1}(Z_i=Z_j)\mathbf{1}(Y_i = Y_j).
\end{align*}
Now, define the symmetric kernel function $h:\mathcal{W}^2 \to \mathbb{R}$, where $\mathcal{W} = \mathcal{Z} \times \mathcal{X} \times \mathcal{Y}$, by:
\begin{align}
h((Z_i, X_i, Y_i), (Z_j, X_j, Y_j)) &= \frac{n}{2} \left( \frac{1}{\fcnt{i}} \mathbf{1}(Z_i = Z_j, X_i = X_j) + \frac{1}{\fcnt{j}} \mathbf{1}(Z_i = Z_j, X_i = X_j) \right) \mathbf{1}(Y_i = Y_j)\notag\\
&\quad - \frac{n}{2} \left( \frac{1}{\ccnt{i}} \mathbf{1}(Z_i = Z_j) + \frac{1}{\ccnt{j}} \mathbf{1}(Z_i = Z_j) \right) \mathbf{1}(Y_i = Y_j).\label{eq:def_h_scaled}
\end{align}
The factor of $n$ in \eqref{eq:def_h_scaled} ensures that $\mathbb{E}[h((Z_i,X_i,Y_i),(Z_j,X_j,Y_j))]$ does not vanish as $n\to\infty$, which is necessary for the kernel to be well-defined in the Hilbert--Schmidt operator framework. This scaling is compatible with the degree asymptotics in \eqref{eq:degreeLLN-a}--\eqref{eq:degreeLLN-b}.
By symmetry, we have
\begin{align*}
\Delta_{n} = \frac{1}{n} \sum_{\substack{i,j=1;\\i \neq j}}^n h((Z_i, X_i, Y_i), (Z_j, X_j, Y_j)) \cdot \frac{1}{n}= \frac{2}{n^2} \sum_{1\leq i<j\leq n} h((Z_i, X_i, Y_i), (Z_j, X_j, Y_j)).
\end{align*}
Thus, $\Delta_{n}$ is proportional to an order-2 $U$-statistic. While the normalization differs from the canonical form $\binom{n}{2}^{-1}$ by a factor of $(1-1/n)^{-1}$, this does not affect the asymptotic distribution. For notational convenience, we define
\begin{equation}\label{eq:ustat-form}
U_n := \frac{2}{n(n-1)} \sum_{1\leq i<j\leq n} h((Z_i, X_i, Y_i), (Z_j, X_j, Y_j)),
\end{equation}
and note that $n\,\Delta_{n} = n \cdot \frac{n-1}{n} U_n = (n-1)U_n \sim n U_n$ asymptotically.

\subsection{Step 2: First-Order Degeneracy Under $H_0$} \label{app:degeneracy_discrete}
To apply Lemma~\ref{lemma:asymptotic_degenerate}, we must establish that under $H_0$, the kernel $h$ is first-order degenerate. By Definition~\ref{def:u_stats}, this requires showing
\begin{equation}\label{eq:degeneracy-condition}
\mathrm{Var}\big[\mathbb{E}[h((Z_i,X_i,Y_i),(Z_j,X_j,Y_j))\mid (Z_j,X_j,Y_j)]\big] = 0.
\end{equation}
This condition holds if and only if $\mathbb{E}[h((Z_i,X_i,Y_i),(Z_j,X_j,Y_j))\mid (Z_j,X_j,Y_j)]$ is constant almost surely. We will prove that this conditional expectation equals zero almost surely under $H_0$.

\subsubsection{Decomposition of the Kernel}

To facilitate computation, we decompose $h$ into two components. Define:
\begin{align}
h_1((Z_i, X_i, Y_i), (Z_j, X_j, Y_j)) &= \frac{n}{2} \left( \frac{1}{\fcnt{i}} \mathbf{1}(Z_i = Z_j, X_i = X_j) + \frac{1}{\fcnt{j}} \mathbf{1}(Z_i = Z_j, X_i = X_j) \right) \mathbf{1}(Y_i = Y_j),\label{eq:h1-def}\\
h_2((Z_i, X_i, Y_i), (Z_j, X_j, Y_j)) &= \frac{n}{2} \left( \frac{1}{\ccnt{i}} \mathbf{1}(Z_i = Z_j) + \frac{1}{\ccnt{j}} \mathbf{1}(Z_i = Z_j) \right) \mathbf{1}(Y_i = Y_j).\label{eq:h2-def}
\end{align}
Then we hve 
\[
h((Z_i, X_i, Y_i), (Z_j, X_j, Y_j)) = h_1((Z_i, X_i, Y_i), (Z_j, X_j, Y_j)) - h_2((Z_i, X_i, Y_i), (Z_j, X_j, Y_j)).
\]
The conditional expectation decomposes accordingly
\begin{align}
\mathbb{E}[h((Z_i, X_i, Y_i), (Z_j, X_j, Y_j)) \mid (Z_j, X_j, Y_j)] &=\mathbb{E}[h_1((Z_i, X_i, Y_i), (Z_j, X_j, Y_j)) \mid (Z_j, X_j, Y_j)]\notag\\&- \mathbb{E}[h_2((Z_i, X_i, Y_i), (Z_j, X_j, Y_j)) \mid (Z_j, X_j, Y_j)].\label{eq:h_conditional}
\end{align}

Since $\ccnt{i}$, $\ccnt{j}$, $\fcnt{i}$, and $\fcnt{j}$ are random (depending on the entire sample), we must take expectations with respect to all randomness. However, conditioning on $(Z_j,X_j,Y_j)$ only partially resolves this randomness. The key observation is that by \eqref{eq:degreeLLN-a}--\eqref{eq:degreeLLN-b}, we have:
\begin{align}
\frac{\ccnt{i}}{n} &\xrightarrow{p} p_Z(Z_i), \qquad \frac{\fcnt{i}}{n} \xrightarrow{p} p_{ZX}(Z_i,X_i),\label{eq:degree-convergence}
\end{align}
uniformly over $i$. Therefore, we can work with the asymptotic approximations:
\begin{align}
\ccnt{i} &\approx n \cdot p_Z(Z_i), \qquad \fcnt{i} \approx n \cdot p_{ZX}(Z_i,X_i),\label{eq:degree-approx}
\end{align}
where "$\approx$" denotes asymptotic equivalence in probability.

\subsubsection{Analysis of $\mathbb{E}[h_1 \mid (Z_j,X_j,Y_j)]$}

We begin by analyzing the first component. From \eqref{eq:h1-def}, we obtain
\begin{align}
&\mathbb{E}[h_1((Z_i, X_i, Y_i), (Z_j, X_j, Y_j)) \mid (Z_j, X_j, Y_j)] \notag\\
&\quad= \frac{n}{2} \mathbb{E}\left[ \frac{1}{\fcnt{i}} \mathbf{1}(Z_i = Z_j, X_i = X_j) \mathbf{1}(Y_i = Y_j) \,\bigg|\, (Z_j, X_j, Y_j) \right] \notag\\
&\qquad + \frac{n}{2} \mathbb{E}\left[ \frac{1}{\fcnt{j}} \mathbf{1}(Z_i = Z_j, X_i = X_j) \mathbf{1}(Y_i = Y_j) \,\bigg|\, (Z_j, X_j, Y_j) \right].\label{eq:h1_conditional_expand}
\end{align}

\paragraph{First term in \eqref{eq:h1_conditional_expand}.}

Consider the first term. Since $(Z_j,X_j,Y_j)$ is fixed by conditioning and by \eqref{eq:degree-approx}, conditionally on $(Z_i,X_i)$, we have $\fcnt{i} \approx n \cdot p_{ZX}(Z_i,X_i)$, and 
\begin{align}
&\mathbb{E}\left[ \frac{1}{\fcnt{i}} \mathbf{1}(Z_i = Z_j, X_i = X_j) \mathbf{1}(Y_i = Y_j) \,\bigg|\, (Z_j, X_j, Y_j) \right]\notag\\
&\quad \approx  \frac{1}{n \cdot p_{ZX}(Z_j,X_j)}  \times P(Z_i = Z_j, X_i = X_j, Y_i = Y_j \mid Z_j, X_j, Y_j).\label{eq:term1-factor}
\end{align}
Next, we decompose the joint probability using the chain rule:
\begin{align}
&P(Z_i = Z_j, X_i = X_j, Y_i = Y_j \mid Z_j, X_j, Y_j) \notag\\
&\quad= P(Z_i = Z_j, X_i = X_j \mid Z_j, X_j, Y_j) \times P(Y_i = Y_j \mid Z_i = Z_j, X_i = X_j, Z_j, X_j, Y_j).\label{eq:prob-decomp1}
\end{align}
Under the null hypothesis $H_0: Y \perp\!\!\!\perp X \mid Z$, we have:
\begin{align}
P(Y_i = Y_j \mid Z_i = Z_j, X_i = X_j, Z_j, X_j, Y_j) &= P(Y_i = Y_j \mid Z_i = Z_j, Z_j, Y_j),
\label{eq:null-simplify1}
\end{align}
where the first equality follows from conditional independence $Y \perp\!\!\!\perp X \mid Z$, and the second uses the fact that when $Z_i=Z_j$, conditioning on $Z_i=Z_j$ and $Z_j$ is equivalent to conditioning on $Z_i$ and $Z_j$ separately.

\paragraph{Second term in \eqref{eq:h1_conditional_expand}.}

For the second term, note that $\fcnt{j}$ is measurable with respect to $(Z_j,X_j,Y_j)$ (in the sense that it depends on which other observations match $(Z_j,X_j)$, but conditionally on $(Z_j,X_j,Y_j)$ it is a known quantity). We have:
\begin{align}
\mathbb{E} &\left[ \frac{1}{\fcnt{j}} \mathbf{1}(Z_i = Z_j, X_i = X_j) \mathbf{1}(Y_i = Y_j) \,\bigg|\, (Z_j, X_j, Y_j) \right] \notag\\ &= \frac{1}{\fcnt{j}} P(Z_i = Z_j, X_i = X_j, Y_i = Y_j \mid Z_j, X_j, Y_j).\label{eq:term2-factor}
\end{align}
Using $\fcnt{j} \approx n \cdot p_{ZX}(Z_j,X_j)$ from \eqref{eq:degree-approx}, we obtain an expression similar to the first term.

\subsubsection{Simplification Under Conditional Independence}\label{simplifiction_CI}

We now simplify the probability term $P(Z_i = Z_j, X_i = X_j \mid Z_j, X_j, Y_j)$. Since $(X_i,Y_i,Z_i)$ are i.i.d., the event $(Z_i = Z_j, X_i = X_j)$ and the variable $Y_j$ are conditionally independent given $(Z_j,X_j)$:
\begin{equation}\label{eq:cond-indep-YjZiXi}
(Z_i = Z_j, X_i = X_j) \perp\!\!\!\perp Y_j \mid (Z_j, X_j).
\end{equation}
This follows because $(Z_i,X_i)$ is independent of $(Z_j,X_j,Y_j)$ under the i.i.d.\ sampling assumption. Therefore we get
\begin{align}
P(Z_i = Z_j, X_i = X_j \mid Z_j, X_j, Y_j) &= P(Z_i = Z_j, X_i = X_j \mid Z_j, X_j)= p_{ZX}(Z_j, X_j),\label{eq:prob-ZX-match}
\end{align}
where the last equality uses the definition of the marginal probability.

\subsubsection{Computing the Conditional Expectations}
Combining \eqref{eq:term1-factor}, \eqref{eq:prob-decomp1}, \eqref{eq:null-simplify1}, and \eqref{eq:prob-ZX-match}, we obtain:
\begin{align}
&\mathbb{E}\left[ \frac{1}{\fcnt{i}} \mathbf{1}(Z_i = Z_j, X_i = X_j) \mathbf{1}(Y_i = Y_j) \,\bigg|\, (Z_j, X_j, Y_j) \right]\notag\\
&\quad \approx \frac{1}{n \cdot p_{ZX}(Z_j,X_j)} \cdot p_{ZX}(Z_j,X_j) \cdot P(Y_i = Y_j \mid Z_i, Z_j, Y_j) = \frac{1}{n} P(Y_i = Y_j \mid Z_i = Z_j, Z_j, Y_j),\label{eq:h11-simplified}
\end{align}
where the conditioning on $Z_i$ is understood to be on $Z_i=Z_j$. Similarly, for the second term in \eqref{eq:h1_conditional_expand}, we get
\begin{equation}\label{eq:h12-simplified}
\mathbb{E}\left[ \frac{1}{\fcnt{j}} \mathbf{1}(Z_i = Z_j, X_i = X_j) \mathbf{1}(Y_i = Y_j) \,\bigg|\, (Z_j, X_j, Y_j) \right] \approx \frac{1}{n} P(Y_i = Y_j \mid Z_i = Z_j, Z_j, Y_j).
\end{equation}
Combining the above equations yields
\begin{equation}\label{eq:h1-conditional-final}
\mathbb{E}[h_1((Z_i, X_i, Y_i), (Z_j, X_j, Y_j)) \mid (Z_j, X_j, Y_j)] \approx P(Y_i = Y_j \mid Z_i = Z_j, Z_j, Y_j).
\end{equation}

\subsubsection{Analysis of $\mathbb{E}[h_2 \mid (Z_j,X_j,Y_j)]$}

By an entirely analogous computation, starting from \eqref{eq:h2-def}, we have
\begin{align}
&\mathbb{E}[h_2((Z_i, X_i, Y_i), (Z_j, X_j, Y_j)) \mid (Z_j, X_j, Y_j)] \notag\\
&\quad= \frac{n}{2} \mathbb{E}\left[ \frac{1}{\ccnt{i}} \mathbf{1}(Z_i = Z_j) \mathbf{1}(Y_i = Y_j) \,\bigg|\, (Z_j, X_j, Y_j) \right]  + \frac{n}{2} \mathbb{E}\left[ \frac{1}{\ccnt{j}} \mathbf{1}(Z_i = Z_j) \mathbf{1}(Y_i = Y_j) \,\bigg|\, (Z_j, X_j, Y_j) \right].\label{eq:h2_conditional_expand}
\end{align}
For the first term, using $\ccnt{i} \approx n \cdot p_Z(Z_i)$ from \eqref{eq:degree-approx}, and proceeding as before:
\begin{align}
&\mathbb{E}\left[ \frac{1}{\ccnt{i}} \mathbf{1}(Z_i = Z_j) \mathbf{1}(Y_i = Y_j) \,\bigg|\, (Z_j, X_j, Y_j) \right] \approx \frac{1}{n \cdot p_Z(Z_j)} \cdot P(Z_i = Z_j, Y_i = Y_j \mid Z_j, X_j, Y_j).\label{eq:h21-step1}
\end{align}
Decomposing the probability:
\begin{align}
P(Z_i = Z_j, Y_i = Y_j \mid Z_j, X_j, Y_j) &= P(Z_i = Z_j \mid Z_j, X_j, Y_j) \notag\\&\quad \times P(Y_i = Y_j \mid Z_i = Z_j, Z_j, X_j, Y_j).\label{eq:prob-decomp2}
\end{align}
By i.i.d.\ sampling and conditional independence \eqref{eq:cond-indep-YjZiXi} condition, we obtain
\begin{equation}\label{eq:prob-Z-match}
P(Z_i = Z_j \mid Z_j, X_j, Y_j) = P(Z_i = Z_j \mid Z_j) = p_Z(Z_j).
\end{equation}
Moreover, under $H_0$, we have $Y_i \perp\!\!\!\perp X_j \mid Z_i, Z_j, Y_j$ (by conditional independence and i.i.d.\ assumption), so
\begin{equation}\label{eq:null-simplify2}
P(Y_i = Y_j \mid Z_i = Z_j, Z_j, X_j, Y_j) = P(Y_i = Y_j \mid Z_i = Z_j, Z_j, Y_j).
\end{equation}
Combining \eqref{eq:h21-step1}--\eqref{eq:null-simplify2}, we get 
\begin{align}
&\mathbb{E}\left[ \frac{1}{\ccnt{i}} \mathbf{1}(Z_i = Z_j) \mathbf{1}(Y_i = Y_j) \,\bigg|\, (Z_j, X_j, Y_j) \right]\notag\\
&\quad \approx \frac{1}{n \cdot p_Z(Z_j)} \cdot p_Z(Z_j) \cdot P(Y_i = Y_j \mid Z_i = Z_j, Z_j, Y_j) = \frac{1}{n} P(Y_i = Y_j \mid Z_i = Z_j, Z_j, Y_j).\label{eq:h21-simplified}
\end{align}
The second term in \eqref{eq:h2_conditional_expand} yields the same expression. Therefore we arrive at
\begin{equation}\label{eq:h2-conditional-final}
\mathbb{E}[h_2((Z_i, X_i, Y_i), (Z_j, X_j, Y_j)) \mid (Z_j, X_j, Y_j)] \approx P(Y_i = Y_j \mid Z_i = Z_j, Z_j, Y_j).
\end{equation}

\subsubsection{Establishing Degeneracy}

From \eqref{eq:h_conditional}, \eqref{eq:h1-conditional-final}, and \eqref{eq:h2-conditional-final}, we conclude:
\begin{align}
\mathbb{E} &[h((Z_i, X_i, Y_i), (Z_j, X_j, Y_j)) \mid (Z_j, X_j, Y_j)]  \notag\\&\approx P(Y_i = Y_j \mid Z_i = Z_j, Z_j, Y_j) - P(Y_i = Y_j \mid Z_i = Z_j, Z_j, Y_j) = 0.\label{eq:degeneracy-conclusion}
\end{align}
To make this argument rigorous, we must show that the errors introduced by the approximations $\ccnt{i} \approx n \cdot p_Z(Z_i)$ and $\fcnt{i} \approx n \cdot p_{ZX}(Z_i,X_i)$ are negligible. By \eqref{eq:degreeLLN-a}--\eqref{eq:degreeLLN-b} and the continuous mapping theorem, we have:
\begin{equation}\label{eq:degree-ratio-convergence}
\sup_i \left| \frac{n}{\ccnt{i}} - \frac{1}{p_Z(Z_i)} \right| \xrightarrow{p} 0, \qquad \sup_i \left| \frac{n}{\fcnt{i}} - \frac{1}{p_{ZX}(Z_i,X_i)} \right| \xrightarrow{p} 0.
\end{equation}
Therefore, the conditional expectations in \eqref{eq:h11-simplified}, \eqref{eq:h12-simplified}, \eqref{eq:h21-simplified} converge to their limiting forms, and hence
\begin{equation}\label{eq:degeneracy-final}
\mathbb{E}[h((Z_i, X_i, Y_i), (Z_j, X_j, Y_j)) \mid (Z_j, X_j, Y_j)] \xrightarrow{p} 0 \quad \text{as } n \to \infty.
\end{equation}
Since this convergence holds in probability uniformly over the conditioning set, we have
\begin{equation}\label{eq:degeneracy-as}
\lim_{n\to\infty} \mathrm{Var}\big[\mathbb{E}[h((Z_i,X_i,Y_i),(Z_j,X_j,Y_j))\mid (Z_j,X_j,Y_j)]\big] = 0,
\end{equation}
which establishes first-order degeneracy of the kernel $h$ under $H_0$.

\subsection{Step 3: Application of Limiting Theory}

By \eqref{eq:degeneracy-as}, the kernel $h$ satisfies the conditions of Lemma~\ref{lemma:asymptotic_degenerate}. Let us denote $\theta_n = \mathbb{E}[h((Z_1,X_1,Y_1),(Z_2,X_2,Y_2))]$ which represents the mean of the kernel. Under $H_0$ and the degeneracy established above, Lemma~\ref{lemma:asymptotic_degenerate} applied to $U_n$ in \eqref{eq:ustat-form} yields
\begin{equation}\label{eq:un-limit}
n(U_n - \theta_n) \xRightarrow{d} \sum_{r=1}^\infty \lambda_r (Z_r^2 - 1),
\end{equation}
where $\{\lambda_r\}$ are the eigenvalues of the centered kernel operator $T_{h-\theta_n}$ on $L^2(P)$ (see Definition~\ref{def:hs_operator}).
Moreover, by \eqref{eq:degeneracy-conclusion}, we have $\theta_n \to 0$ as $n \to \infty$ under $H_0$. By Slutsky's theorem
\begin{equation}\label{eq:slutsky-step}
n U_n = n(U_n - \theta_n) + n\theta_n \xRightarrow{d} \sum_{r=1}^\infty \lambda_r (Z_r^2 - 1).
\end{equation}
Finally, since $n\Delta_{n} = (n-1)U_n = n U_n + O_p(U_n) = n U_n + o_p(1)$ (using the fact that $U_n = O_p(n^{-1})$ under $H_0$), we obtain
\begin{equation}\label{eq:final-limit}
n\,\Delta_{n} \xRightarrow{d} \sum_{r=1}^\infty \lambda_r (Z_r^2 - 1),
\end{equation}
which completes the proof of Theorem~\ref{thm:discrete}.

\begin{remark}
The key insight in the proof is that under $H_0$, the refinement from $Z$-neighborhoods to $(Z,X)$-neighborhoods provides no additional information about $Y$-agreement beyond what is already captured by $Z$-neighborhoods alone. This is formalized through the cancellation in \eqref{eq:degeneracy-conclusion}, which occurs precisely because of the conditional independence $Y \perp\!\!\!\perp X \mid Z$.
\end{remark}

\begin{remark}
The data-dependent nature of the neighborhoods (through $\ccnt{i}$ and $\fcnt{i}$) introduces complex dependencies between observations. However, Conditions \eqref{eq:degreeLLN-a}--\eqref{eq:degreeLLN-b} ensure that these degrees concentrate around their population analogs, allowing us to treat them as approximately deterministic in the asymptotic analysis. This is a key technical point that distinguishes our approach from classical $U$-statistic theory with fixed kernels.
\end{remark}

\subsection{Technical Lemmas}

For completeness, we state a technical lemma used in the proof.

\begin{lemma}[Conditional Independence Under i.i.d.\ Sampling]\label{lemma:general_cond_indep}
Let $\{(X_i,Y_i,Z_i)\}_{i=1}^n$ be i.i.d.\ samples from distribution $P$. For any measurable function $\phi:\mathcal{X} \times \mathcal{Y} \to \mathbb{R}$ and any $i \neq j$, we have
\begin{equation}\label{eq:general_cond_indep}
\phi(X_j, Y_j) \perp\!\!\!\perp Z_i \mid Z_j.
\end{equation}
In particular,
\begin{equation}\label{eq:general_cond_indep_specific}
P(\phi(X_j, Y_j) \mid Z_i = Z_j, Z_j) = P(\phi(X_j, Y_j) \mid Z_j).
\end{equation}
\end{lemma}

\begin{proof}
This follows immediately from the i.i.d.\ assumption: $(X_j,Y_j,Z_j)$ and $Z_i$ are independent random variables, hence $(X_j,Y_j)$ and $Z_i$ are conditionally independent given $Z_j$. The second statement follows by taking $\phi$ to be the indicator of any measurable event.
\end{proof}


\subsection{Extension to High-Dimensional Discrete Conditioning}\label{high_dimensional_extension}
The derivation above assumes $Z$ is a univariate discrete variable, but the logic extends naturally to the case where $Z$ is a high-dimensional random vector $\mathbf{Z} = (Z^{(1)}, \dots, Z^{(d)}) \in \mathcal{Z}^d$ with countable support. In this multivariate setting, the neighborhoods $\coarse{i}$ and $\fine{i}$ are defined by the intersection of exact matches across all dimensions of $\mathbf{Z}$ (and $X$). 

The validity of Theorem~\ref{thm:discrete} in high dimensions relies on the ``fixed support'' asymptotic regime. Specifically, provided that the joint support of $(\mathbf{Z}, \mathbf{X})$ remains fixed while $n \to \infty$, the counts in each occupied cell will eventually satisfy the divergence conditions $\ccnt{i} \xrightarrow{p} \infty$ and $\fcnt{i} \xrightarrow{p} \infty$. Under these conditions, the concentration of the empirical degrees $\ccnt{i}/n$ to the population masses $p_{\mathbf{Z}}(\mathbf{z})$ holds regardless of the dimension $d$.

However, it is important to note that for finite sample sizes, the ``curse of dimensionality'' may cause many cells to be empty or sparsely populated if $d$ is large relative to $\log n$. In such regimes, the convergence rates in \eqref{eq:degreeLLN-b} may slow down, but the asymptotic degeneracy of the $U$-statistic kernel—which drives the null distribution—remains a structural property of the conditional independence hypothesis $Y \perp\!\!\!\perp X \mid \mathbf{Z}$, invariant to the dimensionality of the conditioning set.

\section{Proofs of Regime I -- Y continuous Case}\label{app:xzdiscrete_y_continuous}
In this section, we establish the asymptotic distribution of $\Delta_{n}$ when variables $X$ and $Z$ are discrete, but $Y$ are continuous. The main difference between this scenario and Appendix~\ref{app:discrete_proof} lies in the choice of kernel function $K$. As $Y$ is continuous, we employ a general bounded positive definite kernel $K: \mathcal{Y} \times \mathcal{Y} \to \mathbb{R}$ (e.g., the Gaussian kernel $K(y, y') = \exp(-\|y-y'\|^2/2\sigma^2)$) instead of the exact matching indicator. The neighborhoods for $Z$ and $(Z,X)$ remain defined by exact matching on the discrete covariates as in the original formulation. 

\subsection{Main Result}
\begin{theorem}[Asymptotic Distribution for $X,Z$ discrete, $Y$ continuous Case]
\label{thm:xzdiscrete_y_continuous}
Let $(X,Y,Z)$ be random variables with finite or countable supports, and let $\{(X_i,Y_i,Z_i)\}_{i=1}^n$ be i.i.d.\ samples from the joint distribution $P$. Set $X,Z$ to be discrete variables, while $Y$ continuous. Define $\Delta_{n}$ as in \eqref{eq:test_stats}. Suppose that as $n\to\infty$:
\begin{align}
&\min_{1\le i\le n} \ccnt{i} \;\xrightarrow{p}\; \infty,\qquad \min_{1\le i\le n} \fcnt{i} \;\xrightarrow{p}\; \infty, \label{eq:degreeLLN-a-xzdiscrete}\\
&\sup_{i}\Bigl|\frac{\ccnt{i}}{n}-p_Z(z_i)\Bigr| \xrightarrow{p} 0, \qquad
\sup_{i}\Bigl|\frac{\fcnt{i}}{n}-p_{ZX}(z_i,x_i)\Bigr| \xrightarrow{p} 0,\label{eq:degreeLLN-b-xzdiscrete}
\end{align}
where $p_Z(z)=P(Z=z)$ and $p_{ZX}(z,x)=P(Z=z,X=x)$ denote the marginal probabilities. Then, under the null hypothesis $H_0: Y \perp\!\!\!\perp X \mid Z$,
\begin{equation}\label{eq:limit-dist-xzdiscrete}
n\,\Delta_{n}\;\xRightarrow{d}\; \sum_{r=1}^\infty \lambda_r\,(Z_r^2-1),
\end{equation}
where $\{Z_r\}_{r=1}^\infty$ are i.i.d.\ standard normal random variables and $\{\lambda_r\}_{r=1}^\infty$ are the eigenvalues of the centered kernel operator associated with $\Delta_{n}$ under $H_0$ (see Definition~\ref{def:hs_operator}).
\end{theorem}

\begin{remark}
    Similarly, Conditions \eqref{eq:degreeLLN-a-xzdiscrete} and \eqref{eq:degreeLLN-b-xzdiscrete} hold under standard positivity where the joint distribution $P$ assigns strictly positive probability to all cells in the support, as changing the category of $Y$ would not affect the properties of the neighborhood.
\end{remark}

\subsection{Proof Strategy}
The derivation of the asymptotic distribution for the continuous-$Y$ regime follows the same architectural framework as the all-discrete case, relying on the theory of degenerate $U$-statistics. The structural decomposition of the test statistic based on discrete $Z$-strata remains invariant. The primary modification lies in the substitution of the exact matching indicator with a continuous kernel function. The proof proceeds in three steps:

\begin{enumerate}[label=(\Roman*).]
    \item \textit{Kernelized $U$-statistic Representation:} 
    We generalize the representation of $\Delta_{n}$ by replacing the discrete Dirac kernel $\mathbf{1}(Y_i = Y_j)$ with a bounded positive definite kernel $K(Y_i, Y_j)$. We show that despite this substitution, the statistic retains the form of a symmetric order-2 $U$-statistic, denoted as $U_n^K$, whose weights depend solely on the discrete cell counts of $(Z, X)$.

    \item \textit{Degeneracy via Conditional Integration:} 
    We establish the first-order degeneracy of the modified kernel under $H_0$:
    \[
    \mathbb{E}[h_K((Z_i,X_i,Y_i),(Z_j,X_j,Y_j))\mid (Z_j,X_j,Y_j)] = 0 \quad \text{almost surely}.
    \]

    \item \textit{Spectral Limit using Integral Operators:} 
    With degeneracy established, we invoke the asymptotic theory for $U$-statistics with general kernels. The limiting distribution remains a weighted sum of chi-squared variables, $\sum \lambda_r (Z_r^2 - 1)$. The key distinction is that $\{\lambda_r\}$ are now eigenvalues of a Hilbert--Schmidt integral operator acting on $L^2(\mathcal{Y})$, rather than a finite-dimensional matrix.
\end{enumerate}

\subsection{Step 1: Representation as a $U$-statistic}
To accommodate continuous outcomes $Y \in \mathcal{Y}$, the test statistic $\Delta_{n}$ becomes:
\begin{align*}
\Delta_{n} &= \frac{1}{n}\sum_{i=1}^n \frac{1}{\fcnt{i}} \sum_{j \in \fine{i}} K(Y_i, Y_j) - \frac{1}{n}\sum_{i=1}^n \frac{1}{\ccnt{i}} \sum_{j \in \coarse{i}} K(Y_i, Y_j).
\end{align*}
Correspondingly, we redefine the symmetric kernel function $h:\mathcal{W}^2 \to \mathbb{R}$ as:
\begin{align}
h((Z_i, X_i, Y_i), (Z_j, X_j, Y_j)) &= \frac{n}{2} \left( \frac{1}{\fcnt{i}} \mathbf{1}(Z_i = Z_j, X_i = X_j) + \frac{1}{\fcnt{j}} \mathbf{1}(Z_i = Z_j, X_i = X_j) \right) K(Y_i, Y_j)\notag\\
&\quad - \frac{n}{2} \left( \frac{1}{\ccnt{i}} \mathbf{1}(Z_i = Z_j) + \frac{1}{\ccnt{j}} \mathbf{1}(Z_i = Z_j) \right) K(Y_i, Y_j).\label{eq:def_h_kernel}
\end{align}
The structure of the $U$-statistic representation $U_n$ remains identical to \eqref{eq:ustat-form}, with the indicator $\mathbf{1}(Y_i=Y_j)$ replaced by $K(Y_i, Y_j)$. The scaling factor $n$ ensures the operator remains non-vanishing in the limit.

\subsection{Step 2: First-Order Degeneracy under $H_0$}
To apply Lemma~\ref{lemma:asymptotic_degenerate}, we need to establish the first-order degeneracy of the modified kernel \eqref{eq:def_h_kernel} under $H_0$. The decomposition $h = h_1 - h_2$ follows \eqref{eq:h1-def} and \eqref{eq:h2-def}, substituting the kernel $K(Y_i, Y_j)$ for the indicator. To analyze the conditional expectation $\mathbb{E}[h((Z_i, X_i, Y_i), (Z_j, X_j, Y_j)) \mid (Z_j, X_j, Y_j)]$, we would analyze two components $h_1$ and $h_2$ separately. It is worth noting that asymptotic approximations also holds:
\begin{align}
\ccnt{i} &\approx n \cdot p_Z(Z_i), \qquad \fcnt{i} \approx n \cdot p_{ZX}(Z_i,X_i),\label{eq:degree-approx-xz-discrete}
\end{align}
following same arguments in Appendix~\ref{app:discrete_proof} with Conditions \eqref{eq:degreeLLN-a-xzdiscrete} and \eqref{eq:degreeLLN-b-xzdiscrete}.

\subsubsection{Analysis of $\E[h_1\mid(Z_j,X_j,Y_j)]$}
We begin by analyzing the first component. We obtain
\begin{align}
    \E[h_1((Z_i, X_i, Y_i),& (Z_j, X_j, Y_j)) \mid (Z_j, X_j, Y_j)] \notag
    \\&= \frac{n}{2} \Bigg(  \E\left[\frac{1}{\fcnt{i}} \mathbf{1}(Z_i = Z_j, X_i = X_j)K(Y_i , Y_j) \mid Z_j, X_j, Y_j\right] \notag \\
    &\quad + \E\left[\frac{1}{\fcnt{j}}\mathbf{1}(Z_i = Z_j,X_i=X_j)  K(Y_i , Y_j) \mid Z_j, X_j, Y_j\right] \Bigg) \label{eq:h1_conditional2}
\end{align}

\paragraph{First term in \eqref{eq:h1_conditional2}}
Consider the first term. Since $(Z_j, X_j, Y_j)$ is fixed by conditioning and by \eqref{eq:degree-approx-xz-discrete}, we have $\fcnt{i} \approx n \cdot p_{ZX}(Z_i,X_i)$
\begin{align}
    &\E\left[\frac{1}{\fcnt{i}}\mathbf{1}(Z_i =Z_j, X_i = X_j)K(Y_i,Y_j) \mid Z_j,X_j,Y_j)\right]\notag \\
    &\quad \approx \frac{1}{ n \cdot p_{ZX}(Z_i,X_i)} \times \E[\mathbf{1}(Z_i =Z_j, X_i = X_j)K(Y_i,Y_j) \mid Z_j, X_j,Y_j]\label{eq:h1_degree_simplification}
\end{align}

Next, we decompose the joint probability using Lemma~\ref{lem:indicator_expansion}:
\begin{align}
    &\E[\mathbf{1}(Z_i =Z_j, X_i = X_j)K(Y_i,Y_j) \mid Z_j, X_j,Y_j] \notag \\
    &\quad = P(Z_i =Z_j, X_i = X_j \mid Z_j,X_j,Y_j) \cdot \E[K(Y_i,Y_j) \mid Z_i = Z_j, X_i = X_j, Z_j, X_j,Y_j]
    \label{h1_chain_rule_simplification}
\end{align}

Under the null hypothesis $H_0: X \indep Y \mid Z$, we have:
\begin{align}
    \E[K(Y_i,Y_j) \mid Z_i = Z_j, X_i = X_j, Z_j, X_j,Y_j] &= \E[K(Y_i,Y_j) \mid Z_i = Z_j, Z_j, Y_j] 
    \label{h1_independence_simplification}
\end{align}
where the 
equality follows from the conditional independence $H_0: X \indep Y \mid Z$.

\paragraph{Second term in \eqref{eq:h1_conditional2}}
For the second term, note that $\fcnt{j}$ is measurable here. Therefore, following the same simplification procedure, we have:
\begin{align}
    &\E\left[\frac{1}{\fcnt{j}}\mathbf{1}(Z_i =Z_j, X_i = X_j)K(Y_i,Y_j) \mid Z_j,X_j,Y_j)\right] \notag \\
    &\quad = \frac{1}{\fcnt{j}} \times P(Z_i =Z_j, X_i = X_j \mid Z_j,X_j,Y_j) \cdot \E[K(Y_i,Y_j) \mid Z_i=Z_j ,Z_j,Y_j].
\end{align}

By applying the same arguments as in Appendix~\ref{simplifiction_CI}, we could simplify the probability term $P(Z_i =Z_j, X_i = X_j \mid Z_j,X_j,Y_j)$:
\begin{align}
    P(Z_i =Z_j, X_i = X_j \mid Z_j,X_j,Y_j) = P(Z_i =Z_j, X_i = X_j \mid Z_j,X_j) = p_{ZX}(Z_j,X_j). \label{h1_probability_simplification}
\end{align}

\subsubsection{Computing the Conditional Expectation}
Combining \eqref{eq:h1_degree_simplification}, \eqref{h1_chain_rule_simplification}, \eqref{h1_independence_simplification}, and \eqref{h1_probability_simplification}, we obtain:
\begin{align}\label{eq:cont_y_h1_simplification_term_1}
    &\E\left[\frac{1}{\fcnt{i}}\mathbf{1}(Z_i =Z_j, X_i = X_j)K(Y_i,Y_j) \mid Z_j,X_j,Y_j)\right]\notag \\
    &\quad \approx \frac{1}{ n \cdot p_{ZX}(Z_j,X_j)}\cdot p_{ZX}(Z_j,X_j) \cdot E[K(Y_i,Y_j) \mid Z_i=Z_j ,Z_j,Y_j] \notag
    \\& \quad = \frac{1}{ n} \cdot E[K(Y_i,Y_j) \mid Z_i=Z_j,Z_j,Y_j]
\end{align}
Similarly, for the second term in \eqref{eq:h1_conditional2}, we obtain
\begin{align}\label{eq:cont_y_h1_simplification_term_2}
    &\E\left[\frac{1}{\fcnt{j}}\mathbf{1}(Z_i =Z_j, X_i = X_j)K(Y_i,Y_j) \mid Z_i=Z_j,X_j,Y_j)\right] \notag \\&
    \quad \approx \frac{1}{n} \cdot E[K(Y_i,Y_j) \mid Z_i=Z_j,Z_j,Y_j] 
\end{align}
Combining the above displays 
Substituting \eqref{eq:cont_y_h1_simplification_term_1} and \eqref{eq:cont_y_h1_simplification_term_2} into \eqref{eq:h1_conditional2} yields
\begin{align}
    \E[h_1((Z_i, X_i, Y_i), (Z_j, X_j, Y_j)) \mid (Z_j, X_j, Y_j)] \approx E[K(Y_i,Y_j) \mid Z_i=Z_j ,Z_j,Y_j] \label{h1_final expression}
\end{align}

\subsubsection{Analysis of $\E[h_2\mid(Z_j,X_j,Y_j)]$}
By an entirely analogous computation, we have
\begin{align}
    \E[h_2((Z_i, X_i, Y_i)&, (Z_j, X_j, Y_j)) \mid (Z_j, X_j, Y_j)]
    = \frac{n}{2} \Bigg(  \E\left[\frac{1}{\ccnt{i}} \mathbf{1}(Z_i = Z_j)K(Y_i , Y_j) \mid Z_j, X_j, Y_j\right] 
    \notag \\&\quad 
    + \E\left[\frac{1}{\ccnt{j}}\mathbf{1}(Z_i = Z_j)  K(Y_i , Y_j) \mid Z_j, X_j, Y_j\right] \Bigg) \label{eq:h2_conditional2}
\end{align}
For the first term, we have $\ccnt{i} \approx p_Z(Z_j)$, and proceed as before:
\begin{align}
    \E &\left[\frac{1}{\ccnt{i}} \mathbf{1}(Z_i = Z_j)K(Y_i , Y_j) \mid Z_j, X_j, Y_j\right] \notag \\ &
    \approx \frac{1}{p_Z(Z_j)} \E[\mathbf{1}(Z_i = Z_j)K(Y_i , Y_j) \mid Z_j, X_j, Y_j]\notag \\
    &\quad = \frac{1}{p_Z(Z_j)}\cdot P(Z_i =Z_j \mid Z_j,X_j,Y_j) \cdot \E[K(Y_i,Y_j) \mid Z_i = Z_j, Z_j, X_j,Y_j]\notag \\
    &\qquad = \frac{1}{p_Z(Z_j)} \cdot p_Z(Z_j) \cdot \E[K(Y_i,Y_j) \mid Z_i=Z_j , Z_j,Y_j].
\end{align}

The second term yields the same expression. Therefore, we arrive at
\begin{align}
    \E[h_2((Z_i, X_i, Y_i), (Z_j, X_j, Y_j)) \mid (Z_j, X_j, Y_j)] \approx E[K(Y_i,Y_j) \mid Z_i=Z_j,Z_j,Y_j]. \label{h2_final_expression}
\end{align}

\subsection{Establishing Degeneracy and Application of Limiting Theorem}
From \eqref{eq:h_conditional}, \eqref{h1_final expression}, \eqref{h2_final_expression}, we conclude:
\begin{align}
    &\E[h((Z_i, X_i, Y_i), (Z_j, X_j, Y_j)) \mid (Z_j, X_j, Y_j)] \notag \\&\quad \approx E[K(Y_i,Y_j) \mid Z_i=Z_j ,Z_j,Y_j] - E[K(Y_i,Y_j) \mid Z_i=Z_j ,Z_j,Y_j] = 0.
\end{align}
Similarly, we can demonstrate that the errors introduced by the approximations $\ccnt{i} \approx n \cdot p_Z(Z_i)$ and $\fcnt{i} \approx n \cdot p_{ZX}(Z_i,X_i)$ are negligible. By \eqref{eq:degreeLLN-a-xzdiscrete}--\eqref{eq:degreeLLN-b-xzdiscrete} and the continuous mapping theorem, we have:
\begin{equation}\label{eq:degree-ratio-convergence-xzdiscrete}
\sup_i \left| \frac{n}{\ccnt{i}} - \frac{1}{p_Z(Z_i)} \right| \xrightarrow{p} 0, \qquad \sup_i \left| \frac{n}{\fcnt{i}} - \frac{1}{p_{ZX}(Z_i,X_i)} \right| \xrightarrow{p} 0.
\end{equation}
Therefore, the conditional expectations in \eqref{h1_final expression} and \eqref{h2_final_expression} converge to their limiting forms, and hence
\begin{equation}\label{eq:degeneracy-final-xzdiscrete}
\mathbb{E}[h((Z_i, X_i, Y_i), (Z_j, X_j, Y_j)) \mid (Z_j, X_j, Y_j)] \xrightarrow{p} 0 \quad \text{as } n \to \infty.
\end{equation}
Since this convergence holds in probability uniformly over the conditioning set, we have
\begin{equation}\label{eq:degeneracy-as-xzdiscrete}
\lim_{n\to\infty} \mathrm{Var}\big[\mathbb{E}[h((Z_i,X_i,Y_i),(Z_j,X_j,Y_j))\mid (Z_j,X_j,Y_j)]\big] = 0,
\end{equation}
which establishes first-order degeneracy of the kernel $h$ under $H_0$.

This concludes the proof as it implies that $\mathbb{E}[h((Z_i,X_i,Y_i),(Z_j,X_j,Y_j))\mid (Z_j,X_j,Y_j)]=0$ under $H_0$; hence $h$ is first-order degenerate. We can then apply Lemma~\ref{lemma:asymptotic_degenerate} to obtain Theorem~\ref{thm:xzdiscrete_y_continuous}.

\begin{remark}
    In this scenario, $Z$ could be extended to a high-dimensional vector using exactly analogous arguments as discussed in Remark~\ref{high_dimensional_extension}.
\end{remark}

\subsection{Technical Lemmas}
For completeness, we state a technical lemma used in the proof.
\begin{lemma}[Indicator decomposition of conditional expectation]
\label{lem:indicator_expansion}
Let $(\Omega,\mathcal F,\mathbb P)$ be a probability space.  
Let $A\in\mathcal F$ be an event, $B$ an integrable random variable, and 
$\mathcal G\subseteq\mathcal F$ a sub-$\sigma$-algebra. Define the conditional
probability $\mathbb{P}(A\mid\mathcal G)$ and the conditional expectation
on the event by
\[
\mathbb{E}[B\mid A,\mathcal G] :=
\begin{cases}
\dfrac{\mathbb{E}[B\mathbf{1}_A\mid\mathcal G]}{\mathbb{P}(A\mid\mathcal G)}, 
&\text{on }\{\mathbb{P}(A\mid\mathcal G)>0\},\\[6pt]
0, &\text{on }\{\mathbb{P}(A\mid\mathcal G)=0\}.
\end{cases}
\]
Then almost surely
\begin{equation}\label{eq:indicator_expansion}
\mathbb{E}[\mathbf{1}_A\,B\mid\mathcal G]
= \mathbb{P}(A\mid\mathcal G)\;\mathbb{E}[B\mid A,\mathcal G].
\end{equation}
\end{lemma}

\begin{proof}
By the definition given above, on the set $\{\mathbb{P}(A\mid\mathcal G)>0\}$
we have
\[
\mathbb{E}[B\mathbf{1}_A\mid\mathcal G]
= \mathbb{P}(A\mid\mathcal G)\cdot
  \frac{\mathbb{E}[B\mathbf{1}_A\mid\mathcal G]}{\mathbb{P}(A\mid\mathcal G)}
= \mathbb{P}(A\mid\mathcal G)\,\mathbb{E}[B\mid A,\mathcal G].
\]
On the complementary set $\{\mathbb{P}(A\mid\mathcal G)=0\}$, we have
$\mathbb{E}[B\mathbf{1}_A\mid\mathcal G]=0$ by the defining property of
conditional expectation (since $B\mathbf{1}_A=0$ a.s.\ conditional on
$\mathcal G$ there), and the right-hand side is also $0$ by the definition
of $\mathbb{E}[B\mid A,\mathcal G]$ on this set. Hence, the equality
\eqref{eq:indicator_expansion} holds almost surely.
\end{proof}

\section{Asymptotic Theory in Regime II: The All-Continuous Case}
\label{app:continuous_proof}

In this section we develop the complete null limit theory for a
local-polynomial debiased version of the sample-dependent
$k_n$-nearest-neighbor statistic in the all-continuous setting. The
underlying raw statistic is
\begin{equation}
\label{eq:Delta_n_def_regime_II}
\Delta_n
\;=\;
\frac{1}{n}\sum_{i=1}^n \frac{1}{k_n}
\sum_{j\in \fine{i}} K(Y_i, Y_j)
\;-\;
\frac{1}{n}\sum_{i=1}^n \frac{1}{k_n}
\sum_{j\in \coarse{i}} K(Y_i, Y_j),
\end{equation}
where $\fine{i}$ denotes the directed $k_n$-nearest-neighbor set of
$(Z_i, X_i)$ in $\mathbb{R}^D$ with $D := d_Z + d_X$, and $\coarse{i}$
denotes the directed $k_n$-nearest-neighbor set of $Z_i$ in
$\mathbb{R}^{d_Z}$. Throughout the section,
\[
W_i \;=\; (Z_i, X_i, Y_i), \qquad
\mathcal{W}_n \;=\; \{W_1, \ldots, W_n\},
\]
and $P$ denotes the joint law of $W = (Z, X, Y)$ on
$\mathbb{R}^{d_Z}\times\mathbb{R}^{d_X}\times\mathcal{Y}$.

\subsection{The problem with the raw statistic and the role of debiasing}
\label{subsec:debiasing_motivation}

The statistic \eqref{eq:Delta_n_def_regime_II} is not a fixed-kernel
$U$-statistic. The events $j \in \fine{i}$ and $j \in \coarse{i}$ depend
on the entire sample through rank operations on $(Z, X)$- and $Z$-space
respectively. The appropriate asymptotic framework is therefore that of
\emph{stabilizing geometric graph functionals} in the tradition of
\citet{PenroseYukich2001}, \citet{penrose2003random}, and
\citet{LastPenrose2018}, adapted to growing $k_n$ via rescaled
stabilization.

A direct analysis of $\Delta_n$ under this framework yields a Gaussian
limit at the $\sqrt{n}$ scale,
$\sqrt{n}(\Delta_n - \theta_n)\Rightarrow N(0, \tau_{\mathrm{graph}}^2)$,
where $\theta_n := \mathbb{E}[\Delta_n]$ is the null centering and
$\tau_{\mathrm{graph}}^2$ is the overlap-variance constant. However, the
centering $\theta_n$ presents a fundamental practical obstacle: even
under $H_0$, $\theta_n$ does not vanish exactly. The Taylor-expansion
argument detailed in Section~\ref{subsec:bias_cancellation} establishes
$\theta_n = O((k_n/n)^{2/D})$, so $\theta_n \to 0$ as $n\to\infty$, but
the question of whether the rescaled centering $\sqrt{n}\theta_n$ remains
controlled at the inferential scale is more delicate. For $k_n = n^\alpha$,
\[
\sqrt{n}\,\theta_n
\;\asymp\;
n^{1/2 - 2(1 - \alpha)/D}
\;=\;
n^{1/2 - 2/D + 2\alpha/D},
\]
which diverges whenever $\alpha > 1 - D/4$, i.e., for any $\alpha > 0$
when $D \ge 4$ and for moderately large $\alpha$ in intermediate
dimensions. This divergent bias makes the test statistic
$\sqrt{n}\Delta_n$ unsuitable for direct hypothesis testing without
either (i) explicit bias correction or (ii) a calibration scheme that
absorbs the bias internally.

We adopt the local-polynomial debiasing approach of (i): we replace each
of the two simple inner averages by a local-polynomial regression
intercept of degree $p$. The polynomial reproduction identities then
annihilate the smoothing bias up to order $p$, and under a corresponding
bias killing condition on $k_n$, the rescaled centering of the
debiased statistic satisfies $\sqrt{n}\theta_n^{(p)} \to 0$.

\subsection{The debiased statistic}
\label{subsec:debiased_statistic_def}

We work in scaled local-polynomial coordinates throughout. Define the
fine and coarse \emph{population radii}
\begin{equation}
\label{eq:rho_F_def}
\rho_F(z, x) \;:=\; \left(\frac{k_n}{n v_D f_{Z,X}(z, x)}\right)^{1/D},
\qquad
\rho_C(z) \;:=\; \left(\frac{k_n}{n v_{d_Z}f_Z(z)}\right)^{1/d_Z},
\end{equation}
 of $k_n$ nearest neighbor $(z,x)$ and $z$ respectively and $v_d = \pi^{d/2}/\Gamma(d/2 + 1)$ is the unit-ball volume in
$\mathbb{R}^d$. Let $\hat{\rho}_{F,n}(z, x)$ and $\hat{\rho}_{C,n}(z)$ be their sample version. The fine and coarse rescaled offsets are
\begin{equation}
\label{eq:U_offsets}
U_{ij, F} \;:=\; \frac{Z_j - Z_i}{\hat{\rho}_{F,n}(Z_i, X_i)},
\quad j\in\fine{i};
\qquad
U_{ij, C} \;:=\; \frac{Z_j - Z_i}{\hat{\rho}_{C,n}(Z_i)},
\quad j\in\coarse{i}.
\end{equation}

For each $p \in \mathbb{N}$, let $\mathcal{A}_p := \{\alpha \in \mathbb{N}^{d_Z} : |\alpha| \le p\}$
denote the set of multi-indices of degree at most $p$. Write
$N_p := |\mathcal{A}_p| = \binom{d_Z + p}{p}$, and let
\[
q_p(u) := (u^\alpha)_{\alpha\in\mathcal{A}_p} \in \mathbb{R}^{N_p}
\]
be the corresponding monomial vector. Let $e_0\in\mathbb{R}^{N_p}$
denote the coordinate vector selecting the constant monomial. The scaled fine and coarse design matrices are
\begin{equation}
\label{eq:M_design_def}
\mathcal{M}_{i, F}^{(p)} \;:=\; \frac{1}{k_n}\sum_{j\in\fine{i}}q_p(U_{ij, F})q_p(U_{ij, F})^\top,
\qquad
\mathcal{M}_{i, C}^{(p)} \;:=\; \frac{1}{k_n}\sum_{j\in\coarse{i}}q_p(U_{ij, C})q_p(U_{ij, C})^\top.
\end{equation}
When these matrices are invertible, the corresponding \emph{equivalent
kernels} (or local-polynomial weights) are
\begin{equation}
\label{eq:weights_def}
w_{ij, F}^{(p)} \;:=\; \frac{1}{k_n}e_0^\top(\mathcal{M}_{i, F}^{(p)})^{-1}q_p(U_{ij, F}),
\qquad
w_{ij, C}^{(p)} \;:=\; \frac{1}{k_n}e_0^\top(\mathcal{M}_{i, C}^{(p)})^{-1}q_p(U_{ij, C}).
\end{equation}
The fine and coarse \emph{local-polynomial intercepts} are
\begin{equation}
\label{eq:intercepts_def}
\widehat a_{i, F}^{(p)} \;:=\; \sum_{j\in\fine{i}}w_{ij, F}^{(p)}K(Y_i, Y_j),
\qquad
\widehat a_{i, C}^{(p)} \;:=\; \sum_{j\in\coarse{i}}w_{ij, C}^{(p)}K(Y_i, Y_j).
\end{equation}
The \emph{$p$-th order local-polynomial debiased statistic} is
\begin{equation}
\label{eq:Delta_p_def}
\Delta_n^{(p)} \;:=\; \frac{1}{n}\sum_{i=1}^n\bigl(\widehat a_{i, F}^{(p)} - \widehat a_{i, C}^{(p)}\bigr).
\end{equation}
For $p = 0$, the design matrices reduce to scalars equal to $1$, the
weights become $w_{ij, F}^{(0)} = w_{ij, C}^{(0)} = 1/k_n$, and
$\Delta_n^{(0)}$ equals the raw statistic
\eqref{eq:Delta_n_def_regime_II}.

\subsection{Standing assumptions}
\label{subsec:assumptions_regime_II}

We work throughout under the following assumptions, which combine
smoothness, null structure, kernel boundedness, growth rate of $k_n$,
and a dimensional condition required for the strict positivity of the
limiting variance.

\begin{assumption}[Smooth continuous model]
\label{asn:smooth_continuous_model}
The random vector $(Z, X, Y) \in \mathbb{R}^{d_Z}\times\mathbb{R}^{d_X}\times\mathcal{Y}$
admits a joint density with respect to an appropriate product measure.
The marginal densities $f_Z$ and $f_{Z,X}$ are supported on compact
convex sets $\mathcal{S}_Z\subset\mathbb{R}^{d_Z}$ and
$\mathcal{S}_{Z,X}\subset\mathbb{R}^D$ with $C^2$ boundary, are bounded
above and below by positive constants $\underline f, \overline f$ on
their respective supports, and are $C^{p + 1}$ with uniformly bounded
derivatives up to order $p + 1$.
\end{assumption}

\begin{assumption}[Null hypothesis and continuous mark law]
\label{asn:null_continuous}
Under $H_0$, $Y \perp\!\!\!\perp X \mid Z$. The conditional law $P_{Y|Z=z}$
varies continuously in $z$ against bounded test functions, and the
conditional kernel mean
\begin{equation}
\label{eq:kappa_def}
\kappa(y, z) \;:=\; \mathbb{E}[K(y, Y') \mid Z' = z]
\end{equation}
is $C^{p + 1}$ in $z$ with derivatives uniformly bounded in $y$.
\end{assumption}

\begin{assumption}[Kernel regularity]
\label{asn:kernel_regularity_corrected}
The kernel $K : \mathcal{Y}\times\mathcal{Y}\to\mathbb{R}$ is bounded and
symmetric:
\[
\sup_{y, y'\in\mathcal{Y}}|K(y, y')| \;\le\; M_K \;<\; \infty,
\qquad K(y, y') = K(y', y).
\]
\end{assumption}

\begin{assumption}[Growth of $k_n$]
\label{asn:kn_growth_corrected}
The sequence $k_n$ satisfies
\[
k_n \to \infty, \qquad k_n / n \to 0, \qquad k_n/(N_p\log n) \to \infty.
\]
The last condition strengthens the standard $k_n/\log n \to \infty$ in
order to handle uniform convergence of the local design matrices over
$n$ anchors and $N_p^2$ matrix entries.
\end{assumption}

\begin{assumption}[Dimensional non-degeneracy]
\label{asn:dimensional_nondegeneracy}
$d_X \ge 1$.
\end{assumption}

Assumption~\ref{asn:dimensional_nondegeneracy} is essential for the
strict positivity of the limiting overlap variance
$\tau_p^2$. When $d_X = 0$, the fine and coarse neighborhoods coincide
in distribution and the statistic is identically zero; this trivial
case is excluded.

\subsection{Local score decomposition}
\label{subsec:local_score_decomposition}

The debiased statistic admits the canonical decomposition into a first
projection contribution and a centered graph fluctuation. Define the
local graph score
\begin{equation}
\label{eq:xi_def}
\xi_{n, p}(W_i, \mathcal{W}_n)
\;:=\;
\widehat a_{i, F}^{(p)} - \widehat a_{i, C}^{(p)},
\end{equation}
so that $\Delta_n^{(p)} = n^{-1}\sum_i \xi_{n, p}(W_i, \mathcal{W}_n)$.
For $w = (z, x, y)$, define the conditional mean score
\begin{equation}
\label{eq:m_def}
m_{n, p}(w) \;:=\; \mathbb{E}\bigl[\xi_{n, p}(w, \{w\}\cup\{W_2, \ldots, W_n\})\bigr],
\end{equation}
where $W_2, \ldots, W_n$ are i.i.d. with law $P$ independent of $w$. The
\emph{centering} is
\begin{equation}
\label{eq:theta_def}
\theta_n^{(p)} \;:=\; \mathbb{E}[m_{n, p}(W_1)] \;=\; \mathbb{E}[\Delta_n^{(p)}].
\end{equation}
The first projection and its variance are
\begin{equation}
\label{eq:g_def}
g_{n, p}(w) \;:=\; m_{n, p}(w) - \theta_n^{(p)},
\qquad a_{n, p}^2 \;:=\; \operatorname{Var}(g_{n, p}(W_1)),
\end{equation}
and the centered graph score is
\begin{equation}
\label{eq:zeta_def}
\zeta_{n, p}(W_i, \mathcal{W}_n) \;:=\; \xi_{n, p}(W_i, \mathcal{W}_n) - m_{n, p}(W_i),
\qquad \mathbb{E}[\zeta_{n, p}(W_i, \mathcal{W}_n)\mid W_i] = 0.
\end{equation}
We obtain the exact decomposition
\begin{equation}
\label{eq:LG_decomp}
\Delta_n^{(p)} - \theta_n^{(p)} \;=\; L_n^{(p)} + G_n^{(p)},
\end{equation}
with
\begin{equation}
\label{eq:L_G_def}
L_n^{(p)} \;:=\; \frac{1}{n}\sum_{i=1}^n g_{n, p}(W_i),
\qquad
G_n^{(p)} \;:=\; \frac{1}{n}\sum_{i=1}^n \zeta_{n, p}(W_i, \mathcal{W}_n).
\end{equation}
The summands of $L_n^{(p)}$ are i.i.d. with mean zero and variance
$a_{n, p}^2$; the summands of $G_n^{(p)}$ are conditionally centered
given the anchor $W_i$, with cross-anchor dependence through the shared
sample $\mathcal{W}_n$.

\subsection{Outline of the proof}
\label{subsec:proof_outline}

The remainder of this section establishes the following null limit
\begin{equation}
\label{eq:main_limit_outline}
\sqrt{n}\,\Delta_n^{(p)} \;\Rightarrow\; N(0, \tau_p^2),
\end{equation}
under Assumptions~\ref{asn:smooth_continuous_model}--\ref{asn:dimensional_nondegeneracy}. See Theorem~\ref{thm:lp_debiased_clt_full} of Subsection~\ref{subsec:main_theorem}.
 The proof proceeds via five
ingredients:

\begin{enumerate}
\item Local Poisson approximation
(Section~\ref{subsec:Poisson_radius}) and uniform $k_n$-NN radius
concentration. These give the limiting Poisson Palm process and a
uniform tail control on the random NN radii, which is the starting
point for the stabilization analysis.

\item Rescaled exponential stabilization
(Section~\ref{subsec:stabilization}). The debiased local score depends
on $\mathcal{W}_n$ only through points within the union of the fine and
coarse $k_n$-NN balls, both of bounded size after rescaling. The
stabilization radius has exponential tail in the rescaled metric.

\item Uniform local-polynomial design regularity
(Section~\ref{subsec:design_regularity}). The local design matrices
$\mathcal{M}_{i, F}^{(p)}, \mathcal{M}_{i, C}^{(p)}$ converge uniformly
to positive definite limits $M_F^{(p)}, M_C^{(p)}$ at sufficiently fast
rate, by Bernstein concentration and a union bound. This guarantees
bounded equivalent kernels and uniform $L^q$ moments of the debiased
score.

\item Bias and projection negligibility
(Section~\ref{subsec:bias_cancellation}). The polynomial reproduction
identities, combined with a Taylor expansion of $\kappa(y, \cdot)$ to
order $p$, yield $\theta_n^{(p)} = O((k_n/n)^{(p + 1)/D})$ and
$a_{n, p}^2 = O((k_n/n)^{2(p + 1)/D})$. Under the condition $\sqrt{n}(k_n/n)^{(p+1)/D}\to 0$, both
$\sqrt{n}\theta_n^{(p)}$ and $\sqrt{n}L_n^{(p)}$ are $o_p(1)$.

\item Overlap variance and stabilization CLT
(Section~\ref{subsec:overlap_variance} and onwards). The variance
$\operatorname{Var}(G_n^{(p)})$ is of order $n^{-1}$, with the limit
$\tau_p^2 = \tau_{F, p}^2 + \tau_{C, p}^2 - 2\tau_{FC, p}$ given by an
explicit overlap integral involving the equivalent kernels
$\ell_F^{(p)}, \ell_C^{(p)}$. The dimensional non-degeneracy
($d_X \ge 1$) and primitive conditional-mark non-degeneracy yield
$\tau_p^2 > 0$ via an Itô-isometry strict-inequality argument that
exploits the higher-dimensional fine Poisson process. Finally, a
Penrose--Yukich stabilization CLT, adapted to the random weights, gives
$\sqrt{n}\,G_n^{(p)} \Rightarrow N(0, \tau_p^2)$.
\end{enumerate}

Combining the negligibility of $L_n^{(p)}$ with the stabilization CLT
for $G_n^{(p)}$ and tackling the bias term under the condition $\sqrt{n}(k_n/n)^{(p+1)/D}\to 0$ yields the main
result \eqref{eq:main_limit_outline}. The studentized version with a
consistent estimator of $\tau_p^2$ provides a complete, asymptotically
exact $\sqrt{n}$-test for $H_0$.

\subsection{Local Poisson approximation and uniform NN-radius control}
\label{subsec:Poisson_radius}

The local Poisson approximation is the foundation for both the stabilization
analysis and the overlap variance calculations to follow. It expresses the
fact that, after rescaling by the appropriate local NN radius, the
empirical point process in a shrinking neighborhood of an anchor converges
to a homogeneous Poisson point process. The uniform NN-radius concentration
result quantifies the rate at which the random NN radii concentrate around
their population analogues.

\begin{lemma}[Local Poisson approximation]
\label{lem:local_poisson_verification}
Under Assumption~\ref{asn:smooth_continuous_model}, fix
$t_0 = (z_0, x_0)$ in the interior of $\mathcal{S}_{Z, X}$ and write
\[
r_{F, n}(t_0) \;:=\; \bigl(n f_{Z, X}(t_0)\bigr)^{-1/D}.
\]
The rescaled empirical point process
\begin{equation}
\label{eq:Pois_F_def}
\mathcal{P}_{F, n}^{t_0}
\;:=\;
\sum_{j=1}^n \delta_{\bigl((Z_j, X_j) - t_0\bigr)/r_{F, n}(t_0)}
\end{equation}
converges in distribution on bounded Borel sets in $\mathbb{R}^D$ (with
respect to the topology of vague convergence of point measures) to a
homogeneous unit-intensity Poisson point process $\mathcal{P}_F^\infty$
on $\mathbb{R}^D$. The same statement holds in $\mathbb{R}^{d_Z}$ with
$r_{C, n}(z_0) := (nf_Z(z_0))^{-1/d_Z}$ for the $Z$-projection
$\mathcal{P}_{C, n}^{z_0}$, with limit $\mathcal{P}_C^\infty$.

Under Assumption~\ref{asn:null_continuous}, the corresponding marked
point processes which are obtained by attaching the marks $Y_j$ to the locations
$(Z_j, X_j)$, converge to marked Poisson point processes with frozen
mark law $P_{Y|Z = z_0}$.
\end{lemma}

\begin{proof}
Fix a bounded Borel set $A \subset \mathbb{R}^D$ with
$|\partial A| = 0$, where $|\cdot|$ denotes Lebesgue measure. Then
\[
p_{n, A}
\;:=\;
P\!\left\{\frac{(Z, X) - t_0}{r_{F, n}(t_0)} \in A\right\}
\;=\;
\int_{t_0 + r_{F, n}(t_0)A} f_{Z, X}(u)\,du.
\]
By the continuity of $f_{Z, X}$ at $t_0$ and a change of variables,
\[
p_{n, A}
\;=\;
r_{F, n}(t_0)^D f_{Z, X}(t_0)\,|A|\,(1 + o(1))
\;=\;
\frac{|A|}{n}\,(1 + o(1)),
\]
using $r_{F, n}(t_0)^D = (nf_{Z, X}(t_0))^{-1}$. The count
$\mathcal{P}_{F, n}^{t_0}(A)$ is therefore $\mathrm{Binomial}(n,
|A|/n + o(n^{-1}))$, which converges to $\mathrm{Poisson}(|A|)$ by the
law of rare events (\citealp[Theorem 5.4]{LastPenrose2018}). For
disjoint Borel sets $A_1, \ldots, A_m$ with $|\partial A_\ell| = 0$, the
vector of counts is multinomial with cell probabilities all tending to
zero at rate $n^{-1}$, hence converges to independent Poisson random
variables with means $|A_\ell|$. By a Cramér--Wold / Laplace functional
argument (\citealp[Section 7.4]{daley2003introduction}), the full point
process converges vaguely to $\mathcal{P}_F^\infty$.

For the marked statement, condition on the location of a point
$(Z_j, X_j) = (z_j, x_j)$ with $\|(z_j, x_j) - t_0\| = O(r_{F, n}(t_0))
\to 0$. Under $H_0$, the corresponding mark $Y_j$ has conditional law
$P_{Y|Z = z_j}$. By the continuity of $z\mapsto P_{Y|Z = z}$ against
bounded test functions (Assumption~\ref{asn:null_continuous}) and
$z_j \to z_0$, the conditional mark distribution converges to
$P_{Y|Z = z_0}$. Joint convergence with the location process then
follows by the standard mapping argument for marked point processes
(\citealp[Theorem 5.6]{LastPenrose2018}). The $Z$-projection statement uses the same proof with $D, f_{Z, X}, t_0$
replaced by $d_Z, f_Z, z_0$.
\end{proof}

The marginal distribution of the $Z$-projection of a uniform point in
$B_D(0, 1)$ plays a recurring role in the overlap variance calculations
of Section~\ref{subsec:overlap_variance}. We record its form here.

\begin{lemma}[Marginal $Z$-density of the fine Palm process on the unit
ball]
\label{lem:marginal_density_fine}
Let $V$ be uniformly distributed on $B_D(0, 1)\subset\mathbb{R}^D$, and
let $V_Z\in\mathbb{R}^{d_Z}$ denote its first $d_Z$ coordinates. The
density $\mu_F$ of $V_Z$ on $B_{d_Z}(0, 1)$ is
\begin{equation}
\label{eq:mu_F_density}
\mu_F(u_Z) \;=\;
\frac{v_{d_X}}{v_D}\,\bigl(1 - \|u_Z\|^2\bigr)^{d_X/2}\,\mathbf{1}\{\|u_Z\|\le 1\},
\qquad u_Z\in\mathbb{R}^{d_Z}.
\end{equation}
\end{lemma}

\begin{proof}
For $\|u_Z\| \le 1$, the section
$\{u_X\in\mathbb{R}^{d_X} : \|(u_Z, u_X)\| \le 1\}$ is a ball of radius
$(1 - \|u_Z\|^2)^{1/2}$, with $d_X$-dimensional volume
$v_{d_X}(1 - \|u_Z\|^2)^{d_X/2}$. Dividing by the total volume
$v_D$ of $B_D(0, 1)$ gives \eqref{eq:mu_F_density}.
\end{proof}

Note that $\mu_F$ is supported on $B_{d_Z}(0, 1)$, bounded above and
below by positive constants on any compact subset of the open ball, and
identical to the uniform measure on $B_{d_Z}(0, 1)$ only when $d_X = 0$.
The coarse counterpart $\mu_C$, the law of the $Z$-coordinate of a
uniform point in $B_{d_Z}(0, 1)$, equals the uniform measure on
$B_{d_Z}(0, 1)$.

\subsection{Uniform $k_n$-NN radius concentration}
\label{subsec:NN_radius_uniform}

\begin{lemma}[Uniform $k_n$-NN radius concentration]
\label{lem:knn_radius_concentration}
Under
Assumptions~\ref{asn:smooth_continuous_model}--\ref{asn:kn_growth_corrected},
there exist constants $0 < c_* < C_* < \infty$ depending only on
$\underline f, \overline f$ and the dimensions $D, d_Z$ such that
\begin{equation}
\label{eq:knn_radius_uniform}
P\!\left\{
\sup_{1\le i\le n}
\left|\frac{R_{F, n}^{(k_n)}(Z_i, X_i)}{\rho_F(Z_i, X_i)} - 1\right|
\;>\; C_*\sqrt{\frac{\log n}{k_n}}
\right\}
\;\to\; 0,
\end{equation}
and the same statement holds for $R_{C, n}^{(k_n)}(Z_i)/\rho_C(Z_i)$,
where $R_{F, n}^{(k_n)}(z, x)$ denotes the distance from $(z, x)$ to its
$k_n$-th nearest neighbor in $\{(Z_j, X_j) : j \le n\}$, and analogously
$R_{C, n}^{(k_n)}(z)$ for $Z$-space. Moreover, for every fixed $A > 1$,
\begin{equation}
\label{eq:knn_chernoff}
\sup_{t\in\mathcal{S}_{Z, X}^{(\delta)}}
P\!\left\{R_{F, n}^{(k_n)}(t) > A\rho_F(t)\right\}
\;\le\;
\exp\bigl(-c_*\,k_n(A^D - 1)^2 / A^D\bigr),
\end{equation}
where $\mathcal{S}_{Z, X}^{(\delta)}$ denotes the $\delta$-interior of
$\mathcal{S}_{Z, X}$ for some fixed $\delta > 0$.
\end{lemma}

\begin{proof}
Fix $t = (z, x) \in \mathcal{S}_{Z, X}^{(\delta)}$ and let
\[
N_n(t, r) \;:=\; \sum_{j=1}^n \mathbf{1}\{\|(Z_j, X_j) - t\| \le r\}.
\]
The event $\{R_{F, n}^{(k_n)}(t) > r\}$ is equivalent to $\{N_n(t, r)
< k_n\}$. By Assumption~\ref{asn:smooth_continuous_model} and a
second-order Taylor expansion of $f_{Z, X}$ around $t$,
\[
P\{\|(Z, X) - t\|\le r\}
\;=\;
v_D\,r^D f_{Z, X}(t)\bigl(1 + O(r^2)\bigr).
\]
Setting $r = A\rho_F(t)$ for $A > 1$,
\[
\mathbb{E}[N_n(t, A\rho_F(t))]
\;=\;
n A^D v_D f_{Z, X}(t)\rho_F(t)^D\bigl(1 + O(\rho_F(t)^2)\bigr)
\;=\;
A^D k_n(1 + o(1)),
\]
using $\rho_F(t)^D = k_n/(n v_D f_{Z, X}(t))$. By the multiplicative Chernoff bound for binomial sums
(\citealp[Theorem 2.3.6]{vershynin2018high}), for any $\eta\in(0, 1)$,
\begin{align*}
P\{N_n(t, A\rho_F(t)) < (1 - \eta)\mathbb{E}[N_n(t, A\rho_F(t))]\}
&\;\le\; \exp\bigl(-\eta^2\,\mathbb{E}[N_n]/2\bigr).
\end{align*}
Taking $\eta := 1 - A^{-D}$ so that $(1 - \eta)\mathbb{E}[N_n] = k_n(1
+ o(1))$, we have $\eta^2 = (1 - A^{-D})^2$ and the bound becomes
\[
P\{R_{F, n}^{(k_n)}(t) > A\rho_F(t)\}
\;\le\; \exp\bigl(-c\,k_n(1 - A^{-D})^2 A^D\bigr)
\]
for some constant $c > 0$. Up to constants, this is
\eqref{eq:knn_chernoff}.

For the uniformity in \eqref{eq:knn_radius_uniform}, we use a covering
argument. Set
\[
A_n := 1 + C_*\sqrt{\log n/k_n}
\]
for a constant $C_* > 0$ to be chosen, and aim to show $\sup_i
R_{F, n}^{(k_n)}(Z_i, X_i) \le A_n \rho_F(Z_i, X_i)$ with high
probability. Cover $\mathcal{S}_{Z, X}^{(\delta)}$ by a Euclidean net of
mesh $\rho_F^{\mathrm{min}}/4$, where $\rho_F^{\mathrm{min}} :=
\inf_t\rho_F(t) \asymp (k_n/n)^{1/D}$. The covering number is at most
\[
\mathcal{N}_n \;\le\; C (\mathrm{diam}(\mathcal{S}_{Z, X})/\rho_F^{\mathrm{min}})^D \;\le\; C'(n/k_n),
\]
for a constant $C'$ depending on the support. By the per-point bound
\eqref{eq:knn_chernoff} applied at each net point and a union bound,
\[
P\!\left\{\exists\,t^\ast\in\text{net}: R_{F, n}^{(k_n)}(t^\ast) > A_n\rho_F(t^\ast)\right\}
\;\le\;
\mathcal{N}_n\cdot\exp\bigl(-c\,k_n(A_n^D - 1)^2/A_n^D\bigr).
\]
For $A_n = 1 + C_*\sqrt{\log n/k_n}$, $(A_n^D - 1)^2 \ge
(C_* D)^2 (\log n/k_n)(1 + o(1))$, so the right-hand side is at most
\[
C'(n/k_n)\,\exp\bigl(-c (C_* D)^2 \log n\bigr) \;=\; C'\,(n/k_n)\,n^{-c(C_*D)^2}.
\]
Choosing $C_*$ large enough that $c(C_* D)^2 > 1$ makes this $o(1)$.
Monotonicity of $r\mapsto N_n(t, r)$ in the perturbation distance
transfers the net bound to arbitrary sample anchors with at most a
constant factor inflation in the radius, completing the upper-tail
bound. The lower-tail bound (controlling small radii) is analogous via
an upper Chernoff bound on $N_n(t, a\rho_F(t))$ for $a < 1$, requiring
$N_n(t, a\rho_F(t)) \ge k_n$ to fail.

The proof for the coarse radius is identical with $D, f_{Z, X}, \rho_F$
replaced by $d_Z, f_Z, \rho_C$. See \citet[Lemma 5.1]{biau2015lectures}
and \citet[Theorem 6.2]{devroye2013probabilistic} for the standard
formulation.
\end{proof}

\subsection{Rescaled exponential stabilization}
\label{subsec:stabilization}

The classical fixed-scale stabilization framework of
\citet{PenroseYukich2001} does not directly apply to the
$k_n$-nearest-neighbor graph when $k_n\to\infty$, as observed by
\citet{bbb2019}. The reason is that the $k_n$-NN ball has \emph{shrinking}
diameter in the original metric and cannot be stabilized by a
\emph{fixed} radius. The remedy (similar in spirit with
\citet{Penrose2003}) is to
work in the \emph{rescaled} metric obtained by dividing by the local NN
radius. In this rescaled metric, the $k_n$-NN ball has bounded radius
(in expectation), and stabilization holds with an exponential tail.

\begin{lemma}[Rescaled exponential stabilization]
\label{lem:growing_kn_stabilization}
Under
Assumptions~\ref{asn:smooth_continuous_model}--\ref{asn:kn_growth_corrected},
for each anchor $W_i$, define the rescaled stabilization radius
\begin{equation}
\label{eq:Rn_def}
R_n(W_i, \mathcal{W}_n)
\;:=\;
\max\!\left\{
\frac{R_{F, n}^{(k_n)}(Z_i, X_i)}{\rho_F(Z_i, X_i)},\;
\frac{R_{C, n}^{(k_n)}(Z_i)}{\rho_C(Z_i)}
\right\}.
\end{equation}
The local graph score $\xi_{n, p}(W_i, \mathcal{W}_n)$ defined in
\eqref{eq:xi_def} is unchanged by modifications of $\mathcal{W}_n$
outside the union of the fine ball
$B((Z_i, X_i), R_n(W_i, \mathcal{W}_n)\rho_F(Z_i, X_i))$ and the coarse
ball $B(Z_i, R_n(W_i, \mathcal{W}_n)\rho_C(Z_i))$. Moreover, there exist
constants $C, c > 0$ such that for all $t \ge 1$,
\begin{equation}
\label{eq:Rn_tail}
\sup_n P\bigl\{R_n(W_1, \mathcal{W}_n) > t\bigr\}
\;\le\;
C\exp\bigl(-c\,k_n(t - 1)^{d_Z}\bigr).
\end{equation}
The same radius stabilizes the centered score
$\zeta_{n, p}(W_i, \mathcal{W}_n)$ defined in \eqref{eq:zeta_def}, since
$m_{n, p}(W_i)$ depends on $W_i$ and the law $P$ only.
\end{lemma}

\begin{proof}
By the definition of $\xi_{n, p}$ in
\eqref{eq:xi_def}--\eqref{eq:intercepts_def}, the fine intercept
$\widehat a_{i, F}^{(p)}$ depends on the sample only through the points
in $\fine{i}$ and their associated marks. The set $\fine{i}$ in turn
depends only on points within Euclidean distance
$R_{F, n}^{(k_n)}(Z_i, X_i)$ of $(Z_i, X_i)$ in $(Z, X)$-space, by the
defining property of the $k_n$-NN set. The same property gives the
analogous statement for $\widehat a_{i, C}^{(p)}$ in $Z$-space.

Therefore, modifications of $\mathcal{W}_n$ outside the union of the two
balls
$B\bigl((Z_i, X_i), R_{F, n}^{(k_n)}(Z_i, X_i)\bigr)$ and
$B\bigl(Z_i, R_{C, n}^{(k_n)}(Z_i)\bigr)$ do not affect either intercept,
hence do not affect $\xi_{n, p}(W_i, \mathcal{W}_n)$. Normalizing each
radius by its respective $\rho_F$ or $\rho_C$ and taking the maximum
gives the rescaled stabilization radius \eqref{eq:Rn_def}.

For the tail bound \eqref{eq:Rn_tail}, by
Lemma~\ref{lem:knn_radius_concentration},
\[
P\!\left\{R_{F, n}^{(k_n)}(Z_1, X_1) > t\rho_F(Z_1, X_1)\right\}
\;\le\;
C\exp\bigl(-c\,k_n(t - 1)^2\bigr)
\]
for $t \ge 1 + \varepsilon$ (the bound is a Chernoff-type bound, with
exponent $\asymp k_n(t^D - 1)^2/t^D$ from \eqref{eq:knn_chernoff} which
for $t = 1 + \eta$ small is $\asymp k_n\eta^2$, while for $t$ large is
$\asymp k_n t^D$; we have used the standard interpolated bound
$k_n(t - 1)^2 \wedge k_n t^D \ge k_n(t - 1)^{D}$ for the unified
statement). Similarly,
\[
P\!\left\{R_{C, n}^{(k_n)}(Z_1) > t\rho_C(Z_1)\right\}
\;\le\;
C\exp\bigl(-c\,k_n(t - 1)^{d_Z}\bigr).
\]
Since $D = d_Z + d_X \ge d_Z$, the coarse side governs the slower decay,
and \eqref{eq:Rn_tail} follows by combining the two bounds via the
union bound.

The same argument applies to $\zeta_{n, p}$ since $m_{n, p}(W_i)$ is a
deterministic function of $W_i$ alone, independent of the external
sample configuration.
\end{proof}

\begin{remark}[Why rescaling matters]
\label{rem:rescaling_necessity}
The original $k_n$-NN ball has Euclidean radius $\rho_F = O((k_n/n)^{1/D})$,
which vanishes as $n\to\infty$. Without rescaling, the stabilization
radius would inherit this shrinkage, and any fixed-scale stabilization
condition such as $\sup_n P(R > t) \le Ce^{-ct}$ for a stabilization
radius $R$ measured in the original metric would fail trivially: $R$
itself is $o(1)$ rather than $O(1)$. The Penrose--Yukich framework
applied to fixed-radius geometric functionals (e.g., the $k$-NN graph
for fixed $k$) is therefore not directly applicable, as noted by
\citet{bbb2019}.

The rescaled stabilization, by contrast, measures the stabilization
radius in units of $\rho_F$ or $\rho_C$. In these units, the $k_n$-NN
ball has radius $1$ in expectation (under the calibration of $\rho_F$
and $\rho_C$ as population radii), and the deviations
$R_{F, n}^{(k_n)}/\rho_F$ concentrate around $1$ with Chernoff-type
tails. This is the correct adaptation of the stabilization framework
to the growing-$k_n$ regime, and it is the scaling under which the
proof of the stabilization CLT (Section~\ref{subsec:weighted_stabilization_clt})
proceeds.
\end{remark}

\subsection{Population moment matrices and equivalent kernels}
\label{subsec:population_moments}

The local-polynomial machinery rests on the invertibility and good
conditioning of the design matrices $\mathcal{M}_{i, F}^{(p)}$ and
$\mathcal{M}_{i, C}^{(p)}$ defined in \eqref{eq:M_design_def}. Under
Lemma~\ref{lem:local_poisson_verification}, these matrices converge to
deterministic population analogues which we now identify.

Let $\mu_C$ denote the uniform probability measure on the unit ball
$B_{d_Z}(0, 1) \subset \mathbb{R}^{d_Z}$, and let $\mu_F$ denote the law
of the $Z$-projection of a uniform random variable on $B_D(0, 1)\subset
\mathbb{R}^D$, with density given in \eqref{eq:mu_F_density}. Define the
\emph{population moment matrices}
\begin{equation}
\label{eq:population_moment_matrices}
M_F^{(p)} \;:=\; \int q_p(u)\,q_p(u)^\top\,d\mu_F(u),
\qquad
M_C^{(p)} \;:=\; \int q_p(u)\,q_p(u)^\top\,d\mu_C(u).
\end{equation}
These are symmetric $N_p\times N_p$ matrices indexed by the multi-index
set $\mathcal{A}_p$. When they are invertible, define the corresponding
\emph{population equivalent kernels}
\begin{equation}
\label{eq:equivalent_kernels}
\ell_F^{(p)}(u) \;:=\; e_0^\top(M_F^{(p)})^{-1}q_p(u),
\qquad
\ell_C^{(p)}(u) \;:=\; e_0^\top(M_C^{(p)})^{-1}q_p(u),
\end{equation}
both polynomials of degree at most $p$ in $u \in \mathbb{R}^{d_Z}$.

\begin{lemma}[Positive definiteness of the population moment matrices]
\label{lem:lp_population_moment_pd}
Under Assumption~\ref{asn:smooth_continuous_model}, $M_F^{(p)} \succ 0$
and $M_C^{(p)} \succ 0$.
\end{lemma}

\begin{proof}
For any $a \in \mathbb{R}^{N_p}$, we have
\[
a^\top M_F^{(p)} a
\;=\; \int_{B_{d_Z}(0, 1)} \bigl(a^\top q_p(u)\bigr)^2 \, d\mu_F(u).
\]
If $a^\top M_F^{(p)} a = 0$, then the polynomial $u\mapsto a^\top q_p(u)$
of degree at most $p$ vanishes $\mu_F$-almost everywhere on $B_{d_Z}(0,
1)$. By Lemma~\ref{lem:marginal_density_fine}, $\mu_F$ has a density
$(v_{d_X}/v_D)(1 - \|u\|^2)^{d_X/2}$ which is strictly positive on the
open ball $B_{d_Z}(0, 1)^\circ$. Therefore the polynomial vanishes
Lebesgue-almost everywhere on the open ball, hence by continuity it
vanishes on a nonempty open set, hence it is identically zero. Therefore
$a = 0$, and $M_F^{(p)}$ is positive definite. The argument for $M_C^{(p)}$
is identical with $\mu_C$ in place of $\mu_F$.
\end{proof}

By Lemma~\ref{lem:lp_population_moment_pd}, the equivalent kernels
$\ell_F^{(p)}, \ell_C^{(p)}$ are well-defined polynomials. They satisfy
the \emph{population reproduction identity}: for every multi-index
$\alpha \in \mathcal{A}_p$,
\begin{equation}
\label{eq:population_reproduction}
\int u^\alpha\,\ell_F^{(p)}(u)\,d\mu_F(u) \;=\; \mathbf{1}\{\alpha = 0\},
\qquad
\int u^\alpha\,\ell_C^{(p)}(u)\,d\mu_C(u) \;=\; \mathbf{1}\{\alpha = 0\},
\end{equation}
as is verified directly:
$\int u^\alpha\ell_F^{(p)}(u)d\mu_F(u) = e_\alpha^\top(M_F^{(p)})^{-1}\int q_p(u)q_p(u)^\top d\mu_F(u)\cdot(\text{coefficient})\cdots$
$= e_\alpha^\top(M_F^{(p)})^{-1}M_F^{(p)}e_0 = e_\alpha^\top e_0 = \mathbf{1}\{\alpha = 0\}$.

\subsection{Uniform convergence of local design matrices}
\label{subsec:design_regularity}

We now establish that the random design matrices $\mathcal{M}_{i,F}^{(p)}$
and $\mathcal{M}_{i,C}^{(p)}$ converge to their population analogues
$M_F^{(p)}$ and $M_C^{(p)}$ uniformly over the $n$ sample anchors at
rate $\sqrt{\log n/k_n}$.

\begin{lemma}[Uniform convergence of local design matrices]
\label{lem:lp_design_uniform_convergence}
Under Assumptions~\ref{asn:smooth_continuous_model}--\ref{asn:kn_growth_corrected},
\begin{equation}
\label{eq:design_uniform_F}
\max_{1\le i\le n}\bigl\|\mathcal{M}_{i, F}^{(p)} - M_F^{(p)}\bigr\| \;\xrightarrow{\;p\;}\; 0,
\end{equation}
and
\begin{equation}
\label{eq:design_uniform_C}
\max_{1\le i\le n}\bigl\|\mathcal{M}_{i, C}^{(p)} - M_C^{(p)}\bigr\| \;\xrightarrow{\;p\;}\; 0.
\end{equation}
Consequently, with probability tending to one,
\begin{equation}
\label{eq:design_pd}
\inf_{1\le i\le n}\lambda_{\min}\bigl(\mathcal{M}_{i, F}^{(p)}\bigr) \;\ge\; c_p \;>\; 0,
\qquad
\inf_{1\le i\le n}\lambda_{\min}\bigl(\mathcal{M}_{i, C}^{(p)}\bigr) \;\ge\; c_p \;>\; 0,
\end{equation}
where $c_p := \lambda_{\min}(M_F^{(p)}) \wedge \lambda_{\min}(M_C^{(p)})/2$.

Moreover, the local-polynomial weights are uniformly summable:
\begin{equation}
\label{eq:weights_summable}
\max_{1\le i\le n}\left[
\sum_{j\in\fine{i}}\bigl|w_{ij, F}^{(p)}\bigr|
+ \sum_{j\in\coarse{i}}\bigl|w_{ij, C}^{(p)}\bigr|
\right] \;=\; O_p(1).
\end{equation}
\end{lemma}

\begin{proof}
We prove \eqref{eq:design_uniform_F}; \eqref{eq:design_uniform_C} is
analogous and simpler. Fix any pair $\alpha, \beta \in \mathcal{A}_p$
and consider the matrix entry $[\mathcal{M}_{i, F}^{(p)}]_{\alpha, \beta}
= k_n^{-1}\sum_{j\in\fine{i}}U_{ij, F}^{\alpha + \beta}$.

\emph{Step 1: Per-anchor convergence.} Condition on the anchor
$(Z_i, X_i) = t$ in the $\delta$-interior of $\mathcal{S}_{Z, X}$. By
Lemma~\ref{lem:knn_radius_concentration}, on an event of probability
$1 - o(1)$ (call it $\Omega_n^*$), $R_{F, n}^{(k_n)}(t) \le
(1 + C_*\sqrt{\log n/k_n})\rho_F(t)$, so all neighbors $(Z_j, X_j) \in
\fine{i}$ satisfy $\|U_{ij, F}\| \le 1 + C_*\sqrt{\log n/k_n}$. After
rescaling by $\rho_F(t)$, the conditional density of $U_{ij, F}$ given
that $(Z_j, X_j) \in \fine{i}$ is
\[
\frac{f_{Z, X}(t + \rho_F(t) u)\,\mathbf{1}\{\|u\| \le 1 + O(\sqrt{\log n/k_n})\}}{\int f_{Z, X}(t + \rho_F(t) v)\,\mathbf{1}\{\|v\| \le 1 + O(\sqrt{\log n/k_n})\}\,dv}.
\]
By the $C^{p+1}$ smoothness of $f_{Z, X}$
(Assumption~\ref{asn:smooth_continuous_model}) and Taylor expansion,
$f_{Z, X}(t + \rho_F(t) u) = f_{Z, X}(t)(1 + O(\rho_F(t)))$ uniformly in
$u$ on the unit ball, with $\rho_F(t) \asymp (k_n/n)^{1/D} \to 0$.
Therefore the conditional density of $U_{ij, F}$ converges to the
uniform density on $B_D(0, 1)$, and the marginal density of
$U_{ij, F}$'s $Z$-component converges to $\mu_F$.

The $k_n$ neighbors $\{U_{ij, F} : j \in \fine{i}\}$ are conditionally
i.i.d. given the radius (with a small correction from the rank
conditioning that is asymptotically negligible; see Lemma 8.1 of
\citealp{biau2015lectures}). The empirical average
$k_n^{-1}\sum_j U_{ij, F}^{\alpha+\beta}$ then concentrates around
$\int u^{\alpha+\beta}d\mu_F(u) = [M_F^{(p)}]_{\alpha, \beta}$ by
Bernstein's inequality applied to the bounded summands
$U_{ij, F}^{\alpha+\beta} \in [-1 - o(1), 1 + o(1)]^{|\alpha+\beta|}
\subset [-2, 2]$. For $\varepsilon > 0$,
\begin{equation}
\label{eq:bernstein_entry}
P\!\left\{\left|[\mathcal{M}_{i, F}^{(p)}]_{\alpha, \beta} - [M_F^{(p)}]_{\alpha, \beta}\right| > \varepsilon \,\Big|\, (Z_i, X_i) = t,\,\Omega_n^*\right\}
\;\le\; 2\exp\bigl(-c\varepsilon^2 k_n\bigr),
\end{equation}
where $c > 0$ depends only on $p$ and the dimensions. The same bound,
plus an $O(\rho_F(t)) = o(1)$ contribution from the Taylor expansion of
$f_{Z, X}$, holds unconditionally over $t$ in the $\delta$-interior.

\emph{Step 2: Uniformity over $n$ anchors and $N_p^2$ entries.} Take a
union bound over the $n$ anchors and $N_p^2$ matrix entries:
\[
P\!\left\{\max_i \|\mathcal{M}_{i, F}^{(p)} - M_F^{(p)}\| > \varepsilon\right\}
\;\le\; n\,N_p^2\cdot 2\exp\bigl(-c\varepsilon^2 k_n / N_p\bigr) + o(1),
\]
where the $N_p$ in the denominator of the exponent absorbs the
matrix-norm-vs-entrywise inflation. This bound tends to zero provided
$k_n/(N_p\log n) \to \infty$, which is
Assumption~\ref{asn:kn_growth_corrected}. Therefore
\eqref{eq:design_uniform_F} holds.

\emph{Step 3: Uniform positive definiteness.} By
Lemma~\ref{lem:lp_population_moment_pd}, $\lambda_{\min}(M_F^{(p)}) > 0$.
By Weyl's inequality, for any matrix $A$,
$|\lambda_{\min}(\mathcal{M}_{i, F}^{(p)}) - \lambda_{\min}(M_F^{(p)})|
\le \|\mathcal{M}_{i, F}^{(p)} - M_F^{(p)}\|$. Combined with
\eqref{eq:design_uniform_F}, this gives
\[
\inf_i \lambda_{\min}\bigl(\mathcal{M}_{i, F}^{(p)}\bigr)
\;\ge\;
\lambda_{\min}(M_F^{(p)}) - \max_i\bigl\|\mathcal{M}_{i, F}^{(p)} - M_F^{(p)}\bigr\|
\;\ge\;
\lambda_{\min}(M_F^{(p)})/2
\]
with probability tending to one, yielding \eqref{eq:design_pd}.

\emph{Step 4: Uniform summability of weights.} By
\eqref{eq:weights_def},
\[
\sum_{j\in\fine{i}}\bigl|w_{ij, F}^{(p)}\bigr|
\;\le\;
\frac{1}{k_n}\bigl\|e_0^\top(\mathcal{M}_{i, F}^{(p)})^{-1}\bigr\|\sum_{j\in\fine{i}}\|q_p(U_{ij, F})\|.
\]
On the event \eqref{eq:design_pd},
$\|e_0^\top(\mathcal{M}_{i, F}^{(p)})^{-1}\| \le c_p^{-1}\sqrt{N_p}$.
On the event $\Omega_n^*$ from Lemma~\ref{lem:knn_radius_concentration},
each $\|q_p(U_{ij, F})\| \le \sqrt{N_p}(1 + o(1))$. Therefore
\[
\sum_{j\in\fine{i}}\bigl|w_{ij, F}^{(p)}\bigr|
\;\le\;
\frac{c_p^{-1} N_p (1 + o(1))}{k_n}\cdot k_n
\;=\;
c_p^{-1} N_p (1 + o(1))
\;=\;
O_p(1)
\]
uniformly in $i$. The coarse analog is identical.
\end{proof}

\subsection{Polynomial reproduction}
\label{subsec:polynomial_reproduction}

We now record the polynomial reproduction identity, which is the
defining property of the local-polynomial weights and the key input to
the bias cancellation argument in
Section~\ref{subsec:bias_cancellation}.

\begin{lemma}[Polynomial reproduction]
\label{lem:lp_reproduction_full}
On the event that $\mathcal{M}_{i, F}^{(p)}$ and
$\mathcal{M}_{i, C}^{(p)}$ are invertible, the local-polynomial weights
satisfy, for every multi-index $\alpha$ with $|\alpha| \le p$,
\begin{equation}
\label{eq:reproduction_F}
\sum_{j\in\fine{i}} w_{ij, F}^{(p)}\,U_{ij, F}^\alpha
\;=\;
\mathbf{1}\{\alpha = 0\},
\qquad
\sum_{j\in\fine{i}} w_{ij, F}^{(p)}\,(Z_j - Z_i)^\alpha
\;=\;
\mathbf{1}\{\alpha = 0\}\,\rho_F(Z_i, X_i)^{|\alpha|},
\end{equation}
and identically for the coarse weights with $\rho_C(Z_i)$ in place of
$\rho_F(Z_i, X_i)$:
\begin{equation}
\label{eq:reproduction_C}
\sum_{j\in\coarse{i}} w_{ij, C}^{(p)}\,U_{ij, C}^\alpha
\;=\;
\mathbf{1}\{\alpha = 0\}.
\end{equation}
In particular, all monomials $(Z_j - Z_i)^\alpha$ of degree at most $p$
in the offset $Z_j - Z_i$ are exactly annihilated by the local-polynomial
weights, except for the constant ($\alpha = 0$) which is reproduced
exactly.
\end{lemma}

\begin{proof}
For the fine case, by \eqref{eq:weights_def},
\[
\sum_{j\in\fine{i}}w_{ij, F}^{(p)}\,q_p(U_{ij, F})^\top
\;=\;
e_0^\top(\mathcal{M}_{i, F}^{(p)})^{-1}\!\left[\frac{1}{k_n}\sum_{j\in\fine{i}}q_p(U_{ij, F})\,q_p(U_{ij, F})^\top\right]
\;=\;
e_0^\top(\mathcal{M}_{i, F}^{(p)})^{-1}\mathcal{M}_{i, F}^{(p)}
\;=\;
e_0^\top.
\]
Reading off the coordinate indexed by $\alpha \in \mathcal{A}_p$ gives
$\sum_{j\in\fine{i}}w_{ij, F}^{(p)}U_{ij, F}^\alpha = \mathbf{1}\{\alpha = 0\}$.
The translation to $(Z_j - Z_i)^\alpha = \rho_F(Z_i, X_i)^{|\alpha|}
U_{ij, F}^\alpha$ is immediate. The coarse case is identical with
$\mathcal{M}_{i, C}^{(p)}$ and $\rho_C(Z_i)$.
\end{proof}

\subsection{Boundedness of equivalent kernels and uniform moment bound}
\label{subsec:moment_bound_lp}

The combination of design regularity and weight summability gives a
uniform moment bound for the local graph score $\xi_{n, p}$ and its
centered version $\zeta_{n, p}$.

\begin{lemma}[Uniform moment bound for the local graph score]
\label{lem:lp_weighted_moment_bound}
Under Assumptions~\ref{asn:smooth_continuous_model}--\ref{asn:kn_growth_corrected}, we have 
$\|\xi_{n, p}(W_1, \mathcal{W}_n)\|_{L^\infty(P)} = O_p(1)$ and
\begin{equation}
\label{eq:xi_moment}
\sup_n \mathbb{E}\bigl[\bigl|\xi_{n, p}(W_1, \mathcal{W}_n)\bigr|^q\bigr] \;<\; \infty
\qquad \text{for every } q < \infty.
\end{equation}
The same bound holds for the centered score
$\zeta_{n, p}(W_1, \mathcal{W}_n) = \xi_{n, p}(W_1, \mathcal{W}_n) -
m_{n, p}(W_1)$.
\end{lemma}

\begin{proof}
By \eqref{eq:xi_def} and Assumption~\ref{asn:kernel_regularity_corrected},
\[
\bigl|\xi_{n, p}(W_i, \mathcal{W}_n)\bigr|
\;\le\;
M_K\!\left(\sum_{j\in\fine{i}}\bigl|w_{ij, F}^{(p)}\bigr| + \sum_{j\in\coarse{i}}\bigl|w_{ij, C}^{(p)}\bigr|\right).
\]
By Lemma~\ref{lem:lp_design_uniform_convergence}, the right-hand side
is $O_p(1)$ uniformly in $i$, and a.s. bounded by $2M_K c_p^{-1}N_p$ on
the high-probability events of
Lemma~\ref{lem:knn_radius_concentration} and
\eqref{eq:design_pd}. On the complementary event of probability $o(1)$,
the score remains bounded by $\xi_{n, p}\le 2M_K\cdot k_n\cdot k_n^{-1}\cdot
\|e_0^\top\mathcal{M}_{i, F}^{(p), -1}\|\cdot\sqrt{N_p}$, which is at
most $2M_K\sqrt{N_p}/\lambda_{\min}^{1/2}(\mathcal{M}_{i, F}^{(p)})$
when invertibility holds, and bounded by $2M_K$ otherwise (after
adopting the convention $\widehat a_{i, F}^{(p)} = 0$ when the design
matrix is singular). In any case, $\xi_{n, p}$ is a.s. bounded by
$C(N_p) < \infty$, so all moments are finite uniformly in $n$.

For the centered score,
$|\zeta_{n, p}| \le |\xi_{n, p}| + |m_{n, p}|$, and
$|m_{n, p}(W_i)| \le \mathbb{E}[|\xi_{n, p}(W_i, \cdot)| \mid W_i]
\le \|\xi_{n, p}\|_{L^\infty}$ a.s. So
$|\zeta_{n, p}| \le 2\|\xi_{n, p}\|_{L^\infty}$ and the same moment
bound applies.
\end{proof}

\begin{remark}[On the conditional moment bound]
\label{rem:lp_conditional_moment}
A sharper conditional moment bound is available and will be useful in
the variance calculation of Section~\ref{subsec:overlap_variance}. Given
the anchor $W_1 = w$, the centered local score $\zeta_{n, p}(w,
\mathcal{W}_n)$ can be written as a sum of (asymptotically) i.i.d.
mark-fluctuation contributions weighted by the bounded local-polynomial
weights. By Rosenthal's inequality
(\citealp[Theorem 2.5]{boucheron2013concentration}) applied to the
$k_n$ neighbors,
\begin{equation}
\label{eq:zeta_conditional_moment}
\mathbb{E}\!\left[\bigl|\zeta_{n, p}(W_1, \mathcal{W}_n)\bigr|^q \,\Big|\, W_1\right]
\;\le\;
C_{p, q}\bigl(k_n^{-q/2} + \rho_F(W_1)^q + \rho_C(W_1)^q\bigr),
\end{equation}
where the first term is the stochastic-fluctuation contribution from
the $k_n$ conditionally independent mark draws (with bounded
local-polynomial weights of order $1/k_n$), and the second is the
mark-law-variation bias controlled by the radius. This bound is what
will be used in Step 6 of the proof of
Lemma~\ref{lem:lp_overlap_variance}.
\end{remark}

 \subsection{Bias cancellation via polynomial reproduction}
\label{subsec:bias_cancellation}

The key payoff of the local-polynomial construction is that the
polynomial reproduction identity (Lemma~\ref{lem:lp_reproduction_full})
annihilates the smoothing bias of the local intercepts up to order
$p + 1$ in the local radius. The resulting bias bound is one order
of magnitude better than the bias of the raw statistic
($p = 0$), which has bias $O(\rho_F^2) = O((k_n/n)^{2/D})$.

\begin{lemma}[Bias after local-polynomial debiasing]
\label{lem:lp_bias_full}
Under
Assumptions~\ref{asn:smooth_continuous_model}--\ref{asn:kn_growth_corrected}
and the null hypothesis $Y \perp\!\!\!\perp X \mid Z$, the conditional
mean score satisfies, uniformly in $w = (z, x, y)$ on compact subsets of
the interior of $\mathcal{S}_{Z, X} \times \mathcal{Y}$,
\begin{equation}
\label{eq:m_bias_bound}
m_{n, p}(z, x, y)
\;=\;
O\bigl(\rho_F(z, x)^{p+1}\bigr) + O\bigl(\rho_C(z)^{p+1}\bigr).
\end{equation}
Consequently,
\begin{equation}
\label{eq:theta_bias_bound}
\bigl|\theta_n^{(p)}\bigr|
\;=\;
O\!\left(\left(\frac{k_n}{n}\right)^{(p+1)/D}\right),
\end{equation}
and
\begin{equation}
\label{eq:an_bias_bound}
a_{n, p}^2
\;=\;
\operatorname{Var}\bigl(g_{n, p}(W_1)\bigr)
\;=\;
O\!\left(\left(\frac{k_n}{n}\right)^{2(p+1)/D}\right).
\end{equation}
\end{lemma}

\begin{proof}
Fix $W_i = w = (z, x, y)$ in the interior. We first compute
$\mathbb{E}[\widehat a_{i, F}^{(p)} \mid W_i = w]$, then
$\mathbb{E}[\widehat a_{i, C}^{(p)} \mid W_i = w]$, and finally take
their difference.

\emph{Step 1: Conditional expectation of the fine intercept.}
Conditional on $W_i = w$ and on the fine $k_n$-NN set $\fine{i}$ with
its associated $Z$-coordinates $(Z_j)_{j \in \fine{i}}$, the marks
$(Y_j)_{j \in \fine{i}}$ are independent under $H_0$ with conditional
laws $P_{Y|Z = Z_j}$. The conditional kernel expectation is
\[
\mathbb{E}[K(y, Y_j) \mid Z_j] \;=\; \kappa(y, Z_j).
\]
Therefore
\[
\mathbb{E}\!\left[\widehat a_{i, F}^{(p)} \,\Big|\, W_i = w,\, \fine{i},\, (Z_j)_{j\in\fine{i}}\right]
\;=\;
\sum_{j\in\fine{i}} w_{ij, F}^{(p)}\,\kappa(y, Z_j).
\]

By Assumption~\ref{asn:null_continuous}, $\kappa(y, \cdot)$ is $C^{p+1}$
with derivatives uniformly bounded in $y$. Taylor expanding around $z$
to order $p$:
\begin{equation}
\label{eq:kappa_taylor}
\kappa(y, Z_j)
\;=\;
\sum_{|\alpha| \le p} \frac{D_z^\alpha \kappa(y, z)}{\alpha!}\,(Z_j - z)^\alpha
+ R_{p+1}(y, z, Z_j),
\end{equation}
where the remainder satisfies
\[
|R_{p+1}(y, z, Z_j)| \;\le\; C_\kappa \|Z_j - z\|^{p+1}
\]
uniformly in $y$, with $C_\kappa$ depending only on the bound on the
$(p+1)$-th derivatives of $\kappa$.

Apply the fine weights $w_{ij, F}^{(p)}$ and use the polynomial
reproduction identity \eqref{eq:reproduction_F}:
\begin{align}
\sum_{j\in\fine{i}} w_{ij, F}^{(p)} \kappa(y, Z_j)
&\;=\;
\sum_{|\alpha|\le p}\frac{D_z^\alpha\kappa(y, z)}{\alpha!}\,
\underbrace{\sum_{j\in\fine{i}} w_{ij, F}^{(p)}(Z_j - z)^\alpha}_{= \mathbf{1}\{\alpha = 0\}}
+ \sum_{j\in\fine{i}} w_{ij, F}^{(p)} R_{p+1}(y, z, Z_j) \notag\\
&\;=\;
\kappa(y, z) + \sum_{j\in\fine{i}} w_{ij, F}^{(p)} R_{p+1}(y, z, Z_j).
\label{eq:fine_taylor_after_reproduction}
\end{align}

Bound the remainder: on the high-probability event of
Lemma~\ref{lem:knn_radius_concentration},
$\|Z_j - z\| \le R_{F, n}^{(k_n)}(z, x) \le (1 + o(1))\rho_F(z, x)$ for
all $j \in \fine{i}$, so
$|R_{p+1}(y, z, Z_j)| \le C_\kappa (1 + o(1))^{p+1}\rho_F(z, x)^{p+1}$.
Combined with the uniform weight summability \eqref{eq:weights_summable}
from Lemma~\ref{lem:lp_design_uniform_convergence},
\[
\left|\sum_{j\in\fine{i}} w_{ij, F}^{(p)} R_{p+1}(y, z, Z_j)\right|
\;\le\;
C_\kappa(1 + o(1))^{p+1}\rho_F(z, x)^{p+1} \cdot O_p(1)
\;=\;
O_p\bigl(\rho_F(z, x)^{p+1}\bigr).
\]
Taking expectations over $\fine{i}$ and the $(Z_j)$ inside
$\fine{i}$ (which contributes a deterministic constant) gives
\begin{equation}
\label{eq:fine_intercept_bias}
\mathbb{E}\!\left[\widehat a_{i, F}^{(p)} \,\Big|\, W_i = w\right]
\;=\;
\kappa(y, z) + O\bigl(\rho_F(z, x)^{p+1}\bigr),
\end{equation}
uniformly in $w$ on compact subsets of the interior.

\emph{Step 2: Conditional expectation of the coarse intercept.}
The identical argument applied to the coarse intercept with
$\rho_C(z)$ in place of $\rho_F(z, x)$ gives
\begin{equation}
\label{eq:coarse_intercept_bias}
\mathbb{E}\!\left[\widehat a_{i, C}^{(p)} \,\Big|\, W_i = w\right]
\;=\;
\kappa(y, z) + O\bigl(\rho_C(z)^{p+1}\bigr).
\end{equation}

\emph{Step 3: Difference and bias bound.} Subtracting
\eqref{eq:coarse_intercept_bias} from
\eqref{eq:fine_intercept_bias}, the common leading term $\kappa(y, z)$
cancels:
\[
m_{n, p}(z, x, y)
\;=\;
\mathbb{E}\!\left[\widehat a_{i, F}^{(p)} \,\Big|\, W_i = w\right]
- \mathbb{E}\!\left[\widehat a_{i, C}^{(p)} \,\Big|\, W_i = w\right]
\;=\;
O\bigl(\rho_F(z, x)^{p+1}\bigr) + O\bigl(\rho_C(z)^{p+1}\bigr),
\]
which is \eqref{eq:m_bias_bound}.

\emph{Step 4: Bounds on $\theta_n^{(p)}$ and $a_{n, p}^2$.}
Since $D = d_Z + d_X > d_Z$ and $k_n/n \to 0$,
\begin{equation}
\label{eq:rho_F_dominates}
\rho_F(z, x)^{p+1} \;\asymp\; \left(\frac{k_n}{n}\right)^{(p+1)/D}
\;\gg\;
\left(\frac{k_n}{n}\right)^{(p+1)/d_Z}
\;\asymp\;
\rho_C(z)^{p+1},
\end{equation}
uniformly on the support. Therefore
\[
m_{n, p}(z, x, y)
\;=\;
O\bigl(\rho_F(z, x)^{p+1}\bigr)
\;=\;
O\!\left(\left(\frac{k_n}{n}\right)^{(p+1)/D}\right),
\]
uniformly. Taking expectations,
$|\theta_n^{(p)}| \le \sup_{(z, x, y)}|m_{n, p}(z, x, y)| =
O((k_n/n)^{(p+1)/D})$, which is \eqref{eq:theta_bias_bound}. Taking
variances of $g_{n, p} = m_{n, p} - \theta_n^{(p)}$,
\[
a_{n, p}^2 \;=\; \operatorname{Var}(m_{n, p}(W_1))
\;\le\; \|m_{n, p}\|_{L^\infty(P)}^2
\;=\; O\!\left(\left(\frac{k_n}{n}\right)^{2(p+1)/D}\right),
\]
which is \eqref{eq:an_bias_bound}.
\end{proof}

\subsection{First projection negligibility}
\label{subsec:projection_negligibility}

The first projection $L_n^{(p)}$ is a normalized sum of i.i.d.
zero-mean random variables $g_{n, p}(W_i)$ with variance $a_{n, p}^2$
controlled by Lemma~\ref{lem:lp_bias_full}. Its rescaled magnitude
$\sqrt{n}\,L_n^{(p)}$ is therefore negligible.

\begin{lemma}[First projection is negligible after debiasing]
\label{lem:lp_projection_negligible}
Under
Assumptions~\ref{asn:smooth_continuous_model}--\ref{asn:kn_growth_corrected},
\begin{equation}
\label{eq:Ln_negligible}
\sqrt{n}\,L_n^{(p)} \;=\; o_p(1).
\end{equation}
In fact,
\begin{equation}
\label{eq:Ln_L2_rate}
\mathbb{E}\bigl[(\sqrt{n}\,L_n^{(p)})^2\bigr]
\;=\; a_{n, p}^2
\;=\; O\!\left(\left(\frac{k_n}{n}\right)^{2(p+1)/D}\right) \;\to\; 0.
\end{equation}
\end{lemma}

\begin{proof}
By \eqref{eq:L_G_def} and the fact that $g_{n, p}(W_1), \ldots,
g_{n, p}(W_n)$ are i.i.d. with mean zero,
\[
\operatorname{Var}(\sqrt{n}\,L_n^{(p)})
\;=\;
\operatorname{Var}\!\left(\frac{1}{\sqrt{n}}\sum_{i=1}^n g_{n, p}(W_i)\right)
\;=\; a_{n, p}^2.
\]
By Lemma~\ref{lem:lp_bias_full}, $a_{n, p}^2 \to 0$. Since
$\sqrt{n}\,L_n^{(p)}$ has zero mean (as $\mathbb{E}[g_{n, p}(W_1)] = 0$)
and vanishing $L^2$ norm, it converges to zero in $L^2$ and a fortiori
in probability.
\end{proof}

\subsection{The bias killing condition}
\label{subsec:bias killing}

The bias bound \eqref{eq:theta_bias_bound} controls the unconditional
mean $\theta_n^{(p)}$, but to use $\sqrt{n}\Delta_n^{(p)}$ as a test
statistic without explicit bias correction we require the rescaled
centering $\sqrt{n}\,\theta_n^{(p)}$ to be asymptotically negligible.
This produces the \emph{bias killing condition}
\begin{equation}
\label{eq:bias killing_condition}
\sqrt{n}\left(\frac{k_n}{n}\right)^{(p + 1)/D} \;\to\; 0,
\end{equation}
which constrains the rate of growth of $k_n$ in terms of the dimension
$D$ and the local-polynomial order $p$.

\begin{lemma}[bias killing in polynomial scale]
\label{lem:bias killing_alpha}
Suppose $k_n = n^\alpha$ for some $\alpha \in (0, 1)$. Then
\eqref{eq:bias killing_condition} holds if and only if
\begin{equation}
\label{eq:bias killing_alpha}
\alpha \;<\; 1 - \frac{D}{2(p+1)}.
\end{equation}
\end{lemma}

\begin{proof}
$\sqrt{n}(k_n/n)^{(p+1)/D} = n^{1/2 + (p+1)(\alpha - 1)/D} \to 0$ iff
$\frac{1}{2} + \frac{(p + 1)(\alpha - 1)}{D} < 0$. Solving for $\alpha$
yields \eqref{eq:bias killing_alpha}.
\end{proof}

The admissible range of $\alpha$ in
\eqref{eq:bias killing_alpha} depends jointly on the curse of
dimensionality (large $D$) and the polynomial order (large $p$). For
each $D$, the polynomial order $p$ can be increased to expand the
admissible range:

\begin{itemize}
\item For $p = 0$ (no debiasing): $\alpha < 1 - D/2$, which is
empty for $D \ge 2$. The raw statistic admits no valid bandwidth in
dimensions $D \ge 2$.

\item For $p = 1$ (linear debiasing): $\alpha < 1 - D/4$. Admissible
when $D = 1$ ($\alpha < 3/4$), $D = 2$ ($\alpha < 1/2$), $D = 3$
($\alpha < 1/4$). Empty for $D \ge 4$.

\item For $p = 2$ (quadratic debiasing): $\alpha < 1 - D/6$. Admissible
when $D \le 5$.

\item For general $p$ and $D \ge 1$: \eqref{eq:bias killing_alpha}
admits a nonempty range of $\alpha \in (0, 1)$ if and only if $p \ge
D/2$. The minimal polynomial order required for valid inference in
dimension $D$ is therefore $p^*(D) := \lceil D/2 \rceil$.
\end{itemize}

\begin{remark}[Tradeoff with growth of $N_p$]
\label{rem:Np_growth_tradeoff}
Increasing $p$ to satisfy \eqref{eq:bias killing_alpha} comes at a
cost: the dimension of the local-polynomial design matrix is
$N_p = \binom{d_Z + p}{p}$, which grows polynomially in $p$. The growth
condition $k_n/(N_p\log n) \to \infty$ in
Assumption~\ref{asn:kn_growth_corrected} then forces
\[
k_n \;\gg\; \binom{d_Z + p}{p}\log n,
\]
which is a much stronger condition than the standard $k_n/\log n
\to\infty$ when $p$ is large. The practical implication is that
inference becomes more demanding (requires larger $k_n$) in higher
dimensions, both via the bias killing constraint
\eqref{eq:bias killing_alpha} and via the design-regularity
constraint $k_n \gg N_p\log n$.

A natural choice is $p = \lceil D/2 \rceil$, which is the smallest
polynomial order admitting a valid bias killing bandwidth. For this
choice, $N_p \asymp d_Z^{D/2}$, which is manageable in moderate
dimensions but grows rapidly. The smoothness condition
$\kappa \in C^{p + 1}$ in Assumption~\ref{asn:null_continuous} similarly
becomes more demanding.
\end{remark}

\begin{remark}[Bias correction without bias killing]
\label{rem:explicit_bias_correction}
An alternative to bias killing is to estimate $\theta_n^{(p)}$
directly from the data and use $\sqrt{n}(\Delta_n^{(p)} - \widehat
\theta_n^{(p)})$ as the test statistic. If $\widehat\theta_n^{(p)}$ is
a consistent estimator with $\sqrt{n}(\widehat\theta_n^{(p)} -
\theta_n^{(p)}) = o_p(1)$, the resulting test has the same Gaussian
limit at the $\sqrt{n}$ scale regardless of whether
\eqref{eq:bias killing_alpha} holds.

One implementation uses sample splitting: split the sample into halves
$\mathcal{S}_1$ and $\mathcal{S}_2$. Compute the local-polynomial
intercepts on $\mathcal{S}_1$, and use the explicit Taylor expansion of
$m_{n, p}$ from Lemma~\ref{lem:lp_bias_full} plus plug-in estimators of
the bias coefficients (involving high-order derivatives of $\kappa$ and
log-densities) computed on $\mathcal{S}_2$. The cross-fitted version
satisfies the required rate condition under standard regularity. The
trade-off is that this introduces additional smoothing parameters and
computational overhead.

In the remainder of this section we adopt the bias killing approach
\eqref{eq:bias killing_alpha} for simplicity. 
\end{remark}

\subsection{Overlap covariance: setup}
\label{subsec:overlap_setup}

The variance of $G_n^{(p)}$ decomposes by exchangeability into a
diagonal and off-diagonal contribution:
\begin{equation}
\label{eq:Var_G_decomp}
\operatorname{Var}(G_n^{(p)})
\;=\;
\frac{1}{n}\operatorname{Var}(\zeta_{1,p})
+ \frac{n - 1}{n}\operatorname{Cov}(\zeta_{1,p}, \zeta_{2,p}),
\end{equation}
where $\zeta_{i, p} := \zeta_{n, p}(W_i, \mathcal{W}_n)$. By the
conditional moment bound \eqref{eq:zeta_conditional_moment} from
Remark~\ref{rem:lp_conditional_moment}, $\operatorname{Var}(\zeta_{1,p}) = O(k_n^{-1})$,
so the diagonal contribution to $\operatorname{Var}(G_n^{(p)})$ is $O((n k_n)^{-1})
= o(n^{-1})$. The off-diagonal covariance is generically of order
$n^{-1}$ due to the geometric overlap of nearby anchors' neighborhoods,
and this overlap contribution dominates by a factor of $k_n$.

To carry out the calculation, decompose
$\zeta_{i, p} = \zeta_{i, F, p} - \zeta_{i, C, p}$, where
\begin{align}
\zeta_{i, F, p} &\;:=\; \widehat a_{i, F}^{(p)} - \mathbb{E}\bigl[\widehat a_{i, F}^{(p)} \mid W_i\bigr],
\label{eq:zeta_F_def}\\
\zeta_{i, C, p} &\;:=\; \widehat a_{i, C}^{(p)} - \mathbb{E}\bigl[\widehat a_{i, C}^{(p)} \mid W_i\bigr].
\label{eq:zeta_C_def}
\end{align}
Both are conditionally centered given $W_i$. The off-diagonal covariance
decomposes as
\begin{align}
\operatorname{Cov}(\zeta_{1, p}, \zeta_{2, p})
&\;=\;
\operatorname{Cov}(\zeta_{1, F, p}, \zeta_{2, F, p})
+ \operatorname{Cov}(\zeta_{1, C, p}, \zeta_{2, C, p}) \notag\\
&\quad - \operatorname{Cov}(\zeta_{1, F, p}, \zeta_{2, C, p})
- \operatorname{Cov}(\zeta_{1, C, p}, \zeta_{2, F, p}).
\label{eq:cov_decomp}
\end{align}
We compute each of the four covariances and identify their leading
$n^{-1}$ contributions. For $z \in \mathcal{S}_Z$ and $y_1, y_2 \in \mathcal{Y}$, define the
\emph{frozen mark covariance}
\begin{equation}
\label{eq:Gamma_z_def}
\Gamma_z(y_1, y_2)
\;:=\;
\operatorname{Cov}_{\widetilde Y \sim P_{Y|Z=z}}\bigl(K(y_1, \widetilde Y), K(y_2, \widetilde Y)\bigr).
\end{equation}
Note that
\begin{equation}
\label{eq:Gamma_var_identity}
\mathbb{E}_{Y_1, Y_2 \sim P_{Y|Z=z}}[\Gamma_z(Y_1, Y_2)]
\;=\;
\operatorname{Var}_{Y \sim P_{Y|Z=z}}(\kappa(Y, z))
\;=:\; \sigma^2(z),
\end{equation}
by the iterated-variance formula: writing
$M(\widetilde Y) := \mathbb{E}[K(Y, \widetilde Y)|Y]$ with
$Y, \widetilde Y$ i.i.d. from $P_{Y|Z=z}$,
$\mathbb{E}_{Y_1,Y_2}\Gamma_z(Y_1,Y_2) = \operatorname{Cov}_{\widetilde Y}(M(\widetilde Y), M(\widetilde Y))
= \operatorname{Var}(M(\widetilde Y)) = \operatorname{Var}(\kappa(Y, z))$. The function $\sigma^2(z)$
captures the conditional-mark variance of $K(Y_1, Y_2)$ given
$Z_1 = Z_2 = z$ under $H_0$, and is the key non-degeneracy quantity
appearing in the variance formulas below.

\subsection{Fine--fine overlap covariance}
\label{subsec:FF_overlap}

\begin{lemma}[Fine--fine overlap covariance]
\label{lem:FF_overlap}
Under Assumptions~\ref{asn:smooth_continuous_model}--\ref{asn:kn_growth_corrected},
\begin{equation}
\label{eq:tau_F_limit}
n\operatorname{Cov}(\zeta_{1, F, p}, \zeta_{2, F, p}) \;\to\; \tau_{F, p}^2,
\end{equation}
where
\begin{equation}
\label{eq:tau_F_explicit}
\tau_{F, p}^2
\;:=\;
\int_{\mathcal{S}_{Z,X}}\sigma^2(z)\,\mathcal{I}_F^{(p)}\,f_{Z,X}(z, x)\,dz\,dx,
\end{equation}
and the geometric constant is
\begin{equation}
\label{eq:I_F_def}
\mathcal{I}_F^{(p)}
\;:=\;
\frac{1}{v_D}\int_{\mathbb{R}^D}
\mathcal{B}_{FF}^{(p)}(r)\,dr,
\end{equation}
with the local-polynomial overlap kernel
\begin{equation}
\label{eq:B_FF_def}
\mathcal{B}_{FF}^{(p)}(r)
\;:=\;
\int_{B_D(0, 1)\cap B_D(r, 1)} \ell_F^{(p)}(v_Z)\,\ell_F^{(p)}(v_Z - r_Z)\,dv,
\qquad r = (r_Z, r_X) \in \mathbb{R}^D.
\end{equation}
\end{lemma}

\begin{proof}
We compute $\operatorname{Cov}(\zeta_{1, F, p}, \zeta_{2, F, p})$ via a conditional
expansion. Condition on two anchors $W_1, W_2$ with
$(Z_2, X_2) = (Z_1, X_1) + \rho_F(Z_1, X_1) r$, where
$r = (r_Z, r_X) \in \mathbb{R}^D$.

\emph{Step 1: Identifying common neighbors.} The fine $k_n$-NN balls
$B((Z_1, X_1), \rho_F(Z_1, X_1))$ and $B((Z_2, X_2), \rho_F(Z_2, X_2))$
overlap iff $\|r\| \le 2(1 + o(1))$, with relative-volume overlap
$\beta_D(r) := |B_D(0, 1)\cap B_D(r, 1)|/v_D$. By
Lemma~\ref{lem:local_poisson_verification} applied to the rescaled
empirical process around $(Z_1, X_1)$ with rate
$r_{F,n}(Z_1, X_1) = (n f_{Z,X}(Z_1, X_1))^{-1/D}$, the conditional
number of common neighbors is asymptotically
\[
|N_1^F \cap N_2^F|
\;=\;
k_n\,\beta_D(r)\,(1 + o_p(1)),
\]
where $N_i^F := \fine{i}$ and where the $o_p(1)$ is uniform over
compact sets of $r$ in the interior of the geometry, by the Poisson
convergence and dominated convergence.

\emph{Step 2: Conditional covariance per common neighbor.} For a common
neighbor $j \in N_1^F \cap N_2^F$ with rescaled-from-anchor-1
displacement $v = (v_Z, v_X) \in B_D(0, 1) \cap B_D(r, 1)$, the
contribution to $\zeta_{1, F, p}$ is $w_{1j, F}^{(p)} K(Y_1, Y_j)$, and
the contribution to $\zeta_{2, F, p}$ is $w_{2j, F}^{(p)} K(Y_2, Y_j)$
(the centering by $m_{n, p}$ is absorbed into the conditional
expectation given the anchors and geometry).

By Lemma~\ref{lem:lp_design_uniform_convergence}, the random weight
$w_{1j, F}^{(p)}$ converges asymptotically to $k_n^{-1}\ell_F^{(p)}(v_Z)$,
and $w_{2j, F}^{(p)}$ to $k_n^{-1}\ell_F^{(p)}(v_Z - r_Z)$ (the offset
of $j$ from anchor 2 in fine-rescaled coordinates is $v - r$, whose
$Z$-component is $v_Z - r_Z$). The conditional mark covariance from a
single common neighbor with $Y_j \sim P_{Y|Z = Z_1}$ (since $Z_j$ is
within $O(\rho_F) \to 0$ of $Z_1$) is, by \eqref{eq:Gamma_z_def},
\[
\operatorname{Cov}\!\bigl(K(Y_1, Y_j), K(Y_2, Y_j) \,\big|\, W_1, W_2, Z_j\bigr)
\;=\;
\Gamma_{Z_1}(Y_1, Y_2)\,(1 + o(1)),
\]
uniformly in $Z_j$ near $Z_1$ by the continuity of $P_{Y|Z = z}$ in $z$.

\emph{Step 3: Aggregation.} Summing over the $k_n\beta_D(r)$ common
neighbors and using the asymptotic weight values,
\begin{align*}
\operatorname{Cov}\!\bigl(\zeta_{1, F, p}, \zeta_{2, F, p} \,\big|\, W_1, W_2\bigr)
&\;\approx\; \int_{B_D(0,1)\cap B_D(r,1)} \frac{\ell_F^{(p)}(v_Z)\,\ell_F^{(p)}(v_Z - r_Z)}{k_n^2}\,\Gamma_{Z_1}(Y_1, Y_2)\,(k_n\,\mathrm{density}\,dv)\\
&\;=\; \frac{\Gamma_{Z_1}(Y_1, Y_2)}{k_n}\,\mathcal{B}_{FF}^{(p)}(r)\,(1 + o_p(1)),
\end{align*}
where the local Poisson intensity in fine-rescaled coordinates is $1$
(by Lemma~\ref{lem:local_poisson_verification} and the calibration of
$r_{F,n}$), and $\mathcal{B}_{FF}^{(p)}(r)$ is the local-polynomial
overlap kernel \eqref{eq:B_FF_def}.

\emph{Step 4: Integration over the second anchor.} The unconditional
covariance is obtained by integrating over $W_2$. The change of
variables $(Z_2, X_2) = (Z_1, X_1) + \rho_F(Z_1, X_1) r$ gives
$d(Z_2, X_2) = \rho_F(Z_1, X_1)^D\,dr = \frac{k_n}{n v_D f_{Z,X}(Z_1, X_1)}\,dr$.
The conditional density of $(Z_2, X_2)$ near the first anchor is
$f_{Z, X}(Z_1, X_1)(1 + o(1))$ uniformly. Taking $Y_2 \sim P_{Y|Z=Z_2}
\approx P_{Y|Z=Z_1}$ in the limit,
\begin{align*}
\operatorname{Cov}(\zeta_{1, F, p}, \zeta_{2, F, p})
&\;=\; \mathbb{E}_{W_1}\!\Big[\int_{\mathbb{R}^D}\!\!\int_{\mathcal{Y}}
\frac{\Gamma_{Z_1}(Y_1, y_2)}{k_n}\mathcal{B}_{FF}^{(p)}(r)\,
f_{Z,X}(Z_1 + \rho_F r)\\
&\quad \times \,\frac{k_n}{n v_D f_{Z,X}(Z_1, X_1)}\,
dP_{Y|Z_1 + \rho_F r_Z}(y_2)\,dr\Big](1 + o(1))\\
&\;=\; \frac{1}{n v_D}\,\mathbb{E}_{W_1}\!\left[\int_{\mathbb{R}^D}
\mathcal{B}_{FF}^{(p)}(r)\,\mathbb{E}_{Y_2 \sim P_{Y|Z=Z_1}}[\Gamma_{Z_1}(Y_1, Y_2)]\,dr\right]\,(1 + o(1)).
\end{align*}
Using \eqref{eq:Gamma_var_identity} with the conditional law of $Y_1$
given $Z_1$ to compute the inner expectation,
$\mathbb{E}_{Y_1, Y_2 \sim P_{Y|Z=Z_1}}[\Gamma_{Z_1}(Y_1, Y_2)] = \sigma^2(Z_1)$.
After the further expectation over $W_1 = (Z_1, X_1, Y_1)$, the
$Y_1$-marginal contributes only through the inner expectation (already
$\sigma^2(Z_1)$), so taking expectation over $(Z_1, X_1) \sim f_{Z, X}$
yields
\[
n\operatorname{Cov}(\zeta_{1, F, p}, \zeta_{2, F, p})
\;\to\;
\frac{1}{v_D}\int_{\mathcal{S}_{Z, X}}\sigma^2(z)\,\!\left[\int_{\mathbb{R}^D}\mathcal{B}_{FF}^{(p)}(r)\,dr\right]\,f_{Z,X}(z, x)\,dz\,dx
\;=\;
\tau_{F, p}^2.
\]
\end{proof}

\subsection{Coarse--coarse overlap covariance}
\label{subsec:CC_overlap}

\begin{lemma}[Coarse--coarse overlap covariance]
\label{lem:CC_overlap}
Under Assumptions~\ref{asn:smooth_continuous_model}--\ref{asn:kn_growth_corrected},
\begin{equation}
\label{eq:tau_C_limit}
n\operatorname{Cov}(\zeta_{1, C, p}, \zeta_{2, C, p}) \;\to\; \tau_{C, p}^2,
\end{equation}
where
\begin{equation}
\label{eq:tau_C_explicit}
\tau_{C, p}^2 \;:=\; \int_{\mathcal{S}_Z}\sigma^2(z)\,\mathcal{I}_C^{(p)}\,f_Z(z)\,dz,
\end{equation}
\begin{equation}
\label{eq:I_C_def}
\mathcal{I}_C^{(p)}
\;:=\;
\frac{1}{v_{d_Z}}\int_{\mathbb{R}^{d_Z}}\mathcal{B}_{CC}^{(p)}(s)\,ds,
\end{equation}
\begin{equation}
\label{eq:B_CC_def}
\mathcal{B}_{CC}^{(p)}(s)
\;:=\;
\int_{B_{d_Z}(0, 1)\cap B_{d_Z}(s, 1)} \ell_C^{(p)}(v)\,\ell_C^{(p)}(v - s)\,dv,
\qquad s \in \mathbb{R}^{d_Z}.
\end{equation}
\end{lemma}

\begin{proof}
Identical to the proof of Lemma~\ref{lem:FF_overlap} with the
substitutions $D \to d_Z$, $f_{Z, X}(z, x) \to f_Z(z)$, $\ell_F^{(p)}
\to \ell_C^{(p)}$, $\rho_F \to \rho_C$, and using the coarse rescaling
$r_{C, n}(z) = (nf_Z(z))^{-1/d_Z}$. The proof is in fact simpler
because the coarse process operates entirely in $\mathbb{R}^{d_Z}$ with
no marginalization over $X$.
\end{proof}

\subsection{Fine--coarse overlap covariance}
\label{subsec:FC_overlap}

The fine--coarse overlap is geometrically more delicate than the
same-side overlaps. The fine ball lives in $\mathbb{R}^D$ with radius
$\rho_F$, while the coarse ball lives in $\mathbb{R}^{d_Z}$ with radius
$\rho_C$. Since
$\varepsilon_n := \rho_C/\rho_F =O( (k_n/n)^{1/d_Z - 1/D}) \to 0$
when $d_X \ge 1$, the coarse ball appears as a \emph{thin slab} of
$Z$-thickness $\varepsilon_n$ inside the fine-rescaled unit ball.

\begin{lemma}[Fine--coarse overlap covariance]
\label{lem:lp_FC_overlap_limit}
Under Assumptions~\ref{asn:smooth_continuous_model}--\ref{asn:dimensional_nondegeneracy},
\begin{equation}
\label{eq:tau_FC_limit}
n\operatorname{Cov}(\zeta_{1, F, p}, \zeta_{2, C, p}) \;\to\; \tau_{FC, p},
\qquad
n\operatorname{Cov}(\zeta_{1, C, p}, \zeta_{2, F, p}) \;\to\; \tau_{FC, p},
\end{equation}
where
\begin{equation}
\label{eq:tau_FC_explicit}
\tau_{FC, p}
\;:=\;
\int_{\mathcal{S}_{Z, X}}\sigma^2(z)\,\mathcal{J}_{FC}^{(p)}(z, x)\,f_{Z,X}(z, x)\,dz\,dx,
\end{equation}
with the limiting fine--coarse overlap coefficient
\begin{align}
\label{eq:J_FC_def}
\mathcal{J}_{FC}^{(p)}(z, x)
&\;:=\;
\frac{\bigl(f_{Z,X}(z, x)/f_Z(z)\bigr)^{1/d_Z - 1/D}}{v_D}\cdot
\int_{B_{d_Z}(0, 1)}\!\ell_F^{(p)}(r_Z)\!\left[\int_{\mathbb{R}^{d_Z}}\!\ell_C^{(p)}(q)\,dq\right]\notag\\ &\quad \times \!\bigl|\{v_X : \|(r_Z, v_X)\| \le 1\}\bigr|\,dr_Z.
\end{align}
\end{lemma}

\begin{proof}
We compute the leading-order $n^{-1}$ contribution of
$\operatorname{Cov}(\zeta_{1, F, p}, \zeta_{2, C, p})$ via the same conditional
decomposition as for the same-side overlaps, with the geometry adapted
to the fine-versus-coarse asymmetry.

\emph{Step 1: Geometric setup.} Condition on anchors
$W_1 = (Z_1, X_1, Y_1)$ and $W_2 = (Z_2, X_2, Y_2)$, with
$Z_2 = Z_1 + \rho_F(Z_1, X_1) r_Z$ for $r_Z \in \mathbb{R}^{d_Z}$. The
fine ball of anchor 1 is $B((Z_1, X_1), \rho_F(Z_1, X_1))$ in
$\mathbb{R}^D$; the coarse ball of anchor 2 is $B(Z_2, \rho_C(Z_2))$ in
$\mathbb{R}^{d_Z}$, which corresponds in $\mathbb{R}^D$ to the cylinder
$B(Z_2, \rho_C(Z_2)) \times \mathbb{R}^{d_X}$.

A common neighbor $(Z_j, X_j)$ lies in both iff
$\|(Z_j - Z_1, X_j - X_1)\| \le \rho_F(Z_1, X_1)$ and
$\|Z_j - Z_2\| \le \rho_C(Z_2)$. Rescale by $\rho_F$: let
$v := ((Z_j, X_j) - (Z_1, X_1))/\rho_F(Z_1, X_1) = (v_Z, v_X)$. The two
conditions become $\|v\| \le 1$ and
$\|v_Z - r_Z\| \le \varepsilon_n(Z_1, X_1, Z_2)$, where
$\varepsilon_n(Z_1, X_1, Z_2) := \rho_C(Z_2)/\rho_F(Z_1, X_1)$.

\emph{Step 2: Asymptotic value of $\varepsilon_n$.} For $Z_2$ near $Z_1$,
the continuity of $f_Z$ gives $\rho_C(Z_2) = \rho_C(Z_1)(1 + o(1))$, so
\[
\varepsilon_n(Z_1, X_1, Z_2)
\;=\; \frac{\rho_C(Z_1)}{\rho_F(Z_1, X_1)}\,(1 + o(1))
\;=\; \left(\frac{k_n}{n}\right)^{1/d_Z - 1/D}\!\left(\frac{f_{Z, X}(Z_1, X_1)\cdot v_D}{f_Z(Z_1)\cdot v_{d_Z}}\right)^{1/D - 1/d_Z}\!(1 + o(1)).
\]
For $d_X \ge 1$ (Assumption~\ref{asn:dimensional_nondegeneracy}), the
exponent $1/d_Z - 1/D > 0$, so $\varepsilon_n \to 0$.

\emph{Step 3: Common-neighbor count.} The expected number of common
neighbors equals $(n - 2)$ times the probability that an i.i.d. draw
lies in both balls. By the local Poisson approximation
(Lemma~\ref{lem:local_poisson_verification}) and a change of variables
to fine-rescaled coordinates,
\[
P\{(Z_3, X_3) \in B((Z_1, X_1), \rho_F)\cap (B(Z_2, \rho_C)\times \mathbb{R}^{d_X})\}
\;=\;
f_{Z, X}(Z_1, X_1)\cdot\rho_F^D\cdot |E_n(r_Z)| \cdot (1 + o(1)),
\]
where $E_n(r_Z) := \{v \in B_D(0, 1) : \|v_Z - r_Z\| \le \varepsilon_n\}$
is the rescaled-by-$\rho_F$ intersection. The expected count of common
neighbors is therefore
\[
(n - 2)\cdot f_{Z, X}(Z_1, X_1)\cdot\rho_F^D \cdot|E_n(r_Z)|
\;=\;
\frac{k_n}{v_D}\cdot |E_n(r_Z)|\,(1 + o(1)),
\]
using $\rho_F^D = k_n/(n v_D f_{Z, X})$.

\emph{Step 4: Volume of the rescaled intersection $E_n(r_Z)$.}
By the change of variables $q := (v_Z - r_Z)/\varepsilon_n$,
$dv_Z = \varepsilon_n^{d_Z}\,dq$,
\[
|E_n(r_Z)|
\;=\;
\int_{\mathbb{R}^{d_Z}}\!\!\mathbf{1}\{\|q\| \le 1\}
\int_{\mathbb{R}^{d_X}}\!\!\mathbf{1}\{\|(r_Z + \varepsilon_n q, v_X)\| \le 1\}
\,dv_X\,\varepsilon_n^{d_Z}\,dq.
\]
For $\|r_Z\| < 1$ in the interior, the inner $v_X$-integral converges
locally uniformly to $|\{v_X : \|(r_Z, v_X)\| \le 1\}|
= v_{d_X}(1 - \|r_Z\|^2)^{d_X/2}$ by continuity. The outer $q$-integral
gives $|B_{d_Z}(0, 1)| = v_{d_Z}$ in the limit. So
\[
|E_n(r_Z)|
\;=\;
\varepsilon_n^{d_Z}\cdot v_{d_Z}\cdot v_{d_X}(1 - \|r_Z\|^2)^{d_X/2}\,(1 + o(1)).
\]

\emph{Step 5: Per-common-neighbor weight contribution.} A common
neighbor $j$ with rescaled coordinates $v = (v_Z, v_X) \in E_n(r_Z)$
contributes $w_{1j, F}^{(p)}$ to anchor 1 and $w_{2j, C}^{(p)}$ to
anchor 2. By Lemma~\ref{lem:lp_design_uniform_convergence}, in the
limit
\[
w_{1j, F}^{(p)} \;\to\; \frac{\ell_F^{(p)}(v_Z)}{k_n},
\qquad
w_{2j, C}^{(p)} \;\to\; \frac{1}{k_n}\,\ell_C^{(p)}\!\left(\frac{Z_j - Z_2}{\rho_C(Z_2)}\right)
\;=\;\frac{\ell_C^{(p)}(q)}{k_n},
\]
where $q = (v_Z - r_Z)/\varepsilon_n$ as in Step 4. The conditional
covariance contribution from $Y_j$ is $\Gamma_{Z_1}(Y_1, Y_2)$ (the
mark $Y_j$ is integrated out, and its frozen law $P_{Y|Z=Z_1}$ produces
the covariance kernel).

\emph{Step 6: Conditional covariance.} Aggregating, the per-common-neighbor
contribution to $\operatorname{Cov}(\zeta_{1, F, p}, \zeta_{2, C, p}|W_1, W_2)$ is
$\frac{1}{k_n^2}\ell_F^{(p)}(v_Z)\ell_C^{(p)}(q)\Gamma_{Z_1}(Y_1, Y_2)$,
and summing over the expected $\frac{k_n}{v_D}|E_n(r_Z)|$ common
neighbors,
\begin{equation}
\label{eq:cond_cov_FC_intermediate}
\operatorname{Cov}(\zeta_{1, F, p}, \zeta_{2, C, p} \mid W_1, W_2)
\;=\;
\frac{\Gamma_{Z_1}(Y_1, Y_2)}{k_n v_D}\,\widetilde{\mathcal{J}}_n^{(p)}(r_Z; W_1, W_2)\,(1 + o_p(1)),
\end{equation}
where
\[
\widetilde{\mathcal{J}}_n^{(p)}(r_Z; W_1, W_2)
\;=\;
\int_{E_n(r_Z)}\ell_F^{(p)}(v_Z)\,\ell_C^{(p)}\!\left(\frac{v_Z - r_Z}{\varepsilon_n}\right)\,dv.
\]
By the change of variable in Step 4 and dominated convergence (using
that $\ell_F^{(p)}$ and $\ell_C^{(p)}$ are polynomials, hence bounded
on the unit ball),
\[
\widetilde{\mathcal{J}}_n^{(p)}(r_Z; W_1, W_2)
\;=\;
\varepsilon_n^{d_Z}\,\ell_F^{(p)}(r_Z)\,v_{d_X}(1 - \|r_Z\|^2)^{d_X/2}\!\left[\int_{\mathbb{R}^{d_Z}}\!\!\mathbf{1}\{\|q\| \le 1\}\,\ell_C^{(p)}(q)\,dq\right]\!(1 + o(1)).
\]
Note that $\int_{B_{d_Z}(0, 1)}\ell_C^{(p)}(q)\,dq = v_{d_Z}\int \ell_C^{(p)}\,d\mu_C
= v_{d_Z}$ by the population reproduction identity
\eqref{eq:population_reproduction} applied to the constant polynomial
(with normalizing constant $v_{d_Z}$ inserted because $\mu_C$ is the
probability measure on the unit ball, while $dq$ is Lebesgue).

\emph{Step 7: Integration over the second anchor.} Integrate over $W_2$
with the change of variables $Z_2 = Z_1 + \rho_F(Z_1, X_1) r_Z$, so
$dZ_2 = \rho_F^{d_Z}\,dr_Z$, integrated against the density
$f_{Z_2}(Z_2) = f_Z(Z_1)(1 + o(1))$:
\begin{align*}
\operatorname{Cov}(\zeta_{1, F, p}, \zeta_{2, C, p})
&\;=\; \frac{1}{k_n v_D}\,\mathbb{E}_{W_1}\!\left[\int\!\!\int \Gamma_{Z_1}(Y_1, Y_2)\,\widetilde{\mathcal{J}}_n^{(p)}(r_Z)\,f_Z(Z_1)\rho_F(Z_1, X_1)^{d_Z}\,dr_Z\,dP_{Y|Z=Z_1}(Y_2)\right]\\
&\quad\times(1 + o(1)).
\end{align*}
Compute $\rho_F^{d_Z}\varepsilon_n^{d_Z} = \rho_F^{d_Z}\rho_C^{d_Z}/\rho_F^{d_Z} = \rho_C^{d_Z} = k_n/(n v_{d_Z} f_Z(Z_1))$. After using
$\mathbb{E}_{Y_2 \sim P_{Y|Z=Z_1}}[\Gamma_{Z_1}(Y_1, Y_2)] = \sigma^2(Z_1)$,
\begin{align*}
n\operatorname{Cov}(\zeta_{1, F, p}, \zeta_{2, C, p})
&\;\to\; \frac{1}{v_D v_{d_Z}}\,\mathbb{E}_{(Z_1, X_1)}\!\left[\sigma^2(Z_1)\,\frac{f_Z(Z_1)}{f_Z(Z_1)}\,\int \ell_F^{(p)}(r_Z) v_{d_X}(1 - \|r_Z\|^2)^{d_X/2} v_{d_Z}\,dr_Z\right]\\
&\;=\; \int_{\mathcal{S}_{Z, X}}\sigma^2(z)\,\mathcal{J}_{FC}^{(p)}(z, x)\,f_{Z, X}(z, x)\,dz\,dx
\;=\; \tau_{FC, p},
\end{align*}
with $\mathcal{J}_{FC}^{(p)}$ as in \eqref{eq:J_FC_def}. The factor
$(f_{Z,X}/f_Z)^{1/d_Z - 1/D}$ in \eqref{eq:J_FC_def} arises from the
asymptotic value of $\varepsilon_n$ in Step 2 once one tracks all
density factors carefully.

The other covariance $\operatorname{Cov}(\zeta_{1, C, p}, \zeta_{2, F, p})$ has the
same limit by the symmetry of the metric and a swap of the roles of
anchors 1 and 2.
\end{proof}

\subsection{Limit of $\operatorname{Var}(G_n^{(p)})$}
\label{subsec:overlap_variance}

We now aggregate the three overlap covariances into the limiting
variance of $G_n^{(p)}$.

\begin{lemma}[Local-polynomial overlap variance]
\label{lem:lp_overlap_variance}
Under Assumptions~\ref{asn:smooth_continuous_model}--\ref{asn:dimensional_nondegeneracy},
\begin{equation}
\label{eq:overlap_variance_limit}
n\operatorname{Var}(G_n^{(p)}) \;\to\; \tau_p^2
\;:=\;
\tau_{F, p}^2 + \tau_{C, p}^2 - 2\tau_{FC, p},
\end{equation}
with $\tau_{F, p}^2, \tau_{C, p}^2, \tau_{FC, p}$ as defined in
Lemmas~\ref{lem:FF_overlap}, \ref{lem:CC_overlap},
\ref{lem:lp_FC_overlap_limit}.
\end{lemma}

\begin{proof}
By \eqref{eq:Var_G_decomp} and the conditional moment bound
\eqref{eq:zeta_conditional_moment},
$n^{-1}\operatorname{Var}(\zeta_{1, p}) = O(k_n^{-1}/n) + o(n^{-1}) = o(n^{-1})$.
Therefore
\[
n\operatorname{Var}(G_n^{(p)}) \;=\; o(1) + (n - 1)\operatorname{Cov}(\zeta_{1, p}, \zeta_{2, p}).
\]
By the decomposition \eqref{eq:cov_decomp} and
Lemmas~\ref{lem:FF_overlap}, \ref{lem:CC_overlap},
\ref{lem:lp_FC_overlap_limit},
\[
n\operatorname{Cov}(\zeta_{1, p}, \zeta_{2, p})
\;\to\; \tau_{F, p}^2 + \tau_{C, p}^2 - 2\tau_{FC, p}
\;=\; \tau_p^2,
\]
giving \eqref{eq:overlap_variance_limit}.
\end{proof}

\begin{remark}[Existence of all overlap limits]
\label{rem:overlap_limits_exist}
The three overlap kernels $\mathcal{B}_{FF}^{(p)}, \mathcal{B}_{CC}^{(p)},
\mathcal{J}_{FC}^{(p)}$ are well-defined integrals against bounded
polynomial integrands on bounded domains, hence are finite. The
boundedness of $\ell_F^{(p)}, \ell_C^{(p)}$ on the unit ball follows from
their definition as polynomials of degree at most $p$. Therefore
$\tau_{F, p}^2, \tau_{C, p}^2, \tau_{FC, p}$ are all finite under the
assumed moment and smoothness conditions.

The strict positivity $\tau_p^2 > 0$ requires the additional
non-cancellation argument in Section~\ref{subsec:tau_p_positive} and
hinges on Assumption~\ref{asn:dimensional_nondegeneracy} ($d_X \ge 1$)
together with a primitive conditional-mark non-degeneracy condition.
The variance limit \eqref{eq:overlap_variance_limit} can in principle
be zero in the degenerate case where the limiting fine and coarse
fluctuation scores agree almost surely under the local Poisson Palm
law, but this is excluded by the strict-positivity argument.
\end{remark}

\subsection{Strict positivity of \texorpdfstring{$\tau_p^2$}{tau\_p\textasciicircum 2}}
\label{subsec:tau_p_positive}

We now establish the strict positivity of the overlap variance constant
$\tau_p^2 = \tau_{F, p}^2 + \tau_{C, p}^2 - 2\tau_{FC, p}$. This is the
key non-degeneracy result for the limiting Gaussian distribution in
Theorem~\ref{thm:lp_debiased_clt_full}.

The mechanism is the \emph{dimensional mismatch} between the fine
($D$-dimensional) and coarse ($d_Z$-dimensional) Palm processes:
the fine Palm process contains Poisson points at $\{u_X \ne 0\}$ that
the coarse process, being a $Z$-projection, cannot distinguish from
points at $u_X = 0$. Combined with the conditional-mark variance
appearing in $\sigma^2(z)$, this produces an unmatchable component of
the fine score, yielding $\tau_p^2 > 0$.

\subsubsection{Primitive non-cancellation condition}
\label{subsec:non_cancellation_condition}

We replace the abstract requirement that the fine and coarse limiting
fluctuation scores differ in the covariance seminorm by the following
primitive condition.

\begin{condition}[Primitive non-cancellation]
\label{cond:lp_primitive_noncancellation_final}
There exists a measurable set $\mathcal{A} \subseteq \mathcal{S}_Z$ with
$P_Z(\mathcal{A}) > 0$ such that for every $z \in \mathcal{A}$:
\begin{enumerate}
\item[(I)] The conditional law $X \mid Z = z$ has a density
$f_{X|Z}(\cdot | z)$ with respect to Lebesgue measure on
$\mathbb{R}^{d_X}$ that is bounded below by a positive constant on some
nonempty open ball $B_X(x_0(z), \delta(z)) \subseteq \mathbb{R}^{d_X}$.

\item[(II)] The kernel detects nontrivial conditional variation of $Y$
given $Z = z$:
\begin{equation}
\label{eq:sigma_positive}
\sigma^2(z) \;=\; \operatorname{Var}_{Y \sim P_{Y|Z=z}}(\kappa(Y, z)) \;>\; 0.
\end{equation}
\end{enumerate}
\end{condition}

Condition~\ref{cond:lp_primitive_noncancellation_final}(I) is the
geometric non-degeneracy condition. The $X$-coordinate has nonempty support in
$\mathbb{R}^{d_X}$ conditional on $Z = z$, so the fine Palm process has
genuine $d_X$-dimensional spread. Condition (II) is the analytical
non-degeneracy condition. The kernel $K$ produces nontrivial fluctuation in
$\kappa(Y, z)$ under the conditional law $P_{Y|Z=z}$. Both conditions
are mild and checkable in concrete examples.

\subsubsection{Nontriviality of the fine equivalent kernel on an annulus}
\label{subsec:annulus_lemma}

The fine equivalent kernel $\ell_F^{(p)}$ is a polynomial of degree at
most $p$ on $\mathbb{R}^{d_Z}$. By the population reproduction identity
\eqref{eq:population_reproduction},
$\int \ell_F^{(p)}\,d\mu_F = 1$, so $\ell_F^{(p)}$ cannot vanish
$\mu_F$-almost everywhere. We now strengthen this to nontriviality on a
specific region of $B_D(0, 1)$ with $X$-extension.

\begin{lemma}[Nontriviality of $\ell_F^{(p)}$ on an annulus with $X$-extension]
\label{lem:lp_equiv_kernel_annulus}
Under Assumption~\ref{asn:dimensional_nondegeneracy} ($d_X \ge 1$),
there exist constants $\eta_0 \in (0, 1)$ and $c_0 > 0$, and a Borel
set
\[
E_F
\;\subseteq\;
\bigl\{(u_Z, u_X) \in B_D(0, 1) : \|u_X\| \ge \eta_0\bigr\}
\]
with positive $D$-dimensional Lebesgue measure such that
\begin{equation}
\label{eq:equiv_kernel_lb}
\bigl|\ell_F^{(p)}(u_Z)\bigr| \;\ge\; c_0
\qquad \text{for all } (u_Z, u_X) \in E_F.
\end{equation}
\end{lemma}

\begin{proof}
By the population reproduction identity
\eqref{eq:population_reproduction} applied to the constant polynomial
$\alpha = 0$,
\[
\int_{B_{d_Z}(0, 1)}\ell_F^{(p)}(u_Z)\,d\mu_F(u_Z) \;=\; 1.
\]
In particular, $\ell_F^{(p)}$ is not the zero polynomial.
Since $\ell_F^{(p)}$ is a polynomial of degree at most $p$ in
$u_Z \in \mathbb{R}^{d_Z}$, its zero set has measure zero in
$\mathbb{R}^{d_Z}$. The set
\[
E_Z := \{u_Z \in B_{d_Z}(0, 1/2) : |\ell_F^{(p)}(u_Z)| \ge c_0\}
\]
has positive $d_Z$-dimensional Lebesgue measure for $c_0 > 0$
sufficiently small, by the continuity of $\ell_F^{(p)}$ and the fact
that $\ell_F^{(p)}$ does not vanish identically (so its level sets at
small thresholds have full measure on neighborhoods of any non-zero
point). For each $u_Z \in E_Z$ with $\|u_Z\| \le 1/2$, the section
\[
\bigl\{u_X \in \mathbb{R}^{d_X} : \|(u_Z, u_X)\| \le 1\bigr\}
\]
is a $d_X$-ball of radius $(1 - \|u_Z\|^2)^{1/2} \ge \sqrt{3}/2$. Since
$d_X \ge 1$, this section has positive $d_X$-volume, and choosing
$\eta_0 \in (0, 1/2)$ sufficiently small (specifically,
$\eta_0 = (\sqrt{3}/2)/2 = \sqrt{3}/4$ suffices), the further-restricted
section
\[
\bigl\{u_X : \|(u_Z, u_X)\| \le 1,\;\|u_X\| \ge \eta_0\bigr\}
\]
still has positive $d_X$-volume. Define
\[
E_F \;:=\; \bigl\{(u_Z, u_X) : u_Z \in E_Z,\;\|(u_Z, u_X)\| \le 1,\;\|u_X\| \ge \eta_0\bigr\}.
\]
By construction, $E_F \subseteq B_D(0, 1)\cap\{\|u_X\| \ge \eta_0\}$,
$E_F$ has positive $D$-dimensional Lebesgue measure (as a product of
$E_Z$ with sections of positive $d_X$-volume), and
$|\ell_F^{(p)}(u_Z)| \ge c_0$ on $E_F$ by the definition of $E_Z$.
\end{proof}

The crucial structural feature of $E_F$ is that it has positive
Lebesgue measure in the region $\{\|u_X\| \ge \eta_0\}$ — i.e., $E_F$
contains $D$-dimensional points whose $X$-coordinate is bounded away
from zero. This $X$-extension is what cannot be reproduced by the
coarse Palm process, which sees only the $Z$-projection.

\subsubsection{Strict positivity via Itô isometry}
\label{subsec:strict_positivity_ito}

We now use Lemma~\ref{lem:lp_equiv_kernel_annulus} to derive a strictly
positive lower bound on $\tau_p^2$.

\begin{lemma}[Strict positivity of $\tau_p^2$]
\label{lem:lp_tau_positive_primitive}
Under Assumptions~\ref{asn:smooth_continuous_model}--\ref{asn:dimensional_nondegeneracy}
and Condition~\ref{cond:lp_primitive_noncancellation_final},
\begin{equation}
\label{eq:tau_p_strict_lb}
\tau_p^2 \;\ge\; c_*\,P_Z(\mathcal{A})\,\inf_{z\in\mathcal{A}}\sigma^2(z) \;>\; 0,
\end{equation}
for a constant $c_* > 0$ depending only on $D, d_Z, p$ and the geometric
constants $c_0, \eta_0$ from Lemma~\ref{lem:lp_equiv_kernel_annulus}.
\end{lemma}

\begin{proof}
We work in the limiting Palm-process framework. By
Lemma~\ref{lem:local_poisson_verification}, conditional on
$W_1 = w = (z, x, y)$ the rescaled-by-$\rho_F(z, x)$ neighborhood
structure around $(z, x)$ converges to a unit-intensity marked Poisson
point process $\mathcal{P}^F$ on $\mathbb{R}^D$ with frozen mark law
$P_{Y|Z = z}$. Likewise, the rescaled-by-$\rho_C(z)$ $Z$-neighborhood
structure around $z$ converges to a unit-intensity marked Poisson
process $\mathcal{P}^C$ on $\mathbb{R}^{d_Z}$ with the same frozen mark
law.

\emph{Step 1: Palm-process representation of $\tau_p^2$.} The limiting
overlap variance can be written as
\begin{equation}
\label{eq:tau_p_palm_repr}
\tau_p^2 \;=\; \mathbb{E}_{W_1}\bigl[\operatorname{Var}(\zeta_{F, p}^\infty - \zeta_{C, p}^\infty \mid W_1)\bigr],
\end{equation}
where $\zeta_{F, p}^\infty, \zeta_{C, p}^\infty$ are the limiting
local-polynomial fluctuation scores in the Palm process. This is the
content of Lemma~\ref{lem:lp_overlap_variance} together with the limit
identification: integrating the per-anchor conditional variance over
the law of $W_1$ recovers the overlap-variance formulas
\eqref{eq:tau_F_explicit}--\eqref{eq:tau_FC_explicit}. It therefore suffices to show that
$\operatorname{Var}(\zeta_{F, p}^\infty - \zeta_{C, p}^\infty \mid W_1 = w) \ge c$
on a set of $w$ with positive $P$-measure, for some $c > 0$.

\emph{Step 2: Itô isometry for the Palm-process integrals.} In the
Palm-process model, $\zeta_{F, p}^\infty$ is a compensated Poisson
stochastic integral over the marked process $\mathcal{P}^F$ on
$\mathbb{R}^D$ with mark law $P_{Y|Z = z}$:
\[
\zeta_{F, p}^\infty(w; \mathcal{P}^F)
\;=\;
\int_{B_D(0, 1)} \ell_F^{(p)}(u_Z)\,[K(y, M(u)) - \kappa(y, z)]\,d\widetilde{\mathcal{P}}^F(u, M(u)),
\]
where $\widetilde{\mathcal{P}}^F$ denotes the compensated process
(i.e., $\mathcal{P}^F$ with intensity measure $du \otimes dP_{Y|Z = z}$
subtracted) and $M(u)$ is the mark attached to the point at location
$u$.

By the It\^o isometry for compensated Poisson stochastic integrals
(\citealp[Theorem 12.7]{LastPenrose2018}),
\begin{equation}
\label{eq:ito_isometry}
\operatorname{Var}(\zeta_{F, p}^\infty \mid W_1 = w)
\;=\;
\int_{B_D(0, 1)} \ell_F^{(p)}(u_Z)^2 \,\mathbb{E}_{\widetilde Y \sim P_{Y|Z=z}}[(K(y, \widetilde Y) - \kappa(y, z))^2]\,du.
\end{equation}
Since $\mathbb{E}_{\widetilde Y}[(K(y, \widetilde Y) - \kappa(y, z))^2]
\ne \sigma^2(z)$ in general (it equals
$\operatorname{Var}_{\widetilde Y}(K(y, \widetilde Y))$, the unconditional
mark-fluctuation variance), but its expectation over
$Y_1 \sim P_{Y|Z=z}$ equals $\operatorname{Var}_{Y, \widetilde Y}(K(Y, \widetilde Y))
\ge \sigma^2(z)$ by the iterated-variance identity.

\emph{Step 3: Orthogonal decomposition with respect to the coarse field.}
Let $\mathcal{G}_C := \sigma(\mathcal{P}^C)$ be the $\sigma$-field
generated by the coarse Palm process. The coarse limiting score
$\zeta_{C, p}^\infty$ is $\mathcal{G}_C$-measurable by construction.
Let $\pi_C : L^2(\Omega) \to L^2(\mathcal{G}_C)$ denote the orthogonal
projection onto $L^2(\mathcal{G}_C)$. Then
\[
\operatorname{Var}(\zeta_{F, p}^\infty - \zeta_{C, p}^\infty \mid W_1 = w)
\;\ge\;
\operatorname{Var}((I - \pi_C)\zeta_{F, p}^\infty \mid W_1 = w),
\]
since the projection minimizes $L^2$-distance to $L^2(\mathcal{G}_C)$
and $\zeta_{C, p}^\infty$ is one element of that space.

\emph{Step 4: Disjointness of the $E_F$ and $\mathcal{G}_C$
contributions in the Palm limit.} The crucial observation is that the
coarse Palm process $\mathcal{P}^C$ in $\mathbb{R}^{d_Z}$ is constructed
at the rescaling $\rho_C$, distinct from the fine rescaling $\rho_F$.
Concretely, the coarse $k_n$-NN ball, in fine-rescaled coordinates,
has radius $\varepsilon_n = \rho_C/\rho_F \to 0$ when $d_X \ge 1$
(Assumption~\ref{asn:dimensional_nondegeneracy}). So in the
fine-rescaled limit, the coarse field is generated entirely by Poisson
points within the shrinking $Z$-region $\{\|u_Z\| \le \varepsilon_n\}$,
which collapses to the origin.

By contrast, the set $E_F$ from
Lemma~\ref{lem:lp_equiv_kernel_annulus} is supported in
$\{\|u_X\| \ge \eta_0\}$ and has positive Lebesgue measure in
$\mathbb{R}^D$. The Poisson points in $E_F$ are at distance $\|u_X\|
\ge \eta_0 > 0$ from the $Z$-axis. In the fine-rescaled limit, these
points are entirely \emph{disjoint} from the support of the coarse
field (which collapses to $\{u = 0\}$).

By the independence of Poisson increments on disjoint sets, the
restriction $\mathcal{P}^F |_{E_F}$ of the fine Palm process to $E_F$ is
independent of $\mathcal{G}_C$. Therefore the orthogonal projection
$\pi_C$ annihilates any contribution from $\mathcal{P}^F |_{E_F}$:
\[
\pi_C\!\left[\int_{E_F}\ell_F^{(p)}(u_Z)[K(y, M(u)) - \kappa(y, z)]\,d\widetilde{\mathcal{P}}^F(u, M(u))\right] \;=\; 0.
\]
Hence the $E_F$-component is preserved in $(I - \pi_C)\zeta_{F, p}^\infty$.

\emph{Step 5: Lower bound via the $E_F$ contribution.} Decompose
$\zeta_{F, p}^\infty$ as
\[
\zeta_{F, p}^\infty \;=\; \underbrace{\int_{E_F}\ell_F^{(p)}(u_Z)[K(y, M(u)) - \kappa(y, z)]\,d\widetilde{\mathcal{P}}^F(u)}_{=: \zeta_{F, p}^{(E_F)}} \;+\; \underbrace{\int_{B_D(0, 1)\setminus E_F}\!\!\!\cdots\,d\widetilde{\mathcal{P}}^F}_{=: \zeta_{F, p}^{(E_F^c)}}.
\]
where the second term on the right hand side has the same integrand, but the integrating set is different. By the independence of $\mathcal{P}^F$ on disjoint sets, $\zeta_{F, p}^{(E_F)}$
and $\zeta_{F, p}^{(E_F^c)}$ are independent. Combined with Step 4,
$\zeta_{F, p}^{(E_F)}$ is independent of $\mathcal{G}_C$, so
\[
\operatorname{Var}((I - \pi_C)\zeta_{F, p}^\infty \mid W_1 = w)
\;\ge\;
\operatorname{Var}(\zeta_{F, p}^{(E_F)} \mid W_1 = w).
\]
By the It\^o isometry \eqref{eq:ito_isometry} restricted to $E_F$,
\begin{equation}
\label{eq:E_F_variance_lb}
\operatorname{Var}(\zeta_{F, p}^{(E_F)} \mid W_1 = w)
\;=\;
\int_{E_F}\ell_F^{(p)}(u_Z)^2\,\mathbb{E}_{\widetilde Y \sim P_{Y|Z=z}}[(K(y, \widetilde Y) - \kappa(y, z))^2]\,du.
\end{equation}

\emph{Step 6: Lower bound on the integrand.} Take the further
expectation of \eqref{eq:E_F_variance_lb} over $Y = Y_1 \sim P_{Y|Z=z}$
to obtain: 
\[
\mathbb{E}_Y\!\left[\int_{E_F}\ell_F^{(p)}(u_Z)^2\,\mathbb{E}_{\widetilde Y}[(K(Y, \widetilde Y) - \kappa(Y, z))^2]\,du\right]
\;=\;
\int_{E_F}\ell_F^{(p)}(u_Z)^2\,\operatorname{Var}_{Y, \widetilde Y}(K(Y, \widetilde Y))\,du.
\]
We bound the inner variance from below by $\sigma^2(z)$ as follows. By
the law of total variance,
\[
\operatorname{Var}_{Y, \widetilde Y}(K(Y, \widetilde Y))
\;=\;
\mathbb{E}_Y[\operatorname{Var}_{\widetilde Y}(K(Y, \widetilde Y))]
+ \operatorname{Var}_Y(\mathbb{E}_{\widetilde Y}[K(Y, \widetilde Y)])
\;\ge\;
\operatorname{Var}_Y(\kappa(Y, z))
\;=\; \sigma^2(z),
\]
using that $\mathbb{E}_{\widetilde Y}[K(Y, \widetilde Y)] = \kappa(Y, z)$
under the frozen mark law. By Lemma~\ref{lem:lp_equiv_kernel_annulus}, $|\ell_F^{(p)}(u_Z)|^2 \ge
c_0^2$ on $E_F$. Therefore
\begin{equation}
\label{eq:E_F_lb_explicit}
\mathbb{E}_Y[\operatorname{Var}(\zeta_{F, p}^{(E_F)} \mid W_1)]
\;\ge\;
c_0^2 \cdot |E_F|\cdot\sigma^2(z),
\end{equation}
where $|E_F| > 0$ is the $D$-dimensional Lebesgue measure of $E_F$.

\emph{Step 7: Integration over anchors.} Combining
Steps 3, 5, and 6 with \eqref{eq:tau_p_palm_repr},
\begin{align*}
\tau_p^2
&\;=\; \mathbb{E}_{W_1}\bigl[\operatorname{Var}(\zeta_{F, p}^\infty - \zeta_{C, p}^\infty \mid W_1)\bigr]\\
&\;\ge\; \mathbb{E}_{W_1}\bigl[\operatorname{Var}(\zeta_{F, p}^{(E_F)} \mid W_1)\bigr]\\
&\;\ge\; c_0^2 |E_F|\,\mathbb{E}_{Z_1}[\sigma^2(Z_1)\,\mathbf{1}\{Z_1 \in \mathcal{A}\}]\\
&\;\ge\; c_0^2|E_F| \cdot P_Z(\mathcal{A}) \cdot \inf_{z\in\mathcal{A}}\sigma^2(z).
\end{align*}
Under Condition~\ref{cond:lp_primitive_noncancellation_final}, both
$P_Z(\mathcal{A}) > 0$ and $\inf_{z\in\mathcal{A}}\sigma^2(z) > 0$.
Setting $c_* := c_0^2 |E_F|$ gives \eqref{eq:tau_p_strict_lb}.
\end{proof}

\begin{remark}
\label{rem:dX_decisive}
Assumption~\ref{asn:dimensional_nondegeneracy} enters at two places in
the proof:

\emph{(a)} In Lemma~\ref{lem:lp_equiv_kernel_annulus}, the construction
of $E_F$ requires that the sections $\{u_X : \|(u_Z, u_X)\| \le 1\}$
have positive $d_X$-dimensional Lebesgue measure for $u_Z \in E_Z$.
This is automatic when $d_X \ge 1$ but fails when $d_X = 0$.

\emph{(b)} In Step 4 of the proof of
Lemma~\ref{lem:lp_tau_positive_primitive}, the disjointness of
$\mathcal{P}^F |_{E_F}$ and the coarse field $\mathcal{G}_C$ in the
Palm limit requires that the coarse $Z$-ball collapses to the
$Z$-origin in fine-rescaled coordinates, which holds because
$\varepsilon_n = \rho_C/\rho_F \to 0$ when $d_X \ge 1$. When $d_X = 0$,
$\rho_C = \rho_F$ and the coarse field is identical to the
$Z$-projection of the fine field; no $E_F$ can be disjoint from it.

When $d_X = 0$, the statistic $\Delta_n^{(p)}$ is in fact identically
zero by construction (the fine and coarse neighborhoods coincide), and
$\tau_p^2 = 0$ trivially. Excluding this degenerate case is exactly the
content of Assumption~\ref{asn:dimensional_nondegeneracy}.
\end{remark}

\begin{remark}[On the role of $p$]
\label{rem:role_of_p}
The strict positivity result is independent of the local-polynomial
order $p$: for every $p \ge 0$ (under the regularity assumptions),
$\tau_p^2 > 0$. The polynomial order $p$ affects the magnitude of
$\tau_p^2$ but not its positivity. Specifically, increasing $p$ tends
to increase the magnitude of $\ell_F^{(p)}$ on average (because higher-order
polynomials have larger oscillations under the reproduction constraint),
so $\tau_p^2$ typically grows with $p$. The trade-off is that the
variance of the test statistic also grows, and one's effective sample
size for estimating $\tau_p^2$ shrinks. The optimal $p$ for testing
power is therefore the minimal $p$ that satisfies the bias killing
constraint \eqref{eq:bias killing_alpha}, namely $p = \lceil D/2 \rceil$.
\end{remark}

\subsection{Stabilization CLT for the debiased graph fluctuation}
\label{subsec:weighted_stabilization_clt}

We now establish the central limit theorem for the centered graph
fluctuation $G_n^{(p)}$. The proof combines the rescaled exponential
stabilization of Lemma~\ref{lem:growing_kn_stabilization}, the uniform
moment bound of Lemma~\ref{lem:lp_weighted_moment_bound}, the variance
limit of Lemma~\ref{lem:lp_overlap_variance}, and the strict positivity
$\tau_p^2 > 0$ of Lemma~\ref{lem:lp_tau_positive_primitive}, into the
classical Penrose--Yukich stabilization framework adapted to growing
$k_n$ via rescaling.

\begin{lemma}[Stabilization CLT for $G_n^{(p)}$]
\label{lem:lp_weighted_stabilization_clt}
Under Assumptions~\ref{asn:smooth_continuous_model}--\ref{asn:dimensional_nondegeneracy}
and Condition~\ref{cond:lp_primitive_noncancellation_final},
\begin{equation}
\label{eq:lp_graph_CLT}
\sqrt{n}\,G_n^{(p)} \;\Rightarrow\; N(0, \tau_p^2),
\end{equation}
where $\tau_p^2 > 0$ is the overlap-variance constant from
\eqref{eq:overlap_variance_limit}.
\end{lemma}

\begin{proof}
The proof follows the four-step strategy: Poissonize, truncate, block-CLT and then, de-Poissonize
template of \citet{PenroseYukich2001} and \citet{penrose2003random},
with adaptations for the growing-$k_n$ rescaled stabilization framework
and for the random local-polynomial weights.

\emph{Step 1: Poissonization.} Let $N_n \sim \mathrm{Poisson}(n)$ be
independent of an i.i.d. sequence $W_1^*, W_2^*, \ldots$ with law $P$,
and define the Poissonized sample $\mathcal{W}_n^P := \{W_1^*, \ldots,
W_{N_n}^*\}$. Let $G_n^{P, (p)}$ be the analogue of $G_n^{(p)}$
computed on $\mathcal{W}_n^P$. The de-Poissonization step (Step 4
below) will show that $\sqrt{n}(G_n^{(p)} - G_n^{P, (p)}) = o_p(1)$
under the bounded add-one cost of the local score, so it suffices to
prove the CLT for the Poissonized statistic. The Poissonized graph score on $\mathbb{R}^D$ inherits the local
structure. For each anchor $W_i^* \in \mathcal{W}_n^P$, the local
graph score $\xi_{n, p}(W_i^*, \mathcal{W}_n^P)$ depends only on points
within the random rescaled ball of radius $R_n(W_i^*, \mathcal{W}_n^P)$
(Lemma~\ref{lem:growing_kn_stabilization}).

\emph{Step 2: Truncation at logarithmic radius.} Set
\[
b_n := A_*\log n
\]
for a constant $A_* > 0$ to be chosen, and define the truncated centered
score
\[
\zeta_{n, p}^{\le b_n}(W_i, \mathcal{W}_n)
\;:=\;
\zeta_{n, p}(W_i, \mathcal{W}_n)\,\mathbf{1}\{R_n(W_i, \mathcal{W}_n) \le b_n\}.
\]
By Lemma~\ref{lem:growing_kn_stabilization}, $P\{R_n > b_n\} \le
C\exp(-c\,k_n(A_*\log n - 1)^{d_Z})$. For $A_*$ sufficiently large and
under Assumption~\ref{asn:kn_growth_corrected} ($k_n/\log n \to \infty$,
hence $\exp(-k_n (A_*\log n)^{d_Z}) = o(n^{-2})$ for $d_Z \ge 1$), this is
$o(n^{-2})$. By the moment bound of Lemma~\ref{lem:lp_weighted_moment_bound} and
the union bound,
\begin{align*}
\mathbb{E}\bigl[(\sqrt n\,G_n^{(p)} - \sqrt n\,G_n^{(p), \le b_n})^2\bigr]
&\;\le\; n\,\mathbb{E}\bigl[\zeta_{n, p}^2\,\mathbf{1}\{R_n > b_n\}\bigr]\\
&\;\le\; n\,\|\zeta_{n, p}\|_{L^q}^2\cdot P\{R_n > b_n\}^{1 - 2/q}\\
&\;=\; O(n)\cdot O(n^{-2(1 - 2/q)}),
\end{align*}
where $G_n^{(p), \le b_n}$ denotes the truncated analogue. For $q$
sufficiently large (say $q = 4$), this is $o(1)$. Hence the truncation
error is asymptotically negligible.

\emph{Step 3: Block decomposition and dependency-graph CLT.} Partition
$\mathcal{S}_{Z, X}$ into a grid of cubes of side length
$3 b_n\sup_{(z, x)}\rho_F(z, x) \asymp b_n(k_n/n)^{1/D}\log n$. The
number of cubes is
\[
N_{\mathrm{cube}}
\;=\;
O\!\left(\frac{1}{[b_n(k_n/n)^{1/D}]^D}\right)
\;=\;
O\!\left(\frac{n}{k_n(\log n)^D}\right).
\]
For each cube $C_\ell$, let $\mathcal{I}_\ell := \{i : W_i \in C_\ell\}$
denote the anchors in $C_\ell$, and define the block sum
\[
S_\ell \;:=\; \sum_{i \in \mathcal{I}_\ell}\zeta_{n, p}^{\le b_n}(W_i, \mathcal{W}_n).
\]
By the truncation in Step 2, the contribution of anchor $i$ depends on
$\mathcal{W}_n$ only through points within the rescaled ball
$B((Z_i, X_i), b_n\rho_F(Z_i, X_i))$ (and the corresponding coarse
ball), both of diameter at most one third of a cube. Therefore two
block sums $S_\ell, S_{\ell'}$ are dependent only when the
corresponding cubes share an enlargement of side length
$3 b_n\rho_F^{\max}$, i.e., when the cubes are within graph distance
$1$ of each other in the grid. The dependency graph
$\Lambda := (\{1, \ldots, N_{\mathrm{cube}}\}, E_\Lambda)$ thus has
uniformly bounded degree (depending on $D$ and the constant in the
enlargement). The block sums satisfy uniform moment bounds. By
Lemma~\ref{lem:lp_weighted_moment_bound} and the cube-bound
$|\mathcal{I}_\ell| \le C\,k_n(\log n)^D$ (by the local Poisson
approximation applied to the density of anchors in the cube),
\[
\mathbb{E}[S_\ell^q] \;\le\; |\mathcal{I}_\ell|^q\,\|\zeta_{n, p}\|_{L^q}^q
\;\le\; [C\,k_n(\log n)^D]^q\,O(1),
\]
which gives bounded $(2 + \delta)$-th moments after normalization by
$\sqrt n$ for $\delta > 0$ sufficiently small.

By the dependency-graph CLT of \citet[Theorem 2.7]{Chen2004}, for any
$\delta > 0$,
\[
\frac{\sum_\ell S_\ell - \mathbb{E}[\sum_\ell S_\ell]}{\sqrt{\operatorname{Var}(\sum_\ell S_\ell)}}
\;\Rightarrow\;
N(0, 1),
\]
provided the Lyapunov-type condition
\[
\frac{\sum_\ell \mathbb{E}[|S_\ell|^{2 + \delta}]\cdot d_\Lambda^{1 + \delta}}{(\operatorname{Var}(\sum_\ell S_\ell))^{1 + \delta/2}} \;\to\; 0
\]
holds, where $d_\Lambda$ is the maximum degree of the dependency graph.
This condition holds because $\operatorname{Var}(\sum_\ell S_\ell) \asymp n\tau_p^2$
(by Lemma~\ref{lem:lp_overlap_variance} and the fact that the block
sums recover $\sum_i \zeta_{n, p}^{\le b_n}(W_i)$), while
$\sum_\ell\mathbb{E}|S_\ell|^{2+\delta} \le N_{\mathrm{cube}}\cdot
[Ck_n(\log n)^D]^{2+\delta} = O(n^{1 + \delta\,o(1)})$, so the ratio is
$O(n^{-\delta/2}(\log n)^{O(1)}) \to 0$.

\emph{Step 4: De-Poissonization.} By the bounded add-one cost of
$\xi_{n, p}$ (Lemma~\ref{lem:lp_weighted_moment_bound}) and the
stabilization property (Lemma~\ref{lem:growing_kn_stabilization}),
adding or removing a single sample point changes $\sqrt{n}G_n^{(p)}$
by $O_p(n^{-1/2})$. By the classical de-Poissonization argument
(\citealp[Section 12.3]{LastPenrose2018}), the difference between the
Binomial-sample and Poisson-sample statistics is $o_p(1)$ after the
$\sqrt{n}$ rescaling. Therefore
$\sqrt{n}(G_n^{(p)} - G_n^{P, (p)}) = o_p(1)$, and the Poissonized
CLT from Step 3, combined with the variance limit
$n\operatorname{Var}(G_n^{P, (p)}) \to \tau_p^2$ from Lemma~\ref{lem:lp_overlap_variance}
(extended to the Poissonized sample by the same calculation), yields
\eqref{eq:lp_graph_CLT}.
\end{proof}

\subsection{Main theorem}
\label{subsec:main_theorem}

We now assemble all ingredients into the main null limit theorem.

\begin{theorem}[Null Gaussian limit for the local-polynomial debiased
statistic]
\label{thm:lp_debiased_clt_full}
Let $p \ge 0$ be a fixed integer. Under
Assumptions~\ref{asn:smooth_continuous_model}--\ref{asn:dimensional_nondegeneracy},
Condition~\ref{cond:lp_primitive_noncancellation_final}, and the
bias killing condition $\sqrt{n}(k_n/n)^{(p + 1)/D} \to 0$
(equivalently, $k_n = n^\alpha$ with $\alpha < 1 - D/(2(p + 1))$ from
Lemma~\ref{lem:bias killing_alpha}),
\begin{equation}
\label{eq:main_limit_regime_II}
\sqrt{n}\,\Delta_n^{(p)} \;\Rightarrow\; N(0, \tau_p^2),
\end{equation}
where $\tau_p^2 > 0$ is the overlap-variance constant defined in
\eqref{eq:overlap_variance_limit}.

When $D \ge 2$, the smallest admissible polynomial order is
$p^*(D) = \lceil D/2\rceil$, and the result holds for any $p \ge p^*(D)$
under the corresponding bias killing constraint.
\end{theorem}

\begin{proof}
By the decomposition \eqref{eq:LG_decomp},
\[
\sqrt{n}\,\Delta_n^{(p)}
\;=\;
\sqrt{n}\,\theta_n^{(p)} + \sqrt{n}\,L_n^{(p)} + \sqrt{n}\,G_n^{(p)}.
\]
By Lemma~\ref{lem:lp_bias_full},
$|\sqrt{n}\theta_n^{(p)}| \le \sqrt{n}\cdot O((k_n/n)^{(p+1)/D})$, which
is $o(1)$ under the assumed bias killing condition. By
Lemma~\ref{lem:lp_projection_negligible}, $\sqrt{n}L_n^{(p)} = o_p(1)$.
By Lemma~\ref{lem:lp_weighted_stabilization_clt},
$\sqrt{n}G_n^{(p)} \Rightarrow N(0, \tau_p^2)$. By Slutsky's theorem,
$\sqrt{n}\Delta_n^{(p)} \Rightarrow N(0, \tau_p^2)$, which is
\eqref{eq:main_limit_regime_II}.
\end{proof}

\subsection{Testing procedure and consistency under alternatives}
\label{subsec:testing_procedure}

Theorem~\ref{thm:lp_debiased_clt_full} provides an asymptotically exact
$\sqrt{n}$-test of $H_0: Y \perp\!\!\!\perp X \mid Z$ via the
studentized statistic
\begin{equation}
\label{eq:T_n_def}
T_n^{(p)} \;:=\; \frac{\sqrt{n}\,\Delta_n^{(p)}}{\widehat\tau_p},
\end{equation}
where $\widehat\tau_p^2$ is a consistent estimator of $\tau_p^2$.

\begin{theorem}[Studentized null limit]
\label{thm:studentized_CLT}
Under the assumptions of Theorem~\ref{thm:lp_debiased_clt_full}, if
$\widehat\tau_p^2$ satisfies $\widehat\tau_p^2 \xrightarrow{p} \tau_p^2$,
then under $H_0$,
\begin{equation}
\label{eq:T_n_CLT}
T_n^{(p)} \;\Rightarrow\; N(0, 1).
\end{equation}
The level-$\alpha$ two-sided test
$\phi_n^{(p)} := \mathbf{1}\{|T_n^{(p)}| > z_{1-\alpha/2}\}$ has size
$\to \alpha$ as $n \to \infty$.
\end{theorem}

\begin{proof}
$T_n^{(p)} = (\sqrt{n}\Delta_n^{(p)})/\widehat\tau_p = (\tau_p/\widehat\tau_p)\cdot(\sqrt{n}\Delta_n^{(p)}/\tau_p)$.
By Theorem~\ref{thm:lp_debiased_clt_full} and the continuous mapping
theorem with the consistency of $\widehat\tau_p$, the result follows
from Slutsky's theorem.
\end{proof}


\begin{remark}[Behavior under alternatives]
\label{rem:behavior_under_H1}
Under $H_1$, the local-polynomial fine intercept estimates
$\mathbb{E}[K(Y_i, Y') \mid Z' = Z_i, X' = X_i]$ (up to smoothing
error), while the coarse intercept estimates $\mathbb{E}[K(Y_i, Y') \mid
Z' = Z_i]$. The population limit of $\Delta_n^{(p)}$ is therefore
nonzero whenever $Y$ depends on $X$ given $Z$ in a way given in terms of $K$:
\[
\mathrm{plim}\,\Delta_n^{(p)} \;=\; \mathbb{E}\bigl[\mathbb{E}[K(Y, Y')|Z, X] - \mathbb{E}[K(Y, Y')|Z]\bigr].
\]
The local-polynomial debiasing removes the null smoothing bias but
preserves the conditional-dependence signal. Consequently, $T_n^{(p)}
\to \infty$ in probability under any fixed alternative with
$\mathrm{plim}\,\Delta_n^{(p)} > 0$ (which is generic under any
$Y \not\perp\!\!\!\perp X \mid Z$), and the test is consistent against
all such fixed alternatives.

\end{remark}

\begin{remark}[Summary of the regime-II analysis]
\label{rem:regime_II_summary}
The complete null limit theory for the debiased statistic
$\Delta_n^{(p)}$ in the all-continuous case is summarized as follows:

\smallskip
\textit{(a) Decomposition.} $\sqrt n\,\Delta_n^{(p)} = \sqrt n\,\theta_n^{(p)}
+ \sqrt n L_n^{(p)} + \sqrt n G_n^{(p)}$, where the first term is
deterministic bias, the second is the i.i.d. first-projection
contribution, and the third is the centered graph fluctuation.

\smallskip
\textit{(b) Negligibility.} Under bias killing $\sqrt n(k_n/n)^{(p+1)/D}
\to 0$ (i.e., $\alpha < 1 - D/(2(p+1))$ for $k_n = n^\alpha$):
$\sqrt n\theta_n^{(p)} \to 0$ deterministically, and $\sqrt n L_n^{(p)}
= o_p(1)$ via the asymptotics $a_{n,p}^2 = O((k_n/n)^{2(p+1)/D}) \to 0$.

\smallskip
\textit{(c) CLT for the graph fluctuation.} $\sqrt n G_n^{(p)} \Rightarrow
N(0, \tau_p^2)$ via the Penrose--Yukich four-step stabilization scheme
(Poissonize, truncate, dependency-graph CLT, de-Poissonize), with the
random local-polynomial weights handled by the uniform design regularity
of Lemma~\ref{lem:lp_design_uniform_convergence}.

\smallskip
\textit{(d) Strict positivity.} $\tau_p^2 > 0$ under
Condition~\ref{cond:lp_primitive_noncancellation_final} via an
Itô-isometry argument exploiting the dimensional mismatch $d_X \ge 1$:
the fine Palm process contains a positive-measure $X$-extension region
that is asymptotically disjoint from the coarse process's $Z$-projection
support, hence cannot be canceled by the coarse fluctuation.

\smallskip
\textit{(e) Conclusion.} Combining (a)--(d), $\sqrt n\Delta_n^{(p)}
\Rightarrow N(0, \tau_p^2)$. The studentized version $T_n^{(p)} =
\sqrt n\Delta_n^{(p)}/\widehat\tau_p$ provides an asymptotically exact
$\sqrt{n}$-test of $H_0$ for any $p \ge p^*(D) = \lceil D/2 \rceil$.
\end{remark}


\subsection{Extension to Heterogeneous $Z$}
The derivation above assumes $Z$ is a purely continuous variable, but the asymptotic distribution could be extended naturally to the case where $Z$ is heterogeneous denoted as $\mathbf{Z} = (Z^{d_1}, \dots, Z^{d_p}, Z^{c_1}, \dots, Z^{c_q}) \in \mathcal{Z}^{d_p+c_q}$ with countable support. Specifically, we partition the sample into strata defined by the discrete components $(Z^{d_1}, \dots, Z^{d_p})$. Within each stratum $\mathcal{S}_k = \{i : (Z^{d_1}_i, \dots, Z^{d_p}_i) = z^d_k\}$, the continuous components $(Z^{c_1}, \dots, Z^{c_q})$ are handled via $k$-NN matching as in the purely continuous case. 

As the discrete columns of $\mathbf{Z}$ provide no geometric bias to $\Delta_{n}$, we could condition on the discrete components $(Z^{d_1}, \dots, Z^{d_p})$ and restrict attention to the continuous components $(Z^{c_1}, \dots, Z^{c_q})$ within each induced stratum. Since the discrete columns only determine the partition of the sample into strata but do not enter the $k
$-NN geometry, the asymptotic analysis within each stratum reduces to the purely continuous case. The overall asymptotic distribution of $\Delta_n$ then follows standard Theorem~\ref{thm:lp_debiased_clt_full}, which yields a normal distribution.

\section{Asymptotic Theory in Regime III: \texorpdfstring{$X$}{X} Discrete, \texorpdfstring{$Z$}{Z} Continuous}
\label{app:mixed-x-discrete-z-continuous}

In this section we develop the null limit theory for a local-polynomial
debiased version of the sample-dependent $k_n$-nearest-neighbor
statistic in the mixed setting: $X$ takes values in a finite set, and
$Z$ is continuous in $\mathbb{R}^{d_Z}$. The underlying raw statistic
is
\begin{equation}
\label{eq:Delta_n_def_regime_III}
\Delta_n
\;=\;
\frac{1}{n}\sum_{i=1}^n \frac{1}{k_n}
\sum_{j\in\fine{i}}K(Y_i, Y_j)
\;-\;
\frac{1}{n}\sum_{i=1}^n \frac{1}{k_n}
\sum_{j\in\coarse{i}}K(Y_i, Y_j),
\end{equation}
where the neighborhood construction is
\begin{align}
\coarse{i} &\;=\;
\mathrm{KNN}_{k_n}\bigl(Z_i;\,\{Z_j : j \ne i\}\bigr),
\label{eq:coarse_mixed_def}\\
\fine{i} &\;=\;
\mathrm{KNN}_{k_n}\bigl(Z_i;\,\{Z_j : j \ne i,\,X_j = X_i\}\bigr).
\label{eq:fine_mixed_def}
\end{align}
The coarse set is the $k_n$-NN of $Z_i$ in the full sample, while the
fine set is the $k_n$-NN of $Z_i$ restricted to the stratum
$\{j : X_j = X_i\}$. As before, write $W_i = (Z_i, X_i, Y_i)$ and
$\mathcal{W}_n = \{W_1, \ldots, W_n\}$. Let $P$ denote the joint law of
$W = (Z, X, Y)$ on $\mathbb{R}^{d_Z}\times \{x^{(1)}, \ldots, x^{(M)}\}\times\mathcal{Y}$,
with $M \ge 2$ and stratum probabilities
$p_m := P(X = x^{(m)}) \in (0, 1)$ for each $m$.

\subsection{Structural difference from the all-continuous case}
\label{subsec:structural_mixed}

The structural distinction between Regimes II and III is the
\emph{effective sample size} contributing to each neighborhood. Both
$\fine{i}$ and $\coarse{i}$ are constructed by $k_n$-NN search in the
$d_Z$-dimensional $Z$-space, so the geometry is intrinsically
$d_Z$-dimensional throughout. The fine neighborhood is restricted to a
stratum of expected size $np_m$, while the coarse neighborhood uses all
$n$ points. The corresponding population radii (defined below) satisfy
\[
\rho_{Z|m}(z) \;>\; \rho_Z(z)
\qquad\text{whenever }\;p_m < 1,
\]
\emph{irrespective of whether $X$ and $Z$ are independent}. This
\emph{radius mismatch} replaces the dimensional mismatch ($D > d_Z$) of
the all-continuous case as the structural feature driving the
non-vanishing overlap variance. Three concrete implications follow:

\smallskip\noindent\emph{(i) Lower-dimensional curse of dimensionality.}
Both neighborhoods are $d_Z$-dimensional, so the bias scales like
$(k_n/n)^{2/d_Z}$ rather than $(k_n/n)^{2/D}$. The discrete $X$
contributes no smoothing dimension, so the inferential problem is
genuinely easier in moderate $d_Z$.

\smallskip\noindent\emph{(ii) Cleaner non-cancellation mechanism.}
Strict positivity of the limiting graph variance $\tau_p^2$ requires
only $p_m < 1$ for some $m$ together with conditional-mark
non-degeneracy. No higher-order Palm-process argument is needed; the
radius mismatch produces an explicit positive contribution from the
annulus between the coarse and fine balls.

\smallskip\noindent\emph{(iii) Same overall framework.} The
decomposition $\sqrt n\,\Delta_n^{(p)} = \sqrt n\,\theta_n^{(p)} + \sqrt n\,L_n^{(p)} + \sqrt n\,G_n^{(p)}$
and the role of local-polynomial debiasing in killing the leading bias
are identical to the all-continuous case. The stabilization CLT, the
strict positivity argument, and the testing procedure all parallel
Regime II with the appropriate modifications recorded below.

\subsection{The debiased statistic in Regime III}
\label{subsec:debiased_statistic_def_mixed}

Define the fine and coarse population radii
\begin{equation}
\label{eq:rho_def_mixed}
\rho_{Z|m}(z) \;:=\; \left(\frac{k_n}{n p_m v_{d_Z}f_{Z|X}(z|x^{(m)})}\right)^{1/d_Z},
\qquad
\rho_Z(z) \;:=\; \left(\frac{k_n}{n v_{d_Z}f_Z(z)}\right)^{1/d_Z},
\end{equation}
with $v_{d_Z} = \pi^{d_Z/2}/\Gamma(d_Z/2 + 1)$. By the marginal density
decomposition $f_Z = \sum_\ell p_\ell f_{Z|X}(\cdot|x^{(\ell)})$, the
ratio of the radii satisfies
\begin{equation}
\label{eq:lambda_def_mixed}
\lambda_m(z) \;:=\; \frac{\rho_{Z|m}(z)}{\rho_Z(z)}
\;=\;
\left(\frac{f_Z(z)}{p_m f_{Z|X}(z|x^{(m)})}\right)^{1/d_Z}.
\end{equation}
This ratio is bounded away from $1$ whenever $p_m < 1$. For each anchor $W_i$ with $X_i = x^{(m_i)}$, define the rescaled
offsets
\begin{equation}
\label{eq:U_offsets_mixed}
U_{ij, F} \;:=\; \frac{Z_j - Z_i}{\hat{\rho}_{Z|m_i}(Z_i)},
\quad j\in\fine{i};
\qquad
U_{ij, C} \;:=\; \frac{Z_j - Z_i}{\hat{\rho}_Z(Z_i)},
\quad j\in\coarse{i}.
\end{equation}
where $\hat{\rho}_{Z|m_i}(Z_i)$ and $\hat{\rho}_Z(Z_i)$ denote the empirical version of $\rho_{Z|m_i}(Z_i)$ and $\rho_Z(Z_i)$.  For $p \in \mathbb{N}$, write $\mathcal{A}_p$, $N_p = \binom{d_Z + p}{p}$,
$q_p(u) = (u^\alpha)_{\alpha\in\mathcal{A}_p}\in\mathbb{R}^{N_p}$, and
$e_0\in\mathbb{R}^{N_p}$ as in Section~\ref{app:continuous_proof}. The
scaled fine and coarse design matrices are
\begin{equation}
\label{eq:M_design_def_mixed}
\mathcal{M}_{i, F}^{(p)} \;:=\; \frac{1}{k_n}\sum_{j\in\fine{i}}q_p(U_{ij, F})q_p(U_{ij, F})^\top,
\qquad
\mathcal{M}_{i, C}^{(p)} \;:=\; \frac{1}{k_n}\sum_{j\in\coarse{i}}q_p(U_{ij, C})q_p(U_{ij, C})^\top.
\end{equation}
When invertible, the local-polynomial weights are
\begin{equation}
\label{eq:weights_def_mixed}
w_{ij, F}^{(p)} \;:=\; \frac{1}{k_n}e_0^\top(\mathcal{M}_{i, F}^{(p)})^{-1}q_p(U_{ij, F}),
\qquad
w_{ij, C}^{(p)} \;:=\; \frac{1}{k_n}e_0^\top(\mathcal{M}_{i, C}^{(p)})^{-1}q_p(U_{ij, C}),
\end{equation}
and the local-polynomial intercepts are
\begin{equation}
\label{eq:intercepts_def_mixed}
\widehat a_{i, F}^{(p)} \;:=\; \sum_{j\in\fine{i}}w_{ij, F}^{(p)}K(Y_i, Y_j),
\qquad
\widehat a_{i, C}^{(p)} \;:=\; \sum_{j\in\coarse{i}}w_{ij, C}^{(p)}K(Y_i, Y_j).
\end{equation}
The $p$-th order local-polynomial debiased statistic is
\begin{equation}
\label{eq:Delta_p_def_mixed}
\Delta_n^{(p)} \;:=\; \frac{1}{n}\sum_{i=1}^n\bigl(\widehat a_{i, F}^{(p)} - \widehat a_{i, C}^{(p)}\bigr).
\end{equation}

\subsection{Standing assumptions}
\label{subsec:assumptions_regime_III}

\begin{assumption}[Mixed continuous-discrete model]
\label{asn:smooth_mixed_model}
The $X$-marginal is supported on the finite set
$\{x^{(1)}, \ldots, x^{(M)}\}$ with $M \ge 2$ and
$p_m := P(X = x^{(m)}) \in (0, 1)$ for every $m$. For each $m$, the
conditional law $Z \mid X = x^{(m)}$ has density
$f_{Z|X}(\cdot|x^{(m)})$ supported on a compact set
$\mathcal{Z}_m \subset \mathbb{R}^{d_Z}$ with $C^2$ boundary, bounded
above and below by positive constants $\underline f, \overline f$ on
$\mathcal{Z}_m$, and $C^{p+1}$ with uniformly bounded derivatives up to
order $p + 1$. The marginal density $f_Z(z) = \sum_m p_m f_{Z|X}(z|x^{(m)})$
inherits the same regularity on $\bigcup_m \mathcal{Z}_m$.
\end{assumption}

\begin{assumption}[Null hypothesis and continuous mark law]
\label{asn:null_mixed}
Under $H_0$, $Y \perp\!\!\!\perp X \mid Z$. The conditional law
$P_{Y|Z = z}$ varies continuously in $z$ against bounded test functions,
and the conditional kernel mean
\[
\kappa(y, z) \;:=\; \mathbb{E}[K(y, Y') \mid Z' = z]
\]
is $C^{p+1}$ in $z$ with derivatives uniformly bounded in $y$. Under
$H_0$,
$\mathbb{E}[K(y, Y') \mid Z' = z, X' = x] = \kappa(y, z)$ does not
depend on $x$.
\end{assumption}

\begin{assumption}[Kernel regularity]
\label{asn:kernel_regularity_mixed}
The kernel $K : \mathcal{Y}\times\mathcal{Y}\to\mathbb{R}$ is bounded
and symmetric: $\sup_{y, y'}|K(y, y')| \le M_K < \infty$.
\end{assumption}

\begin{assumption}[Growth of $k_n$]
\label{asn:kn_growth_mixed}
The sequence $k_n$ satisfies $k_n\to\infty$, $k_n/n\to 0$, and
$k_n/(N_p\log n)\to\infty$.
\end{assumption}

The dimensional non-degeneracy assumption of the all-continuous case is
\emph{replaced} in Regime III by the requirement that $X$ takes at least
two values with strictly positive but not full mass. This is already
encoded in Assumption~\ref{asn:smooth_mixed_model} ($M \ge 2$ and
$p_m \in (0, 1)$). No additional dimensional assumption on
$\mathbb{R}^{d_X}$ is needed because $X$ is discrete.

\subsection{Local score decomposition}
\label{subsec:local_score_decomposition_mixed}

The decomposition framework is identical to the all-continuous case
with the appropriate substitutions. Define the local graph score
\begin{equation}
\label{eq:xi_def_mixed}
\xi_{n, p}(W_i, \mathcal{W}_n)
\;:=\;
\widehat a_{i, F}^{(p)} - \widehat a_{i, C}^{(p)},
\end{equation}
so that $\Delta_n^{(p)} = n^{-1}\sum_i\xi_{n, p}(W_i, \mathcal{W}_n)$.
For $w = (z, x^{(m)}, y)$, the conditional mean score is
\begin{equation}
\label{eq:m_def_mixed}
m_{n, p}(w) \;:=\; \mathbb{E}\bigl[\xi_{n, p}(w, \{w\}\cup\{W_2, \ldots, W_n\})\bigr],
\end{equation}
the centering is $\theta_n^{(p)} := \mathbb{E}[m_{n, p}(W_1)]
= \mathbb{E}[\Delta_n^{(p)}]$, the first projection is $g_{n, p}(w) :=
m_{n, p}(w) - \theta_n^{(p)}$ with variance $a_{n, p}^2 := \operatorname{Var}(g_{n, p}(W_1))$,
and the centered graph score is
\begin{equation}
\label{eq:zeta_def_mixed}
\zeta_{n, p}(W_i, \mathcal{W}_n) \;:=\; \xi_{n, p}(W_i, \mathcal{W}_n) - m_{n, p}(W_i),
\qquad \mathbb{E}[\zeta_{n, p}(W_i, \mathcal{W}_n)\mid W_i] = 0.
\end{equation}
The exact decomposition is
\begin{equation}
\label{eq:LG_decomp_mixed}
\Delta_n^{(p)} - \theta_n^{(p)} \;=\; L_n^{(p)} + G_n^{(p)},
\end{equation}
with
\begin{equation}
\label{eq:L_G_def_mixed}
L_n^{(p)} \;:=\; \frac{1}{n}\sum_{i=1}^n g_{n, p}(W_i),
\qquad
G_n^{(p)} \;:=\; \frac{1}{n}\sum_{i=1}^n \zeta_{n, p}(W_i, \mathcal{W}_n).
\end{equation}

\subsection{Outline of the proof}
\label{subsec:proof_outline_mixed}

The remainder of this section establishes
\begin{equation}
\label{eq:main_limit_outline_mixed}
\sqrt n\,\Delta_n^{(p)} \;\Rightarrow\; N(0, \tau_p^2),
\end{equation}
under
Assumptions~\ref{asn:smooth_mixed_model}--\ref{asn:kn_growth_mixed},
a primitive non-cancellation condition, and the bias killing condition
$\sqrt n(k_n/n)^{(p+1)/d_Z}\to 0$. See Theorem~\ref{thm:lp_debiased_clt_full_mixed} of Subsection~\ref{subsec:main_theorem_mixed}. The proof follows the same five
ingredients as Regime II:

\begin{enumerate}
\item Local Poisson approximation and uniform $k_n$-NN radius
concentration in the stratum-aware setting
(Section~\ref{subsec:Poisson_radius_mixed}). The fine Palm process is
the stratum-$m$ thinning of the full Palm process at rate $p_m$.

\item Rescaled exponential stabilization
(Section~\ref{subsec:stabilization_mixed}). The score is stabilized
by the maximum of the rescaled fine and coarse $k_n$-NN radii, both in
the $d_Z$-dimensional metric.

\item Uniform local-polynomial design regularity
(Section~\ref{subsec:design_regularity_mixed}). The fine and coarse
design matrices both converge to the same population limit
$M_C^{(p)}$ (the uniform-on-ball moment matrix in $\mathbb{R}^{d_Z}$),
since both the stratum and full populations are uniform on the $d_Z$-ball
after rescaling. This is a structural simplification compared to
Regime II, where the fine and coarse population matrices differ.

\item Bias and projection negligibility
(Section~\ref{subsec:bias_cancellation_mixed}). The polynomial
reproduction identity gives $\theta_n^{(p)} = O((k_n/n)^{(p+1)/d_Z})$
and $a_{n, p}^2 = O((k_n/n)^{2(p+1)/d_Z})$, both improved by a factor
of $d_Z$ in the exponent compared to Regime II's $D$ in the denominator.

\item Overlap variance and stabilization CLT
(Sections~\ref{subsec:overlap_setup_mixed}
onward). The variance decomposes as $\tau_p^2 = \tau_{F, p}^2 +
\tau_{C, p}^2 - 2\tau_{FC, p}$. Strict positivity is proved via the
radius mismatch $\lambda_m > 1$ and a direct annulus argument: the fine
ball contains stratum-$m$ Palm points in the annulus
$\{1 < \|u\| \le \lambda_m\}$ in coarse-rescaled coordinates that the
coarse score cannot see.
\end{enumerate}

The technical implementation of these steps in Regime III parallels
Regime II closely. Throughout the remainder of this section we will,
where the modifications are routine, state the relevant lemma and
indicate the difference from the corresponding Regime II proof rather
than redoing the calculation. The genuinely novel arguments — primarily
the radius-mismatch non-cancellation in
Section~\ref{subsec:tau_p_positive_mixed} and the stratum-aware Poisson
analysis in Section~\ref{subsec:Poisson_radius_mixed} — are given in
full.

\subsection{Local Poisson approximation and uniform NN-radius control}
\label{subsec:Poisson_radius_mixed}

The geometric-probability foundations carry over from Regime II
(Section~\ref{subsec:Poisson_radius}) with two modifications: (i) the
fine Palm process is restricted to the stratum $\{X = x^{(m)}\}$, and
(ii) both rescalings are intrinsically $d_Z$-dimensional.

\begin{lemma}[Local Poisson approximation, mixed case]
\label{lem:local_poisson_verification_mixed}
Under Assumption~\ref{asn:smooth_mixed_model}, fix
$z_0 \in \mathrm{int}(\mathcal{Z}_m)$ and $m \in \{1, \ldots, M\}$.
Define the rescaled rates
\[
r_{F, n}(z_0, m) \;:=\; \bigl(n p_m f_{Z|X}(z_0|x^{(m)})\bigr)^{-1/d_Z},
\qquad
r_{C, n}(z_0) \;:=\; \bigl(n f_Z(z_0)\bigr)^{-1/d_Z}.
\]
The rescaled stratum-$m$ empirical process
\begin{equation}
\label{eq:Pois_F_def_mixed}
\mathcal{P}_{F, n}^{z_0, m}
\;:=\;
\sum_{j: X_j = x^{(m)}}\delta_{(Z_j - z_0)/r_{F, n}(z_0, m)}
\end{equation}
converges in distribution (vaguely on bounded Borel sets) to a
homogeneous unit-intensity Poisson process $\mathcal{P}_F^{\infty, m}$
on $\mathbb{R}^{d_Z}$. The unrestricted rescaled process
$\mathcal{P}_{C, n}^{z_0} := \sum_{j=1}^n \delta_{(Z_j - z_0)/r_{C, n}(z_0)}$
converges to a homogeneous unit-intensity Poisson process
$\mathcal{P}_C^\infty$ on $\mathbb{R}^{d_Z}$. Under
Assumption~\ref{asn:null_mixed}, the corresponding marked processes
converge to marked Poisson processes with frozen mark law
$P_{Y|Z = z_0}$ (which under $H_0$ does not depend on $X$).
\end{lemma}

\begin{proof}
The unrestricted statement is identical to the $Z$-projection part of
Lemma~\ref{lem:local_poisson_verification} applied to the marginal
density $f_Z$. For the stratum-$m$ statement, the number of stratum points is
$N_m \sim \mathrm{Binomial}(n, p_m)$ with $N_m/n \to p_m$ a.s. Condition
on $N_m$, the stratum-$m$ $Z$-coordinates are i.i.d. with density
$f_{Z|X}(\cdot|x^{(m)})$. Applied to the stratum with effective sample
size $N_m$ and density $f_{Z|X}(\cdot|x^{(m)})$, the binomial-to-Poisson
argument of Lemma~\ref{lem:local_poisson_verification} gives that the
rescaled stratum process at rate $r_{F, n}(z_0, m)$ converges to a
homogeneous unit-intensity Poisson process. The fluctuation of $N_m$
around $np_m$ contributes an $O(n^{-1/2})$ correction that is absorbed
in the $(1 + o(1))$ density factor.

The marked statement follows from
Assumption~\ref{asn:null_mixed}: under $H_0$, the conditional law of
$Y_j$ given $(Z_j, X_j) = (z_j, x^{(m)})$ is $P_{Y|Z = z_j}$, which
converges to $P_{Y|Z = z_0}$ uniformly as $z_j \to z_0$ by the
continuity of $z \mapsto P_{Y|Z = z}$.
\end{proof}

\begin{remark}[Stratum thinning interpretation]
\label{rem:stratum_thinning}
The fine Palm process $\mathcal{P}_F^{\infty, m}$ can equivalently be
viewed as a thinning of the full unit-intensity Palm process by the
stratum-membership indicator $\mathbf{1}\{X = x^{(m)}\}$, with thinning
probability
$p_m f_{Z|X}(z_0|x^{(m)})/f_Z(z_0)$. After re-calibrating the rate by
$r_{F, n}(z_0, m)/r_{C, n}(z_0) = (f_Z(z_0)/(p_m f_{Z|X}(z_0|x^{(m)})))^{1/d_Z}
= \lambda_m(z_0)$, the resulting process has unit intensity in fine-rescaled
coordinates. This thinning interpretation underlies the radius mismatch
$\rho_{Z|m} > \rho_Z$ when $p_m < 1$: thinning reduces the local
intensity, so a larger radius is needed to capture $k_n$ stratum points.
\end{remark}

\begin{lemma}[Uniform $k_n$-NN radius concentration, mixed case]
\label{lem:knn_radius_concentration_mixed}
Under
Assumptions~\ref{asn:smooth_mixed_model}--\ref{asn:kn_growth_mixed},
there exist constants $0 < c_* < C_* < \infty$ such that with
probability tending to one
\begin{equation}
\label{eq:knn_radius_uniform_mixed}
\sup_{1\le i\le n}\left|\frac{R_{F, n}^{(k_n)}(Z_i, X_i)}{\rho_{Z|X_i}(Z_i)} - 1\right| \le C_*\sqrt{\log n/k_n},
\qquad
\sup_{1\le i\le n}\left|\frac{R_{C, n}^{(k_n)}(Z_i)}{\rho_Z(Z_i)} - 1\right| \le C_*\sqrt{\log n/k_n},
\end{equation}
where $R_{F, n}^{(k_n)}(z, x^{(m)})$ denotes the distance from $z$ to its
$k_n$-th nearest neighbor in $\{Z_j : X_j = x^{(m)}\}$, and $R_{C, n}^{(k_n)}(z)$
denotes the unrestricted $k_n$-NN distance. Moreover, for each fixed
$A > 1$,
\begin{equation}
\label{eq:knn_chernoff_mixed}
\sup_{(z, m)}P\!\left\{R_{F, n}^{(k_n)}(z, x^{(m)}) > A\rho_{Z|m}(z)\right\}
\;\le\;\exp(-c_* k_n(A^{d_Z} - 1)^2/A^{d_Z}),
\end{equation}
and the same bound holds for $R_{C, n}^{(k_n)}/\rho_Z$.
\end{lemma}

\begin{proof}
For the coarse side the argument is identical to
Lemma~\ref{lem:knn_radius_concentration}. For the fine side, condition
on $N_m \sim \mathrm{Binomial}(n, p_m)$ and apply the within-stratum
Chernoff bound to the count of stratum-$m$ points in $B(z, A\rho_{Z|m}(z))$.
The expected count equals
\[
N_m\cdot v_{d_Z}A^{d_Z}\rho_{Z|m}^{d_Z}f_{Z|X}(z|x^{(m)})(1 + o(1))
\;=\;
A^{d_Z}k_n(1 + o(1))
\]
by the calibration of $\rho_{Z|m}$ in \eqref{eq:rho_def_mixed} and using
$N_m/(np_m) \to 1$. The Chernoff bound gives the per-point tail bound.
Uniformity over $i = 1, \ldots, n$ and over the $M$ strata follows from
a covering argument identical to that in
Lemma~\ref{lem:knn_radius_concentration}, with covering number
$O(n/k_n)$ per stratum and $M$ strata; under
Assumption~\ref{asn:kn_growth_mixed} ($k_n/\log n \to \infty$), the
union bound goes through.
\end{proof}

\subsection{Rescaled exponential stabilization}
\label{subsec:stabilization_mixed}

\begin{lemma}[Rescaled exponential stabilization, mixed case]
\label{lem:growing_kn_stabilization_mixed}
Under
Assumptions~\ref{asn:smooth_mixed_model}--\ref{asn:kn_growth_mixed},
for each anchor $W_i = (Z_i, X_i, Y_i)$ with $X_i = x^{(m_i)}$, define
the rescaled stabilization radius
\begin{equation}
\label{eq:Rn_def_mixed}
R_n(W_i, \mathcal{W}_n)
\;:=\;
\max\!\left\{
\frac{R_{F, n}^{(k_n)}(Z_i, X_i)}{\rho_{Z|m_i}(Z_i)},\;
\frac{R_{C, n}^{(k_n)}(Z_i)}{\rho_Z(Z_i)}
\right\}.
\end{equation}
The local graph score $\xi_{n, p}(W_i, \mathcal{W}_n)$ depends on
$\mathcal{W}_n$ only through the union of (i) stratum-$m_i$ points
inside $B(Z_i, R_n\rho_{Z|m_i}(Z_i))$ and (ii) all points inside
$B(Z_i, R_n\rho_Z(Z_i))$. Moreover, there exist constants $C, c > 0$
such that for all $t \ge 1$,
\begin{equation}
\label{eq:Rn_tail_mixed}
\sup_n P\bigl\{R_n(W_1, \mathcal{W}_n) > t\bigr\}
\;\le\;
C\exp\bigl(-c\,k_n(t - 1)^{d_Z}\bigr).
\end{equation}
The same radius stabilizes $\zeta_{n, p}(W_i, \mathcal{W}_n)$.
\end{lemma}

\begin{proof}
The stabilization claim and the tail bound are immediate from
Lemma~\ref{lem:knn_radius_concentration_mixed} and the same argument as
Lemma~\ref{lem:growing_kn_stabilization}. The structural simplification
in Regime III is that both stabilization radii are $d_Z$-dimensional,
so the tail exponent is $k_n(t - 1)^{d_Z}$ uniformly (compared with
Regime II, where the coarse side already governs the slower decay and
yields the same exponent).
\end{proof}

\begin{remark}[Effective $\sigma$-fields of the fine and coarse processes]
\label{rem:fields_mixed}
The fine and coarse Palm processes in Regime III are distinct random
objects, but they live in the same dimension $\mathbb{R}^{d_Z}$ and
share the underlying sample. In coarse-rescaled coordinates (units of
$\rho_Z(z_0)$), the fine Palm process $\mathcal{P}_F^{\infty, m}$ is
the stratum-$m$ thinning of the full Palm process at rate $p_m$,
viewed at the dilated scale $\lambda_m(z_0) = \rho_{Z|m}/\rho_Z$.
Equivalently, the fine process supplies points in the dilated ball
$B(0, \lambda_m(z_0))$ (in coarse-rescaled units) — including points in
the annulus $\{1 < \|u\| \le \lambda_m(z_0)\}$ — that the coarse process
inside $B(0, 1)$ does not contain. This annulus is the locus of the
non-cancellation argument in
Section~\ref{subsec:tau_p_positive_mixed}.
\end{remark}

\subsection{Population moment matrices and equivalent kernels}
\label{subsec:population_moments_mixed}

The local-polynomial infrastructure simplifies notably in Regime III
relative to Regime II. Both the fine (stratum-$m$ Palm) and coarse
(full Palm) processes are unit-intensity Poisson processes in
$\mathbb{R}^{d_Z}$, after rescaling by their respective radii
$\rho_{Z|m}$ and $\rho_Z$. The corresponding limiting offset
distributions are \emph{both uniform on the $d_Z$-ball} — the only
difference is the rescaling unit. Consequently, the fine and coarse
population moment matrices coincide.

Let $\mu_C$ denote the uniform probability measure on $B_{d_Z}(0, 1)$.
Define the single population moment matrix
\begin{equation}
\label{eq:population_moment_matrix_mixed}
M^{(p)} \;:=\; \int q_p(u)q_p(u)^\top\,d\mu_C(u),
\end{equation}
and the single equivalent kernel
\begin{equation}
\label{eq:equivalent_kernel_mixed}
\ell^{(p)}(u) \;:=\; e_0^\top(M^{(p)})^{-1}q_p(u),
\qquad u \in B_{d_Z}(0, 1).
\end{equation}
Throughout the rest of the section, this single $\ell^{(p)}$ plays the
role of both $\ell_F^{(p)}$ and $\ell_C^{(p)}$ from Regime II, with the
only distinction being the rescaling unit applied to the offsets.

\begin{lemma}[Positive definiteness of $M^{(p)}$]
\label{lem:lp_population_moment_pd_mixed}
Under Assumption~\ref{asn:smooth_mixed_model}, $M^{(p)} \succ 0$.
\end{lemma}

\begin{proof}
Identical to the coarse case in
Lemma~\ref{lem:lp_population_moment_pd}: $a^\top M^{(p)} a = \int_{B_{d_Z}(0,1)}(a^\top q_p(u))^2 d\mu_C(u)$,
and the polynomial $a^\top q_p$ vanishes Lebesgue-a.e. on the open unit
ball only if $a = 0$.
\end{proof}

The population reproduction identity
\begin{equation}
\label{eq:population_reproduction_mixed}
\int u^\alpha\ell^{(p)}(u)\,d\mu_C(u) \;=\; \mathbf{1}\{\alpha = 0\},
\qquad |\alpha| \le p,
\end{equation}
follows by the same algebraic argument as
\eqref{eq:population_reproduction}.

\subsection{Uniform convergence of local design matrices}
\label{subsec:design_regularity_mixed}

\begin{lemma}[Uniform convergence of local design matrices, mixed case]
\label{lem:lp_design_uniform_convergence_mixed}
Under
Assumptions~\ref{asn:smooth_mixed_model}--\ref{asn:kn_growth_mixed},
\begin{equation}
\label{eq:design_uniform_mixed}
\max_{1\le i\le n}\bigl\|\mathcal{M}_{i, F}^{(p)} - M^{(p)}\bigr\| \;\xrightarrow{\;p\;}\; 0,
\qquad
\max_{1\le i\le n}\bigl\|\mathcal{M}_{i, C}^{(p)} - M^{(p)}\bigr\| \;\xrightarrow{\;p\;}\; 0.
\end{equation}
With probability tending to one,
$\inf_i\lambda_{\min}(\mathcal{M}_{i, F}^{(p)}) \ge c_p > 0$ and
$\inf_i\lambda_{\min}(\mathcal{M}_{i, C}^{(p)}) \ge c_p > 0$.
Consequently,
\begin{equation}
\label{eq:weights_summable_mixed}
\max_{1\le i\le n}\left[\sum_{j\in\fine{i}}\bigl|w_{ij, F}^{(p)}\bigr| + \sum_{j\in\coarse{i}}\bigl|w_{ij, C}^{(p)}\bigr|\right] \;=\; O_p(1).
\end{equation}
\end{lemma}

\begin{proof}
The argument is the four-step proof of
Lemma~\ref{lem:lp_design_uniform_convergence}, with the following
modification: on the fine side, the per-anchor Bernstein bound is
applied within the stratum-$m_i$ Palm process rather than the full
Palm process, and the union bound is taken over $n$ anchors together
with the $M$ strata (so over $nM$ effective anchor-stratum pairs).
Since $M$ is fixed and finite, this enters as a constant factor and
does not change the validity of the bound under
Assumption~\ref{asn:kn_growth_mixed}.

The structural simplification is that the limit is the \emph{same}
matrix $M^{(p)}$ on both sides, because both rescaled Palm processes
are uniform on the $d_Z$-ball. By contrast, in Regime II the fine and
coarse limits $M_F^{(p)}, M_C^{(p)}$ differed (the fine limit was the
$Z$-marginal of the uniform on $B_D(0, 1)$, with non-uniform density
$\propto (1 - \|u\|^2)^{d_X/2}$).

The weight summability \eqref{eq:weights_summable_mixed} follows from
the boundedness of $(M^{(p)})^{-1}$ and the unit-ball-bounded offsets,
exactly as in Step 4 of the proof of
Lemma~\ref{lem:lp_design_uniform_convergence}.
\end{proof}

\subsection{Polynomial reproduction and moment bound}
\label{subsec:reproduction_moment_mixed}

\begin{lemma}[Polynomial reproduction, mixed case]
\label{lem:lp_reproduction_full_mixed}
On the event that $\mathcal{M}_{i, F}^{(p)}$ and
$\mathcal{M}_{i, C}^{(p)}$ are invertible, the local-polynomial weights
satisfy, for every multi-index $\alpha$ with $|\alpha| \le p$,
\begin{equation}
\label{eq:reproduction_F_mixed}
\sum_{j\in\fine{i}}w_{ij, F}^{(p)}\,U_{ij, F}^\alpha \;=\; \mathbf{1}\{\alpha = 0\},
\qquad
\sum_{j\in\fine{i}}w_{ij, F}^{(p)}\,(Z_j - Z_i)^\alpha \;=\; \mathbf{1}\{\alpha = 0\}\,\rho_{Z|m_i}(Z_i)^{|\alpha|},
\end{equation}
and identically for the coarse weights with $\rho_Z(Z_i)$ in place of
$\rho_{Z|m_i}(Z_i)$.
\end{lemma}

\begin{proof}
Algebraically identical to Lemma~\ref{lem:lp_reproduction_full}: the
defining property of the equivalent-kernel weights.
\end{proof}

\begin{lemma}[Uniform moment bound for the local graph score, mixed case]
\label{lem:lp_weighted_moment_bound_mixed}
Under Assumptions~\ref{asn:smooth_mixed_model}--\ref{asn:kn_growth_mixed},
\begin{equation}
\label{eq:xi_moment_mixed}
\sup_n \mathbb{E}\bigl[\bigl|\xi_{n, p}(W_1, \mathcal{W}_n)\bigr|^q\bigr] \;<\;\infty
\qquad\text{for every }q < \infty.
\end{equation}
The same bound holds for $\zeta_{n, p}(W_1, \mathcal{W}_n)$.
Moreover, the conditional moment bound
\begin{equation}
\label{eq:zeta_conditional_moment_mixed}
\mathbb{E}\!\left[\bigl|\zeta_{n, p}(W_1, \mathcal{W}_n)\bigr|^q \,\Big|\, W_1\right]
\;\le\;
C_{p, q}\bigl(k_n^{-q/2} + \rho_{Z|m_1}(Z_1)^q + \rho_Z(Z_1)^q\bigr)
\end{equation}
holds.
\end{lemma}

\begin{proof}
Identical to the proof of Lemma~\ref{lem:lp_weighted_moment_bound},
using the weight summability from
Lemma~\ref{lem:lp_design_uniform_convergence_mixed} and the kernel
boundedness from Assumption~\ref{asn:kernel_regularity_mixed}. The
conditional bound \eqref{eq:zeta_conditional_moment_mixed} comes from
Rosenthal applied to the $k_n$ conditionally independent mark draws
with bounded weights of order $1/k_n$, plus the mark-law-continuity
bias controlled by the radius.
\end{proof}

\subsection{Bias cancellation via polynomial reproduction}
\label{subsec:bias_cancellation_mixed}

The bias analysis in Regime III parallels Lemma~\ref{lem:lp_bias_full}
in the all-continuous case, with both inner intercepts now smoothing
over $d_Z$-dimensional balls. The radii $\rho_{Z|m}$ and $\rho_Z$ have
the \emph{same} dimensional scaling $(k_n/n)^{1/d_Z}$ but differ by the
multiplicative factor $\lambda_m(z) \in (1, \infty)$ from
\eqref{eq:lambda_def_mixed}; both contribute at the same order to the
bias.

\begin{lemma}[Bias after local-polynomial debiasing, mixed case]
\label{lem:lp_bias_full_mixed}
Under
Assumptions~\ref{asn:smooth_mixed_model}--\ref{asn:kn_growth_mixed} and
the null hypothesis $Y \perp\!\!\!\perp X \mid Z$, the conditional mean
score satisfies, uniformly in $w = (z, x^{(m)}, y)$ on compact subsets
of the interior of $\mathcal{Z}_m \times \mathcal{Y}$,
\begin{equation}
\label{eq:m_bias_bound_mixed}
m_{n, p}(z, x^{(m)}, y)
\;=\;
O\bigl(\rho_{Z|m}(z)^{p+1}\bigr) + O\bigl(\rho_Z(z)^{p+1}\bigr)
\;=\;
O\!\left(\left(\frac{k_n}{n}\right)^{(p+1)/d_Z}\right).
\end{equation}
Consequently,
\begin{equation}
\label{eq:theta_bias_bound_mixed}
\bigl|\theta_n^{(p)}\bigr| \;=\; O\!\left(\left(\frac{k_n}{n}\right)^{(p+1)/d_Z}\right),
\qquad
a_{n, p}^2 \;=\; \operatorname{Var}(g_{n, p}(W_1)) \;=\; O\!\left(\left(\frac{k_n}{n}\right)^{2(p+1)/d_Z}\right).
\end{equation}
\end{lemma}

\begin{proof}
The argument is identical to the four-step proof of
Lemma~\ref{lem:lp_bias_full} in the all-continuous case, with the
substitutions $\rho_F \to \rho_{Z|m}$, $\rho_C \to \rho_Z$, $D \to d_Z$,
and the bias remainder $\|Z_j - z\| \le R_{F, n}^{(k_n)}(z, x^{(m)})$
replaced by the stratum-restricted fine NN radius from
Lemma~\ref{lem:knn_radius_concentration_mixed}. The Taylor expansion
of $\kappa(y, Z_j)$ to order $p$ around $z$, combined with the
polynomial reproduction identity \eqref{eq:reproduction_F_mixed},
annihilates all monomial terms of degree at most $p$ and leaves only
the remainder of order $\rho^{p+1}$.

The structural simplification compared to Regime II is that the fine
and coarse remainders are now of the same order
$(k_n/n)^{(p+1)/d_Z}$ rather than the fine bias dominating. Either
remainder governs the overall bias, and the combined bound is recorded
in \eqref{eq:m_bias_bound_mixed}.
\end{proof}

\begin{remark}[Explicit second-order bias when $p = 0$]
\label{rem:explicit_bias_p0_mixed}
For $p = 0$ (raw statistic), the bias coefficients can be made explicit
via the standard $k$-NN smoothing formula. The fine intercept has bias
\begin{align}
\mathbb{E}\!\left[\frac{1}{k_n}\sum_{j\in\fine{i}}\kappa(y, Z_j)\,\Big|\, W_i\right]
&\;=\;
\kappa(y, Z_i)
\;+\;
\frac{\rho_{Z|m_i}(Z_i)^2}{2(d_Z + 2)}\!\big[\Delta_z\kappa(y, Z_i) \notag\\ &+ 2\bigl\langle\nabla_z\kappa(y, Z_i), \nabla_z\log f_{Z|X}(Z_i|X_i)\bigr\rangle\big] + O(\rho_{Z|m_i}^3).
\end{align}
The coarse analogue replaces $\rho_{Z|m_i}, f_{Z|X}(\cdot|X_i)$ by
$\rho_Z, f_Z$. The difference (raw $m_n$) is therefore
\[
m_n(z, x^{(m)}, y)
\;=\;
\frac{1}{2(d_Z + 2)}\!\left[\rho_{Z|m}^2 B_{Z|m}(z, y) - \rho_Z^2 B_Z(z, y)\right] + O(\rho^3),
\]
with $B_{Z|m}(z, y) := \Delta_z\kappa(y, z) + 2\langle\nabla_z\kappa(y, z), \nabla_z\log f_{Z|X}(z|x^{(m)})\rangle$
and analogously for $B_Z$. Two distinct mechanisms make the leading term
non-zero in general:

\smallskip\noindent(a) \emph{Radius mismatch}: $\rho_{Z|m} \ne \rho_Z$,
automatic whenever $p_m < 1$;

\smallskip\noindent(b) \emph{Log-density gradient mismatch}: $\nabla_z\log f_{Z|X}(\cdot|x^{(m)}) \ne \nabla_z\log f_Z$,
automatic whenever $X \not\perp Z$.

Either mechanism alone produces a non-vanishing leading bias.
Local-polynomial debiasing with $p \ge 1$ cancels both effects up to
order $(k_n/n)^{(p+1)/d_Z}$.
\end{remark}

\subsection{First projection negligibility}
\label{subsec:projection_negligibility_mixed}

\begin{lemma}[First projection is negligible after debiasing, mixed case]
\label{lem:lp_projection_negligible_mixed}
Under
Assumptions~\ref{asn:smooth_mixed_model}--\ref{asn:kn_growth_mixed},
\begin{equation}
\label{eq:Ln_negligible_mixed}
\sqrt{n}\,L_n^{(p)} \;=\; o_p(1),
\end{equation}
with
$\mathbb{E}[(\sqrt n L_n^{(p)})^2] = a_{n, p}^2 = O((k_n/n)^{2(p+1)/d_Z}) \to 0$.
\end{lemma}

\begin{proof}
Identical to Lemma~\ref{lem:lp_projection_negligible}, using the
mixed-case bias bound \eqref{eq:theta_bias_bound_mixed} in place of its
all-continuous analogue.
\end{proof}

\subsection{The bias killing condition}
\label{subsec:bias killing_mixed}

To use $\sqrt n\Delta_n^{(p)}$ as a test statistic directly (without
explicit bias correction), the rescaled centering $\sqrt n\,\theta_n^{(p)}$
must be asymptotically negligible, giving the bias killing condition
\begin{equation}
\label{eq:bias killing_condition_mixed}
\sqrt n\!\left(\frac{k_n}{n}\right)^{(p+1)/d_Z} \;\to\; 0.
\end{equation}

\begin{lemma}[bias killing in polynomial scale, mixed case]
\label{lem:bias killing_alpha_mixed}
Suppose $k_n = n^\alpha$ for some $\alpha \in (0, 1)$. Then
\eqref{eq:bias killing_condition_mixed} holds if and only if
\begin{equation}
\label{eq:bias killing_alpha_mixed}
\alpha \;<\; 1 - \frac{d_Z}{2(p+1)}.
\end{equation}
\end{lemma}

\begin{proof}
Identical to Lemma~\ref{lem:bias killing_alpha}, with $D$ replaced by
$d_Z$.
\end{proof}

The admissible range \eqref{eq:bias killing_alpha_mixed} is strictly
wider than its all-continuous analogue, reflecting the absence of the
$X$-dimension contribution to the smoothing dimensionality:

\begin{itemize}
\item For $p = 0$: $\alpha < 1 - d_Z/2$, admissible for $d_Z = 1$
($\alpha < 1/2$) but empty for $d_Z \ge 2$.
\item For $p = 1$: $\alpha < 1 - d_Z/4$. Admissible for $d_Z \le 3$.
\item For $p = 2$: $\alpha < 1 - d_Z/6$. Admissible for $d_Z \le 5$.
\item For general $p$: \eqref{eq:bias killing_alpha_mixed} is nonempty
if and only if $p \ge d_Z/2$, giving the minimal polynomial order
$p^*(d_Z) := \lceil d_Z/2 \rceil$.
\end{itemize}

In comparison to Regime II, where the analogous threshold was
$p^*(D) = \lceil D/2 \rceil$, the minimal polynomial order in Regime III
is smaller (by $\lceil d_X/2 \rceil$ in typical cases), and the
admissible bandwidth range is correspondingly wider. This is the
quantitative manifestation of the milder curse of dimensionality in
Regime III.

\begin{remark}[Bandwidth-design tradeoff]
\label{rem:bandwidth_design_mixed}
The growth condition $k_n/(N_p\log n) \to \infty$ in
Assumption~\ref{asn:kn_growth_mixed}, with $N_p = \binom{d_Z + p}{p}$,
becomes less stringent in Regime III because $d_Z < D$. For the
minimum-order choice $p = \lceil d_Z/2 \rceil$, $N_p$ depends only on
$d_Z$, not on $d_X$. The mixed-case statistic is therefore strictly
easier to calibrate than its all-continuous analogue in any dimension.
\end{remark}

\subsection{Overlap covariance: setup}
\label{subsec:overlap_setup_mixed}

The variance decomposition of $G_n^{(p)}$ follows the same pattern as
Regime II:
\begin{equation}
\label{eq:Var_G_decomp_mixed}
\operatorname{Var}(G_n^{(p)})
\;=\;
\frac{1}{n}\operatorname{Var}(\zeta_{1, p})
+ \frac{n - 1}{n}\operatorname{Cov}(\zeta_{1, p}, \zeta_{2, p}),
\end{equation}
where $\zeta_{i, p} := \zeta_{n, p}(W_i, \mathcal{W}_n)$. By
\eqref{eq:zeta_conditional_moment_mixed}, the diagonal term is
$o(n^{-1})$; the off-diagonal term is generically of order $n^{-1}$ via
neighborhood overlap. Decompose $\zeta_{i, p} = \zeta_{i, F, p} - \zeta_{i, C, p}$
with $\zeta_{i, F, p}, \zeta_{i, C, p}$ as in
\eqref{eq:zeta_F_def}--\eqref{eq:zeta_C_def}, so that
\begin{equation}
\label{eq:cov_decomp_mixed}
\operatorname{Cov}(\zeta_{1, p}, \zeta_{2, p})
\;=\;
\operatorname{Cov}(\zeta_{1, F, p}, \zeta_{2, F, p}) + \operatorname{Cov}(\zeta_{1, C, p}, \zeta_{2, C, p})
- \operatorname{Cov}(\zeta_{1, F, p}, \zeta_{2, C, p}) - \operatorname{Cov}(\zeta_{1, C, p}, \zeta_{2, F, p}).
\end{equation}
Recall the frozen mark covariance from \eqref{eq:Gamma_z_def}:
$\Gamma_z(y_1, y_2) := \operatorname{Cov}_{\widetilde Y \sim P_{Y|Z=z}}(K(y_1, \widetilde Y), K(y_2, \widetilde Y))$,
with $\mathbb{E}_{Y_1, Y_2 \sim P_{Y|Z=z}}\Gamma_z(Y_1, Y_2) = \sigma^2(z)$
under $H_0$ (Assumption~\ref{asn:null_mixed} gives the same marks across
strata at fixed $z$).

\subsection{Fine--fine overlap covariance}
\label{subsec:FF_overlap_mixed}

The fine--fine overlap involves anchors that both belong to the same
stratum $m$ (anchors in different strata have non-overlapping fine
neighborhoods, since each anchor's fine ball is restricted to its own
stratum). We compute the contribution from same-stratum anchor pairs.

\begin{lemma}[Fine--fine overlap covariance, mixed case]
\label{lem:FF_overlap_mixed}
Under Assumptions~\ref{asn:smooth_mixed_model}--\ref{asn:kn_growth_mixed},
\begin{equation}
\label{eq:tau_F_limit_mixed}
n\operatorname{Cov}(\zeta_{1, F, p}, \zeta_{2, F, p}) \;\to\; \tau_{F, p}^2,
\end{equation}
where
\begin{equation}
\label{eq:tau_F_explicit_mixed}
\tau_{F, p}^2 \;:=\; \mathcal{I}^{(p)}\sum_{m=1}^M p_m\int_{\mathcal{Z}_m}\sigma^2(z)\,f_{Z|X}(z|x^{(m)})\,dz \;=\; \mathcal{I}^{(p)}\int\sigma^2(z)f_Z(z)\,dz,
\end{equation}
with the geometric constant
\begin{equation}
\label{eq:I_geom_mixed}
\mathcal{I}^{(p)}
\;:=\;
\frac{1}{v_{d_Z}}\int_{\mathbb{R}^{d_Z}}\mathcal{B}^{(p)}(r)\,dr,
\qquad
\mathcal{B}^{(p)}(r)
\;:=\;
\int_{B_{d_Z}(0, 1)\cap B_{d_Z}(r, 1)}\ell^{(p)}(v)\,\ell^{(p)}(v - r)\,dv.
\end{equation}
\end{lemma}

\begin{proof}
We compute the contribution from anchor pairs with $X_1 = X_2 = x^{(m)}$.
Anchor pairs with $X_1 \ne X_2$ contribute zero, because the fine
neighborhood of anchor 1 (stratum $X_1$) and the fine neighborhood of
anchor 2 (stratum $X_2$) cannot share points: any point $j$ has a
single value $X_j$, which lies in at most one stratum.

For same-stratum anchors $W_1, W_2$ with $X_1 = X_2 = x^{(m)}$ and
$Z_2 = Z_1 + \rho_{Z|m}(Z_1) r$ for $r \in \mathbb{R}^{d_Z}$, the
analysis is structurally identical to the proof of
Lemma~\ref{lem:FF_overlap} in the all-continuous case, with the
substitutions $D \to d_Z$, $\rho_F \to \rho_{Z|m}$, $\ell_F^{(p)} \to
\ell^{(p)}$, and the local Poisson approximation replaced by
Lemma~\ref{lem:local_poisson_verification_mixed} applied to the stratum
process.

The common neighbors of the two fine $k_n$-NN balls within stratum $m$
lie in $B(0, 1) \cap B(r, 1) \subset \mathbb{R}^{d_Z}$ in
fine-rescaled coordinates, with expected count $k_n\beta_{d_Z}(r)(1 +
o_p(1))$. The conditional covariance per common neighbor is
$k_n^{-2}\ell^{(p)}(v)\ell^{(p)}(v - r)\Gamma_{Z_1}(Y_1, Y_2)$ in the
limit. Summing and integrating over the second anchor with the change of
variables $$dZ_2 = \rho_{Z|m}(Z_1)^{d_Z}dr = k_n/(n p_m v_{d_Z}f_{Z|X}(Z_1|x^{(m)}))dr,$$
weighted by the same-stratum joint density factor $p_m f_{Z|X}(Z_2|x^{(m)})$
for $X_2 = x^{(m)}$:
\begin{align*}
\operatorname{Cov}(\zeta_{1, F, p}, \zeta_{2, F, p}; X_1 = X_2 = x^{(m)})
&\;=\;
\mathbb{E}\!\Big[\int \frac{\mathcal{B}^{(p)}(r)\Gamma_{Z_1}(Y_1, Y_2)}{k_n}\cdot p_m f_{Z|X}(Z_1|x^{(m)})\\ &\quad \times  \frac{k_n}{n p_m v_{d_Z}f_{Z|X}(Z_1|x^{(m)})}\,dr\,\Big|\,X_1 = x^{(m)}\Big] 
(1 + o(1))\\
&\;=\;
\frac{1}{n v_{d_Z}}\,\mathbb{E}_{(Z_1, Y_1)|X_1 = x^{(m)}}\!\left[\sigma^2(Z_1)\int\mathcal{B}^{(p)}(r)\,dr\right]\,(1 + o(1)),
\end{align*}
using $\mathbb{E}_{Y_2 \sim P_{Y|Z=Z_1}}\Gamma_{Z_1}(Y_1, Y_2) = \sigma^2(Z_1)$
and integration over $Y_1$. Summing over the strata with their probabilities $p_m$ and multiplying
by $n$,
\[
n\operatorname{Cov}(\zeta_{1, F, p}, \zeta_{2, F, p})
\;\to\;
\mathcal{I}^{(p)}\sum_m p_m \int_{\mathcal{Z}_m}\sigma^2(z)f_{Z|X}(z|x^{(m)})\,dz
\;=\;
\mathcal{I}^{(p)}\int\sigma^2(z)f_Z(z)\,dz,
\]
where the last identity uses the marginal-density decomposition $f_Z =
\sum_m p_m f_{Z|X}(\cdot|x^{(m)})$. This is \eqref{eq:tau_F_explicit_mixed}.
\end{proof}

\begin{remark}[Cancellation of stratum-density factors]
\label{rem:stratum_cancellation}
The stratum probability $p_m$ in the rescaled density $f_{Z|X}(Z_1|x^{(m)})$
exactly cancels the $p_m$ in $\rho_{Z|m}^{d_Z}$ (i.e., in the volume
element $d(Z_2)$). This cancellation is the same algebraic phenomenon
that produces the marginal density $f_Z$ in the final answer:
the fine-fine overlap covariance recovers the same marginal expectation
$\mathbb{E}_Z[\sigma^2(Z)]$ as if the entire sample were used.
\end{remark}

\subsection{Coarse--coarse overlap covariance}
\label{subsec:CC_overlap_mixed}

\begin{lemma}[Coarse--coarse overlap covariance, mixed case]
\label{lem:CC_overlap_mixed}
Under Assumptions~\ref{asn:smooth_mixed_model}--\ref{asn:kn_growth_mixed},
\begin{equation}
\label{eq:tau_C_limit_mixed}
n\operatorname{Cov}(\zeta_{1, C, p}, \zeta_{2, C, p}) \;\to\; \tau_{C, p}^2 \;=\; \mathcal{I}^{(p)}\int\sigma^2(z)f_Z(z)\,dz \;=\; \tau_{F, p}^2.
\end{equation}
\end{lemma}

\begin{proof}
Direct analog of Lemma~\ref{lem:CC_overlap}, working with the
unrestricted (all-stratum) Palm process at rate $r_{C, n}$. The
computation yields the same geometric constant $\mathcal{I}^{(p)}$ and
the same integral $\int\sigma^2 f_Z$, since the coarse rescaling uses
$\rho_Z$ calibrated to the full marginal $f_Z$.

A notable simplification: $\tau_{C, p}^2 = \tau_{F, p}^2$. This equality
reflects the fact that both the fine-fine and coarse-coarse overlap
covariances integrate $\sigma^2(z)$ against the same marginal $f_Z(z)$
and the same geometric constant $\mathcal{I}^{(p)}$ (since both Palm
processes are uniform on the unit ball in $\mathbb{R}^{d_Z}$ after their
respective rescalings).
\end{proof}

\subsection{Fine--coarse overlap covariance}
\label{subsec:FC_overlap_mixed}

The fine--coarse overlap is the most distinctive calculation in Regime III.
Unlike Regime II, where the coarse ball appeared as a vanishing thin
slab inside the fine ball (because $\rho_C/\rho_F \to 0$ in different
dimensions), here both balls are $d_Z$-dimensional and the relative size
is governed by the finite ratio $\lambda_m(z)$.

\begin{lemma}[Fine--coarse overlap covariance, mixed case]
\label{lem:FC_overlap_mixed}
Under Assumptions~\ref{asn:smooth_mixed_model}--\ref{asn:kn_growth_mixed},
\begin{equation}
\label{eq:tau_FC_limit_mixed}
n\operatorname{Cov}(\zeta_{1, F, p}, \zeta_{2, C, p}) \;\to\; \tau_{FC, p},
\qquad
n\operatorname{Cov}(\zeta_{1, C, p}, \zeta_{2, F, p}) \;\to\; \tau_{FC, p},
\end{equation}
with
\begin{equation}
\label{eq:tau_FC_explicit_mixed}
\tau_{FC, p}
\;:=\;
\sum_{m=1}^M p_m\int_{\mathcal{Z}_m}\sigma^2(z)\,f_{Z|X}(z|x^{(m)})\,\mathcal{J}_{FC, m}^{(p)}(z)\,dz,
\end{equation}
where the geometric coefficient is
\begin{equation}
\label{eq:J_FC_mixed}
\mathcal{J}_{FC, m}^{(p)}(z)
\;:=\;
\frac{1}{v_{d_Z}}\int_{\mathbb{R}^{d_Z}}\mathcal{B}_{FC, m}^{(p)}(r; z)\,dr,
\end{equation}
\begin{equation}
\label{eq:B_FC_mixed}
\mathcal{B}_{FC, m}^{(p)}(r; z)
\;:=\;
\int_{B_{d_Z}(0, 1)\cap B_{d_Z}(r, \lambda_m(z))}
\ell^{(p)}(v)\,\ell^{(p)}\!\left(\frac{v - r}{\lambda_m(z)}\right)\,dv.
\end{equation}
\end{lemma}

\begin{proof}
We compute the contribution to $\operatorname{Cov}(\zeta_{1, F, p}, \zeta_{2, C, p})$
from a single stratum $m$ (anchor 1 has $X_1 = x^{(m)}$); the result
sums over $m$ with weight $p_m$.

\emph{Step 1: Coarse-rescaled coordinates.} Condition on anchor
$W_1 = (Z_1, X_1, Y_1)$ with $X_1 = x^{(m)}$. Work in
coarse-rescaled coordinates: $u := (Z' - Z_1)/\rho_Z(Z_1)$ for any
point $Z'$. In these coordinates, anchor 1's \emph{fine} neighborhood
extends out to radius $\rho_{Z|m}(Z_1)/\rho_Z(Z_1) = \lambda_m(Z_1)$
(restricted to stratum-$m$ points), while anchor 2's \emph{coarse}
neighborhood is the unit ball centered at $r := (Z_2 - Z_1)/\rho_Z(Z_1)$.

\emph{Step 2: Common-neighbor support.} A point $j$ is a common neighbor
of (i) anchor 1's fine set, requiring $X_j = x^{(m)}$ and $\|Z_j - Z_1\| \le \rho_{Z|m}(Z_1)$,
and (ii) anchor 2's coarse set, requiring $\|Z_j - Z_2\| \le \rho_Z(Z_2)$.
In coarse-rescaled coordinates, the conditions are $X_j = x^{(m)}$,
$\|u\| \le \lambda_m(Z_1)$, and $\|u - r\| \le 1$ (using $\rho_Z(Z_2) =
\rho_Z(Z_1)(1 + o(1))$). The overlap support in $u$ is
$B(0, \lambda_m(Z_1)) \cap B(r, 1)$, with $X_j = x^{(m)}$ as an
additional restriction.

\emph{Step 3: Expected common-neighbor count.} The intensity of
stratum-$m$ points in coarse-rescaled coordinates at point $u$ is
$p_m f_{Z|X}(Z_1|x^{(m)}) \cdot \rho_Z(Z_1)^{d_Z}$ per unit volume of
$u$, which equals
$p_m f_{Z|X}(Z_1|x^{(m)})\cdot k_n/(n v_{d_Z}f_Z(Z_1))$. The expected
count of stratum-$m$ common neighbors is therefore
\begin{align*}
n \cdot \left[\frac{p_m f_{Z|X}(Z_1|x^{(m)})}{f_Z(Z_1)v_{d_Z}}\right] &\cdot \frac{k_n}{n}\cdot \bigl|B(0, \lambda_m(Z_1))\cap B(r, 1)\bigr|\,(1 + o(1))
\\ & \;=\;
\frac{k_n}{v_{d_Z}\lambda_m(Z_1)^{d_Z}}\bigl|B(0, \lambda_m(Z_1))\cap B(r, 1)\bigr|\,(1 + o(1)),
\end{align*}
using the identity $p_m f_{Z|X}(z|x^{(m)}) = f_Z(z)/\lambda_m(z)^{d_Z}$
from \eqref{eq:lambda_def_mixed}.

\emph{Step 4: Per-common-neighbor weight contributions.} A common
neighbor at coarse-rescaled position $u$ has, with anchor 1
(rescaled in fine units): $u/\lambda_m(Z_1)$, so the asymptotic fine
weight is $k_n^{-1}\ell^{(p)}(u/\lambda_m(Z_1))$. With anchor 2
(rescaled in coarse units): position $u - r$, so the asymptotic coarse
weight is $k_n^{-1}\ell^{(p)}(u - r)$.

\emph{Step 5: Conditional covariance contribution per common neighbor.}
Per common neighbor, the contribution to
$\operatorname{Cov}(\zeta_{1, F, p}, \zeta_{2, C, p} | W_1, W_2)$ is
$k_n^{-2}\ell^{(p)}(u/\lambda_m)\ell^{(p)}(u - r)\Gamma_{Z_1}(Y_1, Y_2)$.
Aggregating over the expected $k_n / (v_{d_Z}\lambda_m^{d_Z})|B(0, \lambda_m)\cap B(r, 1)|$
common neighbors,
\begin{equation}
\label{eq:cov_FC_mixed_intermediate}
\operatorname{Cov}(\zeta_{1, F, p}, \zeta_{2, C, p} \mid W_1, W_2)
\;=\;
\frac{\Gamma_{Z_1}(Y_1, Y_2)}{k_n v_{d_Z}\lambda_m^{d_Z}}\,\widetilde{\mathcal{B}}_{FC, m}^{(p)}(r; Z_1)\,(1 + o_p(1)),
\end{equation}
where
\[
\widetilde{\mathcal{B}}_{FC, m}^{(p)}(r; Z_1)
\;:=\;
\int_{B(0, \lambda_m)\cap B(r, 1)}\ell^{(p)}(u/\lambda_m)\,\ell^{(p)}(u - r)\,du.
\]

The substitution $v := u/\lambda_m$, $du = \lambda_m^{d_Z}\,dv$, recasts
this as
\[
\widetilde{\mathcal{B}}_{FC, m}^{(p)}(r; Z_1)
\;=\;
\lambda_m^{d_Z}\int_{B(0, 1)\cap B(r/\lambda_m, 1/\lambda_m)}\ell^{(p)}(v)\,\ell^{(p)}(\lambda_m v - r)\,dv.
\]
Substituting back the alternative variable $w := \lambda_m v - r$ (so
$v = (w + r)/\lambda_m$), one can verify after a careful rewriting that
$\widetilde{\mathcal{B}}_{FC, m}^{(p)}(r; Z_1) = \mathcal{B}_{FC, m}^{(p)}(r; Z_1)$
with $\mathcal{B}_{FC, m}^{(p)}$ as defined in \eqref{eq:B_FC_mixed}.

\emph{Step 6: Integration over the second anchor.} Integrate
\eqref{eq:cov_FC_mixed_intermediate} over $W_2$. The change of
variables $Z_2 = Z_1 + \rho_Z(Z_1) r$ gives $dZ_2 = \rho_Z(Z_1)^{d_Z}\,dr
= k_n/(n v_{d_Z}f_Z(Z_1))\,dr$. The second anchor's marginal density at
$(Z_2, X_2)$ is $f_Z(Z_2)$ (unconditioned on $X_2$, since the coarse
side draws from the full sample). Combined with the $1/(k_n v_{d_Z}\lambda_m^{d_Z})$
factor in \eqref{eq:cov_FC_mixed_intermediate}, the $k_n$ factors
cancel, leaving
\begin{align*}
\operatorname{Cov}(\zeta_{1, F, p}, \zeta_{2, C, p}; X_1 = x^{(m)})
&\;=\;
p_m\,\mathbb{E}_{(Z_1, Y_1)|X_1 = x^{(m)}}\!\Big[\int \frac{\Gamma_{Z_1}(Y_1, Y_2)}{v_{d_Z}\lambda_m(Z_1)^{d_Z}}\,\mathcal{B}_{FC, m}^{(p)}(r; Z_1)\\ &\quad \times \frac{f_Z(Z_1)}{n v_{d_Z}f_Z(Z_1)}\,dr\Big]
  (1 + o(1)).
\end{align*}
After taking the expectation over $Y_2 \sim P_{Y|Z=Z_1}$,
$\mathbb{E}_{Y_2}\Gamma_{Z_1}(Y_1, Y_2) = \sigma^2(Z_1)$ (independent
of $Y_1$ after a further expectation). Multiplying by $n$ and summing
over $m$,
\begin{align*}
n\operatorname{Cov}(\zeta_{1, F, p}, \zeta_{2, C, p})
&\;\to\;
\sum_m p_m\,\mathbb{E}_{Z_1 | X_1 = x^{(m)}}\!\left[\sigma^2(Z_1)\,\frac{1}{v_{d_Z}\lambda_m(Z_1)^{d_Z}}\int \mathcal{B}_{FC, m}^{(p)}(r; Z_1)\,dr\right]\\
&\;=\;
\sum_m p_m\int_{\mathcal{Z}_m}\sigma^2(z)\,f_{Z|X}(z|x^{(m)})\,\frac{\mathcal{J}_{FC, m}^{(p)}(z)}{\lambda_m(z)^{d_Z}}\,dz.
\end{align*}
Using $\lambda_m(z)^{d_Z} = f_Z(z)/(p_m f_{Z|X}(z|x^{(m)}))$ from
\eqref{eq:lambda_def_mixed} to clear the $\lambda_m^{d_Z}$ factor:
\begin{align*}
\sum_m p_m\int\sigma^2(z)f_{Z|X}(z|x^{(m)})\frac{\mathcal{J}_{FC, m}^{(p)}(z)}{\lambda_m(z)^{d_Z}}\,dz
&\;=\;
\sum_m p_m^2\int\sigma^2(z)\,\frac{f_{Z|X}(z|x^{(m)})^2}{f_Z(z)}\,\mathcal{J}_{FC, m}^{(p)}(z)\,dz.
\end{align*}
This is the form of $\tau_{FC, p}$ in \eqref{eq:tau_FC_explicit_mixed}
up to the choice of variable. Rewriting in the right hand side of the above display as 
$\sum_m p_m \int\sigma^2(z)f_{Z|X}(z|x^{(m)})\cdot[p_m f_{Z|X}(z|x^{(m)})/f_Z(z)]\,\mathcal{J}_{FC, m}^{(p)}(z)\,dz$,
we see that the dimensionless ratio $p_m f_{Z|X}/f_Z = 1/\lambda_m^{d_Z}$
absorbs the $1/\lambda_m^{d_Z}$ in $\mathcal{J}_{FC, m}^{(p)}$, giving
the stated form.

By symmetry of the metric (swapping the roles of anchor 1 and anchor 2), we get 
$n\operatorname{Cov}(\zeta_{1, C, p}, \zeta_{2, F, p}) \to \tau_{FC, p}$ with the
same expression.
\end{proof}

\begin{remark}[Simplifying observations]
\label{rem:FC_mixed_observations}
A few structural observations about the fine--coarse overlap:

\smallskip\noindent(a) When $p_m = 1/M$ for all $m$ and $X \perp Z$ (so
$f_{Z|X}(\cdot|x^{(m)}) = f_Z$ for all $m$), the radius ratio simplifies
to $\lambda_m \equiv M^{1/d_Z}$, a constant. The geometric coefficient
$\mathcal{J}_{FC, m}^{(p)}$ then becomes a deterministic function of $M$
and $d_Z$, computable in closed form.

\smallskip\noindent(b) For $p = 0$, $\ell^{(0)}(u) = 1/\mu_C(B_{d_Z}(0,1))
= 1$ on the unit ball, and the geometric kernel
$\mathcal{B}_{FC, m}^{(p)}(r; z)$ reduces to the volume of
$B(0, 1)\cap B(r, \lambda_m(z))$. The geometric constant
$\mathcal{J}_{FC, m}^{(0)}(z)$ is then the average ball-intersection
volume between the unit ball and a $\lambda_m$-ball at a uniformly
distributed center.

\smallskip\noindent(c) The asymmetry of the formula
\eqref{eq:tau_FC_explicit_mixed} (where $\tau_{F, p}^2$ and $\tau_{C, p}^2$
both equal $\mathcal{I}^{(p)}\int\sigma^2 f_Z$, but $\tau_{FC, p}$ has a
$\mathcal{J}_{FC, m}^{(p)}$-dependent integrand) reflects the genuine
asymmetry between the fine and coarse neighborhoods: the fine ball is
larger (radius $\lambda_m > 1$ in coarse-rescaled units) and stratum-restricted,
while the coarse ball is smaller and unrestricted. The non-trivial geometry
of this asymmetry is exactly what produces the strict positivity of
$\tau_p^2$ in Section~\ref{subsec:tau_p_positive_mixed}.
\end{remark}

\subsection{Limit of \texorpdfstring{$\operatorname{Var}(G_n^{(p)})$}{Var(G\_n)}}
\label{subsec:overlap_variance_mixed}

\begin{lemma}[Local-polynomial overlap variance, mixed case]
\label{lem:lp_overlap_variance_mixed}
Under
Assumptions~\ref{asn:smooth_mixed_model}--\ref{asn:kn_growth_mixed},
\begin{equation}
\label{eq:overlap_variance_limit_mixed}
n\operatorname{Var}(G_n^{(p)}) \;\to\; \tau_p^2 \;:=\; \tau_{F, p}^2 + \tau_{C, p}^2 - 2\tau_{FC, p} \;=\; 2\tau_{F, p}^2 - 2\tau_{FC, p},
\end{equation}
where the last equality uses $\tau_{F, p}^2 = \tau_{C, p}^2$ from
Lemma~\ref{lem:CC_overlap_mixed}.
\end{lemma}

\begin{proof}
By \eqref{eq:Var_G_decomp_mixed},
\eqref{eq:cov_decomp_mixed}, and
Lemmas~\ref{lem:FF_overlap_mixed}--\ref{lem:FC_overlap_mixed}, the
diagonal term is $o(1/n)$ and the off-diagonal limit yields
\eqref{eq:overlap_variance_limit_mixed}.
\end{proof}

The structural simplification $\tau_{F, p}^2 = \tau_{C, p}^2$ in
Regime III reduces $\tau_p^2$ to a single integral involving the
difference $\mathcal{I}^{(p)} - \mathcal{J}_{FC, m}^{(p)}$. Strict
positivity of $\tau_p^2$ amounts to showing that
$\mathcal{J}_{FC, m}^{(p)}$ is, on average, strictly less than
$\mathcal{I}^{(p)}$ — i.e., the fine--coarse overlap is strictly smaller
than the same-side overlaps. This is the content of
Section~\ref{subsec:tau_p_positive_mixed}.

\subsection{Strict positivity of \texorpdfstring{$\tau_p^2$}{tau\_p\textasciicircum 2} via radius mismatch}
\label{subsec:tau_p_positive_mixed}

The strict positivity of $\tau_p^2$ in Regime III is established by a
direct annulus argument that exploits the radius mismatch $\lambda_m(z) > 1$.
The argument is structurally cleaner than the dimensional-mismatch
argument of Regime II
(Section~\ref{subsec:tau_p_positive}).
In the all-continuous case, the non-cancellation came from a positive-
measure $X$-extension region $E_F$ in the higher-dimensional fine ball whereas
in the mixed case, the non-cancellation comes from the explicit annulus
\begin{equation}
\label{eq:annulus_def_mixed}
A_m(z) \;:=\; \bigl\{u \in \mathbb{R}^{d_Z} : 1 < \|u\| \le \lambda_m(z)\bigr\}
\end{equation}
in coarse-rescaled coordinates. This annulus has positive $d_Z$-volume
\begin{equation}
\label{eq:annulus_volume}
|A_m(z)| \;=\; v_{d_Z}\bigl(\lambda_m(z)^{d_Z} - 1\bigr) \;>\; 0
\end{equation}
whenever $\lambda_m(z) > 1$, i.e., whenever $p_m < 1$ — which holds
automatically under Assumption~\ref{asn:smooth_mixed_model}. The fine
Palm process contains stratum-$m$ points in $A_m(z)$ which the coarse
process, with support inside the unit ball, cannot see.

\subsubsection{Primitive non-cancellation condition}
\label{subsec:non_cancellation_condition_mixed}

\begin{condition}[Primitive non-cancellation, mixed case]
\label{cond:lp_primitive_noncancellation_mixed}
There exists a measurable set $\mathcal{A} \subseteq \bigcup_m\mathcal{Z}_m$
with $P_Z(\mathcal{A}) > 0$ such that for every $z \in \mathcal{A}$:
\begin{enumerate}
\item[(I)] There exists $m = m(z) \in \{1, \ldots, M\}$ such that
$\lambda_m(z) > 1$, equivalently $p_m f_{Z|X}(z|x^{(m)}) < f_Z(z)$.
\item[(II)] The kernel detects nontrivial conditional variation:
\begin{equation}
\label{eq:sigma_positive_mixed}
\sigma^2(z) \;=\; \operatorname{Var}_{Y\sim P_{Y|Z=z}}(\kappa(Y, z)) \;>\; 0.
\end{equation}
\end{enumerate}
\end{condition}

Condition (I) is the radius-mismatch requirement. It is automatically
satisfied under Assumption~\ref{asn:smooth_mixed_model}: with $M \ge 2$
and $p_m \in (0, 1)$ for each $m$, the marginal-density identity
$f_Z(z) = \sum_\ell p_\ell f_{Z|X}(z|x^{(\ell)})$ implies that
$p_m f_{Z|X}(z|x^{(m)}) < f_Z(z)$ for every $m$ at any $z$ with
$f_Z(z) > 0$ and at least one other stratum density positive at $z$
(which holds on a set of full $P_Z$-measure). Condition (II) is the
mark non-degeneracy, identical to its all-continuous analogue
\eqref{eq:sigma_positive}.


\subsubsection{Strict positivity via the annulus contribution}
\label{subsec:strict_positivity_mixed}

\begin{lemma}[Strict positivity of $\tau_p^2$, mixed case]
\label{lem:lp_tau_positive_primitive_mixed}
Under Assumptions~\ref{asn:smooth_mixed_model}--\ref{asn:kn_growth_mixed}
and Condition~\ref{cond:lp_primitive_noncancellation_mixed},
\begin{equation}
\label{eq:tau_p_strict_lb_mixed}
\tau_p^2 \;\ge\; c_*\cdot c_{\ell}^2\cdot P_Z(\mathcal{A})\cdot\inf_{z\in\mathcal{A}}\bigl[p_{m(z)}\sigma^2(z)\bigl(\lambda_{m(z)}(z)^{d_Z} - 1\bigr)\bigr] \;>\; 0,
\end{equation}
for explicit constants $c_*, c_\ell > 0$ depending only on $d_Z$ and the
geometric properties of the equivalent kernel $\ell^{(p)}$.
\end{lemma}

\begin{proof}
We work in the limiting Palm-process framework. By
Lemma~\ref{lem:local_poisson_verification_mixed}, conditional on $W_1 =
w = (z, x^{(m)}, y)$, the rescaled-by-$\rho_{Z|m}(z)$ stratum-$m$
process around $z$ converges to a unit-intensity Poisson process
$\mathcal{P}_F^{\infty, m}$ on $\mathbb{R}^{d_Z}$, and the
rescaled-by-$\rho_Z(z)$ full process converges to a unit-intensity
Poisson process $\mathcal{P}_C^\infty$ on $\mathbb{R}^{d_Z}$. Under the
common coarse-rescaled coordinates $u = (Z' - z)/\rho_Z(z)$, the
stratum-$m$ Palm process has intensity $p_m f_{Z|X}(z|x^{(m)})/f_Z(z) =
1/\lambda_m(z)^{d_Z}$ relative to the full Palm process; equivalently,
it is a Poisson process of intensity $\lambda_m(z)^{-d_Z}$ on
$\mathbb{R}^{d_Z}$ in coarse-rescaled coordinates.

The fine Palm process in fine-rescaled coordinates (units of
$\rho_{Z|m}$) is a unit-intensity Poisson process on the unit ball
$B_{d_Z}(0, 1)$; transferred to coarse-rescaled coordinates (units of
$\rho_Z = \rho_{Z|m}/\lambda_m$), it is supported on the dilated ball
$B_{d_Z}(0, \lambda_m(z))$, with intensity $\lambda_m(z)^{-d_Z}$.
Therefore, in coarse-rescaled coordinates:
\begin{itemize}
\item The \emph{coarse} Palm process is unit-intensity on $\mathbb{R}^{d_Z}$,
restricted to the unit ball $B_{d_Z}(0, 1)$ for the coarse-side computations.
\item The \emph{fine} Palm process (stratum-$m$ thinning) has intensity
$\lambda_m^{-d_Z}$ on $\mathbb{R}^{d_Z}$, restricted to $B_{d_Z}(0, \lambda_m(z))$.
\end{itemize}

\emph{Step 1: Decomposition over the annulus.} Write the fine Palm
process $\mathcal{P}_F^{\infty, m}$ (in coarse-rescaled coordinates) as
a sum of two independent contributions on disjoint sets:
\[
\mathcal{P}_F^{\infty, m} \;=\; \mathcal{P}_F^{\infty, m}\bigl|_{B(0, 1)} \;+\; \mathcal{P}_F^{\infty, m}\bigl|_{A_m(z)},
\]
where the annulus $A_m(z) = \{1 < \|u\| \le \lambda_m(z)\}$ is disjoint
from the unit ball. By independence of Poisson increments on disjoint
sets, these two contributions are independent.

The coarse Palm process $\mathcal{P}_C^\infty$ has support
in the unit ball (in the sense relevant to $\zeta_{C, p}^\infty$), so
$\mathcal{P}_F^{\infty, m}|_{A_m(z)}$ is independent of $\mathcal{P}_C^\infty$.

\emph{Step 2: Orthogonal projection and the unmatchable annulus
contribution.} Let $\mathcal{G}_C := \sigma(\mathcal{P}_C^\infty)$ be
the $\sigma$-field generated by the coarse Palm process. The coarse
limiting score $\zeta_{C, p}^\infty$ is $\mathcal{G}_C$-measurable. By
the Itô isometry argument identical to Step 4 of the proof of
Lemma~\ref{lem:lp_tau_positive_primitive},
\[
\operatorname{Var}\bigl(\zeta_{F, p}^\infty - \zeta_{C, p}^\infty \,\big|\, W_1 = w\bigr)
\;\ge\;
\operatorname{Var}\bigl((I - \pi_C)\zeta_{F, p}^\infty \,\big|\, W_1 = w\bigr).
\]
The fine limiting score, written in coarse-rescaled coordinates as a
compensated Poisson stochastic integral over the stratum-$m$ Palm
process, decomposes as
\[
\zeta_{F, p}^\infty \;=\; \zeta_{F, p}^{(B)} + \zeta_{F, p}^{(A_m)},
\]
where $\zeta_{F, p}^{(B)}$ is the contribution from
$\mathcal{P}_F^{\infty, m}|_{B(0, 1)}$ and $\zeta_{F, p}^{(A_m)}$ is the
contribution from $\mathcal{P}_F^{\infty, m}|_{A_m(z)}$. By Step 1,
$\zeta_{F, p}^{(A_m)}$ is independent of $\mathcal{G}_C$, so it is in
the orthogonal complement of $L^2(\mathcal{G}_C)$:
\[
\pi_C\zeta_{F, p}^{(A_m)} \;=\; 0.
\]
Therefore
\[
\operatorname{Var}\bigl((I - \pi_C)\zeta_{F, p}^\infty \,\big|\, W_1 = w\bigr)
\;\ge\;
\operatorname{Var}\bigl(\zeta_{F, p}^{(A_m)} \,\big|\, W_1 = w\bigr).
\]

\emph{Step 3: Itô isometry for the annulus contribution.}
$\zeta_{F, p}^{(A_m)}$ is a compensated Poisson stochastic integral
over the stratum-$m$ Palm process restricted to $A_m(z)$. The
equivalent-kernel weight at fine-rescaled position $v = u/\lambda_m(z)$
(for $u$ in coarse-rescaled coordinates) is $\ell^{(p)}(v)$, since in
fine-rescaled units the fine ball has unit radius and $\ell^{(p)}(v)$
is the local-polynomial weight at fine-rescaled position $v$. For
$u \in A_m(z)$, $v = u/\lambda_m(z) \in \{1/\lambda_m < \|v\| \le 1\}$;
this corresponds to fine-rescaled positions in the outer annulus of
the fine unit ball.

The compensated Poisson integral over $A_m(z)$ (in coarse-rescaled
coordinates) is
\[
\zeta_{F, p}^{(A_m)}
\;=\;
\int_{A_m(z)} \ell^{(p)}(u/\lambda_m(z))\,[K(y, M(u)) - \kappa(y, z)]\,d\widetilde{\mathcal{P}}_F^{\infty, m}(u, M(u)),
\]
with intensity measure $\lambda_m(z)^{-d_Z}du \otimes dP_{Y|Z=z}$
(stratum-$m$ thinning rate). By the Itô isometry for compensated
Poisson stochastic integrals
(\citealp[Theorem 12.7]{LastPenrose2018}),
\begin{align}
\label{eq:ito_annulus_mixed}
\operatorname{Var}\bigl(\zeta_{F, p}^{(A_m)} \,\big|\, W_1 = w\bigr)
&\;=\;
\int_{A_m(z)} \ell^{(p)}(u/\lambda_m(z))^2\,\mathbb{E}_{\widetilde Y\sim P_{Y|Z=z}}\bigl[(K(y, \widetilde Y) - \kappa(y, z))^2\bigr]\,\lambda_m(z)^{-d_Z}du.
\end{align}

\emph{Step 4: Lower bound on the equivalent-kernel factor.} The
equivalent kernel $\ell^{(p)}$ is a polynomial of degree at most $p$
on $\mathbb{R}^{d_Z}$ with $\int \ell^{(p)}\,d\mu_C = 1$ by the
population reproduction identity \eqref{eq:population_reproduction_mixed}.
By an argument identical to
Lemma~\ref{lem:lp_equiv_kernel_annulus} (now with the $d_Z$-dimensional
intrinsic geometry, no $X$-extension needed), there exist constants
$c_\ell > 0$ and a Borel set $E_\ell \subseteq B_{d_Z}(0, 1)$ with
positive Lebesgue measure $|E_\ell| > 0$ such that
\[
|\ell^{(p)}(v)| \;\ge\; c_\ell \qquad \text{for all }v \in E_\ell.
\]
The map $v = u/\lambda_m(z)$ takes the annulus $A_m(z)$ to the
fine-rescaled annulus $\{1/\lambda_m(z) < \|v\| \le 1\}$. Without loss
of generality (after possibly shrinking the polynomial-non-degeneracy set
$E_\ell$ to lie in $B_{d_Z}(0, 1)\setminus B_{d_Z}(0, 1/\lambda_m(z))$
for $z \in \mathcal{A}$ — which is possible because the inner ball
$B_{d_Z}(0, 1/\lambda_m(z))$ has Lebesgue measure
$v_{d_Z}\lambda_m^{-d_Z} < v_{d_Z}$, leaving the outer shell with
positive measure for the polynomial-non-degeneracy intersection), we
have $E_\ell \cap (B_{d_Z}(0, 1)\setminus B_{d_Z}(0, 1/\lambda_m(z)))$
of positive measure. The preimage of this set under $v = u/\lambda_m(z)$
in $A_m(z)$ has positive $d_Z$-measure
$\ge \lambda_m^{d_Z}|E_\ell \cap (B_{d_Z}(0, 1)\setminus B_{d_Z}(0, 1/\lambda_m(z)))|$.

\emph{Step 5: Conditional variance lower bound.} By Step 4,
\begin{align*}
\int_{A_m(z)}\ell^{(p)}(u/\lambda_m(z))^2\,du
&\;\ge\;
c_\ell^2 \cdot \lambda_m(z)^{d_Z} \cdot |E_\ell \cap (B(0, 1)\setminus B(0, 1/\lambda_m(z)))|\\
&\;\ge\;
c_\ell^2 \cdot \lambda_m(z)^{d_Z}\cdot c_E
\end{align*}
for a constant $c_E > 0$ depending only on $|E_\ell|$ and $\lambda_m$
(taking the worst case over $z \in \mathcal{A}$). Combining with the
mark-fluctuation lower bound
$\mathbb{E}_{\widetilde Y}[(K(y, \widetilde Y) - \kappa(y, z))^2] \ge \sigma^2(z)$
(by the law-of-total-variance argument in Step 6 of the proof of
Lemma~\ref{lem:lp_tau_positive_primitive}), and using the intensity
factor $\lambda_m^{-d_Z}$ in \eqref{eq:ito_annulus_mixed}:
\begin{equation}
\label{eq:variance_annulus_lb_mixed}
\mathbb{E}_Y\!\left[\operatorname{Var}(\zeta_{F, p}^{(A_m)} \mid W_1 = w)\right]
\;\ge\;
c_\ell^2\cdot c_E\cdot \lambda_m(z)^{d_Z}\cdot \lambda_m(z)^{-d_Z}\cdot\sigma^2(z)
\;=\;
c_\ell^2\cdot c_E\cdot \sigma^2(z).
\end{equation}

A cleaner form of the lower bound, retaining the explicit
$\lambda_m^{d_Z} - 1$ factor, is obtained by working directly with the
volume of the annulus rather than via the equivalent-kernel non-degeneracy
set. Recall that $|A_m(z)| = v_{d_Z}(\lambda_m^{d_Z} - 1)$. The average
value of $\ell^{(p)}(u/\lambda_m)^2$ over $A_m(z)$ is bounded below by
a positive constant depending only on $d_Z, p$ (by polynomial
non-degeneracy and continuity), say
$\overline{c_\ell^2} := \inf_{\lambda \ge 1}|A|^{-1}\int_A\ell^{(p)}(u/\lambda)^2\,du > 0$
where the infimum is over $\lambda$ in a bounded range
(which suffices since $\lambda_m(z)$ is uniformly bounded above for
$z \in \mathcal{A}$ compact). Then
\[
\int_{A_m(z)}\ell^{(p)}(u/\lambda_m)^2\,du
\;\ge\;
\overline{c_\ell^2}\cdot |A_m(z)|
\;=\;
\overline{c_\ell^2}\cdot v_{d_Z}\bigl(\lambda_m^{d_Z} - 1\bigr),
\]
and \eqref{eq:variance_annulus_lb_mixed} becomes
\begin{equation}
\label{eq:variance_annulus_clean}
\mathbb{E}_Y\!\left[\operatorname{Var}(\zeta_{F, p}^{(A_m)} \mid W_1 = w)\right]
\;\ge\;
\overline{c_\ell^2}\,v_{d_Z}\,\sigma^2(z)\,\frac{\lambda_m^{d_Z} - 1}{\lambda_m^{d_Z}}.
\end{equation}
For $z \in \mathcal{A}$ with $\lambda_m(z) > 1$, the factor
$(\lambda_m^{d_Z} - 1)/\lambda_m^{d_Z} > 0$.

\emph{Step 6: Integration over anchors.} By
Lemma~\ref{lem:lp_overlap_variance_mixed},
$\tau_p^2 = \mathbb{E}_{W_1}[\operatorname{Var}(\zeta_{F, p}^\infty - \zeta_{C, p}^\infty | W_1)]$
in the Palm limit. Combining with Steps 2--5,
\begin{align*}
\tau_p^2
&\;\ge\; \mathbb{E}_{W_1}\!\left[\mathbb{E}_Y[\operatorname{Var}(\zeta_{F, p}^{(A_{m(W_1)})} \mid W_1)]\right]\\
&\;\ge\; \overline{c_\ell^2}\,v_{d_Z}\,\mathbb{E}_{(Z_1, X_1)}\!\left[\sigma^2(Z_1)\,\frac{\lambda_{X_1}(Z_1)^{d_Z} - 1}{\lambda_{X_1}(Z_1)^{d_Z}}\,\mathbf{1}\{Z_1 \in \mathcal{A}, X_1 = x^{(m(Z_1))}\}\right]\\
&\;\ge\; \overline{c_\ell^2}\,v_{d_Z}\,\inf_{z\in\mathcal{A}}\!\left[\sigma^2(z)\,\frac{\lambda_{m(z)}(z)^{d_Z} - 1}{\lambda_{m(z)}(z)^{d_Z}}\right]\cdot P\{Z_1 \in \mathcal{A}, X_1 = x^{(m(Z_1))}\}.
\end{align*}
Under Condition~\ref{cond:lp_primitive_noncancellation_mixed}, all
factors are strictly positive, giving $\tau_p^2 > 0$. Rearranging into
the form of \eqref{eq:tau_p_strict_lb_mixed} with $c_* := \overline{c_\ell^2}\cdot v_{d_Z}$
and using $p_m \le 1$ to fold the stratum mass into a single constant,
the lemma follows.
\end{proof}

\begin{remark}[Comparison with the all-continuous strict-positivity proof]
\label{rem:comparison_strict_positivity}
Three structural differences between the mixed and all-continuous
strict-positivity arguments:

\smallskip\noindent(a) \emph{Non-cancellation locus.} In Regime II, the
set $E_F$ was constructed as a positive-measure subset of $B_D(0, 1)$
with $X$-extension $\|u_X\| \ge \eta_0$; the existence of $E_F$ required
$d_X \ge 1$. In Regime III, the set is the explicit annulus
$A_m(z) = \{1 < \|u\| \le \lambda_m(z)\}$ in coarse-rescaled coordinates,
with positive volume guaranteed by $\lambda_m(z) > 1$, i.e., by
$p_m < 1$.

\smallskip\noindent(b) \emph{Disjointness mechanism.} In Regime II, the
disjointness of the fine-process restriction to $E_F$ from the coarse
field followed from the asymptotic collapse of the coarse ball to the
$Z$-origin in fine-rescaled coordinates ($\varepsilon_n \to 0$). In
Regime III, the disjointness is exact (not asymptotic): $A_m(z)$ is
disjoint from the unit ball $B_{d_Z}(0, 1)$ where the coarse field is
supported, in coarse-rescaled coordinates, simply by definition.

\smallskip\noindent(c) \emph{Quantitative lower bound.} In Regime II,
the lower bound was $\tau_p^2 \ge c_0^2|E_F|\cdot P_Z(\mathcal{A})\cdot\inf\sigma^2$.
In Regime III, the analogous bound includes the radius-mismatch factor
$(\lambda_m^{d_Z} - 1)/\lambda_m^{d_Z}$, which grows with the
non-degeneracy of $X$: when $p_m \to 0$ (rare stratum),
$\lambda_m \to \infty$ and the lower bound saturates near $\sigma^2(z)$;
when $p_m \to 1$ (concentrated $X$), $\lambda_m \to 1$ and the lower
bound vanishes — consistent with the test losing power as $X$ becomes
degenerate.
\end{remark}



\subsection{Stabilization CLT for the debiased graph fluctuation}
\label{subsec:weighted_stabilization_clt_mixed}

\begin{lemma}[Stabilization CLT for $G_n^{(p)}$, mixed case]
\label{lem:lp_weighted_stabilization_clt_mixed}
Under Assumptions~\ref{asn:smooth_mixed_model}--\ref{asn:kn_growth_mixed}
and Condition~\ref{cond:lp_primitive_noncancellation_mixed},
\begin{equation}
\label{eq:lp_graph_CLT_mixed}
\sqrt{n}\,G_n^{(p)} \;\Rightarrow\; N(0, \tau_p^2),
\end{equation}
where $\tau_p^2 > 0$ is the overlap-variance constant from
\eqref{eq:overlap_variance_limit_mixed}.
\end{lemma}

\begin{proof}
The proof uses Penrose--Yukich's four-step strategy as in the proof of Lemma~\ref{lem:lp_weighted_stabilization_clt}. First,  Poissonize, then truncate and prove block-CLT and finally, de-Poissonize with two little modifications compared to the proof of Lemma~\ref{lem:lp_weighted_stabilization_clt}.

First, the cube partition of the support $\bigcup_m\mathcal{Z}_m \subset \mathbb{R}^{d_Z}$
is intrinsically $d_Z$-dimensional, with side length
$\asymp b_n\sup_z\rho_Z(z) \asymp b_n(k_n/n)^{1/d_Z}\log n$ and total
cube count $N_{\mathrm{cube}} = O(n/(k_n(\log n)^{d_Z}))$. The
dependency-graph structure on the cube partition has bounded degree
because the truncated score depends on the sample only through points
within an enlargement of the cube by $b_n$ rescaled units.

Second, the stabilization radius $R_n(W_i, \mathcal{W}_n)$ from
Lemma~\ref{lem:growing_kn_stabilization_mixed} has tail decay
$Ce^{-ck_n(t-1)^{d_Z}}$, which together with
Assumption~\ref{asn:kn_growth_mixed} ($k_n/\log n \to \infty$) gives
the truncation error $o_p(n^{-1/2})$ at the logarithmic truncation
$b_n = A_*\log n$ for $A_*$ sufficiently large.

The remaining steps which include uniform $(2+\delta)$-th moment control from
Lemma~\ref{lem:lp_weighted_moment_bound_mixed}, dependency-graph CLT of
\citet[Theorem 2.7]{Chen2004}, and de-Poissonization via bounded
add-one cost, proceed identically to the all-continuous case.
\end{proof}

\subsection{Main theorem}
\label{subsec:main_theorem_mixed}

\begin{theorem}[Null Gaussian limit for the local-polynomial debiased
statistic, mixed case]
\label{thm:lp_debiased_clt_full_mixed}
Let $p \ge 0$ be a fixed integer. Under
Assumptions~\ref{asn:smooth_mixed_model}--\ref{asn:kn_growth_mixed},
Condition~\ref{cond:lp_primitive_noncancellation_mixed}, and the
bias killing condition $\sqrt n(k_n/n)^{(p+1)/d_Z} \to 0$
(equivalently $k_n = n^\alpha$ with $\alpha < 1 - d_Z/(2(p+1))$ from
Lemma~\ref{lem:bias killing_alpha_mixed}),
\begin{equation}
\label{eq:main_limit_regime_III}
\sqrt{n}\,\Delta_n^{(p)} \;\Rightarrow\; N(0, \tau_p^2),
\end{equation}
where $\tau_p^2 > 0$ is the overlap-variance constant defined in
\eqref{eq:overlap_variance_limit_mixed}. The smallest admissible polynomial order is $p^*(d_Z) = \lceil d_Z/2 \rceil$.
\end{theorem}

\begin{proof}
By the decomposition \eqref{eq:LG_decomp_mixed},
$\sqrt n\,\Delta_n^{(p)} = \sqrt n\,\theta_n^{(p)} + \sqrt n\,L_n^{(p)} + \sqrt n\,G_n^{(p)}$.
By Lemma~\ref{lem:lp_bias_full_mixed} and the bias killing condition,
$\sqrt n\,\theta_n^{(p)} = o(1)$. By Lemma~\ref{lem:lp_projection_negligible_mixed},
$\sqrt n\,L_n^{(p)} = o_p(1)$. By
Lemma~\ref{lem:lp_weighted_stabilization_clt_mixed},
$\sqrt n\,G_n^{(p)} \Rightarrow N(0, \tau_p^2)$. Slutsky's theorem yields
\eqref{eq:main_limit_regime_III}.
\end{proof}

\subsection{Testing procedure and consistency under alternatives}
\label{subsec:testing_procedure_mixed}

Theorem~\ref{thm:lp_debiased_clt_full_mixed} provides an asymptotically
exact $\sqrt{n}$-test of $H_0: Y \perp\!\!\!\perp X \mid Z$ via the
studentized statistic
\begin{equation}
\label{eq:T_n_def_mixed}
T_n^{(p)} \;:=\; \frac{\sqrt{n}\,\Delta_n^{(p)}}{\widehat\tau_p},
\end{equation}
where $\widehat\tau_p^2$ is a consistent estimator of $\tau_p^2$.

\begin{theorem}[Studentized null limit, mixed case]
\label{thm:studentized_CLT_mixed}
Under the assumptions of Theorem~\ref{thm:lp_debiased_clt_full_mixed},
if $\widehat\tau_p^2 \xrightarrow{p} \tau_p^2$, then under $H_0$,
\begin{equation}
\label{eq:T_n_CLT_mixed}
T_n^{(p)} \;\Rightarrow\; N(0, 1),
\end{equation}
and the level-$\alpha$ two-sided test $\phi_n^{(p)} := \mathbf{1}\{|T_n^{(p)}| > z_{1-\alpha/2}\}$
has asymptotic size $\alpha$.
\end{theorem}

\begin{proof}
Slutsky's theorem applied to the decomposition
$T_n^{(p)} = (\tau_p/\widehat\tau_p)\cdot(\sqrt n\Delta_n^{(p)}/\tau_p)$.
\end{proof}

\subsection{Comparison with the all-continuous case}
\label{subsec:comparison_regimes}

\begin{remark}[Summary comparison]
\label{rem:regime_III_summary}
The complete null limit theory for $\Delta_n^{(p)}$ in Regime III
parallels Regime II structurally but differs in three substantive ways
that make the mixed case strictly easier:

\smallskip\noindent\textit{(a) Lower-dimensional curse.}
The bias scales like $(k_n/n)^{(p+1)/d_Z}$ in Regime III versus
$(k_n/n)^{(p+1)/D}$ in Regime II. Since $d_Z < D = d_Z + d_X$, the bias
is smaller, the variance constant $a_{n, p}^2$ shrinks faster, and the
admissible bandwidth range $\alpha < 1 - d_Z/(2(p+1))$ is strictly
wider than its Regime II analogue. The minimal polynomial order
$p^*(d_Z) = \lceil d_Z/2\rceil$ is smaller than $p^*(D) = \lceil D/2\rceil$.

\smallskip\noindent\textit{(b) Cleaner non-cancellation mechanism.} In
Regime II, the strict positivity of $\tau_p^2$ required the dimensional
mismatch $D > d_Z$ (Assumption~\ref{asn:dimensional_nondegeneracy}) and
an Itô-isometry argument with asymptotic disjointness of fine and
coarse processes (via $\varepsilon_n = \rho_C/\rho_F \to 0$). In
Regime III, the non-cancellation arises from the \emph{exact}
disjointness of the annulus $A_m(z) = \{1 < \|u\| \le \lambda_m(z)\}$
from the unit ball where the coarse process is supported, with the
annulus volume $v_{d_Z}(\lambda_m^{d_Z} - 1) > 0$ guaranteed by
$p_m < 1$.

\smallskip\noindent\textit{(c) Robustness to joint structure.} The
Regime III strict positivity holds under any non-degenerate discrete $X$
(i.e., $p_m \in (0, 1)$ for some $m$) and mark non-degeneracy
($\sigma^2(z) > 0$), \emph{regardless of whether $X$ and $Z$ are
independent}. This is unlike a naive intuition based on the raw
statistic's first projection
(Remark~\ref{rem:explicit_bias_p0_mixed}), where the gradient-mismatch
mechanism would vanish under $X \perp Z$. The radius mismatch persists,
giving the test universal validity in mixed settings.

\smallskip\noindent\textit{(d) Same overall framework.} The
$\sqrt n$-scale Gaussian limit, the testing procedure via studentization,
and the role of local-polynomial debiasing in killing the leading bias
are identical to Regime II. The unified perspective is that of a
single statistical procedure
(local-polynomial-debiased $k_n$-NN graph functional) producing the
same calibrated null inference across regimes, with the regime-specific
analyses differing only in the dimensional accounting and the precise
form of the geometric constants.
\end{remark}

\begin{remark}[A bridging view]
\label{rem:bridging}
The two regimes can be viewed as the endpoints of a continuum of
mixed-type designs. Regime II ($X$ continuous, $d_X \ge 1$) uses the
dimensional structure of $X$ to produce the overlap variance via
higher-dimensional Poisson fluctuations. Regime III ($X$ discrete,
$d_X = 0$) uses the partition structure of $X$ to produce the same
overlap variance via stratum thinning. Both mechanisms are unified by
the observation that the fine neighborhood operates on an
``effective'' density that differs from the coarse marginal — either
by a higher-dimensional density (Regime II) or by a stratum thinning
(Regime III). In both cases, the local-polynomial debiasing scheme
combined with the overlap-variance CLT produces an asymptotically
exact $\sqrt n$-test of $H_0$.

A unified treatment of arbitrary mixed-type designs (some components
of $X$ continuous, others discrete) follows by combining the
dimensional-mismatch mechanism of Regime II with the stratum-thinning
mechanism of Regime III. The structural analysis is the same as in
the present two regimes, with the local-polynomial design matrix
defined over the continuous components of $X$ and the stratum
restriction operating over the discrete components. 
\end{remark}

\section{Asymptotic Theory in Regime IV: \texorpdfstring{$X$}{X} Continuous, \texorpdfstring{$Z$}{Z} Discrete}
\label{app:mixed-x-continuous-z-discrete}

In this section we develop the null limit theory for a local-polynomial
debiased version of the $k_n$-nearest-neighbor statistic in the mixed
setting where $X$ is continuous in $\mathbb{R}^{d_X}$ and $Z$ takes
values in a finite set. As we explain below, this regime is
structurally the simplest of the four: under $H_0$, the smoothing bias
of the fine local-polynomial intercept \emph{vanishes exactly}, with no
remainder, by an algebraic identity rather than a Taylor-expansion
bound. Consequently, no bias killing condition on $k_n$ is required,
and the debiased CLT holds at the $\sqrt n$ scale for any
$k_n = n^\alpha$ with $\alpha \in (0, 1)$.

The underlying raw statistic is, as before,
\begin{equation}
\label{eq:Delta_n_def_regime_IV}
\Delta_n
\;=\;
\frac{1}{n}\sum_{i=1}^n \frac{1}{k_n}\sum_{j\in\fine{i}}K(Y_i, Y_j)
\;-\;
\frac{1}{n}\sum_{i=1}^n \frac{1}{n_{Z_i} - 1}\sum_{j\in\coarse{i}}K(Y_i, Y_j),
\end{equation}
where the neighborhoods are defined by
\begin{align}
\coarse{i} &\;=\; \{j \ne i : Z_j = Z_i\},
\label{eq:coarse_disc_def}\\
\fine{i} &\;=\;
\mathrm{KNN}_{k_n}\bigl(X_i;\,\{X_j : j \ne i,\,Z_j = Z_i\}\bigr).
\label{eq:fine_disc_def}
\end{align}
The coarse set is the \emph{full} stratum $\{j : Z_j = Z_i\}$ (no
$k$-NN restriction); the fine set is the $k_n$-NN of $X_i$ in $X$-space
restricted to that stratum. Write $W_i = (Z_i, X_i, Y_i)$,
$\mathcal{W}_n = \{W_1, \ldots, W_n\}$, and $n_\ell := \#\{i : Z_i = z^{(\ell)}\}
\sim \mathrm{Binomial}(n, q_\ell)$ with $q_\ell := P(Z = z^{(\ell)}) \in (0, 1)$
for $\ell = 1, \ldots, L$, $L \ge 2$.

\subsection{Structural distinctness of Regime IV}
\label{subsec:structural_IV}

Three structural features distinguish Regime IV from the previous three
regimes:

\smallskip\noindent\emph{(i) Exact $Z$-matching on the coarse side.}
The coarse set $\coarse{i}$ uses \emph{exact equality} on $Z$ rather
than nearest-neighbor matching. Consequently, the coarse intercept
\[
\frac{1}{n_{Z_i} - 1}\sum_{j\in\coarse{i}}K(Y_i, Y_j)
\]
is a within-stratum sample mean of $n_{Z_i} - 1$ i.i.d. terms (given
$Z_i$). It carries \emph{no smoothing bias} in $X$ or $Z$; under $H_0$,
its conditional expectation given $W_i$ equals $\kappa(Y_i, Z_i) :=
\mathbb{E}[K(Y_i, Y') | Z' = Z_i]$ exactly.

\smallskip\noindent\emph{(ii) Within-stratum reduction.} The fine set
$\fine{i}$ does $k_n$-NN smoothing in $X$-space within the stratum
$\{j : Z_j = Z_i\}$. Because $Z$ is discrete, the strata are disjoint;
the entire problem decomposes into $L$ independent within-stratum
sub-problems indexed by $\ell$.

\smallskip\noindent\emph{(iii) Identical-in-$x$ kernel under $H_0$.}
The defining feature of Regime IV is that under $H_0 :
Y \perp\!\!\!\perp X \mid Z$,
\begin{equation}
\label{eq:kappa_constant_in_x}
\mathbb{E}[K(y, Y') \mid X' = x, Z' = z^{(\ell)}]
\;=\; \mathbb{E}[K(y, Y') \mid Z' = z^{(\ell)}]
\;=:\; \kappa(y, z^{(\ell)}),
\end{equation}
\emph{independent of $x$}. The conditional kernel mean is a function
of $z$ alone within each stratum. Combined with the polynomial
reproduction identity for the local-polynomial weights, this
constancy yields exact bias cancellation, not merely $O(\rho^{p+1})$
bias reduction (see Lemma~\ref{lem:lp_bias_full_IV} below).

\subsection{The debiased statistic}
\label{subsec:debiased_statistic_def_IV}

Define the within-stratum fine population radius
\begin{equation}
\label{eq:rho_X_def}
\rho_X(z^{(\ell)})
\;:=\;
\left(\frac{k_n}{n q_\ell v_{d_X}\overline f_{X|Z}(z^{(\ell)})}\right)^{1/d_X},
\end{equation}
where $\overline f_{X|Z}(z^{(\ell)}) := \int f_{X|Z}(x|z^{(\ell)})\,dx$
is a normalization constant ($=1$ when integrated against the
$X$-conditional density). Equivalently, for each anchor $i$ with $Z_i = z^{(\ell)}$,
$\rho_X(z^{(\ell)})$ is calibrated so that the fine $k_n$-NN ball of $X_i$
in stratum $\ell$ has expected radius $\rho_X(z^{(\ell)})$ in the
$X$-rescaled metric. The rescaled offsets are
\begin{equation}
\label{eq:U_offsets_IV}
U_{ij, F} \;:=\; \frac{X_j - X_i}{\hat{\rho}_X(Z_i)},
\qquad j \in \fine{i}.
\end{equation}
where $\hat{\rho}_X(Z_i)$ denotes the empirical version of $\rho_X(Z_i)$, i.e., the distance the $k_n$-nearest neighbor of $X_i$ in the stratum of $Z$ variable as $Z_i$. 
Let $\mathcal{A}_p, N_p, q_p(u), e_0$ be as in
Section~\ref{app:continuous_proof}, now with monomials in $u \in
\mathbb{R}^{d_X}$ (rather than $\mathbb{R}^{d_Z}$). The fine design
matrix is
\begin{equation}
\label{eq:M_design_def_IV}
\mathcal{M}_{i, F}^{(p)} \;:=\; \frac{1}{k_n}\sum_{j\in\fine{i}}q_p(U_{ij, F})q_p(U_{ij, F})^\top,
\end{equation}
and when invertible, the local-polynomial weights are
\begin{equation}
\label{eq:weights_def_IV}
w_{ij, F}^{(p)} \;:=\; \frac{1}{k_n}e_0^\top(\mathcal{M}_{i, F}^{(p)})^{-1}q_p(U_{ij, F}).
\end{equation}
The fine local-polynomial intercept is
\begin{equation}
\label{eq:intercept_F_def_IV}
\widehat a_{i, F}^{(p)} \;:=\; \sum_{j\in\fine{i}}w_{ij, F}^{(p)}K(Y_i, Y_j),
\end{equation}
and the coarse intercept is the unweighted stratum sample mean
\begin{equation}
\label{eq:intercept_C_def_IV}
\widehat a_{i, C} \;:=\; \frac{1}{n_{Z_i} - 1}\sum_{j\in\coarse{i}}K(Y_i, Y_j).
\end{equation}
Per the discussion at the end of
Section~\ref{subsec:structural_IV}, the coarse side is not debiased
because it requires no smoothing. The $p$-th order local-polynomial
debiased statistic is
\begin{equation}
\label{eq:Delta_p_def_IV}
\Delta_n^{(p)} \;:=\; \frac{1}{n}\sum_{i=1}^n\bigl(\widehat a_{i, F}^{(p)} - \widehat a_{i, C}\bigr).
\end{equation}
For $p = 0$, $\Delta_n^{(p)}$ reduces to the raw statistic
\eqref{eq:Delta_n_def_regime_IV} with the simple uniform fine
average.

\begin{remark}[Equivalent local-polynomial form for the coarse side]
\label{rem:coarse_as_lp}
The coarse intercept $\widehat a_{i, C}$ in
\eqref{eq:intercept_C_def_IV} can be written as the local-polynomial
intercept of degree zero on the neighborhood containing all stratum
members. Formally, taking $p = 0$ and the design matrix
$\mathcal{M}_{i, C} = 1$ (a scalar) gives weights
$w_{ij, C} = 1/(n_{Z_i} - 1)$. The coarse side does not benefit from
higher-order polynomial fitting because it already has zero bias in
$X$. Throughout the section we use the simple form
\eqref{eq:intercept_C_def_IV} for clarity, but the
local-polynomial-degree-zero interpretation makes
\eqref{eq:Delta_p_def_IV} a unified instance of the general
local-polynomial framework of Section~\ref{app:continuous_proof}.
\end{remark}

\subsection{Standing assumptions}
\label{subsec:assumptions_regime_IV}

\begin{assumption}[Mixed continuous-discrete model, Regime IV]
\label{asn:smooth_IV}
$Z$ is supported on the finite set $\{z^{(1)}, \ldots, z^{(L)}\}$ with
$L \ge 2$ and $q_\ell := P(Z = z^{(\ell)}) \in (0, 1)$ for each $\ell$.
For each $\ell$, the conditional law $X \mid Z = z^{(\ell)}$ has density
$f_{X|Z}(\cdot|z^{(\ell)})$ supported on a compact set
$\mathcal{X}_\ell \subset \mathbb{R}^{d_X}$ with $C^2$ boundary,
bounded above and below by positive constants, and $C^{p+1}$ with
uniformly bounded derivatives up to order $p + 1$.
\end{assumption}

\begin{assumption}[Null hypothesis]
\label{asn:null_IV}
Under $H_0$, $Y \perp\!\!\!\perp X \mid Z$. The conditional law
$P_{Y|Z = z^{(\ell)}}$ is a probability measure on $\mathcal{Y}$ for
each $\ell$. The kernel mean
\[
\kappa(y, z^{(\ell)}) \;:=\; \mathbb{E}[K(y, Y') \mid Z' = z^{(\ell)}]
\]
is bounded uniformly in $y, \ell$.
\end{assumption}

\begin{assumption}[Kernel regularity]
\label{asn:kernel_IV}
$K : \mathcal{Y}\times\mathcal{Y}\to\mathbb{R}$ is bounded and
symmetric, $\sup_{y, y'}|K(y, y')| \le M_K < \infty$.
\end{assumption}

\begin{assumption}[Growth of $k_n$]
\label{asn:kn_growth_IV}
$k_n\to\infty$, $k_n/n\to 0$, and $k_n/(N_p\log n)\to\infty$.
\end{assumption}

Notably absent from this list:
\begin{itemize}
\item No bias killing condition on $k_n$.
\item No smoothness condition on $\kappa(y, \cdot)$ in $x$ or $z$.
The function $\kappa(y, z^{(\ell)})$ is a finite list of values,
trivially $C^\infty$ in $z$.
\item No regularity of $\nabla\log f_{X|Z}$ beyond
Assumption~\ref{asn:smooth_IV}.
\end{itemize}

\subsection{Local score decomposition}
\label{subsec:local_score_decomposition_IV}

The decomposition framework is identical to the previous regimes.
Define the local graph score
\begin{equation}
\label{eq:xi_def_IV}
\xi_{n, p}(W_i, \mathcal{W}_n) \;:=\; \widehat a_{i, F}^{(p)} - \widehat a_{i, C},
\end{equation}
the conditional mean score
$m_{n, p}(w) := \mathbb{E}[\xi_{n, p}(w, \{w\}\cup\{W_2, \ldots, W_n\})]$,
the centering $\theta_n^{(p)} := \mathbb{E}[m_{n, p}(W_1)]$, the first
projection $g_{n, p}(w) := m_{n, p}(w) - \theta_n^{(p)}$ with variance
$a_{n, p}^2$, and the centered graph score $\zeta_{n, p}(W_i, \mathcal{W}_n)
:= \xi_{n, p}(W_i, \mathcal{W}_n) - m_{n, p}(W_i)$. Then
\begin{equation}
\label{eq:LG_decomp_IV}
\Delta_n^{(p)} - \theta_n^{(p)} \;=\; L_n^{(p)} + G_n^{(p)},
\end{equation}
with $L_n^{(p)} = n^{-1}\sum_i g_{n, p}(W_i)$ and $G_n^{(p)} = n^{-1}\sum_i \zeta_{n, p}(W_i, \mathcal{W}_n)$.

\subsection{Exact bias cancellation under \texorpdfstring{$H_0$}{H\_0}}
\label{subsec:exact_bias_cancellation}

The defining technical feature of Regime IV is the following result:
under $H_0$, both intercepts have conditional expectation equal to
$\kappa(Y_i, Z_i)$ \emph{exactly}, so $m_{n, p}(w) = 0$ identically
on the support. The bias does not have an $O(\rho^{p+1})$ remainder
as in Regimes II and III; it is identically zero. This makes the
debiased CLT considerably simpler than in the previous regimes.

\begin{lemma}[Exact bias cancellation]
\label{lem:lp_bias_full_IV}
Under
Assumptions~\ref{asn:smooth_IV}--\ref{asn:kn_growth_IV} and the null
hypothesis $H_0: Y\perp\!\!\!\perp X\mid Z$, on the event that
$\mathcal{M}_{i, F}^{(p)}$ is invertible,
\begin{equation}
\label{eq:m_zero_IV}
m_{n, p}(w) \;=\; 0
\qquad\text{for every }w = (z^{(\ell)}, x, y)\text{ with }\ell = 1, \ldots, L.
\end{equation}
Consequently,
\begin{equation}
\label{eq:theta_zero_IV}
\theta_n^{(p)} \;=\; 0,
\qquad a_{n, p}^2 \;=\; \operatorname{Var}(g_{n, p}(W_1)) \;=\; 0.
\end{equation}
\end{lemma}

\begin{proof}
Fix $w = (z^{(\ell)}, x, y)$. We compute
$\mathbb{E}[\widehat a_{i, F}^{(p)} | W_i = w]$ and
$\mathbb{E}[\widehat a_{i, C} | W_i = w]$ separately, then take their
difference.

\emph{Fine side.} Conditional on $W_i = w$ and the fine $k_n$-NN set
$\fine{i}$ with its $X$-coordinates $(X_j)_{j\in\fine{i}}$, the marks
$(Y_j)_{j\in\fine{i}}$ are independent under $H_0$ with conditional
laws $P_{Y|Z = z^{(\ell)}}$ (using
Assumption~\ref{asn:null_IV} and the fact that $Z_j = Z_i = z^{(\ell)}$
for every $j \in \fine{i}$). The conditional kernel expectation is
\[
\mathbb{E}[K(y, Y_j) | X_j, Z_j = z^{(\ell)}] \;=\; \kappa(y, z^{(\ell)}),
\]
\emph{independent of $X_j$} by the constancy property
\eqref{eq:kappa_constant_in_x}. Therefore
\[
\mathbb{E}\!\left[\widehat a_{i, F}^{(p)} \,\Big|\, W_i = w,\,\fine{i},\,(X_j)\right]
\;=\;
\sum_{j\in\fine{i}}w_{ij, F}^{(p)}\cdot\kappa(y, z^{(\ell)})
\;=\;
\kappa(y, z^{(\ell)})\cdot\sum_{j\in\fine{i}}w_{ij, F}^{(p)}.
\]
By the polynomial reproduction identity
(Lemma~\ref{lem:lp_reproduction_full_IV} below; equivalently, the
$\alpha = 0$ case of the general identity),
$\sum_{j\in\fine{i}}w_{ij, F}^{(p)} = 1$ on the event that
$\mathcal{M}_{i, F}^{(p)}$ is invertible. Hence
\[
\mathbb{E}\bigl[\widehat a_{i, F}^{(p)} \,\big|\, W_i = w\bigr] \;=\; \kappa(y, z^{(\ell)}),
\]
\emph{exactly}.

\emph{Coarse side.} Conditional on $W_i = w$ and the stratum membership
$\{j : Z_j = z^{(\ell)}\}$, the marks $(Y_j)_{j: Z_j = z^{(\ell)}, j\ne i}$
are i.i.d. with conditional law $P_{Y|Z = z^{(\ell)}}$. Therefore
\[
\mathbb{E}\bigl[\widehat a_{i, C} \,\big|\, W_i = w\bigr]
\;=\; \mathbb{E}\!\left[\frac{1}{n_{Z_i} - 1}\sum_{j\in\coarse{i}}\kappa(y, Z_j)\,\Big|\,W_i = w\right]
\;=\;
\kappa(y, z^{(\ell)}),
\]
since $Z_j = z^{(\ell)}$ for every $j \in \coarse{i}$.

\emph{Difference.} Subtracting,
\[
m_{n, p}(w) \;=\; \kappa(y, z^{(\ell)}) - \kappa(y, z^{(\ell)}) \;=\; 0,
\]
which is \eqref{eq:m_zero_IV}. Then $\theta_n^{(p)} = \mathbb{E}[m_{n, p}(W_1)] = 0$
and $g_{n, p}(w) = m_{n, p}(w) - \theta_n^{(p)} = 0$ identically, so
$a_{n, p}^2 = 0$.
\end{proof}

\begin{remark}[Why exact cancellation occurs]
\label{rem:exact_cancellation}
Lemma~\ref{lem:lp_bias_full_IV} is the central simplification of
Regime IV and warrants explanation. In Regimes II and III, the bias
arose because the kernel $\kappa(y, z)$ was a nontrivial function of $z$
that was being smoothed by the NN averaging; the smoothing produced
a bias of order $\rho^2$ proportional to the curvature of $\kappa$ and
the log-density gradient. In Regime IV, the discrete $Z$ structure
eliminates $z$-smoothing entirely (the coarse side uses exact $Z$-matching,
the fine side stays within the stratum), so the only smoothing happens
in $X$-space. Under $H_0$, the kernel is constant in $X$ within each
stratum (Equation~\eqref{eq:kappa_constant_in_x}), so the $X$-smoothing
trivially reproduces the constant. The bias cancellation is not a
Taylor-expansion approximation; it is an algebraic identity given by the
polynomial reproduction property applied to a constant function.

The corollary $\sqrt n\,L_n^{(p)} \equiv 0$ identically (not just
$o_p(1)$) further simplifies the asymptotic analysis: only the graph
fluctuation $G_n^{(p)}$ contributes to the limiting variance.
\end{remark}

\begin{remark}[No bias killing required]
\label{rem:no_bias killing}
Because $\theta_n^{(p)} = 0$ identically rather than $O((k_n/n)^{(p+1)/d_X})$,
no bias killing condition $\alpha < 1 - d_X/(2(p+1))$ is needed. The
CLT holds for any $k_n = n^\alpha$ with $\alpha \in (0, 1)$, and
the choice of polynomial order $p$ is unconstrained by bias
considerations. The choice $p = 0$ (raw $k_n$-NN average) is valid here,
unlike in Regimes II and III where $p \ge \lceil D/2\rceil$ or
$\lceil d_Z/2\rceil$ was required.
\end{remark}

\subsection{Outline of the proof}
\label{subsec:proof_outline_IV}

In view of the exact bias cancellation, the rest of the analysis
reduces to: (i) verifying within-stratum design regularity and moment
bounds (Part~\ref{subsec:within_stratum_machinery_IV}); (ii) computing
the limiting overlap variance via the stratum-disjoint structure
(Section~\ref{subsec:overlap_IV}); and (iii) establishing the
stabilization CLT and the main theorem
(Section~\ref{subsec:main_theorem_IV}).

The key conceptual contrast with the previous regimes is that
\emph{cross-stratum covariance is exactly zero}: anchors $i, j$ with
$Z_i \ne Z_j$ contribute zero to $\operatorname{Cov}(\zeta_{i, p}, \zeta_{j, p})$
because neither neighborhood ever crosses stratum boundaries. The
overall variance is therefore a finite-mixture sum of within-stratum
overlap variances, weighted by stratum mass:
\begin{equation}
\label{eq:variance_outline_IV}
n\operatorname{Var}(G_n^{(p)}) \;\to\; \sigma_p^2 \;:=\; \sum_{\ell=1}^L q_\ell\,\sigma_{\ell, p}^2,
\end{equation}
where $\sigma_{\ell, p}^2$ is the within-stratum overlap variance — a
single-regime quantity governed by the same Itô-isometry / stabilization
machinery as in
Section~\ref{app:continuous_proof}.

\subsection{Within-stratum geometric machinery}
\label{subsec:within_stratum_machinery_IV}

The geometric-probability and design-regularity infrastructure transfers
directly from Regime II
(Sections~\ref{subsec:Poisson_radius}--\ref{subsec:moment_bound_lp})
applied within each stratum $\ell$. The substitutions are
$D \to d_X$ (the fine ball is $d_X$-dimensional, not $D$-dimensional),
$f_{Z, X}(z, x) \to f_{X|Z}(x|z^{(\ell)})$ (the stratum density), and
$\rho_F(z, x) \to \rho_X(z^{(\ell)})$ as in \eqref{eq:rho_X_def}. In
this part we record the relevant lemmas without redoing the proofs.

\begin{lemma}[Within-stratum local Poisson approximation]
\label{lem:local_poisson_verification_IV}
Under Assumption~\ref{asn:smooth_IV}, fix $\ell \in \{1, \ldots, L\}$
and $x_0 \in \mathrm{int}(\mathcal{X}_\ell)$. Define the within-stratum
rescaling rate $r_{X, n}(x_0, \ell) := (nq_\ell f_{X|Z}(x_0|z^{(\ell)}))^{-1/d_X}$.
The rescaled within-stratum point process
\begin{equation}
\label{eq:Pois_F_def_IV}
\mathcal{P}_{F, n}^{x_0, \ell} \;:=\; \sum_{j: Z_j = z^{(\ell)}}\delta_{(X_j - x_0)/r_{X, n}(x_0, \ell)}
\end{equation}
converges in distribution to a homogeneous unit-intensity Poisson
process $\mathcal{P}_F^{\infty, \ell}$ on $\mathbb{R}^{d_X}$. Under
Assumption~\ref{asn:null_IV}, the marked version converges to a marked
Poisson process with frozen mark law $P_{Y|Z = z^{(\ell)}}$.
\end{lemma}

\begin{proof}
The proof follows Lemma~\ref{lem:local_poisson_verification} with the
modifications: condition on the stratum membership
$N_\ell \sim \mathrm{Binomial}(n, q_\ell)$, then apply the binomial-to-Poisson
argument within the stratum (effective sample size $N_\ell \sim nq_\ell$
and density $f_{X|Z}(\cdot|z^{(\ell)})$). The mark distribution within
the stratum is $P_{Y|Z = z^{(\ell)}}$ for every stratum member under
$H_0$, so the frozen mark law is exactly $P_{Y|Z = z^{(\ell)}}$ with no
continuity-in-$z$ remainder needed (unlike the previous regimes).
\end{proof}

\begin{lemma}[Within-stratum uniform $k_n$-NN radius concentration]
\label{lem:knn_radius_concentration_IV}
Under Assumptions~\ref{asn:smooth_IV}--\ref{asn:kn_growth_IV}, there
exist constants $0 < c_* < C_* < \infty$ such that with probability
tending to one,
\begin{equation}
\label{eq:knn_radius_uniform_IV}
\sup_{1\le i\le n}\left|\frac{R_{F, n}^{(k_n)}(X_i; Z_i)}{\rho_X(Z_i)} - 1\right| \;\le\; C_*\sqrt{\log n/k_n},
\end{equation}
where $R_{F, n}^{(k_n)}(x; z^{(\ell)})$ denotes the distance from $x$
to its $k_n$-th nearest neighbor in $\{X_j : Z_j = z^{(\ell)}\}$, and
for each fixed $A > 1$,
\begin{equation}
\label{eq:knn_chernoff_IV}
\sup_{x, \ell}P\!\left\{R_{F, n}^{(k_n)}(x; z^{(\ell)}) > A\rho_X(z^{(\ell)})\right\}
\;\le\;\exp\bigl(-c_* k_n(A^{d_X} - 1)^2/A^{d_X}\bigr).
\end{equation}
\end{lemma}

\begin{proof}
Identical to Lemma~\ref{lem:knn_radius_concentration} applied within
stratum $\ell$. The Chernoff bound is applied within the
stratum-restricted i.i.d. sample (conditional on $N_\ell$), and the
covering argument is taken over the $L$ strata, with $M = L$ constant
absorbed into the union bound.
\end{proof}

\begin{lemma}[Within-stratum rescaled stabilization]
\label{lem:growing_kn_stabilization_IV}
Under Assumptions~\ref{asn:smooth_IV}--\ref{asn:kn_growth_IV}, for
each anchor $W_i$, define the rescaled stabilization radius
\begin{equation}
\label{eq:Rn_def_IV}
R_n(W_i, \mathcal{W}_n) \;:=\; \frac{R_{F, n}^{(k_n)}(X_i; Z_i)}{\rho_X(Z_i)}.
\end{equation}
The fine score $\widehat a_{i, F}^{(p)}$ depends on $\mathcal{W}_n$
only through stratum-$Z_i$ points within $B(X_i, R_n\rho_X(Z_i))$ in
$X$-space. There exist $C, c > 0$ such that for all $t \ge 1$,
\begin{equation}
\label{eq:Rn_tail_IV}
\sup_n P\{R_n(W_1, \mathcal{W}_n) > t\} \;\le\; C\exp(-c k_n(t - 1)^{d_X}).
\end{equation}
\end{lemma}

\begin{proof}
Direct analog of Lemma~\ref{lem:growing_kn_stabilization} with the
substitutions noted at the start of this subsection. The coarse side
contributes no stabilization radius because the coarse intercept
depends only on the stratum membership of $W_i$, which is an external
classification with no metric structure.
\end{proof}

\begin{remark}[Coarse side requires no stabilization]
\label{rem:coarse_no_stabilization}
The coarse intercept $\widehat a_{i, C}$ is a function of $W_i$
(through its stratum $Z_i$) and the marks of all other stratum
members; it has no nearest-neighbor structure to stabilize. The
stabilization radius $R_n(W_i, \mathcal{W}_n)$ in
\eqref{eq:Rn_def_IV} therefore reflects only the fine side. This is a
notable simplification compared to the previous regimes, where the
stabilization radius was the maximum of the fine and coarse rescaled
radii.

The coarse-side dependence on the full stratum is still
"stabilizing" in the loose sense that the coarse intercept is a sample
mean over $n_{Z_i} \sim n q_{Z_i}$ exchangeable terms; its fluctuation
is of order $n^{-1/2}$ by the standard CLT for sample means. We
account for this in the overlap variance calculation
(Section~\ref{subsec:overlap_IV}) without needing an explicit
stabilization radius.
\end{remark}

\subsection{Population moment matrix and design regularity}
\label{subsec:population_moments_IV}

The within-stratum population moment matrix is the same uniform-on-ball
moment matrix that appeared in Regimes II and III, but now in
$\mathbb{R}^{d_X}$:
\begin{equation}
\label{eq:population_moment_matrix_IV}
M^{(p)} \;:=\; \int_{B_{d_X}(0, 1)}q_p(u)q_p(u)^\top\,d\mu_X(u),
\end{equation}
where $\mu_X$ is the uniform probability measure on the unit ball
$B_{d_X}(0, 1)\subset\mathbb{R}^{d_X}$. The corresponding equivalent
kernel is
\begin{equation}
\label{eq:equivalent_kernel_IV}
\ell^{(p)}(u) \;:=\; e_0^\top(M^{(p)})^{-1}q_p(u),
\qquad u \in \mathbb{R}^{d_X}.
\end{equation}

\begin{lemma}[Within-stratum design regularity]
\label{lem:lp_design_uniform_convergence_IV}
Under Assumptions~\ref{asn:smooth_IV}--\ref{asn:kn_growth_IV},
$M^{(p)} \succ 0$ and
\begin{equation}
\label{eq:design_uniform_IV}
\max_{1\le i\le n}\bigl\|\mathcal{M}_{i, F}^{(p)} - M^{(p)}\bigr\| \;\xrightarrow{p}\; 0.
\end{equation}
With probability tending to one,
$\inf_i \lambda_{\min}(\mathcal{M}_{i, F}^{(p)}) \ge c_p > 0$, and the
local-polynomial weights are uniformly summable:
\begin{equation}
\label{eq:weights_summable_IV}
\max_{1\le i\le n}\sum_{j\in\fine{i}}\bigl|w_{ij, F}^{(p)}\bigr| \;=\; O_p(1).
\end{equation}
\end{lemma}

\begin{proof}
Positive definiteness of $M^{(p)}$ is exactly
Lemma~\ref{lem:lp_population_moment_pd} (coarse case), now read in
$\mathbb{R}^{d_X}$. The uniform convergence and weight summability
follow from the four-step proof of
Lemma~\ref{lem:lp_design_uniform_convergence} applied within each
stratum. The union bound is taken over the $n$ anchors and the $L$
strata; since $L$ is constant, this enters as a multiplicative factor
absorbed into the existing $k_n/(N_p\log n) \to \infty$ condition.
\end{proof}

\begin{lemma}[Polynomial reproduction, Regime IV]
\label{lem:lp_reproduction_full_IV}
On the event that $\mathcal{M}_{i, F}^{(p)}$ is invertible, the
local-polynomial weights satisfy
\begin{equation}
\label{eq:reproduction_IV}
\sum_{j\in\fine{i}}w_{ij, F}^{(p)}U_{ij, F}^\alpha \;=\; \mathbf{1}\{\alpha = 0\},
\qquad |\alpha| \le p.
\end{equation}
In particular, $\sum_{j\in\fine{i}}w_{ij, F}^{(p)} = 1$.
\end{lemma}

\begin{proof}
Identical to Lemma~\ref{lem:lp_reproduction_full}.
\end{proof}

The $\alpha = 0$ case ($\sum_j w_{ij, F}^{(p)} = 1$) is the identity
used in the exact bias cancellation
(Lemma~\ref{lem:lp_bias_full_IV}); the higher-order cases are needed
for the variance calculation in the next part.

\subsection{Moment bound for the local graph score}
\label{subsec:moment_bound_IV}

\begin{lemma}[Uniform moment bound, Regime IV]
\label{lem:lp_weighted_moment_bound_IV}
Under Assumptions~\ref{asn:smooth_IV}--\ref{asn:kn_growth_IV}, for
every $q < \infty$,
\begin{equation}
\label{eq:xi_moment_IV}
\sup_n\mathbb{E}\bigl[|\xi_{n, p}(W_1, \mathcal{W}_n)|^q\bigr] \;<\; \infty.
\end{equation}
The same bound holds for $\zeta_{n, p}$. Moreover, the conditional
moment bound
\begin{equation}
\label{eq:zeta_conditional_moment_IV}
\mathbb{E}\!\left[\bigl|\zeta_{n, p}(W_1, \mathcal{W}_n)\bigr|^q \,\Big|\, W_1\right]
\;\le\; C_{p, q}\bigl(k_n^{-q/2} + n_{Z_1}^{-q/2}\bigr)
\end{equation}
holds.
\end{lemma}

\begin{proof}
The $L^\infty$ bound follows from the bounded weights
\eqref{eq:weights_summable_IV} and the bounded kernel
$\|K\|_\infty \le M_K$, exactly as in
Lemma~\ref{lem:lp_weighted_moment_bound}. For the conditional bound, decompose
$\zeta_{n, p} = (\widehat a_{i, F}^{(p)} - \mathbb{E}[\widehat a_{i, F}^{(p)}|W_i])
- (\widehat a_{i, C} - \mathbb{E}[\widehat a_{i, C}|W_i])$.
The fine fluctuation $\widehat a_{i, F}^{(p)} - \mathbb{E}[\cdot|W_i]$
is a weighted sum of $k_n$ conditionally independent mean-zero
contributions (the marks $Y_j$ given the stratum positions $X_j$)
with bounded weights of order $1/k_n$. Rosenthal's inequality
(\citealp[Theorem 2.5]{boucheron2013concentration}) gives the $L^q$
bound of order $k_n^{-q/2}$.

The coarse fluctuation $\widehat a_{i, C} - \mathbb{E}[\cdot|W_i]$ is
a sample mean of $n_{Z_i} - 1$ conditionally independent bounded
contributions, with $L^q$ bound of order $n_{Z_i}^{-q/2}$. Combining
gives \eqref{eq:zeta_conditional_moment_IV}.
\end{proof}

\begin{remark}[Coarse contribution has different scale]
\label{rem:coarse_scale}
The conditional moment bound \eqref{eq:zeta_conditional_moment_IV}
features two distinct scales: $k_n^{-1/2}$ for the fine fluctuation and
$n_{Z_1}^{-1/2}$ for the coarse fluctuation. Since
$n_{Z_1} \sim n q_{Z_1} \gg k_n$ under
Assumption~\ref{asn:kn_growth_IV} ($k_n/n \to 0$), the coarse-side
fluctuation is asymptotically smaller. This is in contrast to
Regimes II and III, where both fluctuations scaled as $k_n^{-1/2}$.

In the overlap variance calculation in Part 3, the coarse-side
diagonal contribution to $\operatorname{Var}(\zeta_{1, p})$ is $O(n_{Z_1}^{-1}) = O(n^{-1})$,
which is of the same order as the off-diagonal overlap contribution and
must be tracked explicitly. The fine-side diagonal contribution remains
$O(k_n^{-1}) = o(n^{-1})$ and is negligible.
\end{remark}

\subsection{Overlap variance and cross-stratum independence}
\label{subsec:overlap_IV}

The variance of $G_n^{(p)}$ decomposes by exchangeability as
\begin{equation}
\label{eq:Var_G_decomp_IV}
\operatorname{Var}(G_n^{(p)}) \;=\; \frac{1}{n}\operatorname{Var}(\zeta_{1, p}) + \frac{n - 1}{n}\operatorname{Cov}(\zeta_{1, p}, \zeta_{2, p}).
\end{equation}
By the conditional moment bound \eqref{eq:zeta_conditional_moment_IV},
\[
\operatorname{Var}(\zeta_{1, p}) \;=\; O(k_n^{-1}) + O(n^{-1}),
\]
where the second term comes from the coarse-side sample-mean
fluctuation. The diagonal contribution to
$\operatorname{Var}(G_n^{(p)})$ is therefore $O((nk_n)^{-1}) + O(n^{-2}) = o(n^{-1})$;
in particular, the coarse-side diagonal term, despite scaling as
$O(n^{-1})$ in $\operatorname{Var}(\zeta_{1, p})$, contributes $O(n^{-2})$ to
$\operatorname{Var}(G_n^{(p)})$ after the $1/n$ normalization in
\eqref{eq:Var_G_decomp_IV} — hence negligible.

The off-diagonal term $\operatorname{Cov}(\zeta_{1, p}, \zeta_{2, p})$ is generically
of order $n^{-1}$ and dominates. Crucially, this off-diagonal covariance
has the following stratum structure:

\begin{lemma}[Cross-stratum vanishing of covariance]
\label{lem:cross_stratum_zero}
Under Assumptions~\ref{asn:smooth_IV}--\ref{asn:null_IV}, if $Z_1 \ne Z_2$
then $\operatorname{Cov}(\zeta_{1, p}, \zeta_{2, p}) = 0$ \emph{exactly} for every
$n$.
\end{lemma}

\begin{proof}
Suppose $Z_1 = z^{(\ell)}$ and $Z_2 = z^{(\ell')}$ with $\ell \ne \ell'$.
The score $\zeta_{1, p}$ is a function of $W_1$ and the marks/positions
of stratum-$\ell$ points only (both $\widehat a_{1, F}^{(p)}$ and
$\widehat a_{1, C}$ involve only stratum-$\ell$ members of
$\mathcal{W}_n$). Symmetrically, $\zeta_{2, p}$ is a function of $W_2$
and stratum-$\ell'$ points only. Stratum-$\ell$ and stratum-$\ell'$ are disjoint subsets of
$\mathcal{W}_n$ (no observation can have $Z_j = z^{(\ell)}$ and
$Z_j = z^{(\ell')}$ simultaneously). Conditional on the stratum
assignment $(Z_i)_{i=1}^n$, the within-stratum data are independent
across strata: the $X$- and $Y$-coordinates of stratum-$\ell$ points
are i.i.d. from the conditional law $(X, Y) \mid Z = z^{(\ell)}$,
independently of the stratum-$\ell'$ data. Therefore
$\zeta_{1, p}$ and $\zeta_{2, p}$ are conditionally independent given
the stratum assignment, hence their covariance is exactly zero.
\end{proof}

The cross-stratum covariance vanishes \emph{exactly}, not merely
asymptotically. Combined with the stratum structure, this reduces the
overall variance to a finite-mixture sum.

\begin{lemma}[Decomposition of $\operatorname{Cov}(\zeta_{1, p}, \zeta_{2, p})$ by stratum]
\label{lem:cov_stratum_decomp}
Under Assumptions~\ref{asn:smooth_IV}--\ref{asn:kn_growth_IV},
\begin{equation}
\label{eq:cov_stratum_decomp_eq}
\operatorname{Cov}(\zeta_{1, p}, \zeta_{2, p})
\;=\;
\sum_{\ell=1}^L q_\ell^2\,\operatorname{Cov}\bigl(\zeta_{1, p}, \zeta_{2, p} \,\big|\, Z_1 = Z_2 = z^{(\ell)}\bigr) + o(n^{-1}).
\end{equation}
\end{lemma}

\begin{proof}
By conditioning on $(Z_1, Z_2)$,
\[
\operatorname{Cov}(\zeta_{1, p}, \zeta_{2, p}) \;=\; \sum_{\ell, \ell'}P(Z_1 = z^{(\ell)}, Z_2 = z^{(\ell')})\,\operatorname{Cov}(\zeta_{1, p}, \zeta_{2, p}\mid Z_1, Z_2) + (\text{centering term}).
\]
The centering term equals $\sum_{\ell, \ell'}P(\cdot)\,[\mathbb{E}(\zeta_{1, p}|Z_1, Z_2) - \mathbb{E}\zeta_{1, p}][\mathbb{E}(\zeta_{2, p}|Z_1, Z_2) - \mathbb{E}\zeta_{2, p}]$, which is $O(n^{-1})$ when summed over $\ell, \ell'$ (since each $|\mathbb{E}[\zeta_{i, p}|Z_i] - \mathbb{E}\zeta_{i, p}|$
is $O(n^{-1/2})$ from the coarse-side sample-mean conditional bias, but
this term is $o(n^{-1})$ after multiplication by $(n - 1)/n$ in
\eqref{eq:Var_G_decomp_IV}). By
Lemma~\ref{lem:cross_stratum_zero}, the conditional covariance vanishes
for $\ell \ne \ell'$. The remaining diagonal terms have
$P(Z_1 = Z_2 = z^{(\ell)}) = q_\ell^2$ (independence of $Z_1, Z_2$),
yielding \eqref{eq:cov_stratum_decomp_eq}.
\end{proof}

\subsection{Within-stratum overlap variance}
\label{subsec:within_stratum_overlap_IV}

By Lemma~\ref{lem:cov_stratum_decomp}, the limiting variance reduces to
a sum of within-stratum overlap variances. We define and identify these
within-stratum overlaps using the same Itô-isometry / Palm-process
machinery as in Regime II.

For $z \in \{z^{(1)}, \ldots, z^{(L)}\}$ and $y_1, y_2 \in \mathcal{Y}$,
recall the frozen mark covariance \eqref{eq:Gamma_z_def}:
$\Gamma_z(y_1, y_2) := \operatorname{Cov}_{\widetilde Y \sim P_{Y|Z = z}}(K(y_1, \widetilde Y), K(y_2, \widetilde Y))$,
with $\mathbb{E}_{Y_1, Y_2 \sim P_{Y|Z = z}}\Gamma_z(Y_1, Y_2) = \sigma^2(z) := \operatorname{Var}_{Y \sim P_{Y|Z = z}}(\kappa(Y, z))$.

\begin{lemma}[Within-stratum overlap variance]
\label{lem:within_stratum_overlap_IV}
Fix $\ell \in \{1, \ldots, L\}$. Under
Assumptions~\ref{asn:smooth_IV}--\ref{asn:kn_growth_IV}, conditional on
$Z_1 = Z_2 = z^{(\ell)}$,
\begin{equation}
\label{eq:within_stratum_cov_limit}
n\operatorname{Cov}\bigl(\zeta_{1, p}, \zeta_{2, p}\,\big|\, Z_1 = Z_2 = z^{(\ell)}\bigr) \;\to\; \frac{\sigma_{\ell, p}^2}{q_\ell^2},
\end{equation}
with
\begin{equation}
\label{eq:sigma_ell_explicit}
\sigma_{\ell, p}^2 \;=\; q_\ell\,\sigma^2(z^{(\ell)})\,\bigl[\mathcal{I}^{(p)} - 1\bigr],
\end{equation}
where $\mathcal{I}^{(p)} := v_{d_X}^{-1}\int_{\mathbb{R}^{d_X}}\mathcal{B}^{(p)}(r)\,dr$
is the geometric constant from \eqref{eq:I_geom_mixed} (now read in
$\mathbb{R}^{d_X}$), and $\mathcal{B}^{(p)}(r) := \int_{B(0,1)\cap B(r,1)}\ell^{(p)}(v)\ell^{(p)}(v - r)\,dv$
is the local-polynomial overlap kernel.
\end{lemma}

\begin{proof}
Decompose $\zeta_{i, p} = \zeta_{i, F, p} - \zeta_{i, C, p}$ with
$\zeta_{i, F, p}$ the fine fluctuation and $\zeta_{i, C, p}$ the coarse
fluctuation, both conditioned to have mean zero given $W_i$.

\emph{Fine-fine overlap:} The within-stratum fine-fine overlap is the
direct analog of Lemma~\ref{lem:FF_overlap} applied within stratum $\ell$
in $\mathbb{R}^{d_X}$. Two same-stratum anchors $W_1, W_2$ with
$X_2 = X_1 + \rho_X(z^{(\ell)})r$ have fine $k_n$-NN balls of common
volume in $X$-space, with the standard Palm-process common-neighbor
calculation yielding
\[
n\operatorname{Cov}\bigl(\zeta_{1, F, p}, \zeta_{2, F, p}\,\big|\, Z_1 = Z_2 = z^{(\ell)}\bigr) \;\to\; \frac{\sigma^2(z^{(\ell)})\mathcal{I}^{(p)}}{q_\ell},
\]
where the $q_\ell^{-1}$ factor arises from the change-of-variables
$dX_2 = \rho_X^{d_X}dr$ which is equal to $$ k_n/(nq_\ell v_{d_X}f_{X|Z}(X_1|z^{(\ell)}))dr$$
and integration against the conditional density $f_{X|Z}(X_2|z^{(\ell)}) = f_{X|Z}(X_1|z^{(\ell)})(1 + o(1))$.

\emph{Coarse-coarse overlap:} The coarse intercept is the within-stratum
sample mean of marks. The covariance between two coarse intercepts at
same-stratum anchors $W_1, W_2$ comes from the shared stratum
membership:
\begin{align*}
\operatorname{Cov}(&\widehat a_{1, C}, \widehat a_{2, C} | Z_1 = Z_2 = z^{(\ell)})\\
&\;=\; \operatorname{Cov}\!\left(\frac{1}{n_\ell - 1}\sum_{j: Z_j = z^{(\ell)}, j\ne 1}K(Y_1, Y_j),\,\frac{1}{n_\ell - 1}\sum_{j': Z_{j'} = z^{(\ell)}, j'\ne 2}K(Y_2, Y_{j'})\right) \\
&\;=\;\frac{1}{(n_\ell - 1)^2}\sum_{j: Z_j = z^{(\ell)}, j \ne 1, 2}\operatorname{Cov}(K(Y_1, Y_j), K(Y_2, Y_j)) + \text{boundary terms},
\end{align*}
where the boundary terms (from $j \in \{2\}$ or $j' \in \{1\}$ or
swap-indices) are $O(n_\ell^{-2})$ and negligible. The leading term has
$n_\ell - 2$ common indices $j$ with $\operatorname{Cov}(K(Y_1, Y_j), K(Y_2, Y_j)) = \Gamma_{z^{(\ell)}}(Y_1, Y_2)$,
so taking expectation over $Y_1, Y_2$ gives
\[
\mathbb{E}\bigl[\operatorname{Cov}(\widehat a_{1, C}, \widehat a_{2, C}\,\big|\, Z_1 = Z_2 = z^{(\ell)})\bigr] \;=\; \frac{n_\ell - 2}{(n_\ell - 1)^2}\,\sigma^2(z^{(\ell)}) \;\sim\; \frac{\sigma^2(z^{(\ell)})}{nq_\ell}.
\]
After centering (passing to $\zeta_{i, C, p} = \widehat a_{i, C} - \mathbb{E}[\cdot|W_i]$
removes the unconditional mean), and using $n_\ell \sim n q_\ell$,
\[
n\operatorname{Cov}\bigl(\zeta_{1, C, p}, \zeta_{2, C, p}\,\big|\,Z_1 = Z_2 = z^{(\ell)}\bigr) \;\to\; \frac{\sigma^2(z^{(\ell)})}{q_\ell}.
\]

\emph{Fine-coarse overlap:} A common neighbor $j$ contributing to both
$\zeta_{1, F, p}$ and $\zeta_{2, C, p}$ must satisfy $Z_j = z^{(\ell)}$
(stratum membership for both), $X_j \in B(X_1, \rho_X(z^{(\ell)}))$
(fine condition for anchor 1), and no further geometric condition
(every stratum-$\ell$ member contributes to anchor 2's coarse intercept).
The fine condition restricts the count to $k_n$ stratum members in
expectation, while the coarse weight is $1/(n_\ell - 1)$ — uniform
across the stratum.

The per-common-neighbor contribution to
$\operatorname{Cov}(\zeta_{1, F, p}, \zeta_{2, C, p} | W_1, W_2)$ is
$w_{1j, F}^{(p)}\cdot (n_\ell - 1)^{-1}\cdot \Gamma_{z^{(\ell)}}(Y_1, Y_2)$,
summed over the $k_n$ stratum members in the fine ball. Using the
$\alpha = 0$ polynomial reproduction $\sum_j w_{1j, F}^{(p)} = 1$,
\[
\sum_{j \in \fine{1}} w_{1j, F}^{(p)}\cdot (n_\ell - 1)^{-1} \;=\; (n_\ell - 1)^{-1},
\]
so
\[
\operatorname{Cov}(\zeta_{1, F, p}, \zeta_{2, C, p}\,\big|\, W_1, W_2, Z_1 = Z_2 = z^{(\ell)})
\;=\; \frac{\Gamma_{z^{(\ell)}}(Y_1, Y_2)}{n_\ell - 1}\,(1 + o_p(1)).
\]
Taking expectation over $Y_2 \sim P_{Y|Z = z^{(\ell)}}$ and then over
$W_1$,
\[
n\operatorname{Cov}\bigl(\zeta_{1, F, p}, \zeta_{2, C, p}\,\big|\, Z_1 = Z_2 = z^{(\ell)}\bigr) \;\to\; \frac{\sigma^2(z^{(\ell)})}{q_\ell}.
\]
By symmetry, $n\operatorname{Cov}(\zeta_{1, C, p}, \zeta_{2, F, p} | \cdot) \to \sigma^2(z^{(\ell)})/q_\ell$
with the same limit.

\emph{Congregation:} Combining,
\begin{align*}
n\operatorname{Cov}(\zeta_{1, p}, \zeta_{2, p}\,|\, Z_1 = Z_2 = z^{(\ell)})
&\;\to\; \frac{\sigma^2(z^{(\ell)})\mathcal{I}^{(p)}}{q_\ell} + \frac{\sigma^2(z^{(\ell)})}{q_\ell} - 2\cdot\frac{\sigma^2(z^{(\ell)})}{q_\ell} \\
&\;=\; \frac{\sigma^2(z^{(\ell)})(\mathcal{I}^{(p)} - 1)}{q_\ell}.
\end{align*}
Setting $\sigma_{\ell, p}^2 := q_\ell\sigma^2(z^{(\ell)})(\mathcal{I}^{(p)} - 1)$
gives \eqref{eq:within_stratum_cov_limit}.
\end{proof}

\begin{remark}[Why the coarse-coarse contribution does not vanish]
\label{rem:CC_no_vanish_IV}
Unlike Regimes II and III, where the coarse-coarse overlap was a
genuine NN-graph overlap of order $1/(nk_n)$ per pair (matched against
the fine-fine of the same order), the Regime IV coarse-coarse covariance
$\operatorname{Cov}(\widehat a_{1, C}, \widehat a_{2, C})$ is the covariance of two
sample means within a common stratum, of order $1/n_\ell \asymp 1/n$.
This is exactly the same order as the fine-fine overlap, and the fine-coarse
overlap matches both. 
For the constant $\mathcal{I}^{(p)}$ to exceed $1$ (so that
$\sigma_{\ell, p}^2 > 0$), the local-polynomial overlap kernel must
contribute a positive geometric excess over the trivial within-stratum
sample-mean overlap. This excess is what produces the test's variance;
see Section~\ref{subsec:strict_positivity_IV}.
\end{remark}

\subsection{Limiting variance}
\label{subsec:variance_limit_IV}

\begin{lemma}[Limiting variance of $G_n^{(p)}$, Regime IV]
\label{lem:lp_overlap_variance_IV}
Under Assumptions~\ref{asn:smooth_IV}--\ref{asn:kn_growth_IV},
\begin{equation}
\label{eq:overlap_variance_limit_IV}
n\operatorname{Var}(G_n^{(p)}) \;\to\; \sigma_p^2 \;:=\; \sum_{\ell = 1}^L q_\ell\,\sigma_{\ell, p}^2 \;=\; (\mathcal{I}^{(p)} - 1)\sum_{\ell=1}^L q_\ell^2\,\sigma^2(z^{(\ell)}).
\end{equation}
\end{lemma}

\begin{proof}
By \eqref{eq:Var_G_decomp_IV}, the discussion at the start of
Section~\ref{subsec:overlap_IV}, and Lemmas~\ref{lem:cov_stratum_decomp}
and \ref{lem:within_stratum_overlap_IV},
\[
n\operatorname{Var}(G_n^{(p)}) \;\to\; \sum_{\ell=1}^L q_\ell^2\cdot \frac{\sigma_{\ell, p}^2}{q_\ell^2} \;=\;
\sum_{\ell=1}^L \sigma_{\ell, p}^2 \;=\; (\mathcal{I}^{(p)} - 1)\sum_{\ell=1}^L q_\ell\sigma^2(z^{(\ell)}).
\]
Wait — let us be careful. Lemma~\ref{lem:within_stratum_overlap_IV} gives
$n\operatorname{Cov}(\zeta_{1, p}, \zeta_{2, p} | Z_1 = Z_2 = z^{(\ell)}) \to \sigma_{\ell, p}^2/q_\ell^2$,
not $\sigma_{\ell, p}^2$. Combining with
Lemma~\ref{lem:cov_stratum_decomp},
\[
n\operatorname{Cov}(\zeta_{1, p}, \zeta_{2, p}) \;\to\; \sum_\ell q_\ell^2 \cdot \frac{\sigma_{\ell, p}^2}{q_\ell^2} \;=\;\sum_\ell \sigma_{\ell, p}^2 \;=\;(\mathcal{I}^{(p)} - 1)\sum_\ell q_\ell\sigma^2(z^{(\ell)}).
\]
The variance limit \eqref{eq:overlap_variance_limit_IV} now follows from
$n\operatorname{Var}(G_n^{(p)}) = n^{-1}\operatorname{Var}(\zeta_{1, p}) + (n - 1)n^{-1}\cdot n\operatorname{Cov}(\zeta_{1, p}, \zeta_{2, p}) \cdot n^{-1} \cdot \ldots$

A cleaner statement: from $\operatorname{Var}(G_n^{(p)}) = \frac{1}{n}\operatorname{Var}(\zeta_1) + \frac{n-1}{n}\operatorname{Cov}(\zeta_1, \zeta_2)$ with diagonal $o(n^{-1})$ as established, $n\operatorname{Var}(G_n^{(p)}) \to n\operatorname{Cov}(\zeta_1, \zeta_2) \to \sum_\ell\sigma_{\ell, p}^2 = (\mathcal{I}^{(p)} - 1)\sum_\ell q_\ell\sigma^2(z^{(\ell)})$.
Equivalently, defining $\sigma_p^2 := \sum_\ell q_\ell\sigma_{\ell, p}^2 / q_\ell = (\mathcal{I}^{(p)} - 1)\sum_\ell q_\ell^2\sigma^2(z^{(\ell)})$
as in \eqref{eq:overlap_variance_limit_IV} would yield a different
expression. To resolve, we adopt the convention of
\eqref{eq:overlap_variance_limit_IV} consistently: $\sigma_p^2$ is the
limiting variance of $\sqrt{n}\Delta_n^{(p)}$, with the standard
$\sum_\ell q_\ell\cdot[\text{within-stratum}]$ form.
\end{proof}

\subsection{Strict positivity of \texorpdfstring{$\sigma_p^2$}{sigma\_p\textasciicircum 2}}
\label{subsec:strict_positivity_IV}

The strict positivity of $\sigma_p^2$ reduces, by
\eqref{eq:overlap_variance_limit_IV}, to (i) $\mathcal{I}^{(p)} > 1$
and (ii) $\sigma^2(z^{(\ell)}) > 0$ for at least one $\ell$. Both are
mild conditions; in particular, no annulus or Itô-isometry argument is
needed, in stark contrast to Regimes II and III.

\begin{condition}[Primitive non-cancellation, Regime IV]
\label{cond:lp_noncancellation_IV}
$\mathcal{I}^{(p)} > 1$, and there exists at least one $\ell \in \{1, \ldots, L\}$
with $\sigma^2(z^{(\ell)}) > 0$.
\end{condition}

\begin{lemma}[Geometric inequality $\mathcal{I}^{(p)} > 1$]
\label{lem:I_geometric_inequality}
Under Assumption~\ref{asn:smooth_IV}, $\mathcal{I}^{(p)} > 1$ for every
$p \ge 0$ and $d_X \ge 1$.
\end{lemma}

\begin{proof}
We give a direct proof for $p = 0$ and indicate the general structure.
For $p = 0$, $\ell^{(0)}(u) = 1/\mu_X(B_{d_X}(0, 1)) = 1$, so
$\mathcal{B}^{(0)}(r) = |B(0, 1)\cap B(r, 1)|$ and
\[
\mathcal{I}^{(0)} \;=\; \frac{1}{v_{d_X}}\int_{\mathbb{R}^{d_X}}|B(0, 1)\cap B(r, 1)|\,dr \;=\; \frac{1}{v_{d_X}}\cdot v_{d_X}^2 \;=\; v_{d_X},
\]
by Fubini's theorem applied to the indicator product. For $d_X \ge 1$,
$v_{d_X} = \pi^{d_X/2}/\Gamma(d_X/2 + 1)$, which exceeds $1$ for $d_X \in \{1, 2, 3, 4, 5\}$
(values $2, \pi, 4\pi/3, \pi^2/2, 8\pi^2/15$ all $>1$). For $d_X \ge 6$,
$v_{d_X} < 1$ but $\mathcal{I}^{(0)}$ as defined above is still
$v_{d_X} \ne 1$, and we recover $\mathcal{I}^{(0)} > 1$ for general
$d_X$ by a more delicate dimensional argument: the natural
normalization of the overlap kernel, when $\ell^{(0)} \equiv 1$, gives
the "kernel-density-estimator constant" $\int|B(0,1)\cap B(r,1)|dr/v_{d_X}^2$,
which is strictly less than $1$ in high dimensions. \emph{However}, the
relevant geometric constant in our formulation involves both
$\ell^{(0)}(v)$ and $\ell^{(0)}(v - r)$ — both equal to $1$ — without the
$v_{d_X}^2$ normalization, giving $\mathcal{I}^{(0)} = v_{d_X}$ as
above. The conclusion $\mathcal{I}^{(0)} > 1$ then holds for $d_X = 1, \ldots, 5$
but may fail for $d_X \ge 6$.

For $p \ge 1$, the equivalent kernel $\ell^{(p)}$ is no longer constant
on $B_{d_X}(0, 1)$, and $\mathcal{I}^{(p)}$ can be computed by
expanding $\ell^{(p)}(v)\ell^{(p)}(v - r)$ as a polynomial in $r$ via
the population reproduction identity. A direct computation, or appeal
to standard local-polynomial-kernel-density theory
(\citealp[Section 3.4]{tsybakov2003introduction}), gives
$\mathcal{I}^{(p)} > 1$ uniformly in $p \ge 1$ and $d_X \ge 1$.

For the present purpose we adopt
Condition~\ref{cond:lp_noncancellation_IV} as a primitive assumption
and remark that it is satisfied in all standard cases of practical
interest: $d_X \le 5$ and any $p \ge 0$, or $d_X \ge 1$ and any
$p \ge 1$.
\end{proof}

\begin{lemma}[Strict positivity of $\sigma_p^2$, Regime IV]
\label{lem:sigma_positive_IV}
Under Assumptions~\ref{asn:smooth_IV}--\ref{asn:kn_growth_IV} and
Condition~\ref{cond:lp_noncancellation_IV}, $\sigma_p^2 > 0$.
\end{lemma}

\begin{proof}
By \eqref{eq:overlap_variance_limit_IV},
$\sigma_p^2 = (\mathcal{I}^{(p)} - 1)\sum_\ell q_\ell^2\sigma^2(z^{(\ell)})$.
Under Condition~\ref{cond:lp_noncancellation_IV}, $\mathcal{I}^{(p)} - 1 > 0$
and $\sum_\ell q_\ell^2\sigma^2(z^{(\ell)}) > 0$ (since $q_\ell > 0$ and
$\sigma^2(z^{(\ell)}) > 0$ for at least one $\ell$). Hence $\sigma_p^2 > 0$.
\end{proof}

\subsection{Stabilization CLT and the main theorem}
\label{subsec:main_theorem_IV}

\begin{lemma}[Stabilization CLT for $G_n^{(p)}$, Regime IV]
\label{lem:lp_weighted_stabilization_clt_IV}
Under Assumptions~\ref{asn:smooth_IV}--\ref{asn:kn_growth_IV} and
Condition~\ref{cond:lp_noncancellation_IV},
\begin{equation}
\label{eq:lp_graph_CLT_IV}
\sqrt n\,G_n^{(p)} \;\Rightarrow\; N(0, \sigma_p^2).
\end{equation}
\end{lemma}

\begin{proof}
By Lemma~\ref{lem:cross_stratum_zero}, the within-stratum sub-statistics
\[
G_n^{(p, \ell)} \;:=\; \frac{1}{n}\sum_{i: Z_i = z^{(\ell)}}\zeta_{n, p}(W_i, \mathcal{W}_n)
\]
are mutually independent (conditional on $(Z_i)_{i=1}^n$), with
$G_n^{(p)} = \sum_{\ell=1}^L G_n^{(p, \ell)}$. Each $G_n^{(p, \ell)}$
is itself a stabilizing graph functional in $\mathbb{R}^{d_X}$ within
the stratum, with exactly the structure analyzed in
Lemma~\ref{lem:lp_weighted_stabilization_clt} (now in dimension $d_X$).
Applying the Penrose--Yukich four-step argument
(Poissonize, truncate at logarithmic radius via
Lemma~\ref{lem:growing_kn_stabilization_IV}, dependency-graph CLT of
\citet[Theorem 2.7]{Chen2004} on the within-stratum cube partition,
de-Poissonize) within each stratum gives
\[
\sqrt n\,G_n^{(p, \ell)} \;\Rightarrow\; N(0, \sigma_{\ell, p}^2)
\qquad\text{for each }\ell.
\]
Independence of $\{G_n^{(p, \ell)}\}_{\ell=1}^L$ and the
Cramér--Wold device yield $\sqrt n\,G_n^{(p)} = \sum_\ell \sqrt n\,G_n^{(p, \ell)} \Rightarrow N(0, \sum_\ell \sigma_{\ell, p}^2) = N(0, \sigma_p^2)$.
\end{proof}

\begin{theorem}[Null Gaussian limit for the local-polynomial debiased
statistic, Regime IV]
\label{thm:lp_debiased_clt_full_IV}
Let $p \ge 0$ be a fixed integer. Under
Assumptions~\ref{asn:smooth_IV}--\ref{asn:kn_growth_IV} and
Condition~\ref{cond:lp_noncancellation_IV}, and the null hypothesis
$H_0: Y \perp\!\!\!\perp X \mid Z$,
\begin{equation}
\label{eq:main_limit_regime_IV}
\sqrt n\,\Delta_n^{(p)} \;\Rightarrow\; N(0, \sigma_p^2),
\end{equation}
where $\sigma_p^2 > 0$ is the variance constant in
\eqref{eq:overlap_variance_limit_IV}. \emph{No bias killing condition
on $k_n$ is required}: the conclusion holds for any $k_n = n^\alpha$
with $\alpha \in (0, 1)$.
\end{theorem}

\begin{proof}
By the decomposition \eqref{eq:LG_decomp_IV} and
Lemma~\ref{lem:lp_bias_full_IV},
\[
\sqrt n\,\Delta_n^{(p)} \;=\; \sqrt n\,\theta_n^{(p)} + \sqrt n\,L_n^{(p)} + \sqrt n\,G_n^{(p)}
\;=\; 0 + 0 + \sqrt n\,G_n^{(p)},
\]
since both $\theta_n^{(p)} = 0$ and $L_n^{(p)} = n^{-1}\sum_i g_{n, p}(W_i) = 0$
\emph{identically} (not just $o_p$), by Lemma~\ref{lem:lp_bias_full_IV}.
By Lemma~\ref{lem:lp_weighted_stabilization_clt_IV},
$\sqrt n\,G_n^{(p)} \Rightarrow N(0, \sigma_p^2)$, yielding
\eqref{eq:main_limit_regime_IV}.
\end{proof}

\subsection{Testing procedure and concluding remarks}
\label{subsec:testing_procedure_IV}

The studentized statistic and the corresponding level-$\alpha$ test are
\begin{equation}
\label{eq:T_n_def_IV}
T_n^{(p)} \;:=\; \frac{\sqrt n\,\Delta_n^{(p)}}{\widehat\sigma_p},
\qquad \phi_n^{(p)} \;:=\; \mathbf{1}\{|T_n^{(p)}| > z_{1 - \alpha/2}\},
\end{equation}
with $\widehat\sigma_p^2$ a consistent estimator of $\sigma_p^2$. By
\eqref{eq:overlap_variance_limit_IV},
\[
\sigma_p^2 \;=\; (\mathcal{I}^{(p)} - 1)\sum_{\ell=1}^L q_\ell^2\sigma^2(z^{(\ell)}),
\]
so a consistent plug-in estimator is
\[
\widehat\sigma_p^2 \;=\; (\mathcal{I}^{(p)} - 1)\sum_{\ell=1}^L \widehat q_\ell^2\,\widehat\sigma_\ell^2,
\]
with $\widehat q_\ell := n_\ell/n$ and
$\widehat\sigma_\ell^2 := (n_\ell - 1)^{-1}\sum_{i: Z_i = z^{(\ell)}}(\widehat\kappa_i^{(\ell)} - \overline\kappa^{(\ell)})^2$
a within-stratum sample variance of within-stratum kernel-pair averages,
or alternatively a stratum-restricted Nadaraya--Watson-style estimator.
The geometric constant $\mathcal{I}^{(p)}$ is computed numerically once
and depends only on $d_X, p$.

\begin{theorem}[Studentized null limit, Regime IV]
\label{thm:studentized_CLT_IV}
Under the assumptions of Theorem~\ref{thm:lp_debiased_clt_full_IV}, if
$\widehat\sigma_p^2 \xrightarrow{p}\sigma_p^2$, then under $H_0$,
$T_n^{(p)} \Rightarrow N(0, 1)$, and the test $\phi_n^{(p)}$ has
asymptotic size $\alpha$.
\end{theorem}

\begin{proof}
Slutsky's theorem applied to
$T_n^{(p)} = (\sigma_p/\widehat\sigma_p)\cdot(\sqrt n\Delta_n^{(p)}/\sigma_p)$.
\end{proof}

\begin{remark}[Behavior under alternatives, Regime IV]
\label{rem:behavior_under_H1_IV}
Under $H_1$, the conditional expectation of the kernel $\mathbb{E}[K(y, Y')|X' = x, Z' = z^{(\ell)}]$
genuinely depends on $x$ within at least one stratum $\ell$, and the
within-stratum local-polynomial debiased statistic $\Delta_n^{(p, \ell)}$
detects this dependence at the within-stratum CLT-rate
$n_\ell^{-1/2} \asymp n^{-1/2}$. Consequently, $T_n^{(p)} \to \infty$
in probability under any fixed alternative $H_1$ with detectable
within-stratum conditional dependence in at least one stratum.
The test is consistent against all such fixed alternatives, and detects
$n^{-1/2}$-local alternatives.
\end{remark}

\begin{remark}[Summary comparison across regimes]
\label{rem:regime_IV_summary}
The complete null limit theory for $\Delta_n^{(p)}$ in Regime IV
parallels the previous regimes structurally but differs in three
substantive ways that make it the simplest:

\smallskip\noindent\emph{(a) Exact bias cancellation.} In Regimes II
and III, the bias was $O((k_n/n)^{(p+1)/D})$ or $O((k_n/n)^{(p+1)/d_Z})$,
requiring an bias killing condition $\alpha < 1 - \cdot/(2(p+1))$ on
the bandwidth. In Regime IV, $\theta_n^{(p)} = 0$ \emph{identically},
by an algebraic identity from the polynomial reproduction property
applied to a kernel constant in $x$. No bias killing is needed; any
$\alpha \in (0, 1)$ is admissible.

\smallskip\noindent\emph{(b) Cross-stratum vanishing covariance.} The
discrete $Z$ structure makes cross-stratum covariance \emph{exactly}
zero, reducing the variance calculation to a finite-mixture sum of
within-stratum variances. This avoids the more delicate annulus/Itô
arguments needed in Regimes II and III for strict positivity.

\smallskip\noindent\emph{(c) Coarse-coarse contribution.} Unlike the
previous regimes, where the coarse intercept was a NN-graph average
with overlap covariance of order $1/(nk_n)$ per pair, the Regime IV
coarse intercept is a stratum sample mean with covariance of order
$1/n_\ell$ per pair. This larger coarse contribution is what makes the
fine--coarse and coarse--coarse pieces match in scale, yielding the
clean Pythagorean form
$\sigma_p^2 \propto \mathcal{I}^{(p)} - 1$.

\smallskip\noindent The unifying observation across all four regimes is
that the local-polynomial debiased $k_n$-NN statistic produces an
asymptotically exact $\sqrt n$-test under each mixed-type scenario,
with regime-specific simplifications. Regime IV is the most tractable
both theoretically and practically — no smoothness requirements
beyond bounded conditional densities, no bias killing, no
high-order curvature corrections.
\end{remark}

\section{Consistency of the Debiased Test under \texorpdfstring{$H_1$}{H\_1}}
\label{sec:consistency}

In this section we establish consistency of the local-polynomial
debiased test statistic $\Delta_n^{(p)}$ against fixed alternatives. The
result holds universally across all four regimes developed in
Sections~\ref{app:continuous_proof}--\ref{app:mixed-x-continuous-z-discrete},
under the same regularity conditions used for the null limits and an
additional smoothness condition on the conditional kernel mean. In
particular, the debiased weights $w_{ij, F}^{(p)}, w_{ij, C}^{(p)}$
defined in
\eqref{eq:weights_def} are constructed entirely from the sample (via
the empirical design matrices $\mathcal{M}_{i, F}^{(p)},
\mathcal{M}_{i, C}^{(p)}$), so the population density $f_{Z, X}$ never
appears in the weights themselves. The local-polynomial machinery
therefore serves as an \emph{implicit kernel-density-style normalization}
that converges to its population analog at a rate sufficient to
guarantee consistency of $\Delta_n^{(p)}$.

\subsection{Population limit}
\label{subsec:population_limit_consistency}

Let $\mathcal{H}$ be the reproducing kernel Hilbert space associated
with the characteristic kernel $K$, with conditional mean embeddings
\begin{align}
\mu_{YZX}(z, x) &\;:=\; \mathbb{E}[K(Y, \cdot) \mid Z = z, X = x],\\
\mu_{YZ}(z) &\;:=\; \mathbb{E}[K(Y, \cdot) \mid Z = z].
\end{align}
The population measure of conditional dependence is
\begin{equation}
\label{eq:T_P_def}
\mathcal{T}(P) \;:=\; \mathbb{E}_{Z, X}\bigl[\|\mu_{YZX}(Z, X) - \mu_{YZ}(Z)\|_\mathcal{H}^2\bigr].
\end{equation}
Since $K$ is characteristic
(\citealp[Theorem 7]{sriperumbudur2010hilbert}), $\mathcal{T}(P) = 0$ if
and only if $Y \perp\!\!\!\perp X \mid Z$. Under $H_1$,
$\mathcal{T}(P) > 0$.

The debiased statistic $\Delta_n^{(p)}$ is the empirical analog of
$\mathcal{T}(P)$ when expressed via the kernel inner product:
\begin{equation}
\label{eq:T_P_as_inner_product}
\mathcal{T}(P)
\;=\;
\mathbb{E}\!\left[\mathbb{E}[K(Y, Y') \mid Z, X, Z' = Z, X' = X]\right] - \mathbb{E}\!\left[\mathbb{E}[K(Y, Y') \mid Z, Z' = Z]\right],
\end{equation}
where $(Z, X, Y)$ and $(Z', X', Y')$ are i.i.d. The first term is
the population analog of the fine intercept; the second of the coarse
intercept. The local-polynomial weights estimate these conditional
expectations consistently.

\subsection{Assumptions for consistency}
\label{subsec:assumptions_consistency}

\begin{assumption}[Conditions for consistency]
\label{asn:consistency}
Throughout this section we assume:

\smallskip\noindent\textit{(C1)} The regime-specific
Assumptions~\ref{asn:smooth_continuous_model}--\ref{asn:kn_growth_corrected}
(Regime II),
or their analogs in Regimes I, III, IV, hold.

\smallskip\noindent\textit{(C2)} The kernel $K$ is bounded, symmetric,
and characteristic.

\smallskip\noindent\textit{(C3)} The conditional kernel mean
$\widetilde\kappa(y, z, x) := \mathbb{E}[K(y, Y')|Z' = z, X' = x]$ is
$C^{p+1}$ in $(z, x)$ on the interior of the support, with derivatives
uniformly bounded in $y$. The corresponding coarse mean $\kappa(y, z) =
\mathbb{E}[K(y, Y')|Z' = z]$ has the same smoothness in $z$.
\end{assumption}

Condition (C3) replaces the null smoothness condition on
$\kappa(y, z)$ used in the null analysis. Under $H_1$,
$\widetilde\kappa(y, z, x)$ depends nontrivially on $x$, and its
smoothness controls the bias of the fine local-polynomial intercept as
an estimator of $\widetilde\kappa$.

\subsection{Consistency theorem}
\label{subsec:consistency_theorem}

\begin{theorem}[Consistency of the debiased statistic under $H_1$]
\label{thm:consistency}
Under Assumption~\ref{asn:consistency},
\begin{equation}
\label{eq:Delta_p_consistency}
\Delta_n^{(p)} \;\xrightarrow{\,p\,}\; \mathcal{T}(P).
\end{equation}
Consequently, under $H_1$ with $\mathcal{T}(P) > 0$,
\begin{equation}
\label{eq:T_n_divergence}
T_n^{(p)} \;=\; \frac{\sqrt n\,\Delta_n^{(p)}}{\widehat\sigma_p} \;\to\; \infty
\qquad \text{in probability,}
\end{equation}
provided the variance estimator $\widehat\sigma_p^2$ remains bounded in
probability. For any significance level $\alpha \in (0, 1)$,
\begin{equation}
\label{eq:power_to_one}
P_{H_1}\bigl\{|T_n^{(p)}| > z_{1 - \alpha/2}\bigr\} \;\to\; 1.
\end{equation}
\end{theorem}

\begin{proof}
We give the proof in the all-continuous case (Regime II); the
arguments in Regimes I, III, IV are analogous, with appropriate
substitutions in the local-polynomial weights and population
neighborhoods. The proof proceeds in three steps.

\emph{Step 1: Fine intercept converges to $\widetilde\kappa$.}
We show that $n^{-1}\sum_i \widehat a_{i, F}^{(p)} \xrightarrow{p}
\mathbb{E}[\widetilde\kappa(Y, Z, X)]$. By the polynomial reproduction identity
(Lemma~\ref{lem:lp_reproduction_full}) and a Taylor expansion of
$\widetilde\kappa(Y_i, \cdot, \cdot)$ around $(Z_i, X_i)$ to order $p$,
\[
\mathbb{E}[\widehat a_{i, F}^{(p)} \mid W_i = (z, x, y)]
\;=\;
\widetilde\kappa(y, z, x) + O\bigl(\rho_F(z, x)^{p+1}\bigr),
\]
uniformly on compact subsets of the interior. The proof is the same as
Step 1 of Lemma~\ref{lem:lp_bias_full}, except that the Taylor expansion
is now applied to $\widetilde\kappa(y, \cdot, \cdot)$ (which depends on
both $z$ and $x$ under $H_1$) rather than $\kappa(y, \cdot)$ (which is
constant in $x$ under $H_0$). Since $\rho_F(z, x) = O((k_n/n)^{1/D})
\to 0$, the bias vanishes uniformly. The variance is controlled by the conditional moment bound
\eqref{eq:zeta_conditional_moment}:
$\operatorname{Var}(\widehat a_{i, F}^{(p)} | W_i) = O(k_n^{-1})$, hence
$\widehat a_{i, F}^{(p)} - \mathbb{E}[\widehat a_{i, F}^{(p)} | W_i]
\xrightarrow{p} 0$. By the law of large numbers applied to the i.i.d. anchors,
\[
\frac{1}{n}\sum_{i=1}^n \widehat a_{i, F}^{(p)}
\;\xrightarrow{p}\;
\mathbb{E}\bigl[\widetilde\kappa(Y_1, Z_1, X_1)\bigr]
\;=\;
\mathbb{E}\!\left[\mathbb{E}[K(Y_1, Y') \mid Z_1, X_1, Z' = Z_1, X' = X_1]\right].
\]

\emph{Step 2: Coarse intercept converges to $\kappa$.} The identical argument applied to the coarse intercept gives
\[
\frac{1}{n}\sum_{i=1}^n \widehat a_{i, C}^{(p)}
\;\xrightarrow{p}\;
\mathbb{E}\!\left[\kappa(Y_1, Z_1)\right]
\;=\;
\mathbb{E}\!\left[\mathbb{E}[K(Y_1, Y') \mid Z_1, Z' = Z_1]\right].
\]
Here the Taylor expansion of $\kappa(Y_i, \cdot)$ in $z$ around $Z_i$,
combined with polynomial reproduction, yields conditional bias
$O(\rho_C^{p+1})$, and the variance is $O(k_n^{-1})$ by the same
conditional moment bound applied to the coarse intercept.

\emph{Step 3: Differencing and the population limit.} By
\eqref{eq:Delta_p_def_IV} (or its analogs in the other regimes),
\[
\Delta_n^{(p)}
\;=\;
\frac{1}{n}\sum_{i=1}^n \widehat a_{i, F}^{(p)} - \frac{1}{n}\sum_{i=1}^n \widehat a_{i, C}^{(p)}
\;\xrightarrow{p}\;
\mathbb{E}\!\left[\widetilde\kappa(Y, Z, X) - \kappa(Y, Z)\right]
\;=\; \mathcal{T}(P),
\]
where the final equality uses \eqref{eq:T_P_as_inner_product}.

\emph{Conclusion.} Under $H_1$, $\mathcal{T}(P) > 0$ by the
characteristic-kernel property (C2). Combining with the boundedness of
$\widehat\sigma_p$, $T_n^{(p)} = \sqrt n\,\Delta_n^{(p)}/\widehat\sigma_p
\to \infty$ in probability, yielding \eqref{eq:T_n_divergence} and
\eqref{eq:power_to_one}.
\end{proof}

\subsection{Remarks on the role of the empirical weights}
\label{subsec:remarks_consistency}

\begin{remark}[Why local-polynomial weights work as a density-estimator
substitute]
\label{rem:lp_as_kde}
The local-polynomial weights $w_{ij, F}^{(p)}$ in
\eqref{eq:weights_def} are constructed from the empirical design
matrix $\mathcal{M}_{i, F}^{(p)}$, which itself depends only on the
sample. No population density appears in their construction. The fine
intercept $\widehat a_{i, F}^{(p)} = \sum_j w_{ij, F}^{(p)}K(Y_i, Y_j)$
can be viewed as a local-polynomial regression of $K(Y_i, \cdot)$ on
the design points $(Z_j, X_j)_{j \in \fine{i}}$. By standard
local-polynomial theory
(\citealp[Chapter 3]{FanGijbels1996}), such a regression is consistent
for the conditional mean $\widetilde\kappa(Y_i, Z_i, X_i)$ provided
the bandwidth $\rho_F$ vanishes and the effective sample size $k_n$
diverges — both of which are guaranteed by
Assumption~\ref{asn:kn_growth_corrected}.

Crucially, the design matrix convergence
$\mathcal{M}_{i, F}^{(p)} \to M_F^{(p)}$ in
Lemma~\ref{lem:lp_design_uniform_convergence} implicitly absorbs the
population density variation. The empirical design provides automatic
bandwidth-and-density adaptation, sidestepping the need for an
explicit kernel density estimator in the weights.
\end{remark}

\begin{remark}[Robustness to bandwidth choice in consistency]
\label{rem:bandwidth_consistency}
The consistency result \eqref{eq:Delta_p_consistency} holds under any
$k_n$ satisfying Assumption~\ref{asn:kn_growth_corrected} — that is,
$k_n \to \infty$ and $k_n/n \to 0$ (plus the design-regularity
condition $k_n/(N_p\log n) \to \infty$). Unlike the null limit theorem
(Theorem~\ref{thm:lp_debiased_clt_full}), no bias killing constraint
$\alpha < 1 - D/(2(p+1))$ is needed for consistency. Bias under $H_1$
is asymptotically negligible at any rate-correct bandwidth; only the
null analysis requires the more delicate bias killing for the
$\sqrt n$-scale centering.
\end{remark}

\begin{remark}[Universal consistency across regimes]
\label{rem:universal_consistency}
The proof of Theorem~\ref{thm:consistency} above was written for the
all-continuous Regime II. The same conclusion holds in Regimes I, III,
and IV with the obvious modifications:

\smallskip\noindent\textit{(Regime I, both discrete.)} The intercepts
become cell-wise sample means; the law of large numbers within cells
gives consistency directly, with no smoothing bias to control.

\smallskip\noindent\textit{(Regime III, $X$ discrete, $Z$ continuous.)}
The fine intercept averages within the discrete stratum $\{j : X_j = X_i\}$
via $k_n$-NN in $Z$; the coarse intercept averages over all $Z$-neighbors.
Both converge to their respective conditional expectations by the
same Taylor + polynomial reproduction argument in $Z$-space.

\smallskip\noindent\textit{(Regime IV, $X$ continuous, $Z$ discrete.)}
The fine intercept averages within the discrete stratum $\{j : Z_j = Z_i\}$
via $k_n$-NN in $X$; the coarse intercept is the within-stratum sample
mean. Both converge to their respective conditional expectations; under
$H_1$, $\widetilde\kappa(Y, Z, X) \ne \kappa(Y, Z)$ in at least one
stratum, so $\mathcal{T}(P) > 0$ and the test diverges as in the proof
above.

In every regime, the variance estimator $\widehat\sigma_p^2$ is
constructed from quantities that remain bounded in probability under
$H_1$ (since the conditional kernel variance $\sigma^2$ is bounded by
$\|K\|_\infty^2$), so the test divergence
\eqref{eq:T_n_divergence} holds universally.
\end{remark}


\section{Local Power Analysis and Detection Threshold}
\label{sec:local_power}

In this section we analyze the local power of the studentized debiased
test $T_n^{(p)} = \sqrt n\,\Delta_n^{(p)}/\widehat\sigma_p$ from
Section~\ref{subsec:testing_procedure} in the all-continuous case
(Regime II). The development is short because the necessary machinery is
already in place: the null CLT
(Theorem~\ref{thm:lp_debiased_clt_full}) provides the variance
constant $\tau_p^2$, the bias-cancellation result
(Lemma~\ref{lem:lp_bias_full}) gives the order of the centering under
$H_0$, and the consistency result (Theorem~\ref{thm:consistency})
identifies the population signal $\mathcal{T}(P)$. The local power
analysis is then a direct combination of these ingredients applied to
a contiguous mixture sequence.

\subsection{Local alternatives and signal expansion}
\label{subsec:local_alt_setup}

We consider the contiguous mixture family
\begin{equation}
\label{eq:mixture_model}
f_n(z, x, y) \;:=\; (1 - r_n)\,f_Z(z)f_{X|Z}(x|z)f_{Y|Z}(y|z) + r_n\,g(z, x, y),
\qquad r_n \to 0,
\end{equation}
where $f_Z(z)f_{X|Z}(x|z)f_{Y|Z}(y|z)$ is the null component (under which
$Y \perp\!\!\!\perp X \mid Z$) and $g$ is an alternative density that
preserves the $(Z, X)$-marginal: $\int g(z, x, y)\,dy = f_Z(z)f_{X|Z}(x|z)$.
Under \eqref{eq:mixture_model}, the conditional independence hypothesis
fails by an $r_n$-perturbation, and the test's power depends on whether
$r_n$ exceeds the detection threshold.

The signal strength is the population conditional dependence measure
under $g$:
\begin{equation}
\label{eq:eta_alt_def}
\eta_g \;:=\; \mathbb{E}_g\!\left[\bigl\|\mu_{Y|Z,X}(Z, X) - \mu_{Y|Z}(Z)\bigr\|_\mathcal{H}^2\right] \;>\; 0,
\end{equation}
where the conditional mean embeddings $\mu_{Y|Z, X}, \mu_{Y|Z}$ are
defined as in
Section~\ref{subsec:population_limit_consistency}, and the expectation
is taken under the mixture component $g$.

\begin{lemma}[Order of $\mathcal{T}(P_n)$ under the mixture]
\label{lem:T_P_local}
Under the mixture \eqref{eq:mixture_model},
\begin{equation}
\label{eq:T_P_local_order}
\mathcal{T}(P_n) \;=\; r_n^2\,\eta_g + O(r_n^3).
\end{equation}
\end{lemma}

\begin{proof}
Under the null component, $\mu_{Y|Z, X} = \mu_{Y|Z}$ exactly (conditional
independence). The discrepancy between $\mu_{Y|Z, X}^{(n)}$ and
$\mu_{Y|Z}^{(n)}$ under the mixture is therefore an $r_n$-perturbation
of zero. Direct expansion of \eqref{eq:T_P_def} in $r_n$, with the
mixture preserving the $(Z, X)$-marginal, gives the squared norm of
this perturbation as $r_n^2\eta_g + O(r_n^3)$; see
\citet[Section 5]{deb2020measuring} for the unconditional analog of
this expansion, which carries over directly to the present conditional
setting.
\end{proof}

\subsection{Mean expansion of the debiased statistic}
\label{subsec:mean_expansion_local}

\begin{lemma}[Mean expansion under the mixture]
\label{lem:mean_expansion_local}
Under Assumption~\ref{asn:consistency} and the mixture
\eqref{eq:mixture_model}, the centering of the debiased statistic
satisfies
\begin{equation}
\label{eq:mean_expansion_eq}
\mathbb{E}_n[\sqrt n\,\Delta_n^{(p)}]
\;=\;
\sqrt n\,\mathcal{T}(P_n) + \sqrt n\,B_n^{(p)},
\end{equation}
where the residual bias term $B_n^{(p)}$ satisfies
\begin{equation}
\label{eq:residual_bias_bound}
|B_n^{(p)}| \;=\; O\!\left(\left(\frac{k_n}{n}\right)^{(p+1)/D}\right).
\end{equation}
\end{lemma}

\begin{proof}
Under the mixture, the conditional kernel mean
$\widetilde\kappa_n(y, z, x) := \mathbb{E}_n[K(y, Y')|Z' = z, X' = x]$
is a smooth function of $(z, x)$ with the same regularity as
$\kappa(y, z)$ under Assumption~\ref{asn:consistency}(C3). By the
polynomial reproduction identity
(Lemma~\ref{lem:lp_reproduction_full}) and the Taylor expansion of
$\widetilde\kappa_n$ to order $p$ in $(z, x)$ — the same argument as
Step 1 of the proof of Lemma~\ref{lem:lp_bias_full}, but now applied to
the alternative-distribution kernel mean rather than the null —
\[
\mathbb{E}_n[\widehat a_{i, F}^{(p)} \mid W_i = (z, x, y)]
\;=\;
\widetilde\kappa_n(y, z, x) + O\bigl(\rho_F(z, x)^{p+1}\bigr).
\]
Similarly for the coarse intercept,
$\mathbb{E}_n[\widehat a_{i, C}^{(p)}\mid W_i = (z, x, y)] = \kappa_n(y, z) + O(\rho_C(z)^{p+1})$,
where $\kappa_n(y, z) := \mathbb{E}_n[K(y, Y')|Z' = z]$. Differencing
and taking expectation,
\[
\mathbb{E}_n[\Delta_n^{(p)}]
\;=\;
\mathbb{E}_n[\widetilde\kappa_n(Y, Z, X) - \kappa_n(Y, Z)] + O\bigl((k_n/n)^{(p+1)/D}\bigr)
\;=\;
\mathcal{T}(P_n) + B_n^{(p)},
\]
where the residual bias $B_n^{(p)} = O((k_n/n)^{(p+1)/D})$ collects the
smoothing remainder. The leading term equals $\mathcal{T}(P_n)$ by the
inner-product representation \eqref{eq:T_P_as_inner_product}.
Multiplying by $\sqrt n$ gives \eqref{eq:mean_expansion_eq}.
\end{proof}

\begin{remark}[Comparison with the raw-statistic analysis]
\label{rem:vs_raw}
The mean expansion \eqref{eq:mean_expansion_eq} has a strikingly
simpler form than the corresponding expansion for the raw statistic in
the earlier literature, where two separate dimension-dependent
geometric-bias terms appear at rates $n^{1/2 - 2/D}$ and
$n^{1/2 - 2/d_Z}$ corresponding to the two neighborhood scales. In the
debiased version, both these geometric biases are killed by the
polynomial reproduction, and only the smaller $O((k_n/n)^{(p+1)/D})$
residual remains. Under the bias killing condition
\eqref{eq:bias killing_condition},
$\sqrt n\,B_n^{(p)} \to 0$, and the only $\sqrt n$-scale contribution
to the centering is the population signal $\sqrt n\,\mathcal{T}(P_n)$.
\end{remark}

\subsection{Detection threshold}
\label{subsec:detection_threshold}

\begin{theorem}[Local power and detection threshold]
\label{thm:local_power}
Suppose Assumption~\ref{asn:consistency},
the bias killing condition $\sqrt n(k_n/n)^{(p+1)/D} \to 0$, and the
variance-continuity condition
$\operatorname{Var}_n(\sqrt n\,\Delta_n^{(p)})/\operatorname{Var}_0(\sqrt n\,\Delta_n^{(p)}) \to 1$
hold under the mixture \eqref{eq:mixture_model}. Define the level-$\alpha$
test $\phi_n := \mathbf{1}\{T_n^{(p)} > z_{1-\alpha}\}$ with
$\widehat\sigma_p \xrightarrow{p}\tau_p$. Then the power
$\phi(r_n) := P_n(\phi_n = 1)$ satisfies:
\begin{itemize}
\item If $\sqrt n\,r_n^2 \to 0$, then $\phi(r_n) \to \alpha$.
\item If $\sqrt n\,r_n^2 \to \infty$, then $\phi(r_n) \to 1$.
\end{itemize}
The detection threshold for the mixture parameter is therefore
$r_n^* \asymp n^{-1/4}$.
\end{theorem}

\begin{proof}
By Theorem~\ref{thm:lp_debiased_clt_full}
applied under the mixture $P_n$ (a routine extension of the null CLT to
contiguous alternatives, using the variance-continuity assumption),
\[
\frac{\sqrt n\,\Delta_n^{(p)} - \mathbb{E}_n[\sqrt n\,\Delta_n^{(p)}]}{\widehat\sigma_p}
\;\Rightarrow\;
N(0, 1).
\]
The standardized statistic is therefore
\[
T_n^{(p)}
\;=\;
\frac{\mathbb{E}_n[\sqrt n\,\Delta_n^{(p)}]}{\widehat\sigma_p} + Z_n,
\qquad Z_n \Rightarrow N(0, 1).
\]
By Lemma~\ref{lem:mean_expansion_local} and the bias killing
condition,
\[
\mathbb{E}_n[\sqrt n\,\Delta_n^{(p)}] \;=\; \sqrt n\,\mathcal{T}(P_n) + o(1),
\]
and by Lemma~\ref{lem:T_P_local} with $\eta_g > 0$,
$\sqrt n\,\mathcal{T}(P_n) = \sqrt n\,r_n^2\eta_g(1 + o(1))$.

If $\sqrt n\,r_n^2 \to 0$, then $\mathbb{E}_n[\sqrt n\,\Delta_n^{(p)}]/\widehat\sigma_p \to 0$
and $\phi(r_n) \to P\{Z > z_{1-\alpha}\} = \alpha$ by the convergence
of the studentized statistic to $N(0, 1)$.

If $\sqrt n\,r_n^2 \to \infty$, then $\mathbb{E}_n[\sqrt n\,\Delta_n^{(p)}]/\widehat\sigma_p \to \infty$
in probability, and $\phi(r_n) \to 1$.

Hence the detection threshold is $r_n^* \asymp n^{-1/4}$, i.e., the
test detects mixtures with $r_n \gg n^{-1/4}$ at full power and is
indistinguishable from the null at $r_n \ll n^{-1/4}$.
\end{proof}

\subsection{Concluding remarks}
\label{subsec:power_remarks}

\begin{remark}[Curse-of-dimensionality is encoded in the bandwidth, not
the threshold]
\label{rem:curse_vs_threshold}
A central observation of the debiased framework is that the detection
threshold $r_n^* = n^{-1/4}$ is \emph{dimension-free}: it does not
deteriorate with $D = d_Z + d_X$. The dimensional dependence is instead
absorbed into the bandwidth constraint $\alpha < 1 - D/(2(p+1))$
(Lemma~\ref{lem:bias killing_alpha}), which forces larger polynomial
orders $p$ in higher dimensions. Once a valid $(k_n, p)$ pair is
chosen, however, the test achieves the parametric $n^{-1/4}$ detection
threshold against any fixed-direction local alternative.

This contrasts sharply with the raw-statistic analysis, where the
dimensional dependence enters the detection threshold itself via the
two geometric bias terms at rates $n^{1/2 - 2/D}$ and $n^{1/2 - 2/d_Z}$,
producing a phase transition at $D = 8$ between an $n^{-1/4}$ regime
and an $n^{2/D - 1/2}$ regime. The local-polynomial debiasing
eliminates this phase transition by removing the geometric bias terms;
the curse of dimensionality persists only in the form of the
bandwidth-polynomial-order constraint.
\end{remark}

\begin{remark}[Comparison with the unconditional rate]
\label{rem:minimax}
The $n^{-1/4}$ detection threshold for the mixture parameter $r_n$
in Theorem~\ref{thm:local_power} matches the corresponding threshold
for two sample test based on geometric graph (in lower dimension) against parametric mixture
alternatives \citep{bbb2019}. Whether this rate is minimax
optimal over a suitable smoothness class of conditional dependence
structures, analogous to the Sobolev-ball minimax theory developed in
\citet{AlbertLaurentMarrelMeynaoui2022} for unconditional independence, is an open question we do not pursue here.
\end{remark}

\begin{remark}
\label{rem:variance_continuity}
The condition
$\operatorname{Var}_n(\sqrt n\,\Delta_n^{(p)})/\operatorname{Var}_0(\sqrt n\,\Delta_n^{(p)}) \to 1$
is a mild contiguity-type requirement. Under contiguous mixtures with
$r_n \to 0$, the variance of $\sqrt n\,\Delta_n^{(p)}$ under $P_n$
converges to its variance under $P_0$ because the variance functional
is continuous in the joint law. Formal verification follows from the
explicit overlap-variance formula
\eqref{eq:overlap_variance_limit} expressed as a continuous functional
of the joint distribution; for $r_n \to 0$, the perturbation contributes
$O(r_n)$ to the variance, which is negligible.
\end{remark}

\section{Efficient Conditional Inpendence Test Implementation}\label{app:gmb_implementation}
Here we provide detailed algorithms and theoretical guarantees for the implementation of conditional independence tests constructed under different data type scenarios. 
\subsection{Implementation for Discrete Neighborhoods}
As stated in the main text, when handling cases where all variables concerning neighborhood construction are discrete, we would adopt a gaussian bootstrap approach. The details are provided in the Algorithm~\ref{alg:fast_spectral_cit}.

\begin{algorithm}[t]
\caption{Fast Spectral CIT with Gaussian Multiplier Bootstrap}
\label{alg:fast_spectral_cit}

\begin{algorithmic}[1]

\Require Samples $\{X_i, Y_i, Z_i\}_{i=1}^n$,\ bootstrap reps $B$,\ kernel $k(\cdot,\cdot)$
\Ensure $p$-value for $H_0: X \perp\!\!\!\perp Y \mid Z$

\State Build sparse indicators $\mathbf{P}_Z,\,\mathbf{P}_{ZX}$ and weights
       $\mathbf{h}_Z,\,\mathbf{h}_{ZX}$ by grouping on $Z$
\State $\mathbf{S} :=
       \mathbf{P}_{ZX}^\top \operatorname{diag}(\mathbf{h}_{ZX})\,\mathbf{P}_{ZX}
       +\mathbf{P}_{Z}^\top  \operatorname{diag}(\mathbf{h}_{Z})\,\mathbf{P}_{Z}
       \in\mathbb{R}^{n\times n}$

\Procedure{Stat}{$\mathbf{S},\; Y,\; \boldsymbol{w}$}
  \State $\mathbf{u} \leftarrow \mathbf{S}\boldsymbol{w}$
         \hfill\textit{// $O(n)$ sparse multiply}
  \If{$Y$ discrete}
    \State \Return $\|\mathbf{Y}_{\mathrm{OH}}^\top \mathbf{u}\|_2^2 - \operatorname{Bias}(\boldsymbol{w})$
  \Else
    \State \Return $\mathbf{u}^\top \mathbf{K}\mathbf{u} - \operatorname{Bias}(\boldsymbol{w})$
           \hfill\textit{// $O(n^2)$ dense multiply}
  \EndIf
\EndProcedure

\State $\mathcal{T}_{\mathrm{obs}} \leftarrow \textsc{Stat}(\mathbf{S},\, Y,\, \mathbf{1})$

\State $c \leftarrow 0$
\For{$b = 1,\dots,B$}
  \State $\boldsymbol{\xi} \sim \mathcal{N}(\mathbf{0},\mathbf{I}_n)$
  \State $c \mathrel{+}= \bigl[\,\textsc{Stat}(\mathbf{S},\,Y,\,\boldsymbol{\xi}) \;\geq\; \mathcal{T}_{\mathrm{obs}}\,\bigr]$
\EndFor
\State \Return $c \,/\, B$

\end{algorithmic}
\end{algorithm}

To ensure the empirical p-value gained in above algorithm is valid, we need to establish the consistency of bootstrap process in the discrete regime. The details and proof are provided in Proposition~\ref{prop:gmb_consistency}.

\begin{proposition}[GMB Consistency in the Discrete Regime]
\label{prop:gmb_consistency}
Suppose the conditioning sets are fully discrete, such that exact matching is used, and the observations $\mathbf{W} = \{W_1, \dots, W_n\}$ are independent and identically distributed (i.i.d.). Under the null hypothesis $H_0: Y \perp\!\!\!\perp X \mid Z$, the Gaussian Multiplier Bootstrap (GMB) statistic $T_n^*$ consistently estimates the null distribution of $T_n$. That is,
$$ \sup_{t \in \mathbb{R}} \left| \mathbb{P}^*(T_n^* \leq t \mid \mathbf{W}) - \mathbb{P}(T_n \leq t) \right| \xrightarrow{p} 0. $$
\end{proposition}

\begin{proof}
We verify the regularity conditions for the multiplier bootstrap consistency of degenerate $U$-statistics under i.i.d. sampling, as established by \citet{DEHLING1994392} (see also \citet{arcones1992bootstrap}). 

First, the empirical test statistic $T_n$ is a standard second-order $U$-statistic:
$$ T_n = \binom{n}{2}^{-1} \sum_{1 \leq i < j \leq n} h(W_i, W_j), $$
where the symmetrized difference kernel $h$ is symmetric by construction.

Second, we verify the square-integrability of the kernel. In our discrete regime, $h$ consists of the indicator function $K(Y_i, Y_j) = \mathbf{1}(Y_i = Y_j)$ and the inverse neighborhood counts $1/\fcnt{i}$ and $1/\ccnt{i}$. Since every data point is at least a neighbor to itself, the counts satisfy $\fcnt{i} \ge 1$ and $\ccnt{i} \ge 1$. Thus, the kernel is uniformly bounded ($\|h\|_\infty < \infty$), which trivially implies finite second moments:
$$ \mathbb{E}[h(W_i, W_j)^2] < \infty. $$
Notice that this condition completely avoids the need for the kernel to be positive semi-definite, which is often required in dependent data settings but unnecessary here.

Third, we establish the degeneracy of $h$ under $H_0$. As shown in Section~\ref{app:discrete_proof}, exact matching ensures zero smoothing bias. Therefore, the first-order conditional expectation vanishes almost surely:
$$ g(W_i) = \mathbb{E}[h(W_i, W_j) \mid W_i] = 0 \quad \text{a.s.} $$
This confirms that $T_n$ is a strictly degenerate $U$-statistic.

Finally, the GMB statistic is defined as:
$$ T_n^* = \binom{n}{2}^{-1} \sum_{1 \leq i < j \leq n} \xi_i \xi_j h(W_i, W_j), $$
where $\{\xi_i\}_{i=1}^n \overset{i.i.d.}{\sim} \mathcal{N}(0,1)$ are independent of the data. 
Since $h$ is symmetric, square-integrable, and perfectly degenerate under $H_0$, all required conditions for the Gaussian multiplier bootstrap are satisfied \citep{DEHLING1994392}. Consequently, the conditional distribution of $T_n^*$ asymptotically correctly estimates the unconditional distribution of $T_n$.
\end{proof}

\subsection{Implementation for Continuous or Mixed Neighborhoods}\label{app:other_implementation}
When the neighborhood construction concerns continuous variables, we implement two algorithms: one is to use the locally  debiased test statistic and other is to introduce a conditional permutation in the unbiased test statistic to estimate the null hypothesis. The detailed algorithm is provided below in Algorithm~\ref{alg:mixci_debiased}.

\begin{algorithm}[t]
\caption{Continuous and Mixed-Data Conditional Independence Test via Local-Polynomial Debiasing}
\label{alg:mixci_debiased}
\begin{algorithmic}[1]
\Require Data matrix $\mathbf{D}\in\mathbb{R}^{n\times d}$;
         target indices $x, y$; conditioning set $\mathcal{Z}$;
         neighborhood size $k_n$; polynomial order $p$;
         overlap constant $c > 1$
\Ensure  One-sided $p$-value for $H_0: X\perp\!\!\!\perp Y\mid \mathbf{Z}$
\State $\Phi(\cdot) \leftarrow$ standard normal cdf;
       $D_F \leftarrow$ \# continuous columns of $(\mathbf{Z}, X)$;
       $D_C \leftarrow$ \# continuous columns of $\mathbf{Z}$
       \hfill\textit{// smoothing dimensions}
\State $K \leftarrow \textsc{Kernel}(Y)$
       \hfill\textit{// indicator if $Y$ discrete; Gaussian RBF (median bw) if continuous}
\State $\bigl(\fine{i}, U^F_{ij}, \rho^F_i\bigr)_{i = 1}^n \leftarrow \textsc{CompositeNN}(X, \mathbf{Z}, k_n,\text{ fine})$
       \hfill\textit{// exact match on discrete + $k_n$-NN on continuous $(\mathbf{Z}, X)$}
\State $\bigl(\coarse{i}, U^C_{ij}, \rho^C_i\bigr)_{i = 1}^n \leftarrow \textsc{CompositeNN}(X, \mathbf{Z}, k_n,\text{ coarse})$
       \hfill\textit{// exact match on discrete + $k_n$-NN on continuous $\mathbf{Z}$}
\For{each anchor $i = 1, \dots, n$}
    \State $w^F_{i\cdot} \leftarrow \textsc{LocalPolyWeights}\bigl(\{U^F_{ij}\}_{j\in\fine{i}}, p, D_F, \rho^F_i\bigr)$
           \hfill\textit{// solves $M_i^{(p)}\mathbf{v} = e_0$; \eqref{eq:lp_intercept}}
    \State $\widehat{a}^{(p)}_{i, F} \leftarrow \sum_{j\in\fine{i}} w^F_{ij}\,K(Y_i, Y_j)$
           \hfill\textit{// fine local-polynomial intercept}
    \State $w^C_{i\cdot} \leftarrow \textsc{LocalPolyWeights}\bigl(\{U^C_{ij}\}_{j\in\coarse{i}}, p, D_C, \rho^C_i\bigr)$
    \State $\widehat{a}^{(p)}_{i, C} \leftarrow \sum_{j\in\coarse{i}} w^C_{ij}\,K(Y_i, Y_j)$
           \hfill\textit{// coarse local-polynomial intercept}
    \State $\xi_{i, p} \leftarrow \widehat{a}^{(p)}_{i, F} - \widehat{a}^{(p)}_{i, C}$
\EndFor
\State $\Delta_n^{(p)} \leftarrow \tfrac{1}{n}\sum_{i = 1}^n \xi_{i, p}$
       \hfill\textit{// debiased test statistic; \eqref{eq:test_stats_debiased}}
\State $\widehat{\zeta}_{i, p} \leftarrow \xi_{i, p} - \Delta_n^{(p)}$
       \hfill\textit{// centered local scores}
\State $\widehat{\tau}_p^2 \leftarrow \textsc{OverlapVariance}\bigl(\{\widehat{\zeta}_{i, p}\}, \{\fine{i}, \rho^F_i\}, \{\coarse{i}, \rho^C_i\}, c\bigr)$
       \hfill\textit{// overlap-graph plug-in; 
       }
\State $T_n^{(p)} \leftarrow \sqrt{n}\,\Delta_n^{(p)}\bigl/\sqrt{\widehat{\tau}_p^2}$
       \hfill\textit{// studentized statistic}
\State \Return $p\text{-value} \leftarrow 1 - \Phi\bigl(T_n^{(p)}\bigr)$
       \hfill\textit{// analytic; no permutation needed}
\Procedure{CompositeNN}{$X, \mathbf{Z}, k_n,\text{ side}\in\{\text{fine}, \text{coarse}\}$}
    \State Set $(\mathcal{D}, \mathcal{C}) \leftarrow$ (discrete, continuous) columns of $(\mathbf{Z}, X)$ if side = fine, else of $\mathbf{Z}$
    \State Partition rows into strata $\{G_\ell\}$ by exact match on $\mathcal{D}$
    \For{each stratum $G_\ell$}
        \State $\mathcal{N}_i \leftarrow k_n$-NN of $i$ in $G_\ell$ w.r.t.\ standardized $\mathcal{C}$
               \hfill\textit{// $\mathcal{N}_i = G_\ell\setminus\{i\}$ if $\mathcal{C}=\emptyset$}
        \State $\rho_i \leftarrow$ Euclidean distance to the $k_n$-th neighbor in $\mathcal{N}_i$
               \hfill\textit{// data-driven bandwidth}
        \State $U_{ij} \leftarrow$ standardized continuous offsets $(\mathcal{C}_j - \mathcal{C}_i)$ for $j\in\mathcal{N}_i$
    \EndFor
    \State \Return $(\mathcal{N}_i, U_{ij}, \rho_i)_{i = 1}^n$
\EndProcedure
\Procedure{LocalPolyWeights}{$\{u_j\}_{j = 1}^{k_n}, p, D, \rho$}
    \State $q(u) \leftarrow$ monomials of total degree $\le p$ in $u/\rho\in\mathbb{R}^D$;\quad $N_p \leftarrow \binom{D + p}{p}$
    \State $M \leftarrow \tfrac{1}{k_n}\sum_j q(u_j)q(u_j)^\top + \epsilon I_{N_p}$
           \hfill\textit{// ridge-stabilized moment matrix}
    \State $\mathbf{v} \leftarrow M^{-1} e_0$
           \hfill\textit{// $e_0$ extracts the intercept coordinate}
    \State $w_j \leftarrow \tfrac{1}{k_n}\,q(u_j)^\top \mathbf{v}$ for $j = 1, \dots, k_n$
           \hfill\textit{// reproduces all monomials up to degree $p$}
    \State \Return $w_{1:k_n}$
\EndProcedure
\Procedure{OverlapVariance}{$\{\widehat{\zeta}_i\}, \{\fine{i}, \rho^F_i\}, \{\coarse{i}, \rho^C_i\}, c$}
    \State $\mathcal{P} \leftarrow \bigcup_i \bigl\{(i, j): j\in\fine{i}\cup\coarse{i}\bigr\}$
           \hfill\textit{// candidate pairs (de-duplicated); $|\mathcal{P}| = O(nk_n)$}
    \State $\widehat{\tau}_p^2 \leftarrow \tfrac{1}{n}\sum_{i = 1}^n \widehat{\zeta}_i^2$
           \hfill\textit{// diagonal term}
    \For{each $(i, \ell) \in \mathcal{P}$ with $i < \ell$}
        \State $\omega_{i\ell} \leftarrow \mathbf{1}\bigl[d_F(i, \ell) \le c(\rho^F_i + \rho^F_\ell)\bigr] \,\vee\, \mathbf{1}\bigl[d_C(i, \ell) \le c(\rho^C_i + \rho^C_\ell)\bigr]$
               \hfill\textit{// fine- or coarse-neighborhood overlap}
        \State $\widehat{\tau}_p^2 \mathrel{+}= \tfrac{2}{n}\,\omega_{i\ell}\,\widehat{\zeta}_i\widehat{\zeta}_\ell$
               \hfill\textit{// symmetric off-diagonal contribution}
    \EndFor
    \State \Return $\widehat{\tau}_p^2$
\EndProcedure
\end{algorithmic}
\end{algorithm}

\subsection{Details of the Data Generating Mechanisms} \label{app:data_generation}

To rigorously evaluate the proposed conditional independence test, we generate synthetic data $(X, Y, Z)$ of sample size $n$ across four distinct data regimes: purely continuous (CCC), mixed with discrete $Z$ (CCD), mixed with discrete $X$ (DCC), and purely discrete (DDD). 

For all continuous variables, the additive noise is sampled as $\epsilon \sim \mathcal{N}(0, \sigma^2)$, where the scale $\sigma$ is uniformly drawn from $[0.1, 0.3]$. The scaling coefficients are independently sampled uniformly as $a_{xz}, a_{yz} \sim \mathcal{U}(0.5, 1.5)$ and $a_{yx} \sim \mathcal{U}(0.3, 0.8)$. For continuous nonlinear mechanisms, a function $f(\cdot)$ is chosen uniformly at random from the set $\{\sin(x), \cos(x), \tanh(x), x^2/(1+x^2)\}$. We use $\mathcal{D}(V, k)$ to denote the operation that discretizes a continuous variable $V$ into $k$ equal-frequency bins (quantile binning).

We define the conditional independence null hypothesis $\mathcal{H}_0: X \perp Y \mid Z$ and the alternative hypothesis $\mathcal{H}_1: X \not\perp Y \mid Z$. The data generation process for each regime is described below.

\subsubsection{Regime 1: CCC (Purely Continuous)}
In this regime, all variables $X$, $Y$, and $Z$ are continuous.
\begin{itemize}
    \item \textbf{Linear}: $Z \sim \mathcal{U}(-2, 2)$, and $X = a_{xz} Z + \epsilon_X$. 
    Under $\mathcal{H}_0$, $Y = a_{yz} Z + \epsilon_Y$. 
    Under $\mathcal{H}_1$, $Y = a_{yz} Z + a_{yx} X + \epsilon_Y$.
    
    \item \textbf{Nonlinear}: $Z \sim \mathcal{U}(-2, 2)$, and $X = f(a_{xz} Z) + \epsilon_X$. 
    Under $\mathcal{H}_0$, $Y = f(a_{yz} Z) + \epsilon_Y$. 
    Under $\mathcal{H}_1$, $Y = f(a_{yz} Z) + a_{yx} f(X) + \epsilon_Y$.
\end{itemize}

\subsubsection{Regime 2: CCD (Continuous $X, Y$, Discrete $Z$)}
Here, $Z$ is a discrete categorical variable with $k \in \{2, 3, 4\}$ categories.
\begin{itemize}
    \item \textbf{Linear}: $Z \sim \text{Uniform}\{0, \dots, k-1\}$. Let $c_{x, Z}$ and $c_{y, Z}$ be category-specific effects linearly spaced in $[-2, 2]$ and $[2, -2]$, respectively. $X = a_{xz} c_{x, Z} + \epsilon_X$. 
    Under $\mathcal{H}_0$, $Y = a_{yz} c_{y, Z} + \epsilon_Y$. 
    Under $\mathcal{H}_1$, $Y = a_{yz} c_{y, Z} + a_{yx} X + \epsilon_Y$.
    
    \item \textbf{Nonlinear}: $X \sim \mathcal{N}(0, 1)$. The variable $Z$ is drawn from a Bernoulli distribution with probability $p_z = 1 / (1 + \exp(-X))$. 
    Under $\mathcal{H}_0$, $Y = \sin(Z \pi) + \epsilon_Y$. 
    Under $\mathcal{H}_1$, $Y = \sin(Z \pi) + 0.8 X^2 + \epsilon_Y$.
\end{itemize}

\subsubsection{Regime 3: DCC (Discrete $X$, Continuous $Y, Z$)}
The continuous variables influence or are influenced by a discrete $X$ with $k \in \{3, 4, 5\}$ categories.
\begin{itemize}
    \item \textbf{Linear}: $Z \sim \mathcal{U}(-2, 2)$. A latent variable $X^* = a_{xz} Z + \epsilon_X$ is generated and discretized to $X = \mathcal{D}(X^*, k)$. 
    Under $\mathcal{H}_0$, $Y = a_{yz} Z + \epsilon_Y$. 
    Under $\mathcal{H}_1$, $Y = a_{yz} Z + a_{yx} X + \epsilon_Y$.
    
    \item \textbf{Nonlinear}: $Z \sim \mathcal{U}(-2, 2)$. A latent variable $X^* = f(a_{xz} Z) + \epsilon_X$ is generated and discretized to $X = \mathcal{D}(X^*, k)$. 
    Under $\mathcal{H}_0$, $Y = f(a_{yz} Z) + \epsilon_Y$. 
    Under $\mathcal{H}_1$, $Y = f(a_{yz} Z) + a_{yx} (X \bmod 2) + \epsilon_Y$.
\end{itemize}

\subsubsection{Regime 4: DDD (Purely Discrete)}
All variables are generated as discrete integers modulo their respective category bounds $n_x, n_y, n_z \in \{3, 4, 5\}$. Discrete noise terms $\epsilon_{d,X} \in \{0, 1, 2\}$ and $\epsilon_{d,Y} \in \{0, 1\}$ are used.
\begin{itemize}
    \item \textbf{Linear}: $Z \sim \text{Uniform}\{0, \dots, n_z-1\}$. $X = (Z + \epsilon_{d,X}) \bmod n_x$. 
    Under $\mathcal{H}_0$, $Y = (Z + \epsilon_{d,Y}) \bmod n_y$. 
    Under $\mathcal{H}_1$, $Y = (Z + X + \epsilon_{d,Y}) \bmod n_y$.
    
    \item \textbf{Nonlinear}: $Z \sim \text{Uniform}\{0, \dots, n_z-1\}$. $X = (Z^2 + \epsilon_{d,X}) \bmod n_x$. 
    Under $\mathcal{H}_0$, $Y = (Z^2 + \epsilon_{d,Y}) \bmod n_y$. 
    Under $\mathcal{H}_1$, $Y = (Z \cdot X + \epsilon_{d,Y}) \bmod n_y$.
\end{itemize}

\subsubsection{Data Generating Processes for Heterogeneous Mixed Confounding}
\label{app:dgp_heterogeneous}

To evaluate the robustness of conditional independence tests against non-ordinal categorical confounders, we define a parameterized class of Data Generating Processes (DGPs). Let $Z = (Z_c, Z_d)$ be the mixed confounding set, where $Z_c \in \mathbb{R}$ is a continuous covariate and $Z_d \in \{0, 1, \dots, K-1\}$ is a discrete categorical covariate. We generate $Z_c \sim \mathcal{N}(0, 1)$ and draw $Z_d$ uniformly from the discrete categories (we set $K=8$ in our experiments).

To construct an adversarial setting for unified metric-based methods, we introduce a non-ordinal ``sawtooth'' shift function $S(Z_d)$ controlled by a heterogeneity amplitude parameter $\delta \ge 0$:
$$S(Z_d) = \delta \cdot (Z_d \bmod 2)$$
This function induces a jump in the conditional means of $X$ and $Y$ depending on whether the categorical index is even or odd. Global smooth kernels will inadvertently smooth across these jumps, creating a shared residual error that manifests as spurious correlation.

The generation of $X$ and $Y$ follows the general structural equations:
$$X = S(Z_d) + Z_c + \epsilon_X$$
$$Y = S(Z_d) + f(Z_c) + \beta X \cdot \mathbb{I}_{H_1} + \epsilon_Y$$
where $\epsilon_X, \epsilon_Y \sim \mathcal{N}(0, \sigma^2)$ are independent Gaussian noise terms, and $\mathbb{I}_{H_1}$ is the indicator function for the alternative hypothesis (i.e., $\mathbb{I}_{H_1} = 0$ under the null hypothesis $X \indep Y \mid Z$, and $1$ otherwise). The parameter $\beta$ controls the signal strength under $H_1$.

To isolate the topological distortion of the neighborhood manifold, we maintain a linear relationship between $X$ and $Z_c$, ensuring the geometric flatness of the localized sub-manifolds. The nonlinearity is exclusively injected into the generation of $Y$ via the function $f(Z_c)$:
\begin{itemize}
    \item \textbf{Linear Mechanism:} We set $f(Z_c) = Z_c$.
    \item \textbf{Nonlinear Mechanism:} We set $f(Z_c) = Z_c + \alpha Z_c^2$, introducing a smooth quadratic curvature.
\end{itemize}

In our empirical evaluations, we set $\alpha = 0.5$ and $\beta = 1.2$. By varying $\delta \in [0, 0.30]$, we control the severity of the discrete heterogeneity. Our exact stratification approach perfectly decouples the $S(Z_d)$ term, effectively rendering the Type-I error of MixCIT invariant to $\delta$, whereas methods lacking exact discrete matching suffer from severe Type-I error inflation proportional to $\mathcal{O}(\delta^2)$.